\documentclass[11pt]{amsart}

\usepackage[colorlinks=true, pdfstartview=FitV, linkcolor=blue, 
citecolor=blue]{hyperref}

\usepackage{a4wide}
\usepackage{amsmath,amssymb} 
\usepackage{bbm}
\usepackage{graphicx}

\numberwithin{equation}{section}

\theoremstyle{plain}
\newtheorem{theorem}{Theorem}[section]
\newtheorem{proposition}[theorem]{Proposition}
\newtheorem{lemma}[theorem]{Lemma}
\newtheorem{corollary}[theorem]{Corollary}
\newtheorem{fact}[theorem]{Fact}

\theoremstyle{definition}
\newtheorem{definition}[theorem]{Definition}
\newtheorem{remark}[theorem]{Remark}

\newtheorem{example}[theorem]{Example}

\DeclareMathOperator{\SO}{SO}
\DeclareMathOperator{\Orth}{O}
\DeclareMathOperator{\lcm}{lcm}
\DeclareMathOperator{\OG}{O}
\DeclareMathOperator{\OS}{OS}
\DeclareMathOperator{\OC}{OC}
\DeclareMathOperator{\SOS}{SOS}
\DeclareMathOperator{\SOC}{SOC}
\DeclareMathOperator{\scal}{scal}
\DeclareMathOperator{\Scal}{Scal}
\DeclareMathOperator{\norm}{norm}
\DeclareMathOperator{\den}{den}
\DeclareMathOperator{\mul}{MR}
\DeclareMathOperator{\Sim}{sim }
\DeclareMathOperator{\gcld}{gcld}
\DeclareMathOperator{\gcrd}{gcrd}
\DeclareMathOperator{\lcrm}{lcrm}
\DeclareMathOperator{\lclm}{lclm}
\DeclareMathOperator{\cont}{cont}
\DeclareMathOperator{\Res}{res}

\DeclareMathOperator{\nr}{nr}
\DeclareMathOperator{\Real}{Re}
\DeclareMathOperator{\Imag}{Im}
\DeclareMathOperator{\card}{card}

\newcommand{\ii}{\ts\mathrm{i}\ts}
\newcommand{\jj}{\ts\mathrm{j}\ts}
\newcommand{\kk}{\ts\mathrm{k}\ts}
\newcommand{\e}{\ensuremath{\mathrm{e}}}

\newcommand{\ee}{\,\mathrm{e}}
\newcommand{\ts}{\hspace{0.5pt}}
\newcommand{\nts}{\hspace{-0.5pt}}

\newcommand{\similar}{\stackrel{\mathsf{s}}{\sim}}
\newcommand{\Sig}{\varSigma}
\newcommand{\minisquare}{\scriptscriptstyle\square}
\newcommand{\ip}[2]{\langle #1 \ts | \ts #2\rangle}
\newcommand{\bw}{\ts\ts\overline{\nts\nts w \nts\nts}\ts\ts}
\newcommand{\bom}{\ts\ts\overline{\nts\nts \omega \nts\nts}\ts\ts}

\newcommand{\ZZ}{\mathbb{Z}}
\newcommand{\QQ}{\mathbb{Q}}
\newcommand{\RR}{\mathbb{R}}
\newcommand{\NN}{\mathbb{N}}
\newcommand{\Nnull}{\mathbb{N}_{0}^{}}
\newcommand{\CC}{\mathbb{C}}

\newcommand{\HH}{\mathbb{H}}
\newcommand{\LL}{\ensuremath{\mathbb{L}}}
\newcommand{\II}{\mathbb{I}}
\newcommand{\JJ}{\ts\mathbb{J}}
\newcommand{\PP}{\mathbb{P}\ts}
\newcommand{\KK}{\mathbb{K}}

\newcommand{\cC}{\mathcal{C}}

\newcommand{\cO}{\mathcal{O}}

\newcommand{\scO}{{\scriptstyle\mathcal{O}}}
\newcommand{\cS}{\mathcal{S}}
\newcommand{\cT}{\ts\mathcal{T}}

\newcommand{\Gpc}{\vG^{}_{\!\mathsf{pc}}}
\newcommand{\Gbcc}{\vG^{}_{\!\mathsf{bcc}}}
\newcommand{\Gbccstar}{\vG^{*}_{\!\mathsf{bcc}}}
\newcommand{\Gfcc}{\vG^{}_{\!\mathsf{fcc}}}

\newcommand{\vG}{\varGamma}
\newcommand{\vL}{\varLambda}

\newcommand{\one}{\mathbbm{1}}

\newcommand{\exend}{\hfill $\Diamond$}

\newcommand{\myfrac}[2]{\frac{\raisebox{-2pt}{$#1$}}
      {\raisebox{0.5pt}{$#2$}}}
\newcommand{\badfrac}[2]{\frac{\raisebox{-2pt}{$#1$}}{\raisebox{2pt}{$#2$}}}

\begin{document}

\title[Geometric Enumeration Problems]
{Geometric Enumeration Problems for\\[2mm]
Lattices and Embedded $\ZZ\ts\ts$-Modules}

\author{Michael Baake}
\address{Fakult\"at f\"ur Mathematik, Universit\"at Bielefeld, \newline
\hspace*{\parindent}Postfach 100131, 33501 Bielefeld, Germany}
\email{$\{$mbaake,pzeiner$\}$@math.uni-bielefeld.de }

\author{Peter Zeiner}

\begin{abstract}
  In this review, we count and classify certain sublattices of a given
  lattice, as motivated by crystallography. We use methods from
  algebra and algebraic number theory to find and enumerate the
  sublattices according to their index. In addition, we use tools from
  analytic number theory to determine the asymptotic behaviour of the
  corresponding counting functions. Our main focus lies on similar
  sublattices and coincidence site lattices, the latter playing an
  important role in crystallography. As many results are algebraic in
  nature, we also generalise them to $\ZZ\ts$-modules embedded in
  $\RR^d$.
\end{abstract}

\maketitle

\section{Introduction}

Lattices in $\RR^3$ have been used for more than a century in
crystallography,\index{crystallography} as they describe the
translational symmetries of idealised, infinitely extended (periodic)
crystals. As such, they have been studied intensively, together with
space groups, which are finite extensions of lattices (viewed as
Abelian groups) and describe the full symmetry of the crystals;
compare the Epilogue to this volume. Group{\ts}-subgroup relations
have been applied to analyse various aspects such as phase transitions
in crystals.

A special case of the latter is the question of certain kinds of
sublattices of a given lattice. Ideal crystals do not exist in nature,
and the result of crystallisation is very often not a single crystal,
but a mixture of differently orientated crystals of the same kind. The
latter are called \emph{grains}\index{grain~boundary}, and an
important question in crystallography is their mutual orientation and
the border between two neighbouring grains, called a \emph{grain
  boundary}.

To study the latter, one assigns, to each of the two grains, its
corresponding lattice, say $\vG$ and $\vG'$, and computes their
intersection $\vG \cap \vG'$. If the two grains are of the same kind,
the two lattices are related by an orthogonal transformation $R$,
which means that we have $\vG'= R\vG$ for a suitable isometry $R\in
\OG(3,\RR)$. The corresponding sublattice $\vG \cap R\vG$ is called a
\emph{coincidence site lattice}\index{coincidence~site~lattice} (CSL).

It was Friedel in 1911 who first recognised the usefulness of CSLs in
describing and classifying grain boundaries of
crystals~\cite{csl-Friedel11}. Analogous ideas were later used by
Kronberg and Wilson~\cite{csl-KW49}. But it still took some time
before their ideas became popular. In fact, the widespread use of CSLs
was only triggered by a paper of Ranganathan~\cite{csl-Ranga66} in
1966. Many important papers were published in the following years. In
particular, we mention contributions by Grimmer~\cite{csl-Grimmer73,
  csl-Grimmer74, csl-Grimmer84, csl-GBW, csl-GW87, csl-Grimmer89} and
Bollmann~\cite{csl-Boll70,csl-Boll82}.

The discovery of quasicrystals sparked new interest in CSLs, and a
systematic mathematical study started. In particular, the concept of
CSLs was generalised to $\ZZ\ts$-modules embedded in $\RR^d$, which led
to the notion of coincidence site modules (CSMs). They are used to
describe grain boundaries in quasicrystals;\index{quasicrystal}
compare~\cite{csl-BLP96,csl-Pleasants,csl-Warrington} and references
therein.

This new development also triggered a more detailed study of lattices
in dimensions \mbox{$d>3$}, as they are used to generate aperiodic
point sets by the now common cut and project technique;
compare~\cite[Ch.~7]{csl-TAO}.  In particular, lattices in dimension
$d=4$ such as the hypercubic lattices~\cite{csl-Baake-rev,csl-Z3} and
the $A_4$-lattice~\cite{csl-BGHZ08,csl-HZ10} were studied.

Further applications of CSLs can be found in coding theory in
connection with so-called lattice quantisers, where lattices in large
dimensions and with high packing densities are important;
compare~\cite{csl-Slo02a,csl-Slo02b} for general background, as
well as \cite{csl-akh-sal10a} for concrete applications of the
$A_4$-lattice and \cite{csl-akh-sal12} for the hexagonal
lattice. However, not much is known about lattices in dimensions
$d>5$, although there are some partial results for rational lattices
\cite{csl-zou1,csl-zou2,csl-Huck09}.
 
The original concept of CSLs has been generalised in several ways.  In
particular, one may study the intersection of several rotated copies
of a lattice, which are known as multiple CSLs;
compare~\cite{csl-BG,csl-pzmcsl1,csl-BZ07}. They have applications to
so-called multiple junctions~\cite{csl-gerts1,csl-gerts2,csl-gerts3},
which are multiple crystal grains meeting at some common manifold.
Whereas classical CSLs involve only linear isometries, one may
consider affine isometries as well, which is directly related to the
question of coincidences of crystallographic point packings;
compare~\cite{csl-LZ10, csl-LZ14,csl-Loquias10}.  The latter are
connected to the problem of coincidences of coloured lattices and
colour coincidences~\cite{csl-LZ9,csl-Loquias10,csl-LZ15}.

The planar case is certainly the best studied. Here, also a connection
between CSLs and well-rounded sublattices has been established
\cite{csl-BSZ-well}. Moreover, even some results for the hyperbolic
plane \cite{csl-Rodr11} have been found.

Naturally, CSLs are not the only sublattices that are of
interest in crystallography and coding theory. Classifying sublattices
with certain symmetry constraints has a long tradition in mathematics
and in crystallography; compare~\cite{csl-Ruth1,csl-Ruth4} and
references therein. An interesting question is the number of
sublattices that are \emph{similar} to its parent lattice. It has been
answered in detail for a considerable collection of
lattices~\cite{csl-BG2,csl-BM,csl-BHM} in dimensions $d\leq 4$.  For
higher dimensions, some existence results have been obtained by
Conway, Rains and Sloane, who were motivated by problems in coding
theory~\cite{csl-consloa99}.

Actually, some years ago, a close connection between \emph{similar
  sublattices}\index{similar sublattice} (SSLs) and CSLs has been
established~\cite{csl-svenja1}, which was later generalised to
$\ZZ\ts$-modules embedded in $\RR^d$~\cite{csl-svenja2,csl-pzsslcsl1}.
This provides the link for our two main topics, namely the enumeration
of coincidence site lattices and similar sublattices, and its
generalisation to embedded $\ZZ\ts$-modules.\vspace{1mm}

Let us give an outline of this chapter.  Our main focus is on lattices
and certain $\ZZ\ts$-modules, the latter viewed as embedded in some
Euclidean space. This point of view is unusual from an algebraic point
of view, but motivated by the crystallographic applications to
(quasi-)crystals. Therefore, all lattices are regarded as special
cases of embedded $\ZZ\ts$-modules, and one could develop the theory
for embedded modules right from the beginning. However, the lattice
case is without doubt such an important problem in itself that we
prefer to first present the theory for lattices, and generalise
later. In fact, our text is written in such a way that readers
primarily interested in the lattice case can simply skip the
discussions of the more general modules.

The chapter is organised as follows. We start with some basic notions
and facts about lattices in Section~\ref{csl-sec:prelim}. As a
motivation and an introduction to the general theory, we consider a
variety of counting problems of the square lattice in
Section~\ref{csl-sec:square}.  This not only serves to illustrate the
special enumeration problems of SSLs and CSLs we are after, but also
puts them in a broader range of problems to emphasise the connections
to other combinatorial questions. Section~\ref{csl-sec:tools} provides
some useful tools from algebra and analysis.

In Section~\ref{csl-sec:ssl}, we discuss SSLs. After the general theory
in Section~\ref{csl-sec:ssl-gen}, we consider several examples,
including planar lattices (Section~\ref{csl-sec:two-dim}) and rational
lattices in dimensions $d\geq 4$ (Section~\ref{csl-sec:ssl-higher}),
with a detailed presentation of the lattice $A_4$ in
Section~\ref{csl-sec:sim-a4} and the hypercubic lattices in
Section~\ref{csl-sec:ssl-hycub}. The results for lattices are finally
generalised for embedded $\ZZ\ts$-modules in
Section~\ref{csl-sec:ssm}, which also includes the icosian ring as an
example (Section~\ref{csl-sec:ssm-I}). In addition, some examples for
planar modules can already be found in Section~\ref{csl-sec:two-dim}.

From Section~\ref{csl-sec:csl} onwards, we deal with CSLs and
coincidence site modules.  Section~\ref{csl-sec:csl} presents the
general theory, both for simple and multiple CSLs. It includes a
section on some connections with monotiles (Section~\ref
{csl-sec:csl-mono}). In Section~\ref{csl-sec:csl-csm}, we generalise
our results to embedded $\ZZ\ts$-modules and, finally, we investigate
the interrelations between coincidence site modules and similar
submodules in Section~\ref{csl-sec:sslcsl}. This is followed by a
series of examples. In Section~\ref{csl-sec:csl-nfold}, we deal with
planar $\ZZ\ts$-modules. After discussing the cubic lattices in
Section~\ref{csl-sec:cubic}, we move on to the four-dimensional
hypercubic lattices in Section~\ref{csl-sec:cub4} and to the lattice
$A_4$ in Section~\ref{csl-sec:icosianall}, which also covers the
icosian ring as an example of a $\ZZ\ts$-module embedded in $\RR^4$.
Section~\ref{csl-sec:mcsl-cub} is devoted to the multiple CSLs of the
cubic lattices. Finally, we present some (rudimentary) results for
dimensions $d\geq 5$ in Section~\ref{csl-sec:higher}.

Throughout this chapter, ideals play an important role. In almost all
of our examples, we are dealing with principal ideals, which have a
single generating element that is unique up to units.  Although it is
usually more elegant to formulate results in terms of ideals instead
of generating elements, we will frequently prefer to deal with
generating elements. The main reason is that we usually deal with
ideals in algebraic number fields or quaternion algebras, and their
elements can be used to parametrise rotations in dimensions $d\leq
4$. However, rotations are parametrised by concrete complex numbers or
quaternions, respectively, and not by ideals. As we want to emphasise
the direct connection to the rotations and use geometric intuition, we
accept the fact that some equations are more cumbersome when
formulated with quaternions and hold only up to units. For those who
are more interested in an exposition using ideals, we
mention~\cite[Sec.~5]{csl-BLP96}, which shows how to formulate matters
in ideal-theoretic way in the context of quaternion algebras.

As we proceed, we shall prove many of the structural properties and
results$\,$---$\,$in particular, when they are not trivial or not
easily available in the literature.  Otherwise, we state concrete
results without proof, but with proper (and precise) references.

\section{Preliminaries on lattices}
\label{csl-sec:prelim}

Let us begin with some definitions for lattices in $\RR^{d}$ (which
are co-compact discrete subgroups of $\RR^{d}$), where we start from
the notions introduced in \cite[Ch.~3]{csl-TAO} and refer to
\cite{csl-Cassels,csl-Lekker} for further background. In particular, a
lattice $\vG\subset \RR^d$ always has full rank $d$ (as a
$\ZZ\ts$-module), and any lattice basis can also serve as a basis for
$\RR^d$.

\begin{definition}\label{csl-def:comm}
  Two lattices $\vG_1,\vG_2\subset \RR^d$ are called
  \emph{commensurate}\index{lattice!commensurate}, denoted by
  $\vG_1\sim \vG_2$, if $\vG_1\cap \vG_2$ has
  finite index\index{index} in both $\vG_1$ and $\vG_2$.
\end{definition}

In our terminology, commensurateness means that $\vG_1\cap \vG_2$ is a
sublattice (of full rank) of both $\vG_1$ and $\vG_2$.  Actually,
there are several ways to characterise commensurateness
\cite{csl-habil}.

\begin{lemma}\label{csl-lem:comm-lat}
  Let\/ $\vG_1$ and\/ $\vG_2$ be lattices in\/ $\RR^d$. Then, the
  following statements are equivalent.
\begin{enumerate}\itemsep=2pt
\item $\vG_1$ and\/ $\vG_2$ are commensurate.
\item $\vG_1\cap \vG_2$ has finite index in both\/ $\vG_1$ and\/ $\vG_2$.
\item $\vG_1\cap \vG_2$ has finite index in\/ $\vG_1$ or in\/ $\vG_2$.
\item\label{csl-lem:comm-i4-lat} There exist\/ $($positive$\ts )$
  integers\/ $m_1$ and\/ $m_2$ such that\/ $m_1\vG_1\subseteq \vG_2$
  and\/ $m_2\vG_2\subseteq \vG_1$.
\item There exists an integer\/ $m\ne 0$ such that\/ $m \vG_1\subseteq
  \vG_2$ or\/ $m \vG_2\subseteq \vG_1$.
\item $\vG_1\cap \vG_2$ is a lattice\/ $(\nts$of full rank\/ $d\ts\ts )$ in\/
  $\RR^d$.  \qed
\end{enumerate}
\end{lemma}

\begin{figure}
\includegraphics[width=0.9\textwidth]{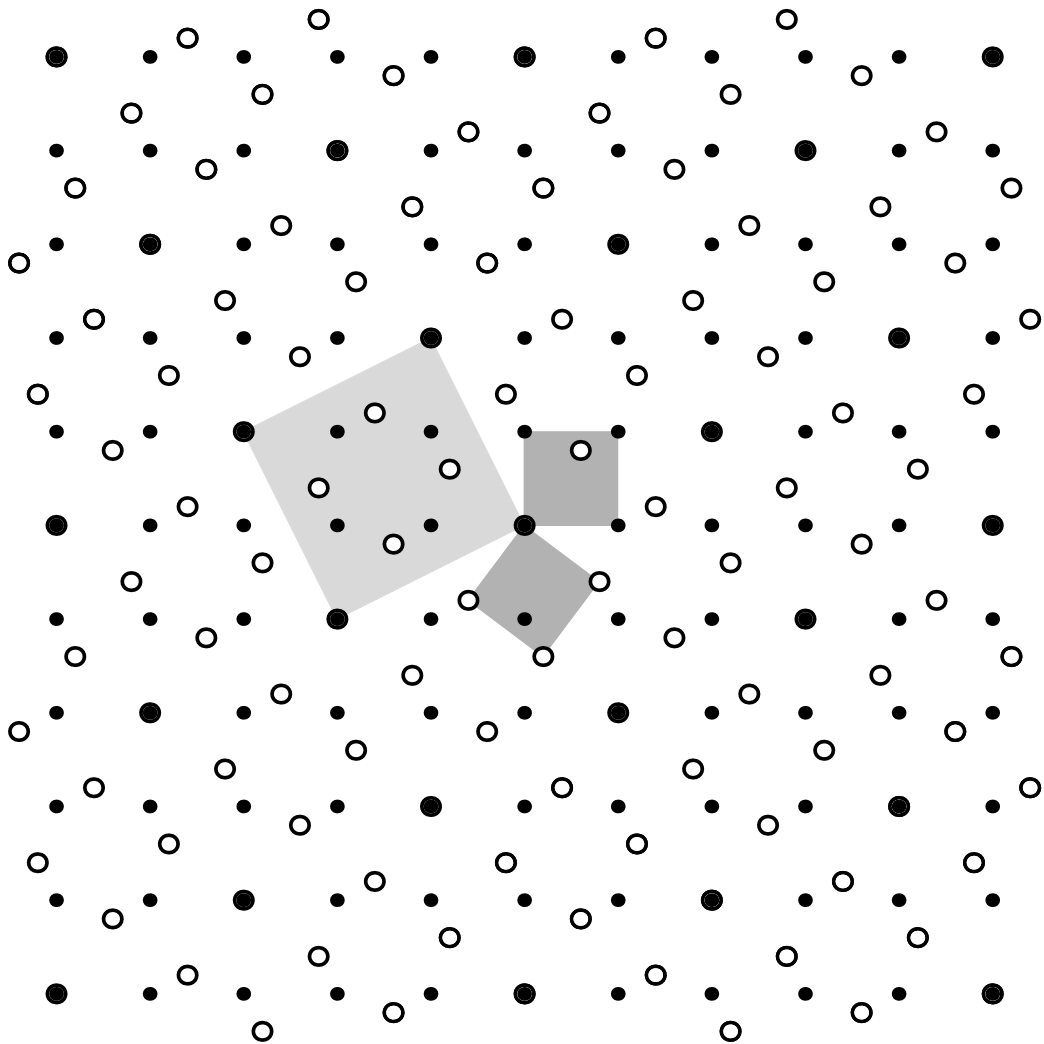}
\caption{A square lattice (all black points) and a rotated copy of it
  (open circles together with large black points), with relative
  rotation angle\break $\alpha=\arctan(\frac{4}{3})\approx
  53.13^\circ$. The large black points mark the intersection of the
  two lattices, which is the CSL and again a square lattice. The
  shaded squares show fundamental domains of the three lattices. The
  larger square is a fundamental domain of the
  CSL.}\index{square~lattice!CSL}
\label{csl-fig:csl}
\end{figure}

As an immediate consequence, for instance via applying property
(\ref{csl-lem:comm-i4}) several times, one obtains that
commensurateness is an equivalence relation.

An example of commensurate lattices is provided by similar
sublattices.  In fact, similarity of lattices is an important concept
to us.  Recall that an invertible linear map $f\! : \,
\RR^d\longrightarrow\RR^d$ is called a \emph{similarity
  transformation} if it is of the form $f=\alpha R$, where $R$ is a
(linear) isometry and $0\ne \alpha \in\RR$. Two lattices $\vG$ and
$\vG'$ are called \emph{similar}, in symbols $\vG \similar \vG'$, if
there exists a similarity transformation from one to the
other. Clearly, similarity of lattices is an equivalence relation.

\begin{definition}
  A similarity transformation that maps a lattice $\vG\in \RR^d$ onto
  a sublattice of $\vG$ is called a \emph{similarity transformation of
    $\vG$}.  A sublattice $\vG'\subseteq \vG $ is called a
  \emph{similar sublattice} (SSL)\index{similar sublattice} of $\vG$
  if $\vG'$ is similar to $\vG$.
\end{definition}

Trivial examples of SSLs are the sublattices $m\vG$, with $m\in\NN$.
Similarly, given an SSL $\vG'\subseteq \vG$, also $m\vG'$ is an
SSL. In order to exclude these cases, we introduce the notion of a
primitive SSL.

\begin{definition}
  An SSL $\vG'\subseteq \vG$ is called
  \emph{primitive}\index{similar sublattice!primitive} if
  $\frac{1}{n}\vG'\subseteq \vG$ with $n\in\NN$ implies that $n=1$.
\end{definition}

In crystallography, the intersection $\vG\cap R\vG$ plays an important
role in describing grain boundaries. If $\vG\cap R\vG$ is a lattice
(of full rank), it is called a \emph{coincidence site lattice}
\index{coincidence~site~lattice} (CSL). A planar example is shown in
Figure~\ref{csl-fig:csl}.  As we have seen, the intersection $\vG\cap
R\vG$ is a lattice if and only if $\vG$ and $R\vG$ are
commensurate. This motivates the following definition.

\begin{definition}\label{csl-def:csl}
  Let $\vG$ be a lattice in $\RR^d$, and let $R\in \OG(d,\RR)$. If
  $\vG$ and $R \vG$ are commensurate, $\vG(R):=\vG \cap R \vG$ is
  called a \emph{coincidence site lattice}
  (CSL)\index{coincidence~site~lattice}.  In this case, $R$ is called
  a \emph{coincidence isometry}\index{isometry!coincidence}.  The
  corresponding index, $\Sig^{}_\vG(R):=[\vG: \vG(R)]$, is called its
  \emph{coincidence index}\index{coincidence~index}.
\end{definition}

Before we embark on a systematic review of CSLs and their properties,
let us embed the study of such lattices, in an illustrative fashion,
into a wider context that is motivated by geometry and combinatorics.

\section{A hierarchy of planar lattice enumeration
problems}\label{csl-sec:square}

It is the intention of this section to shed some more light on the
coincidence problem and how it relates to various types of
index-oriented sublattice enumerations with geometric constraints. Let
us explain this for the square lattice in $\RR^2$ in an informal
manner. The results will be given in closed form in terms of zeta
functions, and explicitly (for small indices) in
Table~\ref{csl-tab:sq} on page~\pageref{csl-tab:sq}.

To this end, let us start with the question of how many
sublattices\index{square~lattice!sublattices} of $\ZZ^2$ have
index\index{index} $m$, without any further restriction.  Let us call
this number $a_m$. Clearly, $a_1^{}=1$ (only $\ZZ^2$ itself is a
sublattice of index $1$) and $a_2^{}=3$ (counting two different
rectangular sublattices and one square sublattice).  In general, one
has $a_{mn} = a_m a_n$ when $m,n\in\NN$ are coprime, and one can
derive, either from \cite[Appendix]{csl-Baake-rev} or from
\cite[Lemma~2 on p.~99]{csl-Serre}, the general result that
\[
    a_m \, = \,  \sigma_1^{}(m) \, = \,
    \mbox{$\sum_{d\mid m}$}\, d\ts ,
\]
where $\sigma_1^{}$ is a divisor function,\index{function!divisor}
whose Dirichlet series\index{Dirichlet~series} generating function
reads
\begin{equation}
  F(s) \, = \sum_{m=1}^{\infty} \frac{a_m}{m^s}
  \, = \,  \zeta(s) \, \zeta(s-1) \ts .  
\end{equation} 
Here, $\zeta(s) = \sum_{m=1}^{\infty} m^{-s}$ is Riemann's
zeta\index{zeta~function!Riemann} function \cite{csl-Apostol}.  From
this, it can be shown that the number of sublattices of index $\le x$,
which is the summatory\footnote{As is well known from number theory,
  arithmetic\index{function!arithmetic} functions such as $m\mapsto
  a_{m}$ are prone to strong fluctuations.  The corresponding
  summatory functions are usually more regular, and show a
  well-defined asymptotic behaviour; see \cite{csl-Apostol} for
  background.} function $\sum_{k\le x} a_k$, grows quadratically as
$x^2 \pi^2/12$; compare \cite[Thm.~324]{csl-Hardy}. More precisely, we
have
\[
   \sum_{m\leq x} a_m \, = \, \myfrac{\pi^{2}}{12} \, x^{2} 
    + \,\cO \bigl(x \log(x)\bigr) \qquad \text{as $x\to\infty$}.
\] 
This counting result is, of course, \emph{algebraic} in nature and
thus applies to any planar lattice, and to the free Abelian group of
rank $2$ in particular (where $a_{m}$ is the number of distinct
subgroups of index\index{index} $m$).

As a first geometric refinement step, let us consider those
sublattices of $\ZZ^{2}$ which are \emph{well-rounded}, which means
that the shortest non-zero lattice vectors span the plane. Here, the
result is considerably more difficult (and the most difficult one for
this informal discussion), and one finds~\cite{csl-BSZ-well} that the
counts $a_{\minisquare}^{\mathsf{wr}} (m)$ lead to the Dirichlet
series\index{Dirichlet~series}
\begin{equation}
 \Phi^{\mathsf{wr}}_{\minisquare} (s) \, = \, \Phi^{\mathsf{pr}}_{\minisquare}(s)
   \bigl( 1 + \phi^{}_0(s) + \phi^{}_1(s) \bigr) ,
\end{equation}
where $\Phi^{\mathsf{pr}}_{\minisquare}(s)$ is the generating function
for all primitive square sublattices given below in
Eq.~\eqref{csl-eq:prim-sq}, together with
\begin{align*}
   \phi^{}_0(s)\, & = \, \myfrac{2}{2^s}
   \sum_{p\in\NN}\, \sum_{p < q < \sqrt{3}p}\myfrac{1}{p^s q^s}\ts , \\[1mm]
  \phi^{}_1(s) \, & = \,
  \myfrac{2}{1+2^{-s}}
  \sum_{k\in\NN}\, \sum_{k < \ell < \sqrt{3}k+\frac{\sqrt{3}-1}{2}}
  \myfrac{1}{(2k+1)^s (2\ell+1)^s} \ts .
\end{align*}
Although no simpler closed expressions for these functions are known,
they can be approximated by explicit formulas involving Riemann's
zeta\index{zeta~function!Riemann} function and a certain
\mbox{$L\ts$-series}~\cite{csl-BSZ-well}; see also below
(Section~\ref{csl-sec:two-dim}) for further details. This enables us
to determine the asymptotic growth rate explicitly, including an error
term. As there are considerably fewer well-rounded sublattices than
sublattices in total, it is not surprising that the growth rate is
smaller, namely $\frac{\log(3) }{2\ts \pi} \, x \log(x)$ as $x \to
\infty$; see~\cite{csl-BSZ-well}. There exists a linear correction
term, and the asymptotic behaviour\index{asymptotic~behaviour} reads
in detail
\begin{align*}
  \sum_{n\leq x} a^{\mathsf{wr}}_{\minisquare} (n) \, & = \, \frac{\log(3) }{3}
      \frac{L(1,\chi^{}_{-4})}{\zeta(2)} \, x \bigl(\log(x) - 1\bigr)
      + c^{}_{\minisquare} x + \cO\bigl(x^{3/4}\log(x)\bigr) \\[1mm]
  & = \, \frac{\log(3) }{2\pi}\, x \log(x) + \left( c^{}_{\minisquare} -
    \frac{\log(3) }{2\pi} \right) x + \cO\bigl(x^{3/4}\log(x)\bigr),
\end{align*}
where, with $\gamma\approx 0.5772157$ denoting the Euler--Mascheroni
constant,
\begin{align*}
  c^{}_{\minisquare}& \, := \,\frac{L(1,\chi^{}_{-4})}{\zeta(2)} \Biggl(
  \frac{\log(3)}{3} \left(
    \frac{L'(1,\chi^{}_{-4})}{L(1,\chi^{}_{-4})} + 3\gamma -
    2\ts\frac{\zeta'(2)}{\zeta(2)}  - \frac{\log(3)}{4} -
    \frac{\log (2)}{6} \right) \\[1mm]
  &\qquad\qquad\qquad\quad + \zeta(2)\, - \sum_{p=1}^\infty \myfrac{1}{p} 
     \biggl(\frac{\log(3)}{2} -
  \sum_{p<q< p\sqrt{3}} \myfrac{1}{q} \biggr)  \\
  &\qquad\qquad\qquad\quad  - \myfrac{4}{3}\sum_{k=0}^\infty 
  \myfrac{1}{2k+1}\biggl(
  \myfrac{1}{4}\log(3) - \!\sum_{k<\ell< k\sqrt{3}+\frac{\sqrt{3}-1}{2}}
  \myfrac{1}{2\ell+1} \biggr) \Biggr) \\[1mm]
  & \,\approx\, 0.6272237 \ts .
\end{align*}
Note that the error term $\cO\bigl(x^{3/4}\log(x)\bigr)$ is certainly
not optimal; see \cite{csl-BSZ-well} for more. Here, the numerical
values are rounded to the last digit displayed, so that the error is
less than $1$ in the last digit. The same rule implicitly applies to
any numerical values that we will give in the following.

\begin{table}
  \caption{Some counts of the enumeration problems for 
    $\ZZ^{2}$. For indices \mbox{$n\leq 60$}, the number of multiple CSLs is
    the same as for CSLs, except for $n=25$, where we have $3$ multiple
    CSLs instead of $2$ simple ones.} \label{csl-tab:sq}
\begin{small}
\begin{tabular}{ | l | r | r | r | r | r | r | r | r | r | r | }
\hline
index $m$ & \hphantom{10}1 & \hphantom{10}2 & \hphantom{10}3 & 
\hphantom{10}4 & \hphantom{10}5 & \hphantom{10}6 & \hphantom{10}7 & 
\hphantom{10}8 & \hphantom{10}9 & \hphantom{1}10  \\
\hline
sublattices & 1 & 3 & 4 & 7 & 6 & 12 & 8 & 15 & 13 & 18 \\
well-rounded & 1 & 1 & 0 & 1 & 2 & 0 & 0 & 1 & 1 & 2 \\
square & 1 & 1 & 0 & 1 & 2 & 0 & 0 & 1 & 1 & 2  \\
prim.\ square & 1 & 1 & 0 & 0 & 2 & 0 & 0 & 0 & 0 & 2  \\
coincidence & 1 & 0 & 0 & 0 & 2 & 0 & 0 & 0 & 0 & 0  \\
\hline
\hline
index $m$ & 11 & 12 & 13 & 14 & 15 & 16 & 17 & 18 & 19 & 20 \\
\hline
sublattices & 12 & 28 & 14 & 24 & 24 & 31 & 18 & 39 & 20 & 42 \\
well-rounded & 0 & 2 & 2 & 0 & 2 & 1 & 2 & 1 & 0 & 2 \\
square & 0 & 0 & 2 & 0 & 0 & 1 & 2 & 1 & 0 & 2 \\
prim.\ square & 0 & 0 & 2 & 0 & 0 & 0 & 2 & 0 & 0 & 0 \\
coincidence & 0 & 0 & 2 & 0 & 0 & 0 & 2 & 0 & 0 & 0 \\
\hline
\hline
index $m$ & 21 & 22 & 23 & 24 & 25 & 26 & 27 & 28 & 29 & 30 \\
\hline 
sublattices & 32 & 36 & 24 & 60 & 31 & 42 & 40 & 56 & 30 & 72 \\
well-rounded & 0 & 0 & 0 & 4 & 3 & 2 & 0 & 0 & 2 & 2 \\
square & 0 & 0 & 0 & 0 & 3 & 2 & 0 & 0 & 2 & 0 \\
prim.\ square & 0 & 0 & 0 & 0 & 2 & 2 & 0 & 0 & 2 & 0 \\
coincidence & 0 & 0 & 0 & 0 & 2 & 0 & 0 & 0 & 2 & 0 \\
\hline
\hline
index $m$ & 31 & 32 & 33 & 34 & 35 & 36 & 37 & 38 & 39 & 40 \\
\hline 
sublattices & 32 & 63 & 48 & 54 & 48 & 91 & 38 & 60 & 56 & 90 \\
well-rounded & 0 & 1 & 0 & 2 & 2 & 1 & 2 & 0 & 0 & 4 \\
square & 0 & 1 & 0 & 2 & 0 & 1 & 2 & 0 & 0 & 2 \\
prim.\ square & 0 & 0 & 0 & 2 & 0 & 0 & 2 & 0 & 0 & 0 \\
coincidence & 0 & 0 & 0 & 0 & 0 & 0 & 2 & 0 & 0 & 0 \\
\hline
\hline
index $m$ & 41 & 42 & 43 & 44 & 45 & 46 & 47 & 48 & 49 & 50 \\
\hline 
sublattices & 42 & 96 & 44 & 84 & 78 & 72 & 48 & 124 & 57 & 93 \\
well-rounded & 2 & 0 & 0 & 0 & 2 & 0 & 0 & 4 & 1 & 3 \\
square & 2 & 0 & 0 & 0 & 2 & 0 & 0 & 0 & 1 & 3 \\
prim.\ square & 2 & 0 & 0 & 0 & 0 & 0 & 0 & 0 & 0 & 2 \\
coincidence & 2 & 0 & 0 & 0 & 0 & 0 & 0 & 0 & 0 & 0 \\
\hline
\hline
index $m$ & 51 & 52 & 53 & 54 & 55 & 56 & 57 & 58 & 59 & 60 \\
\hline 
sublattices & 72 & 98 & 54 & 120 & 72 & 120 & 80 & 90 & 60 & 168 \\
well-rounded & 0 & 2 & 2 & 0 & 0 & 0 & 0 & 2 & 0 & 6 \\
square & 0 & 2 & 2 & 0 & 0 & 0 & 0 & 2 & 0 & 0 \\
prim.\ square & 0 & 0 & 2 & 0 & 0 & 0 & 0 & 2 & 0 & 0 \\
coincidence & 0 & 0 & 2 & 0 & 0 & 0 & 0 & 0 & 0 & 0 \\
\hline
\end{tabular}
\end{small}
\end{table}

Let us next ask how many of the sublattices of $\ZZ^2$ of index $m$
are actually \emph{square}\index{square~lattice!SSL} lattices; see
\cite{csl-BG2,csl-BM1,csl-BM} for various generalisations of this
question. This number can be obtained by counting the lattice points
on circles of radius $\sqrt{m}$ (hence counting solutions of the
Diophantine\index{Diophantine!equation} equation $x^2+y^2=m$) and
afterwards dividing by 4 (the order of $C_4$, the rotation part of the
point symmetry group of $\ZZ^2$). The result is given in
\cite[Chs.~16.9, 16.10 and 17.9]{csl-Hardy} and leads to the Dirichlet
series\index{Dirichlet~series} generating function
\begin{equation}\label{csl-eq:square-sub}
 \Phi^{}_{\minisquare}(s)   \, = \, \zeta^{}_K(s)  \, = \,
            \myfrac{1}{1-2^{-s}} 
            \prod_{p \equiv 1\, (4)} \!
                   \Bigl(\myfrac{1}{1-p^{-s}}\Bigr)^2 
            \prod_{p \equiv 3\, (4)} \myfrac{1}{1-p^{-2s}} \ts , 
\end{equation} 
where here and in what follows
$\zeta^{}_K(s)=\zeta(s)\ts L(s,\chi^{}_{-4})$ is the Dedekind
zeta\index{zeta~function!Dedekind} function of the quadratic
field\index{number~field} $K=\QQ(\ii)$;
compare~\cite[$\S${\ts}11]{csl-Zagier}.  Recall that
\[
    L(s,\chi^{}_{-4}) \, = \ts
    \sum_{n=1}^\infty \frac{\chi^{}_{-4}(n)}{n^s}
\]
is the \mbox{$L\ts$-series} for the Dirichlet character
\[
  \chi^{}_{-4} (n) \, = \,
    \begin{cases} 0 , & \text{if $n$ is even, } \\
    1 , & \text{if $n\equiv 1 \bmod 4$, } \\
    -1 , & \text{if $n\equiv 3 \bmod 4$.}
    \end{cases}
\]
The meaning of Eq.~\eqref{csl-eq:square-sub} becomes clear if one
realises that the square sublattices of $\ZZ^2$ are precisely the
non-trivial ideals of $\ZZ [\ii]$, the ring of Gaussian
integers\index{Gaussian~integer} within the quadratic field $\QQ
(\ii)$; see also \cite[Ex.~2.15]{csl-TAO}.

The growth rate of the number of similar
sublattices\index{square~lattice!SSL} of index $\le x$ is linear,
namely $\frac{\pi}{4}x$, which follows either from the asymptotic
properties of $\zeta^{}_K(s)$ near its rightmost pole (at $s=1$) or
just from counting one quarter of the lattice points inside the circle
of radius $\sqrt{x}$, which is $\frac{\pi}{4}x$ to leading order in
$x$. More precisely, one has 
\[
    \sum_{m\le x} a^{}_{\minisquare} (m) \, = \,
    \myfrac{\pi}{4}\ts x + \cO \bigl(\sqrt{x}\,\bigr);
\]
see \cite[Appendix]{csl-BSZ-well} for details.

In the last case, some of the square lattices fail to be
\emph{primitive}\index{square~lattice!primitive~SSL} (such as $3\ts
\ZZ^2$ etc.), which happens whenever the sublattice is an integer
multiple of $\ZZ^2$ or of one of its primitive sublattices. If we
exclude those, primes $p\equiv 3 \; (4)$ are impossible as divisors of
the index $m$, and some solutions of $p\equiv 1 \; (4)$ also drop out
(whenever the index is divisible by a square).  Now, the generating
function reads\index{Dirichlet~series}
\begin{equation} \label{csl-eq:prim-sq}
  \Phi^{\mathsf{pr}}_{\minisquare}(s) \, = \, \left( 1 + 2^{-s}\right)\!
                 \prod_{p\equiv1 \, (4)}
                 \frac{1+p^{-s}}{1-p^{-s}}
                 \, = \,  \frac{\zeta^{}_K(s)}{\zeta(2s)} \ts , 
\end{equation}
and the asymptotic growth rate\index{asymptotic~behaviour} 
of the number of primitive square
sublattices of index $\le x$ is again linear, this time with
\[
   \sum_{m\le x} a^{\mathsf{pr}}_{\minisquare} (m) \, = \,
    \myfrac{3}{2 \pi}\ts x + \cO \bigl(\sqrt{x}\ts \log(x)\bigr).
\]
The leading term can be determined by counting one quarter of the
visible points \cite{csl-Apostol} in the circle of radius $\sqrt{x}$,
which is $\frac{1}{4}(\pi x \frac{6}{\pi^2})$;
see~\cite[Appendix]{csl-BSZ-well} for the error term.

Finally, let us see how the CSLs\index{square~lattice!CSL} fit into
this picture. As $\SO(2,\RR)$ is Abelian, the symmetry group of any
CSL of $\ZZ^{2}$ must contain a fourfold rotation, which in turn
implies that the CSL must be a square lattice. Hence, any CSL is a
similar sublattice. Note that this is a rather exceptional feature of
the square lattice, which it shares, to our knowledge, only with the
planar hexagonal lattice and some embedded modules in the
plane. Nevertheless, there is a general connection between SSLs and
CSLs, as we shall see in Section~\ref{csl-sec:sslcsl} below.  Note,
however, that not all square sublattices are CSLs$\,$---$\,$they only
are if they are primitive and if their index is \emph{odd}. This
finally results in the generating function of the coincidence problem
described above, namely\index{Dirichlet~series}
\begin{equation}\label{csl-eq:csl-square}
  \Psi^{}_{\minisquare}(s) \, = 
  \sum_{m=1}^{\infty} \frac{a_{\minisquare}^{\mathsf{CSL}}(m)}{m^{s}}\, = \!
  \prod_{p\equiv1 \ts\ts (4)} \frac{1+p^{-s}}{1-p^{-s}}
  \, = \, \left( 1 + 2^{-s}\right)^{-1} \frac{\zeta^{}_K(s)}{\zeta(2s)} \ts .
\end{equation} 
As $\Psi^{}_{\minisquare}(s)$ differs from
$\Phi^{\mathsf{pr}}_{\minisquare}(s)$ only by the factor $\left( 1 +
  2^{-s}\right)^{-1}$, the asymptotic behaviour
\[
   \sum_{m\le x} a_{\minisquare}^{\mathsf{CSL}}(m) \, = \,
    \myfrac{1}{\pi}\, x + \cO \bigl(\sqrt{x}\ts \log(x)\bigr)
\]
is similar to the previous case, with the growth rate being lower by a
factor of $\left( 1 + 2^{-1}\right)^{-1}=\frac{2}{3}$.  As before,
this equation expresses the Dirichlet series generating function in
terms of zeta functions; see Table~\ref{csl-tab:sq} for the first few
terms of the corresponding Dirichlet series.  Similar formulas will
also appear in many of our later examples.  

Now, we may even go a step further and ask for \emph{multiple}
coincidences, i.e., intersections of any finite number of CSLs. As we
shall discuss later (see Section~\ref{csl-sec:csl-nfold} and, in
particular, Example~\ref{csl-ex:mcsl-square}), the set of indices
stays the same, but some additional lattices emerge, which are still
similar sublattices, but not primitive any more.

We hope that this short digression has put the enumeration problem in
a broader perspective, and also in contact with some elementary
questions from analytic number theory.  Of course, we are also
interested in results in higher dimensions, where the picture changes
significantly because $\mathrm{O}(d,\RR)$ is no longer Abelian for
$d>2$. Before we can proceed, we first need to introduce various
methods and tools.

\section{Algebraic and analytic tools}
\label{csl-sec:tools}

In writing the square lattice as $\ZZ^{2}=\ZZ[\ii]$, we can profit
from the algebraic structure of $\ZZ[\ii]$, which is a ring that is a
\emph{principal ideal domain}\index{principal~ideal~domain} (PID). In
fact, it is the ring of integers of the imaginary quadratic field
$\QQ(\ii)$, and as such the maximal order\index{maximal~order} of the
field. Here, the term `order' means that we are dealing with a
finitely generated $\ZZ\ts$-module whose $\QQ$-span is the entire
field. $\ZZ[\ii]$ is maximal for this property in the obvious sense,
and unique as such.\index{number~field}

Note that $\QQ(\ii)$ can also be viewed as a cyclotomic
field,\index{cyclotomic~field} which will later be used for an
extension to analyse similar submodules and coincidence site modules
of the rings $\ZZ[\xi^{}_{n}]$ of integers in the cyclotomic field
$\QQ(\xi^{}_{n})$, where $\xi^{}_{n}$ is a primitive $n$-th root of
unity. We refer to \cite[Sec.~2.5]{csl-TAO} and to \cite{csl-Wash} for
general background in this context. \smallskip

As we shall see, this number-theoretic approach is truly powerful for
planar structures. Consequently, one would like to have related
methods also for higher-dimensional problems. This leads to a
non-commutative generalisation in the form of certain quaternion
algebras and their maximal orders. 

\subsection{Quaternions}
\label{csl-sec:quat}

As quaternions are pivotal in what follows, we briefly recall their
most important properties. For details, we refer to
\cite[Sec.~2.5.4]{csl-TAO}, \cite[Secs.~3 and~4]{csl-BLP96} and to the
general literature~\cite{csl-Koecher,csl-CS03,csl-Hurwitz,csl-Hardy}.

Let $\{\e,\ii,\jj,\kk\}$ be the standard basis of $\RR^4$, where
\[
    \e={(1,0,0,0)}\ts , \quad \ii={(0,1,0,0)}\ts ,\quad 
    \jj={(0,0,1,0)}\ts , \quad \kk={(0,0,0,1)}\ts .  
\]
The \emph{quaternion algebra} over $\RR$ is the associative division
algebra
\[
  \HH\, :=\, \HH(\RR)\, =\, \RR\e+\RR\ii+\RR\jj+\RR\kk\, \simeq\, \RR^4, 
\]
where multiplication is induced by Hamilton's relations
\[
     \ii^2 \, = \, \jj^2 \, = \, \kk^2 \, = \, \ii\jj\kk \, = \, -\e \ts .
\] 
Elements of $\HH$ are called quaternions, and an arbitrary quaternion
$q$ is written as either $q=q^{}_0\e+q^{}_1\ii+q^{}_2\jj+q^{}_3\kk$ or
$q=(q^{}_0,q^{}_1,q^{}_2,q^{}_3)$.  Given two quaternions $q$ and $p$,
their \emph{inner product} is defined by the standard scalar product
of $q$ and $p$ as vectors\footnote{{\ts\ts}We usually identify a quaternion
  $q=(q^{}_0,q^{}_1,q^{}_2,q^{}_3)$ with the corresponding row vector
  $(q^{}_0,q^{}_1,q^{}_2,q^{}_3)$. However, when we use Cayley's
  parametrisation for rotations (see below), we will identify $q$ with
  the corresponding column vector $(q^{}_0,q^{}_1,q^{}_2,q^{}_3)^T$.}
in $\RR^4$.

The \emph{conjugate} of $q=(q^{}_0,q^{}_1,q^{}_2,q^{}_3)$ is
$\overline{q}=(q^{}_0,-q^{}_1,-q^{}_2,-q^{}_3)$, and its \emph{norm}
is $\nr(q)=|q|^2=q\,\overline{q}=q_0^2+q_1^2+q_2^2+q_3^2\in\RR$. One
has $\overline{q\ts p\ts}=\overline{p}\,\ts\overline{\nts q\ts}$ and
$|q\,p|^2=|q|^2|p|^2$ for any $q,p\in\HH$.  Given a quaternion
$q=(q^{}_0, q^{}_1, q^{}_2, q^{}_3)$, its \emph{real} and
\emph{imaginary parts} are defined as $\Real(q)=q^{}_0$ and
$\Imag(q)=q^{}_1\ii+q^{}_2\jj+q^{}_3\kk$, respectively.  It is easy to
verify that $\Real(\HH)$ is the centre of $\HH$, wherefore we can
identify $\e$ with $1$ from now on. The \emph{imaginary space of}
$\HH$ is the three-dimensional subspace $\Imag (\HH)=\{\Imag (q) \mid
q\in\HH\} \simeq\RR^3$ of $\HH$.  For ease of notation, we will
identify $\Imag(\HH)$ and $\RR^3$ and, in addition, also the elements
$(q^{}_1, q^{}_2, q^{}_3)\in \Imag (\HH)$ with the elements
$(0,q^{}_1, q^{}_2, q^{}_3)\in \HH$. \smallskip

Another convenient feature of the quaternions\index{quaternion} is
that they can be used to parametrise rotations in 3 and 4 dimensions;
compare~\cite{csl-Koecher} and~\cite{csl-BLP96}.  In $\RR^3$, any
rotation\index{rotation!Cayley~parametrisation} can be
parametrised by a single quaternion $0\neq q\in\HH$ with
$q=(\kappa,\lambda,\mu,\nu)$ via
\begin{equation}\label{csl-eq:Rpar}
   R(q) \, = \, \myfrac{1}{|q|^2}\!
   \begin{pmatrix}
   \kappa^2 \!+\!\lambda^2\!-\!\mu^2\!-\!\nu^2 & 
   -2\kappa\nu\!+\!2\lambda\mu &
   2\kappa\mu\!+\!2\lambda\nu\\
   2\kappa\nu\!+\!2\lambda\mu & 
   \!\!\!\kappa^2\!-\!\lambda^2\!+\!\mu^2\!-\!\nu^2\!\!\! &
  -2\kappa\lambda\!+\!2\mu\nu\\
  -2\kappa\mu\!+\!2\lambda\nu & 2\kappa\lambda\!+\!2\mu\nu &
   \kappa^2\!-\!\lambda^2\!-\!\mu^2\!+\!\nu^2 
   \end{pmatrix}\! ,
\end{equation}
where elements of $\RR^3$ are written as column vectors.  In
particular, we have
\[
   R(q)\, x \, = \, R(q)\begin{pmatrix} 
   x^{}_1 \\[-0.5ex] x^{}_2 \\[-0.5ex] x^{}_3 \end{pmatrix}
   \, = \, \frac{q\ts x\ts \bar{q}}{|q|^2} 
\]
for any $x\in\RR^3$, where we have again identified $x$ with $(0,
x^{}_1, x^{}_2, x^{}_3)$.  Note that the parametrisation of
Eq.~\eqref{csl-eq:Rpar}, which is known as \emph{Cayley's
  parametrisation}, is only unique up to a scaling factor, meaning
$R(\alpha q)=R(q)$ for $\alpha\ne 0$.

As we are usually interested in specific subgroups of $\SO(3,\RR)$,
let us mention that such subgroups often can be related to suitable
subrings of $\HH$; compare~\cite[Prop.~1]{csl-BLP96}. In particular,
the rotations of $\SO(3,\QQ)$ can be parametrised by integer
quaternions as explained below.

In $\RR^4$, a pair of quaternions\index{quaternion} is needed to
parametrise a rotation\index{rotation!Cayley~parametrisation}
\cite{csl-Koecher,csl-duVal}.  These quaternions are unique up to
positive scaling factors and a common sign change. In particular,
\begin{equation}\label{csl-eq:4dimrot}
  R(p,q) \! : \,
  \RR^4 \xrightarrow{\quad} \RR^4,\quad
  x \longmapsto R(p,q)\ts x\, =\,  
  \frac{p\ts x\ts \bar{q}}{|p \ts q|} 
\end{equation}
defines a rotation in $\RR^4$, whose matrix
representation$\,$---$\,$in abuse of notation also written as
$R(p,q)=\frac{1}{|p \ts q|}M(p,q)\,$---$\,$is explicitly given by
\begin{align*}
&M(p,q)\,=\,
\begin{pmatrix}
\ip{p}{q}&
\ip{p\ii }{q}&
\ip{p\jj }{q}&
\ip{p\kk }{q}\\
\ip{p}{\ii q}&
\ip{p\ii }{\ii q}&
\ip{p\jj }{\ii q}&
\ip{p\kk }{\ii q}\\
\ip{p}{\jj q}&
\ip{p\ii }{\jj q}&
\ip{p\jj }{\jj q}&
\ip{p\kk }{\jj q}\\
\ip{p}{\kk q}&
\ip{p\ii }{\kk q}&
\ip{p\jj }{\kk q}&
\ip{p\kk }{\kk q}
\end{pmatrix}\\[2mm]
&=\,\mbox{\tiny
$\begin{pmatrix}
ak+b\ell+cm+dn & -a\ell+bk+cn-dm & -am-bn+ck+d\ell & -an+bm-c\ell+dk\\
a\ell-bk+cn-dm & ak+b\ell-cm-dn & -an+bm+c\ell-dk & am+bn+ck+d\ell \\
am-bn-ck+d\ell & an+bm+c\ell+dk & ak-b\ell+cm-dn & -a\ell-bk+cn+dm \\
an+bm-c\ell-dk & -am+bn-ck+d\ell & a\ell+bk+cn+dm & ak-b\ell-cm+dn
\end{pmatrix}$} , 
\end{align*}
where $p=(k,\ell,m,n)$ and $q=(a,b,c,d)$. 
Here, $\ip{\cdot}{\cdot}$ denotes the standard (Euclidean)
inner product in $\RR^4$.

A quaternion\index{quaternion} all of whose components are integers
is called a \emph{Lipschitz quaternion}\index{quaternion!Lipschitz}.
The set $\LL$ of Lipschitz quaternions is thus defined as
\begin{equation}
   \LL \, := \, \bigl\{ (q^{}_0,q^{}_1,q^{}_2,q^{}_3)\in\HH 
   \bigm\vert q^{}_0,q^{}_1,q^{}_2,q^{}_3\in\ZZ \bigr\} .
\end{equation}
The Lipschitz quaternions form an order in the quaternion algebra
$\HH(\QQ)$, but not a maximal one.  A \emph{primitive Lipschitz
  quaternion}\index{quaternion!primitive} $q$ is a quaternion in
$\LL$ whose components are relatively prime. Furthermore, a
\emph{Hurwitz quaternion}\index{quaternion!Hurwitz} is a quaternion
whose components are all integers or all half-integers.  The ring
$\JJ$ of Hurwitz quaternions~\cite{csl-Hurwitz} is a maximal
order\index{maximal~order} in the quaternion algebra $\HH(\QQ)$, as is
any of its conjugates (which means that one only has uniqueness up to
conjugacy here). $\JJ$ is given by
\begin{equation}
\begin{split}
    \JJ \, & := \, \bigl\{ (q^{}_0, q^{}_1, q^{}_2, q^{}_3)\in\HH 
          \bigm\vert \text{all } q^{}_i\in\ZZ
    \text{ or all } q^{}_i \in\tfrac{1}{2}+\ZZ \bigr\}  \\[1mm]
    &\hphantom{:}= \, \LL\ts \cup
   \left((\tfrac{1}{2},\tfrac{1}{2},\tfrac{1}{2},
    \tfrac{1}{2})+\LL\right)\! ;
\end{split}   
\end{equation}
compare \cite[Ex.~2.18]{csl-TAO}.  We call a Hurwitz quaternion
$q\in\JJ$ \emph{primitive}\index{quaternion!primitive} if
$\frac{1}{n}q \in \JJ$ with $n\in\NN$ implies $n=1$.  The norm $|q|^2$
of any Hurwitz quaternion is an integer. As quaternions of norm
$|q|^2=2$ play a special role in $\JJ$, we distinguish between
\emph{odd}\index{quaternion!odd} and
\emph{even}\index{quaternion!even} quaternions, where $q\in\JJ$ is
called even or odd depending on whether $|q|^2$ is even or odd. Any
quaternion of norm $|q|^2=2$ can be represented as
$q=(1,1,0,0)u=u'(1,1,0,0)$, where $u,u'$ are unit quaternions. As
the group $\JJ^{\times}$ of unit quaternions has order $24$ and consists
of the quaternions
\[
(\pm 1,0,0,0), (0, \pm 1,0,0), (0,0,\pm 1,0), (0,0,0, \pm 1)
\mbox{ and } (\pm 1,\pm 1,\pm 1,\pm 1),
\]
there are also $24$ quaternions of norm $|q|^2=2$. The latter,
normalised as $\frac{q}{\sqrt{2}}$, together with the units
$u\in\JJ^{\times}$ form a group of order $48$, which is the standard
double cover of the octahedral group $O$.

Recall from \cite{csl-Hurwitz} and \cite[Sec.~2.5.4]{csl-TAO} that
$\JJ$ is a principal ideal ring, which here means that all right
(left) ideals are principal right (left) ideals, and $\JJ$ is also a
maximal order.\index{maximal~order} Thus, for any two right ideals
$a\JJ$ and $b\JJ$, there exist quaternions $g$ and $m$ such that
$g\JJ=a\JJ+b\JJ$ and $m\JJ=a\JJ\cap b\JJ$.  These two quaternions $g$
and $m$ are unique up to multiplication by a unit quaternion from the
right. We call $g$ a \emph{greatest common left divisor}
\index{greatest~common~left~divisor} of $a$ and $b$, and $m$ a
\emph{least common right multiple} \index{least~common~right~multiple}
of $a$ and $b$, in symbols $g=\gcld(a,b)$ and $m=\lcrm(a,b)$. As $g$
and $m$ are unique only up to a unit, these equations only make sense
as a shorthand notation for the corresponding equation of ideals
$g\JJ=\gcld(a,b)\JJ$ or as equations of quaternions that hold up to a
multiplication by a unit quaternion from the right.  In some cases, we
may choose a particular $g$ or $m$. In these cases, the equations
involving them are considered to hold exactly.  Similarly, we define
the \emph{greatest common right divisor} $\gcrd$ and the \emph{least
  common left multiple} $\lclm$.

Similarly, the icosian ring $\II$\index{icosian~ring} is a maximal
order\index{maximal~order} in $\HH(\QQ(\mbox{\small $\sqrt{5}$}\,))$;
we refer to~\cite[Ex.~2.19]{csl-TAO} for the definition. The above
notions can thus analogously be defined for $\II$. One important
difference will emerge from the existence of another maximal order,
$\II'$, which is obtained from~$\II$ by algebraic conjugation in
$\QQ(\mbox{\small $\sqrt{5}$}\,)$, though it is not of the form
$q\ts\II\ts q^{-1}$.  Here, the group of units $\II^{\times}$ is
infinite and isomorphic to $\ZZ[\tau]^{\times} \times \Delta_{H_4}$,
where $\Delta_{H_4}$ is a group of order $120$;
see~\cite[Ex.~2.19]{csl-TAO} for details.

For further properties of maximal orders in quaternion algebras, we
refer to~\cite{csl-Reiner,csl-Vigneras}; see also
\cite[Sec.~4]{csl-BLP96}, where the case of $\HH(K)$ with $K$ a real
algebraic number field is considered in more detail.

\subsection{Tools from analysis}
\label{csl-sec:tool-ana}

In our enumeration problems, we naturally deal with arithmetic
functions $a(m)$, which are functions defined on $\NN$. In many
examples, these functions are
\emph{multiplicative}\index{multiplicative~function}, which means that
$a(mn)=a(m)\, a(n)$ whenever $m$ and $n$ are coprime. Note that,
unless $a\equiv 0$, this implies $a(1)=1$, and $a$ is then completely
determined by its values for $p^n$ with $p$ prime and $n\geq 1$.

In addition, we are often interested in the \emph{summatory
  function}\index{summatory~function}
\[
    A(x)\,=\sum_{m\leq x} a(m)
\]
and its behaviour for large $x$, as this function behaves more
regularly than $a(m)$ itself. This suggests to use generating
functions and to analyse their analytic properties.  In the context of
multiplicative functions (although not restricted to those), a natural
choice for the generating function is a \emph{Dirichlet series} of the
form $F(s)=\sum_{m=1}^{\infty}a(m)\ts m^{-s}$.  Let us recall one
classic result for the case that $a(m)$ is real and non-negative,
which relates $F(s)$ with the asymptotic behaviour of $A(x)$.

\begin{theorem}\label{csl-thm:meanvalues} 
  Let\/ $F(s)$ be a Dirichlet series\index{Dirichlet~series} with
  non-negative coefficients which converges for\/ $\Real (s) > \alpha
  > 0$. Suppose that\/ $F(s)$ is holomorphic at all points of the
  line\/ $\{ \Real (s) = \alpha \}$ except at\/ $s=\alpha$. Here, when
  approaching\/ $\alpha$ from the half-plane to the right of it, we
  assume\/ $F(s)$ to have a singularity of the form\/ $F(s) =
  h(s)/(s-\alpha)^{n+1}$ where\/ $n$ is a non-negative integer, and\/
  $h(s)$ is holomorphic at\/ $s=\alpha$. Then, as\/
  $x\rightarrow\infty$, we have \index{asymptotic~behaviour}
\[
    \qquad\qquad\qquad\;\;
     A(x) \; := \, \sum_{m\leq x} a(m)
          \; \sim \;  \frac{h(\alpha)}{\alpha\cdot n!}
             \; x_{}^{\alpha} \, \bigl(\log(x)\bigr)^{n}  .
    \qquad\qquad\qquad\qed
\]
\end{theorem}

The proof follows easily from Delange's
theorem,\index{Delange~theorem} for instance by taking $q=0$ and
$\omega=n$ in Tenenbaum's formulation; see \cite[Ch.\ II.7, Thm.\
15]{csl-Tenenbaum} and references given there.

\section{Similar sublattices}
\label{csl-sec:ssl}

\subsection{General results}
\label{csl-sec:ssl-gen}

Let us now have a more detailed look at similar sublattices. A
similarity transformation consists of two ingredients, an isometry and
a scaling factor. It thus makes sense to analyse these two parts
independently and to introduce the following notions. We call
\begin{equation}
    \OS(\vG)\, :=\, \{R\in \OG(d,\RR)\mid \exists \ts 
    \alpha\in\RR_{+} \text{ such that } 
    \alpha R \vG \subseteq \vG \}
\end{equation}
the set of all \emph{similarity isometries}\index{isometry!similarity}
of $\vG$. Likewise, we define
\begin{equation}
   \SOS(\vG)\, :=\, \OS(\vG) \cap \ts \SO (d,\RR)
\end{equation}
to be its orientation-preserving part\index{rotation!similarity}. 
The following result is immediate.

\begin{fact}\label{csl-fact:OS-group}
    $\OS(\vG)$ and\/ $\SOS(\vG)$ are subgroups of\/ $\OG(d,\RR)$.   \qed
\end{fact}

One would expect that similar lattices should have related
$\OS$-groups, which is indeed true.

\begin{lemma}\label{csl-lem:OS-sim}
   Let\/ $\vG$ and\/ $\vG'$ be similar lattices 
     with\/ $\vG'=\alpha R \vG$. Then,
\[
      \OS(\vG')\, =\, R\ts \OS(\vG) \ts\ts R^{-1}.
\]
\end{lemma}

\begin{proof}
  If $\alpha \neq 0$, the relation $\OS (\vG) = \OS (\alpha \vG)$ is
  trivial.  The general case follows from the fact that $\beta S \vG
  \subseteq \vG$ holds if and only if $\beta R S \ts R^{-1} \vG'
  \subseteq \vG'$.
\end{proof}

Next, we aim to gain some insight into the scaling factors.  Let us
define
\begin{equation}
  \begin{split}
  \Scal^{}_\vG(R)& \, :=\, \{\alpha\in\RR \mid \alpha R \vG\subseteq \vG \} 
  \quad \text{ and}\\[1mm]
  \scal^{}_\vG(R)&\, :=\, \{\alpha\in\RR \mid \alpha R \vG\sim \vG \}\ts .
  \end{split}
\end{equation}
Note that we have allowed negative values for the scaling factors
here. For an arbitrary but fixed $R\in\OG(d,\RR)$, this ensures that
$\Scal^{}_\vG(R)$ is a $\ZZ\ts$-module and that $\scal^{}_\vG(R)\cup
\{0\}$ is a vector space over the field $\QQ$. As we shall see
shortly, this vector space is one-dimensional if $R\in\OS(\vG)$,
and $\Scal^{}_\vG(R)$ is then a one-dimensional lattice.

We defined $\Scal^{}_\vG(R)$ and $\scal^{}_\vG(R)$ for arbitrary $R\in
\OG(d,\RR)$. However, we are really only interested in the non-trivial
case where $R\in\OS(\vG)$. Clearly, the scaling factor $0$ is always
contained in $\Scal^{}_\vG(R)$, so $\Scal^{}_\vG(R)$ is always
non-empty. By definition, $R\in \OS(\vG)$ if and only if there exists
an $\alpha\in \RR_+$ such that $\alpha R \vG\subseteq \vG$. Hence, we
have the following elementary result.

\begin{fact}[{\cite[Sec.~4]{csl-svenja1}}]\label{csl-lem:scal-nontriv}
   Let\/ $\vG$ be a lattice in\/ $\RR^{d}$ and $R\in\OG(d,\RR)$. Then,
   the following assertions are equivalent.
   \begin{enumerate}\itemsep=2pt
     \item $\Scal^{}_\vG(R)\ne \{ 0 \}$;
     \item $\scal^{}_\vG(R)\ne \varnothing$;
     \item $R\in \OS(\vG)$. \qed
   \end{enumerate}
\end{fact}

One expects that two similar lattices should display the same sets of
scaling factors. Indeed, one has the following result.

\begin{lemma}\label{csl-lem:scal-sim}
  Let\/ $\vG$ and $\vG'$ be similar lattices, with\/ $\vG'=\alpha R
  \vG$. Then,
\[
  \Scal^{}_{\vG'}(S)\, = \,\Scal^{}_\vG(R^{-1}SR) \quad\text{and}\quad
  \scal^{}_{\vG'}(S)\, = \,\scal^{}_\vG(R^{-1}SR).
\]
\end{lemma}
\begin{proof}
  Let $\beta\in\Scal^{}_{\vG'}(S)$, so that $\beta S \vG'\subseteq
  \vG'=\alpha R \vG$.  This is equivalent to $\beta R^{-1}S R \vG
  \subseteq \vG$, from which we infer the first identity. The second
  one follows similarly.
\end{proof}

For a fixed lattice $\vG\subset \RR^{d}$, let us have a closer look at
the elements of $\Scal^{}_\vG(R)$. By basic facts from linear algebra,
we have 
\[
 [\vG:\alpha R\vG\ts ] \, =\, 
 \lvert\det(\alpha R)\rvert \, =\, \alpha^d \ts
  \lvert\det(R)\rvert \, =\, \alpha^d,
\]
whenever $\alpha R$ is a similarity
transformation of $\vG$.  Hence, $\alpha^d$ must be an integer for all
$\alpha \in \Scal^{}_\vG(R)$. More generally, if $\alpha R\vG \sim
\vG$, there exists an integer $m$ such that $m\alpha R\vG$ is a
similar sublattice of $\vG$. Consequently, we have $\alpha^d\in\QQ$
whenever $\alpha \in \scal^{}_\vG(R)$.  We have thus proved the
following.

\begin{lemma}\label{csl-lem:scal-d}
  Let\/ $\vG\subset \RR^{d}$ be a lattice. For any\/ $\alpha \in
  \Scal^{}_\vG(R)$, we have\/ $\alpha^d \in \ZZ$.  Moreover, for any\/
  $\alpha \in \scal^{}_\vG(R)$, we have\/ $\alpha^d \in \QQ$. \qed
\end{lemma}

As a consequence, for any fixed lattice $\vG$, $\Scal^{}_\vG(R)$ is a
discrete and closed set, or in other words, a locally finite
set. Hence, there exists a smallest positive element in
$\Scal^{}_\vG(R)$. This deserves a name.

\begin{definition}
  For any isometry $R\in\OS(\vG)$, the smallest positive element in
  $\Scal^{}_\vG(R)$ is called the
  \emph{denominator}\index{denominator} of $R$, written as
  $\den^{}_\vG(R)$.
\end{definition}

Clearly, one has $\bigl(\den^{}_\vG(R)\bigr)^d\in\NN$. Moreover,
$\den^{}_\vG(R)=1$ is equivalent to $R \vG =\vG$, that is,
$\den^{}_\vG(R)=1$ if and only if $R$ is a symmetry operation of
the lattice $\vG$. In particular, $\den^{}_\vG(\one )=1$.

As $\Scal^{}_\vG(R)$ is a $\ZZ\ts$-module, each integer multiple of
$\den^{}_\vG(R)$ is again an element of $\Scal^{}_\vG(R)$. On the
other hand, each $\alpha\in \Scal^{}_\vG(R)$ must be a multiple of
$\den^{}_\vG(R)$, since otherwise we could find a scaling factor
$\alpha$ with $0 < \alpha < \den^{}_\vG(R)$. This leads to the
following result.

\begin{lemma}\label{csl-lem:scal-den}
  Let\/ $\vG\subset \RR^{d}$ be a lattice. For any isometry\/
  $R\in\OS(\vG)$, we have the relations\/
  $\Scal^{}_\vG(R)=\den^{}_\vG(R)\, \ZZ$ and\/
  $\scal^{}_\vG(R)=\den^{}_\vG(R) \, \QQ^{\times}$.
\end{lemma}

\begin{proof}
  It remains to prove the statement about $\scal^{}_\vG(R)$. By
  definition, $\alpha \in \scal^{}_\vG(R)$ means
  $\alpha R \vG \sim \vG$.  By Lemma~\ref{csl-lem:comm-lat}, there
  exists an $m\in\NN$ such that $m \alpha R \vG \subseteq \vG$, whence
  $m \alpha \in \Scal^{}_\vG(R)=\den^{}_\vG(R)\, \ZZ$ and thus also
  $\scal^{}_\vG(R) \subseteq \den^{}_\vG(R) \, \QQ^{\times}$. On the
  other hand, $\den^{}_\vG(R) R \vG \sim \vG$ and $q\vG \sim \vG$ for
  all $q\in\QQ$ imply that $\alpha R \vG \sim \vG$ for all
  $\alpha \in \den^{}_\vG(R) \, \QQ^{\times}$, which shows that
  $\scal^{}_\vG(R) \supseteq \den^{}_\vG(R) \, \QQ^{\times}$ as well.
\end{proof}

More generally, we have $\scal^{}_\vG(R)=\alpha \ts \QQ^{\times}$ for any
$\alpha \in \scal^{}_\vG(R)$. Note that $\scal^{}_\vG(\one )= \QQ^{\times}$
and $\Scal^{}_\vG(\one )= \ZZ$.

Although we are ultimately more interested in the sets
$\Scal^{}_\vG(R)$, it is worthwhile to discuss $\scal^{}_\vG(R)$ as
these sets are easier to handle.  In particular, we have a natural
group structure on $\{\scal^{}_\vG(R) \mid R\in\OS(\vG)\}$, with the
product of two sets $A$ and $B$ defined in the obvious way as
$AB:=\{\alpha\beta \mid \alpha\in A, \beta\in B\}$, and the inverse of
$A$ given by $A^{-1}=\{ \alpha^{-1} \mid \alpha\in A\}$. The latter is
well defined as $0\not\in A$ whenever $A=\scal^{}_\vG(R)$.

\begin{lemma}\label{csl-lem:scal-group} 
  Let\/ $\vG\subset \RR^{d}$ be a lattice.  For any\/ $R,S\in
  \OS(\vG)$, we have
\[
  \scal^{}_\vG(R)\, \scal^{}_\vG(S)\, =\, \scal^{}_\vG(RS)
  \quad\text{and}\quad
  \scal^{}_\vG(R^{-1})\, =\, \bigl(\scal^{}_\vG(R)\bigr)^{-1}.
\]
\end{lemma}

\begin{proof}
  The group structure is a consequence of the fact that
  commensurateness is an equivalence relation. Alternatively, this
  property also follows from Lemma~\ref{csl-lem:scal-den}, which
  suggests the definition of a natural mapping from the set
  $\{\scal^{}_\vG(R) \mid R\in\OS(\vG)\}$ into the multiplicative
  group $\RR_{+}/\QQ_{+}$, which becomes a group homomorphism with the
  multiplication as defined in this lemma.
\end{proof}

Actually, it is the homomorphism mentioned in the previous proof that
will establish a general connection between SSLs and CSLs, as we shall
discuss in Section~\ref{csl-sec:sslcsl}.

Note that $\{\scal^{}_\vG(R) \mid R\in\OS(\vG)\}$ is an Abelian group,
although $\OS(\vG)$ is not Abelian in general.  In particular, one has
$\scal^{}_\vG(RS)=\scal^{}_\vG(SR)$ even when $RS\ne SR$.
\smallskip

The situation is more involved for the sets $\Scal^{}_\vG(R)$. They do
not form a group, nor even a semigroup, as
$\Scal^{}_\vG(R)\,\Scal^{}_\vG(S)\subseteq \Scal^{}_\vG(RS)$ is
usually a proper inclusion. Nevertheless, we can still extract some
information on the denominator from this inclusion;
compare~\cite{csl-pzsslcsl1,csl-habil}. As
$\den^{}_\vG(R)\ts\den^{}_\vG(S)$ must be in $\scal^{}_\vG(RS)$, it is
an integer multiple of $\den^{}_\vG(RS)$, that is,
\begin{equation}
  \frac{\den^{}_\vG(R)\ts\den^{}_\vG(S)}{\den^{}_\vG(RS)}\,
  \in\,\NN\ts .
\end{equation}
An immediate consequence is that
$\den^{}_\vG(R)\ts\den^{}_\vG(R^{-1})$ is an integer. As
$[\vG: \den^{}_\vG(R) R\vG\ts ]= \bigl(\den^{}_\vG(R)\bigr)^d$ for an
isometry $R\in\OS(\vG)$, we also have
$\bigl(\den^{}_\vG(R)\bigr)^d \vG \subseteq \den^{}_\vG(R) R\vG$, from
which we infer
$\bigl(\den^{}_\vG(R)\bigr)^{d-1} R^{-1} \vG \subseteq \vG$. This
proves
\begin{equation}\label{csl-eq:den-Rinv}
  \frac{\bigl(\den^{}_\vG(R)\bigr)^{d-1}}{\den^{}_\vG(R^{-1})}\,
  \in\,\NN\ts .
\end{equation}

\begin{remark}\label{csl-rem:bound-den}
  Formula \eqref{csl-eq:den-Rinv} implies 
  $\den^{}_\vG(R^{-1})\leq \bigl(\den^{}_\vG(R)\bigr)^{d-1}$.  In
  fact, this upper bound is sharp. As an example, we consider the
  $\ZZ\ts$-span of the vectors $\xi^{i-1}e^{}_i$, where $\xi$ is the
  (positive) $d$-th root of a positive integer and
  $\{e^{}_{1\vphantom{d}},\ldots,e^{}_d\}$ is an orthonormal basis of
  $\RR^d$.  Let $R$ be the rotation that maps $e^{}_i$ onto
  $e^{}_{i+1}$ for $1\leq i \leq d-1$ and $e^{}_d$ onto
  $e^{}_1$. Then, $\den_\vG(R)=\xi$ and $\den_\vG(R^{-1})=\xi^{d-1}$.
  \exend
\end{remark}

The example in Remark~\ref{csl-rem:bound-den} also shows that
$\den^{}_{\vG}(R)$ and $\den^{}_{\vG}(R^{-1})$ generally differ if
$d\geq 3$.  However, in two dimensions, they always agree.

\begin{corollary}
  For any planar lattice\/ $\vG$, one has\/
  $\den^{}_\vG(R^{-1})=\den^{}_\vG(R)$.
\end{corollary}
\begin{proof}
From Eq.~\eqref{csl-eq:den-Rinv}, we infer
\[
\frac{\den^{}_\vG(R)}{\den^{}_\vG(R^{-1})}\, \in\, \NN
\quad \text{as well as} \quad
\frac{\den^{}_\vG(R^{-1})}{\den^{}_\vG(R)}\, \in\, \NN\ts ,
\]
the latter by symmetry. Together, 
they imply  $\den^{}_\vG(R^{-1})=\den^{}_\vG(R)$.
\end{proof}

It is quite useful to understand the relationship between commensurate
lattices in more detail. Using the fact that commensurateness is an
equivalence relation, we can prove the following result.

\begin{lemma}
  If\/ $\vG$ and\/ $\vG'$ are two commensurate lattices in\/
  $\RR^{d}$, one has\/ $\OS(\vG)=\OS(\vG')$ as well as\/
  $\ts\scal_\vG(R)=\scal_{\vG'}(R)$.
\end{lemma}

\begin{proof}
  By Fact~\ref{csl-lem:scal-nontriv}, $R\in \OS(\vG)$ if and only if
  there exists an $\alpha\ne 0$ such that $\alpha R \vG \subseteq \vG$.
  Commensurateness guarantees that there are $m,n\in \NN$ such that
  $m\vG \subseteq \vG'$ and $n\vG' \subseteq \vG$. Thus,
  $\alpha R \vG \subseteq \vG$ implies
  \begin{equation}\label{csl-eq:os-comm}
    mn \alpha R \vG' \,\subseteq\, m\alpha R \vG \,\subseteq\, 
   m\vG \,\subseteq\, \vG',
  \end{equation}
  which gives $\OS(\vG)\subseteq\OS(\vG')$. By symmetry, we conclude
  $\OS(\vG)=\OS(\vG')$. 
  Moreover, Eq.~\eqref{csl-eq:os-comm} shows that $\den_{\vG}(R)$ and
  $\den_{\vG'}(R)$ differ only by a factor $q\in\QQ$, whence one has
  $\scal_\vG(R)=\scal_{\vG'}(R)$ by Lemma~\ref{csl-lem:scal-den}.
\end{proof}

Of course, we cannot expect the sets $\Scal_\vG(R)$ and
$\Scal_{\vG'}(R)$ to be equal. However, as we can sandwich $\vG'$
between appropriately scaled copies of $\vG$, we can derive lower and
upper bounds as follows.

\begin{proposition}\label{csl-prop:comp-Scal}
  Let\/ $\vG^{\prime}$ be a sublattice of\/ $\vG$ of index\/ $m$. Then,
\[
 m\ts\Scal_\vG(R)\,\subseteq\,\Scal_{\vG'}(R)\,\subseteq\,
 \badfrac{1}{m}\ts\Scal_\vG(R)\ts .
\]
Moreover, one has \index{denominator}
\[
   \frac{m\ts\den^{}_\vG(R)}{\den^{}_{\vG'}(R)} \,\in\, \NN
   \quad \text{and} \quad
   \frac{m\ts\den^{}_{\vG'}(R)}{\den^{}_\vG(R)} \,\in\, \NN\ts .
\]
\end{proposition}

\begin{proof}
  If $\alpha\in\Scal_\vG(R)$, then
\[
    \alpha R \vG' \,\subseteq\, 
    \alpha R \vG \subseteq \vG 
    \,\subseteq\, \tfrac{1}{m} \vG'
\]
  shows that $m\alpha \in \Scal_{\vG'}(R)$. Similarly, 
  $\alpha\in\Scal_{\vG'}(R)$ implies
\[
    \alpha R\ts  m\vG \, \subseteq \, 
    \alpha R \vG' \subseteq \vG' 
    \, \subseteq \, \vG \ts,
\]
  which proves $\Scal_{\vG'}(R)\,\subseteq\,  \frac{1}{m}\ts\Scal_\vG(R)$.
  The statement about the denominators now follows from the explicit
  expressions for $\Scal_\vG(R)$ from Lemma~\ref{csl-lem:scal-den},
  or by choosing $\alpha$ to be the denominator in the equations above.
\end{proof}

Let us add that, more generally, one can show that
\[
   m^{}_{1} m^{}_{2} \Scal_{\vG} (R) \, \subseteq \,
   \Scal_{\vG'} (R) \, \subseteq \,
   \badfrac{1}{m^{}_{1} m^{}_{2}} \, \Scal^{}_{\vG} (R)
\]
whenever $m^{}_{1} \vG \subseteq \vG'$ and
$m^{}_{2} \vG' \subseteq \vG$.

Let us conclude these general considerations with a remark on the dual
lattice, defined as $\vG^{*}=\{x\in\RR^{d}\mid \ip{x}{y} \in\ZZ
\text{ for all } y\in\vG\}$; compare \cite[Sec.~3.1]{csl-TAO}.

\begin{lemma}
  If\/ $\vG^*$ is the dual\index{lattice!dual} lattice of\/ $\vG$, one
  has\/ $\OS(\vG)=\OS(\vG^{*})$ together with
\[
    \Scal_{\vG^{*}}(R)\, =\, \Scal_\vG(R^{-1})\ts .
\]
In particular, $\den^{}_{\vG^{*}}(R)=\den^{}_\vG(R^{-1})$.
\end{lemma}
\begin{proof}
  As $\OS(\vG)$ is a group, $R\in\OS(\vG)$ if and only if
  $R^{-1}\in\OS(\vG)$. The latter holds if and only if there is an
  $\alpha\in\RR_{+}$ such that $\alpha R^{-1} \vG \subseteq \vG $.  By
  the definition of the dual lattice, this is equivalent to $\alpha
  \langle x| R^{-1}y\rangle \in\ZZ$ for all $x\in\vG^{*}$ and
  $y\in\vG$.

  Now, $\alpha \langle Rx| y\rangle = \alpha \langle x|
  R^{-1}y\rangle$ shows that $\alpha R \vG^{*}\subseteq \vG^{*}$ holds
  if and only if $\alpha R^{-1} \vG \subseteq \vG $, which proves
  $\OS(\vG)=\OS(\vG^{*})$.  On the other hand, this equation shows
  that $\alpha\in \Scal_{\vG^{*}}(R)$ if and only if $\alpha\in
  \Scal_\vG(R^{-1})$, which completes the proof.
\end{proof}

\subsection{Two dimensions}\label{csl-sec:two-dim}

Let us consider some concrete examples. We start in two dimensions,
where we can make use of the field of complex numbers to characterise
$\SOS(\vG)$ completely. Here, any orientation-preserving similarity
transformation can be represented by complex multiplication, and it
turns out that the semigroup of similarity transformations then forms
a ring, which we call the \emph{multiplier
  ring}\index{multiplier~ring}. The latter is denoted by
$\mul(\vG)$. Actually, there are only two cases. Either
$\SOS(\vG)=\{\pm 1\}$, or equivalently $\mul(\vG)=\ZZ$, in which case
we call $\vG$ \emph{generic}, or $\mul(\vG)=\cO$ is an
order\footnote{{\ts}Note that the symbol $\cO$ occurs with two different
  meanings in this chapter, namely for asymptotic estimates and for
  orders (in the algebraic sense introduced earlier). Since the
  meaning will always be clear from the context, we stick to this
  widely used notation.}  in an imaginary quadratic number
field;\index{number~field} compare~\cite{csl-BSZ-sim} and
\cite{csl-Hardy,csl-Borevic-Sh,csl-Cox} for general
background.\vspace{1mm}

For a more precise formulation, we employ the \emph{similarity class}
of a given lattice $\vG$, denoted by $\Sim(\vG)$, which consists of all
lattices $\vG' \sim \vG$.

\begin{theorem}[{\cite[Prop.~2.3 and 
Thm.~2.6]{csl-BSZ-sim}}]\label{csl-theo:nongenmul}
  If\/ $\vG\subset\RR^{2}$ is a non-generic lattice, its multiplier
  ring\/ $\mul(\vG)$ is an order in an imaginary quadratic
  field. Explicitly, if\/ $\vG \in \Sim \bigl( \langle 1, \tau
  \rangle^{}_\ZZ \bigr)$ with\/ $\tau \in \CC \setminus \RR$ is a
  non-generic lattice, the number\/ $\tau$ is algebraic of degree\/
  $2$ over\/ $\QQ$, and one has
  \[
    \mul ( \vG) \, = \, \langle 1, s \tau \rangle^{}_{\ZZ}
  \]
  for some non-zero integer\/ $s$.

  Moreover, if\/ $K$ is the field of fractions of\/ $\cO:=\mul(\vG)$,
  one has\/ $\mul (\vG) = \cO = \mul (\cO)$. In particular,
\[
\begin{split}
   \SOS (\vG) \, & = \, \SOS (\cO) \, = \, \SOS (\cO_{K}) \\[1mm]
   & = \, \Bigl\{ \myfrac{w}{\lvert w \rvert} \;\big|\; 0\ne w\in \cO \Bigr\}
    \, = \, \Bigl\{ \myfrac{w}{\lvert w \rvert} \;\big|\; 
     0\ne w\in \cO_{K} \Bigr\} ,
\end{split}
\]
  where\/ $\cO_{K}$ is the maximal order\index{maximal~order} of\/ $K$
  and thus contains\/ $\cO$. \qed
\end{theorem}

Note that the group $\SOS(\vG)$ is the same for all lattices in
$\Sim(\vG)$, which follows via Lemma~\ref{csl-lem:OS-sim} from the
fact that the group $\SO(2)$ is Abelian.  Actually, it is the same for
all lattices whose multiplier ring has the same field of fractions,
although the corresponding multiplier rings usually
differ.\index{multiplier~ring}

\begin{example}\label{csl-ex:ssl-sq}
  For the square lattice, which we write as $\ZZ^{2}=\ZZ[\ii]$, we have
  $\mul\bigl(\ZZ[\ii]\bigr)=\ZZ[\ii]$.  This implies
  \[
   \SOS \bigl(\ZZ[\ii]\bigr) \, 
    = \, \Bigl\{ \myfrac{z}{\lvert z \rvert} \;\big|\; 
   0\ne z\in \ZZ[\ii] \Bigr\} 
   \, \simeq\,  C_8 \times \ZZ^{(\aleph_0)} ,
  \]
  where $\ZZ^{(\aleph_0)}$ denotes the countably infinite sum of
  infinite cyclic groups (in contrast to the infinite product).  Here,
  the group $C_8$ is generated by $\frac{1+\ii}{\sqrt{2}}$ and contains
  all units of $\ZZ[\ii]$, whereas a full set of generators of
  $\ZZ^{(\aleph_0)}$ is the set
  $\bigl\{\frac{\pi_{\nts p}}{\sqrt{p}} \mid p\equiv 1 \bmod{4}
  \bigr\}$,
  where $\pi_{\nts p}$ is a Gaussian
  prime\index{Gaussian~integer}\index{prime!Gaussian} such that
  $\pi_{\nts p} \ts\overline{\pi_{\nts p}} = p$. Note that for any $p$
  only one prime of the pair $\pi_{\nts p}$, $\overline{\pi_{\nts p}}$
  is needed, as we have
  $\bigl(\frac{\pi_{\nts p}}{\sqrt{p}}\bigr)^{-1}=
  \frac{\overline{\pi_{\nts p}}}{\sqrt{p}}$.  \exend
\end{example}

The situation is particularly nice if the multiplier ring is a
PID.\index{principal~ideal~domain} In this case, all ideals are
similar sublattices and the situation is completely analogous to that
of the square lattice example in Section~\ref{csl-sec:square}, so we
can write down the generating function explicitly.  This happens for a
finite number of cases only; compare Table~\ref{csl-tab:max}.

\begin{lemma}[{\cite[Thm.~7.30]{csl-Cox}}]\label{csl-lem:CNone}
  There are precisely nine imaginary quadratic
  fields\index{number~field} with class number\/ $1$, meaning
  their maximal orders being PIDs. These are the fields\/
  $K=\QQ(\omega^{}_{0})$ for
  \[
   \omega^{}_{0} \, \in \, \big\{ \tfrac{1+\ii\sqrt{3}}{2}, \ii,
   \tfrac{1+\ii\sqrt{7}}{2}, \ii \sqrt{2}, \tfrac{1+\ii\sqrt{11}}{2},
   \tfrac{1+\ii\sqrt{19}}{2}, \tfrac{1+\ii\sqrt{43}}{2},
   \tfrac{1+\ii\sqrt{67}}{2}, \tfrac{1+\ii\sqrt{163}}{2} \big\} ,
  \]
  which are fields of discriminant\index{discriminant}
\[  
  d_{K}\in  \{-3,-4,-7,-8,-11,-19,-43,-67,-163\} \ts .
\]  
  Here, the maximal order of\/ $K$ is\/ $\cO_K =
  \ZZ[\omega^{}_{0}]$, while one has\/ 
  $\QQ(\omega^{}_{0}) = \QQ(\sqrt{d_{K}}\,)$. \qed
\end{lemma}

\begin{table}
  \caption{Norm forms for the nine maximal orders $\cO_{K}$ of class
    number~$1$ in imaginary quadratic number fields,
    labelled with the
    field discriminant $d_{K}$; see Lemma~\ref{csl-lem:CNone} and
    Remark~\ref{csl-rem:CNone} for details.}
\label{csl-tab:max}
\renewcommand{\arraystretch}{1.25}
\begin{tabular}{|c|c||c|c||c|c|}
\hline
$d_{K}$ &  norm form & 
$d_{K}$ &  norm form & 
$d_{K}$ & norm form  \\ \hline
$-3$   &  $x^2 + xy + y^2$ & 
$-8$   &  $x^2 + 2 y^2$ &
$-43$  &  $x^2 + xy + 11 y^2$ 
\\
$-4$   &  $x^2 + y^2$ & 
$-11$  &  $x^2 + xy + 3 y^2$ &
$-67$  &  $x^2 + xy + 17 y^2$
 \\
$-7$   &  $x^2 + xy + 2 y^2$ &
$-19$ &  $x^2 + xy +  5 y^2$ &
$-163$ & $x^2 + xy + 41 y^2$\\
\hline
\end{tabular}
\end{table}

\begin{remark}\label{csl-rem:CNone}
  In the nine cases of Lemma~\ref{csl-lem:CNone}, the possible indices
  of the similar sublattices of $\cO_K$ are precisely those positive
  integers that can be represented by the corresponding norm forms
  listed in Table~\ref{csl-tab:max}. As a consequence, the Dirichlet
  series generating function for the number of SSLs of a given
  index\index{index} is the zeta\index{zeta~function!Dedekind}
  function of $\cO_K$, which is the Dedekind zeta function
  $\zeta^{}_{K}$ of the quadratic field $K$.\exend
\end{remark}

Let us recall some properties of the Dedekind zeta function
$\zeta^{}_{K}$. The latter is known \cite{csl-Zagier} to factorise as
\begin{equation} \label{csl-quadratic-zeta}
   \zeta^{}_{K} (s) \, = \, \zeta(s) \, L(s,\chi) \ts ,
\end{equation}
where $L(s,\chi)$ is the \mbox{$L\ts$-series} of the
non-trivial character $\chi = \chi^{}_{d_{K}}$ of the quadratic field
$K$.  The latter is a totally multiplicative arithmetic function and
thus given by $\chi^{}_{d_{K}} (1) =1$ together with its values at the
rational primes (that is, primes in 
$\ZZ\subset\QQ$),\index{Dirichlet~character}\vspace*{-2mm}
\[
    \chi^{}_{d_{K}} (p)\, =\, \begin{cases}
    0 \ts , & p \mid d_{K}\ts ,  \\[1mm]
    \bigl( \frac{d_{K_{\vphantom{a}}}}{p}\bigr),
      & 2\ne p \nmid d_{K} \ts ,\\[1mm]
    \bigl( \frac{d_{K_{\vphantom{a}}}}{2}\bigr), 
      & p=2\nmid d_{K}\ts .
    \end{cases}
\]
Here, $\bigl(\frac{m}{p}\bigr)$ and
$\bigl(\frac{m}{2}\bigr)$ denote the
Legendre\index{Legendre~symbol} and the
Kronecker\index{Kronecker~symbol} symbol, respectively, the latter 
defined as\vspace*{-1mm}
\[ 
    \left(\myfrac{m}{2}\right)\, = \, \begin{cases}
    1 \ts , & m \equiv \pm 1 \; (8) \ts , \\
    -1\ts , & m \equiv \pm 3 \; (8) \ts , \\
    0\ts , & m \equiv 0 \; (2)\ts .
    \end{cases} 
\]
This permits a direct calculation of the zeta function via its
Euler\index{Euler~product} product, as the character $\chi(p)$ takes
only the values $0$, $-1$, or $1$, depending on whether the rational
prime $p$ ramifies, is inert, or
splits\index{prime!inert}\index{prime!ramified}\index{prime!splitting}
in the extension from $\QQ$ to $K$. The general formula reads
\begin{equation}\label{csl-quadratic-zeta-two}
\begin{split}
      \zeta^{}_{K} (s) \, &  = \,\prod_{p\in\PP}  
      \myfrac{1}{(1-p^{-s})(1-\chi(p)\ts p^{-s})} \\[1mm]
      & = \!\!
      \prod_{\substack{p\in\PP \\ \chi(p)=0}} \! \myfrac{1}{1-p^{-s}} \!
      \prod_{\substack{p\in\PP \\ \chi(p)=-1}} \! \!\myfrac{1}{1-p^{-2s}} 
      \prod_{\substack{p\in\PP \\ \chi(p)=1}} \! \myfrac{1}{(1-p^{-s})^2} \ts ,
\end{split}
\end{equation}
where $\PP$ denotes the set of rational primes.

The result on the generating functions now reads as follows.

\begin{theorem}[{\cite[Prop.~5.2]{csl-BSZ-sim}}]\label{csl-theo:genfun}
  Let\/ $K$ be any of the nine imaginary quadratic number
  fields\index{number~field} of Lemma\/~$\ref{csl-lem:CNone}$, with\/
  $p^{}_{\sf ram}$ its ramified prime,\index{prime!ramified} which
  is the unique rational prime that divides\/ $d_{K}$. The Dirichlet
  series\index{Dirichlet~series} generating function for the number of
  SSLs of\/ $\cO_{K}$ of a given index is\/
  $D^{}_{\cO_{K}} (s) = \zeta^{}_{K} (s)$ with the Dedekind
  zeta\index{zeta~function!Dedekind} function of\/ $K$ according to
  Eq.~\eqref{csl-quadratic-zeta-two}.

Moreover, the generating function for the primitive
SSLs of\/ $\cO_{K}$ is
\[
     D^{\sf pr}_{\cO_{K}} (s) 
     \, =\,  \frac{D^{}_{\cO_{K}} (s) }{\zeta(2s)}
     \, = \, (1+ p_{\sf ram}^{-s})
     \prod_{p\; \mathrm{splits}} 
     \frac{1+p^{-s}}{1-p^{-s}} \ts ,
\]
where the product runs over all rational primes\/ $p$ that split in
the extension to\/ $K$.\index{prime!splitting} The same generating
function also applies to any other planar lattice\/
$\vG \in \Sim( \cO_K ) $.  \qed
\end{theorem}

In addition to the PIDs mentioned above, there are four additional
(non-maximal) orders with class number~$1$,\index{class~number} which
we have summarised in Table~\ref{csl-tab:non-max}.  Their ideals are
closely related to the ideals of their corresponding maximal
orders. As a consequence, the generating functions $D_{\cO}^{} (s)$
and $D_{\cO}^{\mathsf{pr}} (s)$ possess an Euler product. Their Euler
factors are the same as those of the corresponding maximal order,
except for the Euler factors corresponding to the primes that divide
the \emph{conductor}\index{conductor} $f:=[\cO_K:\cO]$. These special
Euler factors can be calculated explicitly.

\begin{theorem}[{\cite[Sec.~5.2]{csl-BSZ-sim}}]\label{csl-theo:non-max}
  Let\/ $\cO$ be one of the four non-maximal orders of class number\/
  $1$ in imaginary quadratic number fields\index{number~field} as
  given in Table~$\ref{csl-tab:non-max}$.  The sublattice counting
  functions are multiplicative, which implies that their generating
  functions\/ $D_{\cO}^{} (s)$ and\/ $D_{\cO}^{\mathsf{pr}} (s)$ have
  an Euler product expansion.\index{Euler~product} In particular,
  \begin{align*}
    D_{\ZZ[\ii\sqrt{3}\,]}^{\mathsf{pr}} (s) \, & = \,
      \Bigl( 1+ \myfrac{2}{4^s} \Bigr) \Bigl( 1 + \myfrac{1}{3^s} \Bigr)
      \prod_{p\ts\equiv 1\, (3)} \frac{1+p^{-s}}{1-p^{-s}} \ts ,\\[1mm]
    D_{\ZZ[2\ts\ii]}^{\mathsf{pr}} (s) \, & = \,
      \Bigl( 1+ \myfrac{1}{4^s} + \myfrac{2}{8^s} \Bigr) 
      \prod_{p\ts\equiv 1\, (4)} \frac{1+p^{-s}}{1-p^{-s}}\ts , \\[1mm]
    \qquad D_{\ZZ[\frac{1}{2} (1 + \ii \ts 3 \sqrt{3}\,)]}^{\mathsf{pr}} (s) \, & = \,
      \Bigl( 1+ \myfrac{2}{9^s} + \myfrac{3}{27^s} \Bigr) 
      \prod_{p\ts\equiv 1\, (3)} \frac{1+p^{-s}}{1-p^{-s}}\ts , \\[1mm]
    D_{\ZZ[\ii\sqrt{7}\,]}^{\mathsf{pr}} (s) \, & = \,
      \Bigl( 1- \myfrac{2}{2^s} + \myfrac{2}{4^s} \Bigr) 
      \Bigl( 1 + \myfrac{1}{7^s} \Bigr)\!
      \prod_{p\ts\equiv 1,2,4 \, (7)} \frac{1+p^{-s}}{1-p^{-s}} \ts .\qquad
  \end{align*}
Again, the possible indices of the SSLs are precisely those positive integers
that can be represented by the corresponding norm forms, which we have listed
in Table~$\ref{csl-tab:non-max}$. \qed
\end{theorem}

\begin{table}
  \caption{Basic data for the four non-maximal orders of class number~$1$
    in imaginary quadratic number fields,
    labelled with their discriminant $D$;
    see Theorem~\ref{csl-theo:non-max} for details.} 
\label{csl-tab:non-max}
\renewcommand{\arraystretch}{1.25}
\begin{tabular}{|c|c|c|c|c|c|}
\hline
$D$ & $K$ & $\cO$ & norm form & $p\ts | D$ & conductor \\ \hline
$-12$ & $\QQ(\ii\sqrt{3}\,)$ & $\ZZ[\ii\sqrt{3}\,]$ & 
$x^2 + 3 y^2$ & $2,3$ & $2$ \\
$-16$ & $\QQ(\ii)$ & $\ZZ[2\ts\ii]$ & $x^2 + 4 y^2$ & $2$ & $2$ \\
$-27$ & $\QQ(\ii\sqrt{3}\,)$ & $\ZZ[\tfrac{1}{2}
                (1 + \ii \ts 3 \sqrt{3}\,)]$ & 
$x^2 + xy +  7 y^2$ & $3$ & $3$ \\
$-28$ & $\QQ(\ii\sqrt{7}\,)$ & $\ZZ[\ii\sqrt{7}\,]$ & 
$x^2 + 7 y^2$ & $2,7$ & $2$\\
\hline
\end{tabular}
\end{table}

The situation is more involved for class numbers greater than $1$,
where the existence of non-principal ideals complicates the treatment.
In general, the counting functions are no longer multiplicative,
because a product of non-principal ideals may be principal. As a
consequence, not much is known in these cases. However, there is still
one situation that allows further treatment, namely when the
discriminant\index{discriminant} $D$ is one of Euler's
\emph{convenient numbers}, in which case the ideal class group is an
Abelian $2$-group.  The latter implies that we have a natural binary
grading on the ideals, depending on whether they are principal or
not. If the order under investigation is still principal, one can
derive the generating function from the zeta function;
compare~\cite{csl-BSZ-sim}.

\begin{example}
Let us consider $\cO=\ZZ [\ii\sqrt{6}\,]$, where\index{Dirichlet~series}
\[
   D^{\textsf{pr}}_{\ZZ [\ii\sqrt{6}\,]} (s) \, = \prod_{p\equiv 1,7 \, (24)}
   \frac{1+p^{-s}}{1-p^{-s}} \;
   \sum_{m=1}^{\infty} \frac{ b(m)} {m^s}
\]
with $b(1)=1$ and $b(m)=0$ if $p|m$ for some $p\equiv 1,7,13,17,19,23
\bmod 24$. If $m$ is an integer of the form $m=2^{\alpha}\ts 3^{\beta}
\prod_{p\equiv 5,11 \; (24)} p^{\ell_p}$ with $\alpha, \beta\in
\{0,1\}$ and $\ell_p \in \NN_0$, with only finitely many of them $\ne
0$, we have
\[
   b(m) \, = \, \bigl(1 + (-1)^{\alpha + \beta + \sum \ell_p} 
    \bigr)^{\mathrm{card} \{ p > 3 \ts\mid\ts \ell_p \ne\ts 0 \}}.
\]
Obviously, $D^{\textsf{pr}}_{\ZZ [\ii\sqrt{6}\,]} (s)$ possesses no Euler 
product representation.
\exend
\end{example}

\begin{example}
As further cases, let us mention the generating functions for the primitive
SSLs of $\cO=\ZZ [3\ts\ii]$,
\begin{align}
  D^{\sf pr}_{\ZZ[3\ts\ii]} (s) \, & = \!\!
     \sum_{\substack{m\geqslant 1 \\ m\equiv 1\, (3)}} \!\!
     (1+9^s)\, \frac{a^{\sf pr}_{\minisquare} (m)}{(9\ts m)^s} \; + \!\!
     \sum_{\substack{m\geqslant 1 \\ m\equiv 2\, (3)}} \!\!
     \frac{2\ts\ts a^{\sf pr}_{\minisquare} (m)}{(9\ts m)^s}\ts , \\
\intertext{and of $\cO=\ZZ [5\ts\ii]$,}
  D^{\sf pr}_{\ZZ[5\ts\ii]} (s) \, & = \!\!
     \sum_{\substack{m\geqslant 1 \\ m\equiv \pm 1\, (5)}} \!\!
     (1+25^s)\, \frac{a^{\sf pr}_{\minisquare} (m)}{(25\ts m)^s} \; + \!\!\!
     \sum_{\substack{m\geqslant 1 \\ m\equiv 0,\pm 2\, (5)}} \!\!\!
     \frac{2\ts\ts a^{\sf pr}_{\minisquare} (m)}{(25\ts m)^s}\ts ,
\end{align}
where $a^{\sf pr}_{\minisquare} (m)$ is the\index{Dirichlet~series}
number of primitive SSLs of index $m$ of the square lattice; compare
Section~\ref{csl-sec:square}.  \exend
\end{example}

Let us stay in two dimensions a little longer and discuss some
$\ZZ\ts$-modules with $N$-fold rotational symmetry. As this works the
same way as for lattices, we include their discussion here, but see
Section~\ref{csl-sec:ssm} for the general theory behind it.  In
particular, we consider the ring $\ZZ[\xi_n]$ of cyclotomic integers,
where $\xi_n$ is a primitive $n$th root of unity. If
\begin{equation}\label{csl-pid-cyclo}
\begin{split}  
       &n \, \in \, \{3, 4, 5, 7, 8, 9, 11, 12, 13, 15, 16, 17, 
                      19, 20, 21, 24,\\
        & \qquad\quad  25, 27, 28, 32, 33, 35, 36, 40, 44, 45, 
                      48, 60, 84\}\ts  ,
\end{split}    
\end{equation}
the ring $\ZZ[\xi_n]$ is a PID,\index{principal~ideal~domain} which
means that the similar submodules are precisely the ideals of
$\ZZ[\xi_n]$; compare~\cite{csl-Wash,csl-BG2,csl-TAO}. Using the
terminology from above, this implies that $\ZZ[\xi_n]$ is its own
multiplier ring.  Note that $n=3$ and $n=4$ correspond to the
hexagonal and the square lattice, respectively; compare
\cite[Ex.~2.15]{csl-TAO}.  More generally, $\ZZ[\xi_n]$ is a
$\ZZ\ts$-module of rank $d=\phi(n)$, which is larger than $2$ in the
remaining cases; see~\cite[Rem.~3.7 and Ex.~2.16]{csl-TAO}.  Here,
$\phi$ denotes Euler's totient function. In particular, $\ZZ[\xi_n]$
has $N$-fold rotational symmetry, with $N=\lcm(n,2)$.

As mentioned above, the similar submodules are precisely the
non-trivial ideals of $\ZZ[\xi_n]$, which means that the generating
function for the similar submodules of $\ZZ[\xi_n]$ is given
by~\cite{csl-BG2}\index{Dirichlet~series}
\begin{equation}
   \Phi^{}_{\ZZ[\xi_n]}(s) \, =\,  \zeta^{}_{\ts\QQ(\zeta_n)}(s) \, :=\,
   \sum_{\mathfrak{a}} \myfrac{1}{\norm(\mathfrak{a})^s}\ts ,
\end{equation}
where the sum runs over all non-trivial ideals $\mathfrak{a}$ of
$\ZZ[\xi_n]$ and 
\[
  \norm(\mathfrak{a})\, :=\, [\ZZ[\xi_n]: \mathfrak{a}]
\]
denotes the norm of $\mathfrak{a}$. As $\ZZ[\xi_n]$ is a PID for all
$n$ from Eq.~\eqref{csl-pid-cyclo}, the counting function for the
ideals of fixed index\index{index} is multiplicative. This, in turn,
means that $\Phi^{}_{\ZZ[\xi_n]}(s)$ has an Euler product
expansion~\cite{csl-BG2}\index{Euler~product}
\begin{equation}
  \Phi^{}_{\ZZ[\xi_n]}(s)\, =\, \prod_{p \in \PP} E_n(p^{-s})\ts .
\end{equation}
The Euler factors $E_n(p^{-s})$ are of the form
\begin{equation}
  E_n(p^{-s}) \, = \, \myfrac{1}{(1-p^{-\ell s})^m}
  \, = \sum_{j=1}^\infty \binom{j+m-1}{m-1}\myfrac{1}{(p^s)^{\ell j}}
\end{equation}
where $m$ and $\ell$ are certain integers that depend on $p$ and $n$.
If $p$ is coprime to $n$, then $\ell$ is the residue class degree of
$p$, which is the smallest integer $\ell$ such that
$p^\ell \equiv 1 \bmod{n}$, compare \cite[Thm.~2.13]{csl-Wash} and
\cite{csl-BG2}. The integer $m$ is determined by $m\ts
\ell=\phi(n)$.
If $p$ divides $n$ ($p$ is a ramified prime\index{prime!ramified} in
this case), we write $n= r\ts p^t$, where $p^t$ is the maximal power
dividing $n$, so that $r$ is the $p\ts$-free part of $n$. The integers
$\ell$ and $m$ are now calculated by replacing $n$ by $r$ in the
equations above, where $\ell$ is the smallest integer such that
$p^\ell \equiv 1 \bmod{r}$ and $m\ell=\phi(r)$; compare the remarks
after \cite[Thm.~2.13]{csl-Wash} as well as \cite{csl-BG2}.  Explicit
values for $m$ and $\ell$, for all cases of Eq.~\eqref{csl-pid-cyclo},
can be found in~\cite[Tables 1 and 2]{csl-BG2}.

\begin{example}
  Let us take a closer look at $\ZZ[\xi^{}_{5}]=\ZZ[\xi^{}_{10}]$. The
  only ramified prime is $5=\varepsilon (1-\xi^{}_{5})^4$, where
  $\varepsilon= -\xi_5^2/\tau^2$ is a unit in $\ZZ[\xi^{}_{5}]$. In
  terms of ideals, this means $(5)= (1-\xi^{}_{5})^4$. This gives
  $\ell=m=1$ for $p=5$.  In addition, we get $\ell=1$, $m=4$ for
  $p\equiv 1 \bmod 5$, $\ell=2$, $m=2$ for $p\equiv -1 \bmod 5$, and
  $\ell=4$, $m=1$ for $p\equiv \pm 2 \bmod 5$. Thus, we obtain the
  generating function\index{Dirichlet~series}
\[
\begin{split}
  \Phi_{\ZZ[\xi_5]}^{}(s)\, & = \, \myfrac{1}{1-5^{s}}
  \prod_{p\equiv 1\, (5)} \myfrac{1}{(1-p^{s})^4}
  \prod_{p\equiv -1\, (5)} \myfrac{1}{(1-p^{2s})^2}
  \prod_{p\equiv \pm 2\, (5)} \myfrac{1}{1-p^{4s}}  \\[1mm]
  & =\,  1 + \myfrac{1}{5^s} + \myfrac{4}{11^s} + \myfrac{1}{16^s} + 
     \myfrac{1}{25^s} + \myfrac{4}{31^s} + \myfrac{4}{41^s} + 
    \myfrac{4}{55^s} + \myfrac{4}{61^s}  \\[1mm]
     & \qquad + \myfrac{4}{71^s} + \myfrac{1}{80^s} + \myfrac{1}{81^s} 
         + \myfrac{4}{101^s} + \myfrac{10}{121^s} + \myfrac{1}{125^s}
          + \myfrac{4}{131^s} 
          + \ldots
\end{split}
\]
for this case.
\exend
\end{example}

\subsection{Higher dimensions}
\label{csl-sec:ssl-higher}

Let us continue with lattices in higher dimensions. We concentrate on
\emph{rational} lattices\footnote{\label{csl-foot:rat-lat}{\ts}More
  generally, one calls a lattice $\vG$
  rational\index{lattice!rational} if there exists an $\alpha>0$ such
  that $\langle x\ts |\ts y\rangle\in\QQ$ for all $x,y\in\alpha\vG$.
  In this section, we only use the more restrictive definition.}
here, that is, on lattices $\vG$ for which the inner products satisfy
$\langle x\ts |\ts y\rangle\in\QQ$ for all $x,y\in\vG$. For all
scaling factors $\alpha\in\Scal_\vG(R)$, one then has
$\alpha^2\in\QQ$. By an application of Lemma~\ref{csl-lem:scal-d}, we
may conclude that $\alpha^2\in\ZZ$. Moreover, one obtains the stronger
condition $\alpha\in\ZZ$ in odd dimensions, again by
Lemma~\ref{csl-lem:scal-d}. This gives the following result.

\begin{fact}
  For a rational lattice\index{lattice!rational}\/ 
  $\vG\subset\RR^{d}$ with\/ $d$ odd, the
  possible indices of SSLs are exactly the integers of the form\/
  $n^d$ with\/ $n\in\NN$. \qed
\end{fact}

Thus, the question for the possible indices is answered in this case,
and we may proceed with lattices in even dimension, say $d=2k$.  As
$\alpha^2\in\ZZ$, the possible indices of SSLs are all of the form
$c^{k}$ with $c\in\NN$. For an important class of $2k$-dimensional
lattices, an answer was given by Conway, Rains and Sloane
in~\cite{csl-consloa99}.  Let $\ZZ_{p}$ denote the \mbox{$p\ts$-adic}
integers \cite[Ex.~2.10]{csl-TAO} and define the Hilbert
symbol $(a,b)_p$ as
\[
(a,b)_p \, = \,
\begin{cases}
  1,&\mbox{ if } z^2=ax^2+by^2\mbox{ has a non-zero solution in } 
   \ZZ_p, \\   -1,&\mbox{ otherwise.}
\end{cases}
\]
Their result can now be formulated as follows.\footnote{{\ts\ts}The
  authors formulate their results on sublattices in terms of the norm
  $c=\alpha^2$ of a similarity $\sigma=\alpha R$. We prefer to employ
  the index $n=[\vG:\alpha R \vG]=\alpha^d=c^{d/2}$ instead. The use
  of the norm $c$ is natural for rational lattices, as it is always an
  integer in these cases. However, it is less meaningful for general
  lattices, where the natural quantity is the index $n$. To keep our
  notation consistent, we stick to the formulation in terms of the
  index here, which explains the additional exponent $\frac{d}{2}$ in
  our formulation.}

\begin{theorem}[{\cite[Thm.~1]{csl-consloa99}}]\label{csl-theo:conway-sloane}
  Let\/ $\vG\subset\RR^{2k}$ be a rational lattice\index{lattice!rational}. 
  An SSL of index\/
  $c^{k}$ can only exist if the condition
\[
    \left( c, (-1)^k \det(\vG) \right)_p \, = \, 1
\] 
  is satisfied for all primes\/ $p$ that divide\/ $2\ts
  c\det(\vG)$. If\/ $\vG$ is unigeneric and\/ $(r)$-maximal for some\/
  $r\in\QQ$, then this condition is also sufficient.  \qed
\end{theorem}

Here, $(r)$-\emph{maximal} means that $\vG$ is maximal with respect to
the property that $\langle x \ts |\ts x \rangle\in r \ZZ$ for all
$x\in \vG$. It is \emph{unigeneric}\index{lattice!unigeneric} if it is
unique in its genus. Recall that the genus of a rational quadratic
form\index{quadratic~form} is the set of quadratic forms that are
$\RR$-equivalent and $\ZZ_p$-equivalent for any prime $p$;
compare~\cite{csl-Cassels-Q}. In other words, a rational quadratic
form $Q$ is unigeneric if and only if any other quadratic form $Q'$
that is $\ZZ_p$-equivalent to $Q$ for any prime $p$ as well as
$\RR$-equivalent to $Q$ then also is $\ZZ\ts$-equivalent to $Q$. The
correspondence between lattices and quadratic forms then transfers
these notions to lattices.

\begin{example}
  Theorem~\ref{csl-theo:conway-sloane} can now be applied to several
  lattices \cite{csl-consloa99}, which are all unigeneric and $(1)$-
  or $(2)$-maximal.
\begin{enumerate}\itemsep=2pt
\item The root lattice $A_4$ has SSLs of index $c^2$ for
  $c=\nr(z)=zz'$ only, where $z\in\ZZ[\tau]$ with
  $\tau=\bigl(1+\sqrt{5}\, \bigr)/2$ and $z'$ is the algebraic
  conjugate of $z$.  Consequently, rational primes $p \equiv \pm
  2\bmod 5$ appear to even powers in $c$.
\item The hypercubic lattice $\ZZ^6$ has SSLs of index $c^3$ for
  $c=\nr(z)=|z|^2$ only, where $z\in\ZZ[ \ii ]$. Here, rational primes
  $p \equiv 3\bmod 4$ appear to even powers in $c$.
\item The exceptional root lattice $E_6$ has SSLs of index $c^3$ for
  $c=\nr(z)=|z|^2$ only, where $z\in\ZZ[(1+\ii \sqrt{3})/2]$. Rational 
  primes $p \equiv 2\bmod 3$ appear to even powers in
  $c$. \exend
 \end{enumerate}
\end{example}

Further details for the root lattice $A_4$ will be discussed below.
Another interesting consequence of
Theorem~\ref{csl-theo:conway-sloane} is the following result, where
the notation for the lattices is taken from \cite[Ch.~4]{csl-Conway}.

\begin{corollary}[{\cite[Thm.~3]{csl-consloa99}}]
   The lattices\/ $\ZZ^{4m}$, $D_{4m}^{}$ and\/ $D_{4m}^+$ possess SSLs of
    index\/ $c^{2m}$ for all\/ $c\in\NN$. Similarly, the lattices\/ $E_8$,
   $K_{12}$, the Barnes--Wall lattice\/ $BW_{16}$ and the Leech 
   lattice\/ $\Lambda_{24}$ possess SSLs of index\/ $c^4$, $c^6$, $c^8$
   and\/ $c^{12}$, respectively, for all\/ $c\in\NN$.   \qed
\end{corollary}

\subsection{The root lattice \mbox{$A_4$}}\label{csl-sec:sim-a4}

For the lattice $A_4$, we can go further and count the SSLs of a given
index explicitly. Usually, $A_4$ is embedded in $\RR^5$ as a lattice
plane, but this is inconvenient for our purposes and we prefer to look
at it in $\RR^4$, since we want to exploit a useful parametrisation
by quaternions.

Consider the lattice $L\subset\RR^{4}$ that is spanned by the four
vectors
\begin{equation}\label{csl-basis-a4}
(1,0,0,0),\, \tfrac{1}{2}(-1,1,1,1),\, (0,-1,0,0),\,
\tfrac{1}{2}(0,1,\tau-1,-\tau),
\end{equation}
with $\tau=\bigl(1+\sqrt{5}\, \bigr)/2$ as before. Then, $L$ is similar to
$A_{4}$, with the scale reduced by a factor $\sqrt{2}\ts$; compare
\cite[Ex.~3.3]{csl-TAO} or \cite{csl-BHM}. This way, we have
$L\subset\II$, where $\II$\index{icosian~ring} denotes the icosian ring;
see~\cite[Ex.~2.19]{csl-TAO} and references therein.
 
Let us begin by recalling some properties of $L$. Both $L$ and $\II$
are invariant under quaternionic conjugation, so $L=\ts\overline{\nts
  L\nts}\ts$ and $\II=\overline{\II}$, but neither of them is
invariant under algebraic conjugation $\tau \mapsto \tau'$.  Combining
the algebraic conjugation with a permutation of the last two
(quaternionic) components yields another involution, $x\mapsto
\widetilde{x}:=(x'_0, x'_1, x'_3, x'_2)$, which is an involution of
the second kind in the terminology of \cite{csl-Rost} and was called
the \emph{twist map} in~\cite{csl-BHM,csl-BGHZ08}. Note that
$L=\widetilde L$ is invariant under the twist map, which, in addition,
is an anti-automorphism of~$\II$. In other words, the twist map has
the following properties.

\begin{fact}[{\cite[Lemma~1]{csl-BHM}}]
For any\/ $x,y\in\II$ and\/ $\alpha\in \QQ(\tau)$, one has
\begin{enumerate}\itemsep=2pt
  \item $\widetilde{x+y}=\widetilde{x} + \widetilde{y}$ and\/
    $\widetilde{\alpha x} = \alpha' \ts \widetilde{x}$;
  \item $\widetilde{xy}=\widetilde{y}\ts\ts \widetilde{x}$ and\/
    $\widetilde{\widetilde{x}}=x$;
  \item $\widetilde{\overline{x}}=\overline{\widetilde{x}}$ and, for\/
    $x\neq 0$, $(\nts\widetilde{x})^{-1}=\widetilde{x^{-1}}$.\qed
\end{enumerate}
\end{fact}

The twist map is the key to our analysis as it gives us a convenient
parametrisation of the similarity
rotations\index{rotation!similarity}$\,$---$\,$and later also the
coincidence rotations. Furthermore, it provides us with the following
characterisation \cite[Prop.~1]{csl-BHM} of the lattice $L$ as a
subset of $\II$,
\begin{equation}\label{csl-eq:char-L}
  L \, = \, \{x\in \II\mid x=\widetilde{x}\ts\}\ts .
\end{equation}
By Cayley's\index{rotation!Cayley~parametrisation}
parametrisation~\eqref{csl-eq:4dimrot}, we know that any rotation in
$\RR^{4}$ can be written as $R(p,q)x= \frac{1}{|p \ts q|}\ts p\ts
x\bar{q}$. Using the properties of the twist map and the
characterisation of $L$ from above, we immediately see that
$qL\ts\widetilde{q}\subseteq L$ is a similar sublattice of $L$ for any
$q\in\II$. In fact, any SSL of $L$ is of the form $\alpha \ts q
L\ts\widetilde{q}\subseteq L$, with $q\in\II$ and
$\alpha\in\QQ(\tau)$; see~\cite[Cor.~1]{csl-BHM}.

In order to classify the SSLs, it is convenient to introduce a
suitable primitivity notion on $\II$.  A quaternion $q\in\II$ is
called \emph{$\II$-primitive} (or primitive for short) if $\alpha q
\in \II$ with $\alpha\in \QQ(\tau)$ implies
$\alpha\in\ZZ[\tau]$. Equivalently, $q\in\II$ is $\II$-primitive if
the $\II$-content of $q$,
\[
  \cont^{}_{\II}(q) \, :=\, 
  \lcm\bigl\{\alpha\in\ZZ[\tau]\setminus\{0\} \mid
  q\in\alpha\II\bigr\} , 
\]
is a unit in $\ZZ[\tau]$. Note that the notion of an $\lcm$ makes
sense because $\ZZ[\tau]$ is a Euclidean domain. Of course,
$\cont^{}_{\II}(q)$ is defined only up to a unit in $\ZZ[\tau]$.  We
can now fully characterise the SSLs as follows.

\begin{lemma}[{\cite[Cor.~2]{csl-BHM}}]\label{csl-lem:a4-pr-ssl}
  The primitive SSLs of\/ $L$ are precisely the sublattices of the
  form\/ $qL\ts\widetilde{q}$, where\/ $q\in\II$ is\/ $\II$-primitive.
  Consequently, the SSLs of\/ $L$ are precisely the sublattices of
  the form\/ $n \ts qL\ts\widetilde{q}$ with\/ $n\in\NN$ and\/ $q\in\II$
  primitive. \qed
\end{lemma}

As we also want to determine the number of distinct SSLs of a given
index, we need to ensure that we do not count the same SSL twice. In
general, different quaternions may generate the same SSL, so we need a
criterion to determine whether two SSLs $qL\ts\widetilde{q}$ and
$pL\ts\widetilde{p}$ are equal. One first observes that
$L=qL\ts\widetilde{q}$ holds for an $\II$-primitive quaternion $q$ if
and only if $q\in\II^{\times}$, where $\II^{\times}$ is the unit group
in $\II$; see~\cite[Ex.~2.19]{csl-TAO} for an explicit description
and~\cite{csl-Moody94,csl-Patera} for further background.  From here,
one can infer the following result.

\begin{fact}[{\cite[Lemma~5]{csl-BHM}}]\label{csl-fact:a4-eq-pr-ssl}
  For\/ $\II$-primitive quaternions\/ $p,q\in\II$, one has\/
  $pL\widetilde{p}=qL\widetilde{q}$ if and only if\/
  $p\ts\II=q\ts\II$.\qed
\end{fact}

This fact reduces the problem of counting SSLs of $L$ to the problem
of counting primitive right ideals of $\II$. Here, we call a right
ideal $q\II$ primitive if $q$ is $\II$-primitive.

The index of a primitive SSL can be determined by an explicit
calculation.  We mention that $|\widetilde{q}\ts|^2=(|q|^2)'$ holds for
any $q\in\II$.  Recall from \cite[Ex.~2.14]{csl-TAO} that the
\emph{norm} of an element $\alpha\in\QQ(\tau)$ is defined as
\[
  \nr(\alpha) \, =\, \alpha\ts \alpha'\ts .
\]
The index of a primitive SSL $qL\ts\widetilde{q}$ then satisfies $[L:
qL\ts\widetilde{q}\, ]=\nr(|q|^4)$. As $q\ts\II$ has index $\nr(|q|^4)$ in
$\II$ as well, we get the following result.

\begin{lemma}[{\cite[Prop.~4]{csl-BHM}}]\label{csl-lem:a4-biject}
  There is a bijective correspondence between the primitive right
  ideals of\/ $\II$ and the primitive SSLs of $\ts L$, given by\/
  $q\ts\II \leftrightarrow qL\ts\widetilde{q}$. Moreover, one has
  \[
    \bigl[\ts\II: q\II\ts\bigr] \, =\, \nr\bigl(|q|^4\bigr) 
    \, =\, \bigl[L: qL\ts\widetilde{q}\,\bigr],
  \]
  which means that the bijection preserves the index.\qed
\end{lemma}

As a consequence, all possible indices are squares of integers of the
form $k^2+k\ell-\ell^2=\nr(k+\ell \tau)$. In fact, all these indices
are realised \cite{csl-BHM,csl-consloa99}. As the number of right
ideals of $\II$ of a given index is well known, we can deduce the
numbers $b^{}_{\! A_4}(m)$ and $b^{\mathsf{pr}}_{\! A_4}(m)$ of SSLs
and primitive SSLs of index $m$, respectively.  This can efficiently
be done by employing the corresponding Dirichlet series generating
functions. To do so, we first recall the Dirichlet
character\index{Dirichlet~character}
\[
  \chi_{5}^{}(n)=
    \begin{cases}
      0, & \mbox{ if } n\equiv 0 \; (5), \\
      1, & \mbox{ if } n\equiv \pm 1 \; (5), \\
      -1, & \mbox{ if } n\equiv \pm 2 \; (5).
    \end{cases}
\]
Its corresponding \mbox{$L\ts$-series},
$L(s,\chi_5^{})=\sum_{n=1}^\infty\chi_5^{}(n)n^{-s}$, defines (via
analytic continuation) an entire function on the complex plane. The
Dedekind zeta\index{zeta~function!Dedekind} function of $K=\QQ(\tau)$
is given by $\zeta_K^{}(s)=\zeta(s)L(s,\chi_5^{})$, which is a
meromorphic function.  Likewise, the
zeta\index{zeta~function!icosian~ring} function $\zeta_{\ts\II}^{}$ of
the icosian ring \cite{csl-Vigneras,csl-BM}, which counts the right
(or left) ideals of $\II$, is meromorphic in the entire complex plane
and reads
\begin{equation}\label{csl-eq:zetaI}
  \zeta_{\ts\II}^{}(s)\, =\, \zeta_K^{}(2s)\,\zeta_K^{}(2s-1)\ts .
\end{equation}
As the Dirichlet series\index{Dirichlet~series} of the two-sided
ideals is given by $\zeta_K^{}(4s)$, one obtains the zeta function
$\zeta_{\ts\II}^{\mathsf{pr}}$ of the primitive ideals \cite{csl-BM}
as
\begin{equation}
\label{csl-eq:zetaIpr}
  \zeta_{\ts\II}^{\mathsf{pr}}(s)
    \, =\, \frac{\zeta_K^{}(2s)\,\zeta_K^{}(2s-1)}{\zeta_K^{}(4s)}\ts .
\end{equation}
This leads to the following result.

\begin{theorem}[{\cite[Thm.~1]{csl-BHM}}]\label{csl-thm:ssl-a4}
  The Dirichlet series\index{Dirichlet~series} generating functions
  for the numbers\/ $b^{}_{\! A_4}(n)$ and\/ $b^{\mathsf{pr}}_{\!
    A_4}(n)$ of SSLs and primitive SSLs of the root lattice\/ $A_4$ of
  a given index are
  \begin{align*}
    \Phi_{\! A_4}^{}(s)&\, = \sum_{n\in\NN} \frac{b^{}_{\! A_4}(n)}{n^s}
      \, = \, \zeta(4s)\,\frac{\zeta_{\ts\II}^{}(s)}{\zeta_K^{}(4s)}
      \, = \, \frac{\zeta_K^{}(2s)\,\zeta_K^{}(2s-1)}{L(4s,\chi_5^{})} \\ 
  \intertext{and}
   \qquad\qquad\qquad\quad\; \; \,
   \Phi_{\! A_4}^{\mathsf{pr}}(s)&\, = \sum_{n\in\NN} 
      \frac{b^{\mathsf{pr}}_{\! A_4}(n)}{n^s}
      \, = \, \zeta_{\ts \II}^{\mathsf{pr}}(s) 
      \, = \, \frac{\zeta_K^{}(2s)\,\zeta_K^{}(2s-1)}{\zeta_K^{}(4s)}\ts .
      \qquad\qquad\qquad\quad \; \qed
  \end{align*}
\end{theorem}

Both generating functions from Theorem~\ref{csl-thm:ssl-a4} possess
Euler\index{Euler~product} products, which read
\begin{align}
      \Phi_{\! A_4}^{}(s) & \, = \, \myfrac{1}{(1-5^{-2s})(1-5^{1-2s})}
        \prod_{p\equiv \pm 1 (5)} \frac{1+p^{-2s}}{(1-p^{-2s})(1-p^{1-2s})^2}
      \nonumber \\[-1ex] & \\[-2ex]
      & \qquad \times  \prod_{p\equiv \pm 2 (5)} 
        \frac{1+p^{-4s}}{(1-p^{-4s})(1-p^{2-4s})}\nonumber\\
  \intertext{and}
  \quad \Phi_{\! A_4}^{\mathsf{pr}}(s) & \, = \, \myfrac{1+5^{-2s}}{1-5^{1-2s}}
        \prod_{p\equiv \pm 1 (5)} \frac{(1+p^{-2s})^2}{(1-p^{1-2s})^2}
        \prod_{p\equiv \pm 2 (5)} \frac{1+p^{-4s}}{1-p^{2-4s}}\ts .
\end{align}
From these identities, we can obtain explicit expressions for
$b^{}_{\!  A_4}(n)$ and $b^{\mathsf{pr}}_{\! A_4}(n)$, which are
multiplicative arithmetic functions. Thus, they are determined by
their values at prime powers.  As $b^{}_{\!
  A_4}(p^{2r+1})=b^{\mathsf{pr}}_{\!  A_4}(p^{2r+1})=0$, we only need
to state their values for primes at even powers $\geq 2$. The result
is \cite{csl-BHM}
\begin{align*}
  b^{}_{\! A_4}(p^{2r})&=
    \begin{cases}
      \frac{5^{r+1}-1}{4}, & \mbox{ if } p=5,\\[1mm]
      \frac{(r+1)(p^2-1)p^r-2(p^{r+1}-1) }{(p-1)^2}, 
        & \mbox{ if } p\equiv \pm 1\,(5) , \\[1mm]
      \frac{p^{r+2}+p^r-2}{p^2-1} ,
        & \mbox{ if } p\equiv \pm 2\,(5) \mbox{ and  $r$ even,} \\[1mm]
      0, & \mbox{ if } p\equiv \pm 2\,(5) \mbox{ and  $r$ odd,}
    \end{cases}\\
    \intertext{and}
  b^{\mathsf{pr}}_{\! A_4}(p^{2r})&=
    \begin{cases}
      6 \cdot 5^{r-1}, & \mbox{ if } p=5,\\[1mm]
      (r+1)p^r + 2r p^{r-1} + (r-1) p^{r-2},
        & \mbox{ if } p\equiv \pm 1\,(5), \\[1mm]
      p^r + p^{r-2},
        & \mbox{ if } p\equiv \pm 2\,(5)  \mbox{ and  $r$ even,} \\[1mm]
      0, & \mbox{ if } p\equiv \pm 2\,(5) \mbox{ and  $r$ odd.}
    \end{cases}
\end{align*}
It follows from these formulas that all possible indices are not only
realised for some SSL, but even realised for some \emph{primitive}
SSL. In fact, it will turn out that the majority of SSLs of a given
index are primitive. This can be illustrated by comparing the first
few terms of $\Phi_{\! A_4}^{}$ and $\Phi_{\! A_4}^{\mathsf{pr}}$,
\begin{align*}
  \Phi_{\! A_4}^{}(s)&\, = \, 1 +\myfrac{6}{4^{2s}} +
      \myfrac{6}{5^{2s}}+\myfrac{11}{9^{2s}} 
      +\myfrac{24}{11^{2s}} +\myfrac{26}{16^{2s}} +\myfrac{40}{19^{2s}} 
      +\myfrac{36}{20^{2s}} +\myfrac{31}{25^{2s}} 
      + \cdots,\\[1mm]
  \Phi_{\! A_4}^{\mathsf{pr}}(s)&\, =\,  
    1 +\myfrac{5}{4^{2s}} +\myfrac{6}{5^{2s}}+\myfrac{10}{9^{2s}} 
      +\myfrac{24}{11^{2s}} +\myfrac{20}{16^{2s}} +\myfrac{40}{19^{2s}} 
      +\myfrac{30}{20^{2s}} +\myfrac{30}{25^{2s}}
      + \cdots.
\end{align*}
The explicit form of the generating functions $\Phi_{\! A_4}^{}(s)$
and $\Phi_{\! A_4}^{\mathsf{pr}}$ allows us to calculate the
asymptotic behaviour of $b^{}_{\! A_4}(n)$ and $b^{\mathsf{pr}}_{\!
  A_4}(n)$. The result reads as follows.

\begin{corollary}[{\cite[Sec.~4]{csl-BHM}}]
  The asymptotic growth\index{asymptotic~behaviour}
  of the summatory function of\/ 
  $b^{}_{\! A_4}(n)$ is 
  \[
    \sum_{m\leq x} b^{}_{\! A_4}(m) \, \sim \, \frac{\rho}{2}\ts\ts x,\quad 
    \mbox{as } x \to \infty\ts , 
  \]
  where\/ $\rho$ is given by
  \[
    \rho \, = \, \frac{\zeta_K(2) L(1,\chi_5^{})}{L(4,\chi_5^{})}
      \, = \, \myfrac{1}{2}\sqrt{5} \ts\log(\tau) 
      \, \approx\,  0.538\ts 011 \ts .
  \]
  The asymptotic growth\index{asymptotic~behaviour}
  for\/ $b^{\mathsf{pr}}_{\! A_4}(n)$ is also
  linear, now with
  \[
    \qquad\qquad\qquad\quad\qquad
    \rho^{\mathsf{pr}} \, = \, 
    \frac{\zeta_K(2) L(1,\chi_5^{})}{\zeta(4) L(4,\chi_5^{})}
     \, = \, \myfrac{45}{\pi^4}\sqrt{5} \ts\log(\tau) 
   \,\approx\, 0.497\ts 089\ts .
   \qquad\qquad\qquad\quad
  \qed
  \]
\end{corollary}
\begin{proof}[Sketch of proof]
  We apply again Theorem~\ref{csl-thm:meanvalues}, this time to the
  generating functions given in Theorem~\ref{csl-thm:ssl-a4}.  The
  fact that both Dirichlet series are meromorphic functions, which are
  analytic in the half-plane $\{\Real(s)>1\}$ and have the proper
  behaviour on the line $\{\Real(s)=1\}$, implies the linear growth.
  The explicit calculations for the sum
  $\sum_{m\leq x} b^{}_{\! A_4}(m)$ are
  similar\footnote{\label{csl-foot-ssla4}{\ts}Note that a different
    definition for the counting function was applied
    in~\cite{csl-BHM}. There, the function $f(m)=b^{}_{\! A_4}(m^2)$
    was discussed, which makes sense as $b^{}_{\! A_4}(n)$ is non-zero
    only for squares. Correspondingly, the asymptotics for $f(m)$ are
    given by $\sum_{m\leq x} f(m) \, \sim \, \frac{\rho}{2}\ts\ts x^2$
    as $x \to \infty$.} to those from \cite[Sec.~4, p.~1402]{csl-BHM}.
  The case of the primitive SSLs is analogous, and just gives an
  additional factor $\frac{1}{\zeta(4)}$.
\end{proof}

\subsection{Hypercubic lattices in $\RR^{4}$}
\label{csl-sec:ssl-hycub}

There are, up to similarity, two hypercubic lattices in $4$
dimensions, namely the primitive hypercubic lattice $\ZZ^4$ and the
centred hypercubic lattice $D_4^{}$; compare~\cite{csl-Conway}
and~\cite[Ex.~3.2]{csl-TAO}. The latter is similar to its dual lattice
$D_4^{*}$, which we identify with the Hurwitz ring\index{Hurwitz~ring}
$\JJ$.

Recall that any rotation in $4$ dimensions can be parametrised by a
pair of quaternions; compare Section~\ref{csl-sec:quat}. It turns out
that any similarity rotation\index{rotation!similarity} of $\vG \in \{
D_4^{*}, \ZZ^4 \}$ can be parametrised by a pair $(p,q)$ of Hurwitz
quaternions. Moreover, any SSL of $\vG$ is of the form $p\vG \bar{q}$,
where we can choose $p$ to be odd and primitive; compare~\cite[Rem.~1
  and Lemma~2]{csl-BM}.  With this convention, in the case of
$\vG=D_4^{*}=\JJ$, $p$ and $q$ are unique up to multiplication by a
unit of $\JJ$ from the right~\cite[Prop.~3]{csl-BM}. Hence, counting
SSLs of $D_4^{*}$ is equivalent to counting right ideals of~$\JJ$.

The situation is slightly more complicated for $\ZZ^4$, as its
symmetry is lower. As a consequence, there may be three distinct (but,
of course, congruent) SSLs of $\ZZ^4$ that correspond to a single SSL
of $\JJ$. This only happens if the index of the SSL is even. We thus
obtain the following result for the generating functions of the SSLs,
where we make use of the zeta\index{zeta~function!Hurwitz~ring}
function of~$\JJ$, which reads \cite{csl-BM,csl-Vigneras,csl-Reiner}
\begin{equation}\label{csl-zetaHurwitz}
   \zeta^{}_{\JJ}(s) \, = \sum_{I\subseteq\JJ}\myfrac{1}{[\JJ:I]^s}
    \, = \, (1-2^{1-2s}) \, \zeta(2s) \, \zeta(2s-1) \ts .
\end{equation}

\begin{theorem}[{\cite[Thm.~2]{csl-BM}}]\label{csl-thm:hyper}
  The possible indices of similar sublattices of hypercubic lattices
  in\/ $\RR^4$ are precisely the squares of rational integers. The
  number of distinct SSLs of a given index is a multiplicative
  arithmetic function. For the case of\/ $\JJ=D_4^{*}$, the
  corresponding Dirichlet series\index{Dirichlet~series} generating
  function\/ $\Phi_{\JJ}^{}$ reads
  \[
    \Phi_{\JJ}^{}(s) 
    \, = \, \frac{\bigl(\zeta^{}_{\JJ}(s)\bigr)^{2}}{(1+4^{-s})\,\zeta(4s)}
      \, =\,  \frac{\bigl(1-2^{1-2s}\bigr)^2}{1+4^{-s}} 
      \frac{\bigl(\zeta(2s)\,\zeta(2s-1)\bigr)^2}{\zeta(4s)}\ts .
  \]
  The same function also applies to the lattice\/ $D_4^{}$, while we
  obtain
  \[
    \Phi_{\ZZ^4}^{}(s)
    \, = \, \left( 1 + \myfrac{2}{4^s} \right) \Phi_{\JJ}^{}(s)
  \]
for the primitive hypercubic lattice\/ $\ZZ^4$. \qed
\end{theorem}

From the generating functions of Theorem~\ref{csl-thm:hyper}, we can
extract the corresponding counting functions $b_{\JJ}^{}(m)$ and
$b_{\ZZ^4}^{}(m)$. We formulate them in terms of the
function\index{counting~function!hypercubic~lattice}
\begin{equation}\label{csl-eq:gnr}
  g(n,r) \, =\,  (r+1)\, n^r + 
  2 \,\frac{1-(r+1)\ts n^r + r\ts n^{r+1} }{(n-1)^2} 
\end{equation}
for integers $r \geq 0$ and $n>1$.

\begin{corollary}[{\cite[Cor.~1]{csl-BM}}]
  The arithmetic\index{function!arithmetic} functions\/
  $b_{\JJ}^{}(m)$ and\/ $b_{\ZZ^4}^{}(m)$ are
  multiplicative.\index{multiplicative~function} They are non-zero if
  and only if\/ $m$ is a square, and are then determined by
  \[
    b_{\JJ}^{}(p^{2r})  \, =\, 
    \begin{cases}
      1, & \mbox{if } p=2, \\
      g(p,r), & \mbox{if $p$ is an odd prime,}
    \end{cases}
  \]
  for all $r \geq 0$,  and by
  $b_{\ZZ^4}^{}(m) = \left(2+ (-1)^m \right) b_{\JJ}^{}(m)$.
  \qed
\end{corollary}

The first few terms of $\Phi_{\JJ}^{}(s)$ read
\begin{equation}
\begin{split}
  \Phi_{\JJ}^{}(s)  \, =\,
  1 & + \myfrac{1}{4^s} + \myfrac{8}{9^s} + \myfrac{1}{16^s} + \myfrac{12}{25^s}
   + \myfrac{8}{36^s} + \myfrac{16}{49^s} + \myfrac{1}{64^s} + \myfrac{41}{81^s}
   + \myfrac{12}{100^s} \\
   \nonumber
  & 
   + \myfrac{24}{121^s} + \myfrac{8}{144^s} + \myfrac{28}{169^s} 
   + \myfrac{16}{196^s} + \myfrac{96}{225^s} + \myfrac{1}{256^s} 
   + \myfrac{36}{289^s} 
   + \cdots ,
\end{split}
\end{equation}
which corresponds to sequence A045771 in \cite{csl-OEIS}.

\begin{corollary}[{\cite[Cor.~2]{csl-BM}}]
  The asymptotic growth\index{asymptotic~behaviour} of the summatory
  arithmetic function\/ $\sum_{m\leq x} b_{\vG}^{}(m)$ is given
  by{\ts}\footnote{\label{csl-foot-ssld4}Note that in~\cite{csl-BM} the
    asymptotics of the counting function $f^{}_{\vG}(m)=b^{}_{\vG}(m^2)$
    instead of $b^{}_{\vG}(m)$ are discussed; compare
    Footnote~\ref{csl-foot-ssla4}.  Correspondingly, the asymptotics
    for $f^{}_{\!\vG}(m)$ are given by $\sum_{m\leq x} f^{}_{\!\vG}(m) \, \sim
    \, 2 \ts C^{}_{\vG}\, x^2 \log(x)$ as $x \to \infty$.}
  \[
    \sum_{m\leq x} b_{\vG}^{}(m) \, \sim\,  C^{}_{\vG}\, x \log(x)
  \]
  as\/ $x\to\infty$, where the constant\/ $C^{}_{\vG}$ is given by
  \[
  \quad\qquad\qquad\qquad\quad\;\;
   C^{}_{\vG} \,=\, \Res_{s=1}\bigl((s-1)\ts\Phi_{\vG}^{}(s)\bigr) \, =\,
    \begin{cases}
      \frac{1}{8}, & \mbox{for } \vG = \JJ,\\
      \frac{3}{16}, & \mbox{for } \vG = \ZZ^4.  
      \;\,\qquad\qquad\qquad\quad \qed
    \end{cases} 
  \]
\end{corollary}

Finally, let us comment on the primitive SSLs. A pair $(p,q)$ of
Hurwitz quaternions generates a primitive SSL of $\JJ$ if and only if
both $p$ and $q$ are $\JJ$-primitive and at least one of them is
odd.

In this case, the denominator\index{denominator} of the corresponding
rotation is given by
\begin{equation}\label{csl-denD4}
  \den^{}_{\JJ}(R(p,q)) \, =\,  |p \ts q|.
\end{equation}
For $\ZZ^4$, a pair of $\JJ$-primitive quaternions does not
necessarily generate an SSL of $\ZZ^4$. This only works if $p \ts q\in
\ZZ^4$. Consequently, primitive SSLs are either of the form
$p\ts\ZZ^4\bar{q}$ or $2p\ts\ZZ^4\bar{q}$, depending on whether $p \ts
q \in\ZZ^4$ or not. Correspondingly, the denominator for $\ZZ^4$ reads
\begin{equation}\label{csl-denZ4}
  \den^{}_{\ZZ^4}(R(p,q)) \, = \,
  \begin{cases}
    |pq|, & \mbox{if $p \ts q\in \ZZ^4$,}\\
    2\ts |pq|, & \mbox{if $p \ts q\not\in \ZZ^4$.}
  \end{cases}
\end{equation}
As a consequence, we have
$\Phi^{\textsf{pr}}_{\vG}(s)=\Phi^{}_{\vG}(s)/\zeta(4s)$ for
$\vG\in\{\JJ,\ZZ^{4}\}$. Finally, this yields the asymptotic behaviour
\[
    \sum_{m\le x} b^{\textsf{pr}}_{\vG}(m) \, \sim \,
    C^{\textsf{pr}}_{\vG} \, x\ts \log(x)\quad\text{with}\quad
    C^{\mathsf{pr}}_{\vG} \, = \, \begin{cases}
      \frac{45}{4\ts\pi^{4}}, & \text{for $\vG=\JJ$},\\[1mm]
      \frac{135}{8\ts\pi^4}, & \text{for $\vG=\ZZ^{4}$}.
   \end{cases}
\]

\section{Similar submodules}
\label{csl-sec:ssm}

Here, we are interested in $\ZZ\ts$-modules as generalisations of
lattices. As such, they are mainly considered as \emph{geometric} (as
opposed to algebraic) objects. Let us thus begin with a definition of
the geometric setting.

\begin{definition}
  A $\ZZ\ts$-module $M$ of rank $n$ is called (properly)
  \emph{embedded}\index{module!embedded} in $\RR^d$ when $M\subset
  \RR^d$ and when there is a $\ZZ\ts$-basis $\{ b_1, \ldots , b_n\}$
  of $M$ whose $\RR$-span is $\RR^d$.
\end{definition}

In particular, this requires that $n\geq d$, where $n$ is the
\emph{rank} of $M$ and $d$ may be called its \emph{embedding
  dimension}. A lattice is an embedded module with $n=d$. An important
class of embedded modules is given by what we call
$\cS$-lattices.

\begin{definition}\label{csl-def:S-lat}
Let $\cS\subset \RR$ be a ring with identity that is also a finitely
generated, free \mbox{$\ZZ\ts$-module}. Then, we call an embedded
$\ZZ\ts$-module $M\subset\RR^d$ an
\emph{$\cS$-lattice}\index{S-lattice@$\cS$-lattice} if there exist
$d$ linearly independent vectors $b_i\in\RR^d$ such that $M$ is the
$\cS$-span of $\{ b_1, \ldots, b_d \}$, so $M=\langle b_1, \ldots,
b_d \rangle_\cS$.
\end{definition}

We call a $\ZZ\ts$-module $M'\subseteq M$ a (full)
\emph{submodule}\index{submodule} of $M$ if $M'$ and $M$ have the same
rank.\footnote{{\ts}More generally, one calls any $\ZZ\ts$-module
  $M'\subseteq M$ a submodule of $M$ regardless of its rank, but we do
  not need this more general notion in our context.} This implies that
$M'$ and $M$ also have the same embedding dimension, wherefore the index
$[M:M']$ is finite.

Just as for lattices, we define the more general notion of
commensurate modules.

\begin{definition}\label{csl-def:commM}
  Two (properly embedded)\index{module!embedded} $\ZZ\ts$-modules
  $M_1,M_2\subset \RR^d$ are called
  \emph{commensurate}\index{module!commensurate}, which is denoted by
  $M_1\sim M_2$, if their intersection $M_1\cap M_2$ has finite
  index\index{index} in both modules, $M_1$ and $M_2$.
\end{definition}

In our terminology, this means that $M_1$ and $M_2$ are commensurate
if and only if $M_1\cap M_2$ is a submodule of both $M_1$ and $M_2$ in
our above sense.  This implies that $M_1$ and $M_2$ can only be
commensurate if they have the same rank. Once we know that two
embedded modules in $\RR^d$ have the same rank, the situation becomes
easier as we can characterise commensurateness in several ways
\cite{csl-habil}, which we recall here.
\begin{lemma}\label{csl-lem:comm}
  Let\/ $M_1,M_2\subseteq \RR^d$ be two properly embedded
  $\ZZ\ts$-modules of rank\/ $k$. Then, the following statements are
  equivalent.
\begin{enumerate}\itemsep=3pt
\item $M_1$ and\/ $M_2$ are commensurate.
\item $M_1\cap M_2$ has finite index in both\/ $M_1$ and\/ $M_2$.
\item $M_1\cap M_2$ has finite index in\/ $M_1$ or in\/ $M_2$.
\item\label{csl-lem:comm-i4} There exist\/ $($positive$\ts )$
  integers\/ $m_1$ and\/ $m_2$ such that\/ $m_1M_1\subseteq M_2$ and\/
  $m_2M_2\subseteq M_1$.
\item There exists an integer\/ $m$ such that\/ $m M_1\subseteq M_2$
  or\/ $m M_2\subseteq M_1$.
\item $M_1\cap M_2$ has rank\/ $k$.  \qed
\end{enumerate}
\end{lemma}

To continue, two properly embedded modules $M_{1}$ and $M_{2}$ are
called \emph{similar}\index{module!similar}, $M_{1} \similar M_{2}$,
if there exists a similarity transformation between them. Clearly,
similarity of modules is an equivalence relation.

\begin{definition}
  A similarity transformation that maps a module \mbox{$M\subset
    \RR^d$} onto a submodule of $M$ is called a \emph{similarity
  transformation of $M$}\index{similarity~transformation}.
  A submodule $M'\subseteq M$ is called
  a \emph{similar submodule}\index{similar submodule} (SSM) of 
  $M$ if $M' \similar M$.
\end{definition}

We proceed as before and consider coincidence isometries and scaling
factors separately. We first define
\begin{equation}
    \OS(M)\, :=\, \{R\in \OG(d,\RR)\mid \exists \ts 
    \alpha\in\RR_{+} \text{ such that } 
    \alpha R M \subseteq M \}
\end{equation}
whose elements are called \emph{similarity
  isometries}\index{isometry!similarity} of $M$. Similarly, we use
\begin{equation}
   \SOS(M)\, :=\, \OS(M) \cap \ts \SO (d,\RR)
\end{equation}
to denote the set of similarity rotations\index{rotation!similarity}. 
The following results are immediate generalisations of the corresponding
results for lattices in Fact~\ref{csl-fact:OS-group} and
Lemma~\ref{csl-lem:OS-sim}.

\begin{fact}
  $\OS(M)$ and\/ $\ts\SOS(M)$ are subgroups of\/
  $\OG(d,\RR)$. Further, if\/ $M$ and\/ $M'=\alpha R M$ are similar
  modules which are both embedded in\/ $\RR^{d}$, we have
\[
   \hspace{5.4cm}   \OS(M')\, =\, R\ts \OS(M) \ts\ts R^{-1}. 
   \hspace{5cm} \qed
\]
\end{fact}

Next, we consider the scaling factors. We first define
\begin{equation}
\begin{split}
  \Scal^{}_M(R)& \, :=\, \{\alpha\in\RR \mid \alpha R M\subseteq M \} 
\quad \text{ and}\\[1mm]
  \scal^{}_M(R)&\, :=\, \{\alpha\in\RR \mid \alpha R M\sim M \} \ts .
\end{split}
\end{equation}
Again, we have allowed negative values for the scaling factors here to
ensure that $\Scal^{}_M(R)$ is a $\ZZ\ts$-module.  This creates no
problem because $-M=M$.  However, the situation is more complicated
than in the case of lattices, as there are significantly fewer
restrictions on the scaling factors here.

Note that $\Scal^{}_M(R)$ is non-empty for all $R$ as $0 \in
\Scal^{}_M(R)$, but it is non-trivial only if $R\in\OS(M)$, as we have
the following generalisation of Fact~\ref{csl-lem:scal-nontriv}.

\begin{fact}[{\cite[p.~14]{csl-habil}}]\label{csl-lem:scal-nontrivM}
  Let\/ $M\subset\RR^{d}$ be an embedded\/ $\ZZ\ts$-module and consider\/
  $R\in\OG(d,\RR)$. Then, the following properties are equivalent.
   \begin{enumerate}\itemsep=2pt
     \item $\Scal^{}_M(R)\ne \{ 0 \}$;
     \item $\scal^{}_M(R)\ne \varnothing$;
     \item $R\in \OS(M)$. \qed
   \end{enumerate}
\end{fact}

As a first consequence, we mention a result on the possible values of
$\Scal^{}_M(\one)$. Recall that $[x]$ denotes the largest integer
$n\leq x$.

\begin{theorem}[{\cite[Thm.~2.1.6 and Cor.~2.1.7]{csl-habil}}]
  \label{csl-theo:ScalM1}
  Let\/ $M \subset \RR^d$ be an\index{module!embedded} embedded\/
  $\ZZ\ts$-module of rank $k$.  Then, $\Scal^{}_M(\one)$ is a ring
  with unit all elements of which are algebraic integers. Moreover,
  $\Scal^{}_M(\one)$ is a finitely generated, free\/ $\ZZ\ts$-module,
  whose rank is a divisor of\/ $k$ and is at most\/
  $\bigl[\frac{k}{d}\bigr]$.

  Furthermore, $\scal^{}_M(\one)\cup \{0\}$ is the field of fractions
  of\/ $\Scal^{}_M(\one)$.  \qed
\end{theorem}

For $\cS$-lattices, we can immediately determine $\Scal^{}_M(\one)$
and $\scal^{}_M(\one)$.

\begin{fact}\label{csl-fact:scal-S-lat}
  If\/ $M$ is an\/ $\cS$-lattice, then\/ $\Scal^{}_M(\one)=\cS$ and\/ 
  $\scal^{}_M(\one)\cup \{0\}$ is the field of fractions of\/ $\cS$.
\end{fact}

\begin{proof}
  Since $\cS$ is a ring and $M$ is the $\cS$-span of $d$ linearly
  independent vectors $b_i \in \RR^d$, we have $\cS\subseteq
  \Scal^{}_M(\one)$. On the other hand, the linear independence of the
  $b_i$ guarantees $\Scal^{}_M(\one)\ts b_1 \subseteq M\cap \RR\ts b_1
  = \cS\ts b_1$, whence we have the reverse inclusion
  $\Scal^{}_M(\one)\subseteq \cS$. The second part now follows
  immediately from Theorem~\ref{csl-theo:ScalM1}; compare
  also~\cite[Remark~3.11]{csl-svenja2}.
\end{proof}

For general similarity isometries $R$, we have the following result.

\begin{theorem}[{\cite[Thm.~2.1.9]{csl-habil}}]
  Let\/ $M\subset\RR^{d}$ be an embedded\/ $\ZZ\ts$-module.  
  Then, for any isometry\/ \mbox{$R\in\OS(M)$}, $\Scal^{}_M(R)$ 
  is a finitely generated, free\/
  $\ZZ\ts$-module. Moreover, one has\/ $\beta\Scal^{}_M(R)\subseteq
  \Scal^{}_M(R)$ for any\/ $\beta\in \Scal^{}_M(\one)$, and\/
  $\Scal^{}_M(R)$ is thus also a finitely generated\/
  $\Scal^{}_M(\one)$-module.  \qed
\end{theorem}

Observe that $\Scal^{}_M(R)$ is generally not a free
$\Scal^{}_M(\one)$-module, unless $\Scal^{}_M(\one)$ is a PID; see
\cite[p.~15]{csl-habil} for an example.

For lattices, Lemma~\ref{csl-lem:scal-d} asserted that
$\alpha^d\in\ZZ$ for all $\alpha\in\Scal^{}_{\vG}(R)$. The
corresponding result for embedded modules reads as follows.

\begin{theorem}[{\cite[Thm.~2.1.10]{csl-habil}}]\label{thm:alphaScalM}
  As before, let\/ $M\subset\RR^{d}$ be an embedded\/ $\ZZ\ts$-module
  of finite rank.  Then, any\/ $\alpha\in \Scal^{}_M(R)$ is an
  algebraic integer.  If\/ $M$ has rank\/ $k=1$, one always has\/
  $\Scal^{}_{M}(R)=\ZZ$, so\/ $\alpha$ is a rational integer in this
  case.  If\/ $k\geq 2$, the degree of\/ $\alpha$ is at most\/
  $k(k-1)$.  \qed
\end{theorem}

The set $\{\scal_M(R) \mid R\in \OS(M)\}$ has again a group structure,
under the multiplication defined by 
\[
    \scal^{}_M(R)\ts\scal^{}_M(S)\, :=\,
    \{ \alpha\beta \mid \alpha \in \scal^{}_M(R),\beta \in
    \scal^{}_M(S)\}\ts .
\]

We have the following generalisation of Lemma~\ref{csl-lem:scal-group}.

\begin{theorem}[{\cite[Lemmas~2.1.11 and~2.1.12 and Thm.~2.1.12]{csl-habil}}]
\label{csl-theo:scalgroupM}
  Let\/ $M\subset\RR^{d}$ be an embedded\/ $\ZZ\ts$-module. Then, one
  has the following properties.\index{module!embedded}
  \begin{enumerate}\itemsep=2pt
  \item For any\/ $R,S\in \OS(M)$, we have the product relation\/
    \[
    \scal^{}_M(R)\ts\scal^{}_M(S)=\scal^{}_M(RS)
    \]
    together with\/ $\scal^{}_M(R^{-1})\ts\scal^{}_M(R)=\scal^{}_M(\one)$.
  \item $\{ \scal^{}_M(R) \mid R\in \OS(M) \}$ is an Abelian group.
    Its neutral element is\/ $\scal^{}_M(\one)$, and the inverse
    of\/ $\scal^{}_M(R)$ is\/ $\scal^{}_M(R^{-1})$.
  \item $\{ \scal^{}_M(R) \mid R\in \OS(M) \}$ is isomorphic to a
    multiplicative subgroup of the group\/
    $\RR_{+}/(\scal^{}_M(\one)\cap\RR_{+})$.
  \item There exists a natural homomorphism
    \[
    \phi\!:\, \OS(M) \longrightarrow \{ \scal^{}_M(R) \mid R\in \OS(M) \}
    \]
    via\/ $R \mapsto \scal^{}_M(R)$. \qed
\end{enumerate}
\end{theorem}

In fact, this theorem will be the key to establish the connection
between CSMs and SSMs in Section~\ref{csl-sec:sslcsl}.

As $\Scal^{}_M(R)$ need not be a PID, we cannot characterise it by a
denominator as in Section~\ref{csl-sec:ssl}. This makes it more
difficult to establish a connection between the sets $\Scal^{}_M(R)$
for related modules. Nevertheless, there are some results.

\begin{lemma}[{\cite[Lemmas~2.2.1 and~2.2.2]{csl-habil}}]
  If\/ $M$ and\/ $N$ are commensurate modules, one has\/
  $\OS(M)=\OS(N)$ and\/ $\ts\scal_N(R) =\scal_M(R)$ for any\/ $R \in
  \OS(M)=\OS(N)$.\qed
\end{lemma}

For $\Scal^{}_M(R)$, a weaker result applies.

\begin{theorem}[{\cite[Thm.~2.2.3]{csl-habil}}]\label{theo:comp-ScalM}
  Let\/ $N$ be a submodule of\/ $M$ of index\/ $m$. Then, one has
  $m\Scal_M(R)\subseteq\Scal_N(R)\subseteq \frac{1}{m}\Scal_M(R)$.
  \qed
\end{theorem}

Above, we have already considered some examples of planar modules in
Section~\ref{csl-sec:two-dim}. We conclude our discussion of SSMs with
an important example in $\RR^{4}$.

\subsection{The icosian ring}
\label{csl-sec:ssm-I}

We already met the icosian ring $\II$ in connection with the lattice
$A_4$, where it was used as a tool to determine the SSLs of $A_4$. But
it is also interesting to classify the SSMs of $\II$ itself.

Actually, the way to determine the SSMs is completely analogous to the
case of $\JJ$ in the previous section, which is related to the fact
that both $\JJ$ and $\II$ are maximal orders in their corresponding
quaternion algebras; compare~\cite{csl-Reiner}.  Although $\II$ is not
a lattice but a $\ZZ\ts$-module in $\RR^4$, all steps can be easily
generalised for $\II$, as the latter can be viewed as a
$\ZZ[\tau]$-module of rank $4$ (or a $\ZZ[\tau]$-lattice in our above
terminology) that is properly embedded in $\RR^{4}$. Moreover, any
quaternion in $\II$ has a norm which lies in $\ZZ[\tau]$.  Thus, the
zeta function of the number field\index{number~field} $K=\QQ(\tau)$
comes into play again, and we can express the generating function of
the SSMs in terms of $\zeta_{\II}^{}(s)$, which we know from
Eq.~\eqref{csl-eq:zetaI}.\index{zeta~function!icosian~ring}

\begin{theorem}[{\cite[Thm.~3]{csl-BM}}]
  The possible indices of similar submodules of the icosian ring are
  precisely the squares of rational integers that can be represented
  by the quadratic form\/ $x^2+xy-y^2$.\index{quadratic~form} The
  number of SSMs of a given index is a multiplicative arithmetic
  function, whose Dirichlet series\index{Dirichlet~series} generating
  function\/ $\Phi_{\II}^{}$ reads
  \[
    \Phi_{\II}^{}(s) 
    \, =\, \frac{\bigl(\zeta^{}_{\II}(s)\bigr)^2}{\zeta^{}_K(4s)}
    \, = \, \frac{\bigl(\zeta^{}_K(2s)\,\zeta^{}_K(2s-1)\bigr)^2}
    {\zeta^{}_K(4s)}
  \]
  with\/ $K=\ts\QQ(\tau)$. \qed
\end{theorem}

This theorem allows us to infer the corresponding counting function
$b_{\ts\II}^{}(m)$.\index{function!arithmetic} Using the function
$g(n,r)$ defined previously in Eq.~\eqref{csl-eq:gnr}, we obtain the
following explicit result.

\begin{corollary}[{\cite[Cor.~3]{csl-BM}}]
  The arithmetic function\/ $b_{\ts\II}^{}(m)$ is
  multiplicative\index{multiplicative~function} and vanishes unless\/
  $m$ is a square. It is completely determined by specifying\/
  $b_{\ts\II}^{}(p^{2r})$ for all rational primes\/ $p$ and all\/ $r\geq
  0$. With the function\/ $g$ of Eq.~\eqref{csl-eq:gnr}, one
  has\index{counting~function!icosian~ring}
  \[
  \qquad\qquad\quad
   b_{\ts\II}^{}(p^{2r}) \, =\,
    \begin{cases}
      g(5,r), & \mbox{if $p=5$,}\\
      0, &  \mbox{if $p\equiv\pm 2\, (5)$ and $r$ is odd,}\\
      g(p^2,\frac{r}{2}), & \mbox{if $p\equiv\pm 2\, (5)$ 
         and $r$ is even,}\\
      \sum_{\ell=0}^{r} g(p,\ell) g(p,r-\ell),&
      \mbox{if $p\equiv\pm 1\, (5)$.}\qquad\qquad\qquad\quad
      \qquad\qquad \qed
    \end{cases}
  \]
\end{corollary}

The first few terms of $\Phi_{\II}^{}(s)$ read
\begin{align*}
  \Phi_{\II}^{}(s) & \, =\, 
  1 + \myfrac{10}{4^{2s}} + \myfrac{12}{5^{2s}} + \myfrac{20}{9^{2s}} 
   + \myfrac{48}{11^{2s}}
   + \myfrac{66}{16^{2s}} + \myfrac{80}{19^{2s}} 
   + \myfrac{120}{20^{2s}} + \myfrac{97}{25^{2s}} \\[1mm]
  & \qquad \, + \myfrac{120}{29^{2s}}
   + \myfrac{128}{31^{2s}} + \myfrac{200}{36^{2s}} 
   + \myfrac{168}{41^{2s}} + \myfrac{480}{44^{2s}} 
   + \myfrac{240}{45^{2s}} + \cdots
\end{align*}
Along the same lines as before, we can evaluate the asymptotic behaviour.

\begin{corollary}[{\cite[Cor.~4]{csl-BM}}]
  The asymptotic growth\index{asymptotic~behaviour} of the summatory
  arithmetic function\/ $\sum_{m\leq x} b_{\ts\II}^{}(m)$ is given
  by{$\,$}\footnote{{\ts}Compare Footnotes~\ref{csl-foot-ssla4} and
    \ref{csl-foot-ssld4} on pages \pageref{csl-foot-ssla4} and
    \pageref{csl-foot-ssld4}, respectively.}
  \[
    \sum_{m\leq x} b_{\ts\II}^{}(m) \,\sim\, 
    \frac{3 \log(\tau)^2}{5\sqrt{5}}\, x \log(x)
    \,\approx\,  0.062\ts 135\, x \log(x)\vspace{-1pt}
  \]
as\/ $x\to\infty$. \qed
\end{corollary}

Let us now turn our attention to the related problem of coincidence
site lattices. It is less common in the mathematical literature, due
to its origin in crystallography.\index{crystallography} As we shall
see, it is technically more involved and thus less developed from a
structural point of view. Nevertheless, its consideration is
completely natural and intrinsically connected with the SSL problem,
as we shall see later on.

\section{Coincidence site lattices and modules}
\label{csl-sec:csl}

\subsection{Basic facts}
\label{csl-sec:csl-bas}

Let us return to the CSLs, which we have introduced in
Definition~\ref{csl-def:csl}. To parallel our approach to the SSLs, we
introduce the set
\begin{equation}
  \OC(\vG)\, :=\, \{R\in \OG(d,\RR)\mid \vG \sim R \vG \ts  \} ,
\end{equation}
where $\vG\subset \RR^d$ is a (given) lattice.  Likewise, we use the
notation
\begin{equation}
  \SOC(\vG)\, :=\, \{R\in \OC(\vG)\mid \det (R) = 1 \}
\end{equation}
for the set of all orientation-preserving coincidence isometries,
which are also known as coincidence
rotations\index{rotation!coincidence}.  Let us mention that the groups
$\OC(\vG)$ and $\SOC(\vG)$ can be interpreted as \emph{commensurator
  groups}\index{commensurator} of the lattice $\vG$;
compare~\cite{csl-BLP96}.

\begin{fact}[{\cite[Thm.~2.1]{csl-Baake-rev}}]
  The sets\/ $\OC(\vG)$ and\/ $\SOC(\vG)$ are subgroups of\/
  $\OG(d,\RR)$. \qed
\end{fact}

Note that $\OC(\vG)$ contains the symmetry group $\OG(\vG)$ of $\vG$
as a subgroup.  Indeed, $\OG(\vG)$ is precisely the group of all
coincidence isometries of index $\Sig^{}_\vG(R)=[\vG: \vG(R)]=1$;
compare Definition~\ref{csl-def:csl}.

One certainly expects connections between lattices that are closely related.
Here, one has the following elementary result.

\begin{lemma}[{\cite[Cor.~2.1 and
  Lemma~2.6]{csl-Baake-rev}}]\label{csl-lem:comm-csl} Commensurate
  lattices have the same\/ $\OC$-groups. In particular, all
  sublattices of a lattice\/ $\vG$ have the same group of coincidence
  isometries. \qed
\end{lemma}

We have seen earlier in Lemma~\ref{csl-lem:OS-sim} that similar
lattices have conjugate $\OS$-groups.  A corresponding result for
coincidence isometries exists as well.

\begin{lemma}[{\cite[Lemma~2.5]{csl-Baake-rev}}]\label{csl-lem:similar-csl}
  Similar lattices have conjugate\/ $\OC$-groups. In particular, for
  any $0\ne \alpha\in\RR$ and any $R\in\OG(d,\RR)$, one has
\[
  \OC(\alpha R \vG) \,=\, R\, \OC(\vG)\, R^{-1},
\]
together with\/ $\Sig^{}_{\alpha R\vG}(S)=
\Sig^{}_\vG(R^{-1}SR)$. \qed
\end{lemma}

Unsurprisingly, there is also a close connection between a lattice and
its dual lattice; compare~\cite{csl-Baake-rev}.

\begin{lemma}\label{csl-lem:csl-dual}
  Let\/ $\vG^{*}$ be the dual lattice of a lattice\/ $\vG\subseteq \RR^d$.
  Then, $\OC(\vG^{*})=\OC(\vG)$ and\/
  $\Sig^{}_{\vG^{*}}(R)=\Sig^{}_\vG(R)$ for all\/ $R\in \OC(\vG)$.
\end{lemma}

\begin{proof}
  As two lattices are commensurate if and only if their duals are
  commensurate, we have $\vG^{*} \sim R\vG^{*}$ if and only if $\vG
  \sim R\vG$, where one needs the relation $(R\vG)^{*} = R\vG^{*}$. 
  By definition, this implies $\OC(\vG^{*})=\OC(\vG)$. Now,
  \[
   [\vG^{*}:\vG^{*}(R)] \,=\, [\vG^{*}:(\vG + R\vG)^{*}]
   \,=\, [\vG + R \vG : \vG\ts ] \,=\, [\vG: \vG(R)]\ts ,
  \]
  which proves the claim.
\end{proof}

An interesting observation is that the coincidence indices of a
coincidence isometry and its inverse are the same. This fact can be
proved by geometric arguments~\cite{csl-Baake-rev} involving the dual
lattice, which we will repeat here.

\begin{lemma}\label{csl-lem:ind-inv}
  Let\/ $\vG\subseteq \RR^d$ be a lattice. For any\/ $R\in \OC(\vG)$,
  one has
\[
  \Sig^{}_\vG(R)\, =\, \Sig^{}_\vG(R^{-1})\ts .
\]
\end{lemma} 

\begin{proof}
The key is the fact that $[\vG: \vG(R)]$ can be interpreted
geometrically: It is the ratio of the volumes of fundamental cells of
$\vG(R)$ and $\vG$, which is independent of the particular choice of
the latter. As isometries preserve the volume, we have
\[
\begin{split}
  \Sig^{}_\vG(R) \, & =\, [\vG: \vG(R)]\, =\, [R\vG: \vG(R)]
  \, =\, [R\vG: \vG \cap R \vG\ts ]\\
  &=\, [\vG: R^{-1} \vG \cap \vG\ts ] \, = \, \Sig^{}_\vG(R^{-1})\ts ,
\end{split}
\]
which completes the argument.
\end{proof}

As $\OC(\vG)$ is a group, it is natural to ask whether there is a
connection between the indices $\Sig^{}_{\vG}(R_{1})$,
$\Sig^{}_{\vG}(R_{2})$ and $\Sig^{}_{\vG}(R_{1}R_{2})$ for
$R_{1},R_{2}\in\OC(\vG)$.  Although no general formula exists which
expresses one of them in terms of the other two, we have the following
results.

\begin{theorem}[{\cite{csl-pzcsl7},
\cite[Lemma~3.4.3 and Thm.~3.4.4]{csl-habil}}] For any lattice\/
  $\vG\subset \RR^d$ and for any\/ $R_{1},R_{2} \in \OC(\vG)$, one has
  the following relations.
  \begin{enumerate}\itemsep=3pt
    \item $\Sig^{}_{\vG}(R_{1}R_{2})$ divides\/ 
      $\Sig^{}_{\vG}(R_{1})\ts\Sig^{}_{\vG}(R_{2})$.
    \item
      $\Sig^{}_{\vG}(R_{1}R_{2})=\Sig^{}_{\vG}(R_{1})\ts\Sig^{}_{\vG}(R_{2})$
      whenever\/ $\Sig^{}_{\vG}(R_{1})$ and\/ $\Sig^{}_{\vG}(R_{2})$
      are coprime.\qed
  \end{enumerate}
\end{theorem}

\begin{remark}\label{csl-rem:sym-rel}
  In particular, one has $\Sig^{}_{\vG}(R S) = \Sig^{}_{\vG}(R)$ if
  $\Sig^{}_{\vG}(S)=1$, or in other words, if $S\in\OG(\vG)$, which
  means that $S$ is a symmetry operation of $\vG$. Actually, if
  $S\in\OG(\vG)$, one even has $\vG(RS)=\vG(R)$. This motivates us to
  call two coincidence isometries $R$ and $R'$ \emph{symmetry
    related}\index{symmetry~related} if there exists an $S\in
  \OG(\vG)$ such that $R'=RS$. Thus, symmetry-related coincidence
  isometries generate the same CSL, but the converse is not true in
  general; see Example~\ref{csl-ex: Zsqrt3} below for an instance of
  two coincidence isometries that are not symmetry related but
  generate the same CSL.  \exend
\end{remark}

\begin{remark}\label{csl-rem:spec}
  One of the quantities we are after is the set of possible coincidence
  indices. In line with \cite{csl-BLP96}, we call this set,
  \[
    \sigma(\vG)\,=\,\Sig(\OC(\vG))\,=\,
    \{\Sig^{}_\vG(R) \mid R\in \OC(\vG)\}\ts,
  \]
  the \emph{coincidence spectrum}\index{coincidence~spectrum!ordinary}
  of $\vG$.  Sometimes, we call it the ordinary or simple coincidence
  spectrum to distinguish it from the multiple coincidence spectrum,
  which we define later; compare Section~\ref{csl-sec:mcsl}. Likewise,
  $\Sig(\SOC(\vG))=\{\Sig^{}_\vG(R) \mid R\in \SOC(\vG)\}$ is the
  subset of indices of the coincidence rotations. Clearly, we have
  $\Sig(\SOC(\vG))\subseteq \Sig(\OC(\vG))$ in general, but in many
  cases we have $\Sig(\SOC(\vG))= \Sig(\OC(\vG))$. By
  Remark~\ref{csl-rem:sym-rel}, $\Sig(\SOC(\vG))= \Sig(\OC(\vG))$
  whenever an orientation-reversing isometry exists in $\OG(\vG)$, but
  this is only a sufficient condition and by no means a necessary one.
  \exend
\end{remark}

It is not uncommon that one needs to relate the coincidence structure
of a lattice to that of various sublattices. Let us consider some
consequences on the coincidence indices.

\begin{lemma}\label{csl-lem:Sig-sublat}
  Let\/ $\vL$ be a sublattice of\/ $\vG\subseteq \RR^d$ of index\/
  $m$.  Then, $\Sig^{}_{\nts\vL}(R)$ divides\/ $m\Sig^{}_\vG(R)$ and\/
  $\Sig^{}_\vG(R)$ divides\/ $m\Sig^{}_{\nts\vL}(R)$.
\end{lemma}

\begin{proof}
  As $\vL(R)\subseteq \vG(R) \subseteq \vG$, the coincidence index
  $\Sig^{}_\vG(R)$ divides
  \[ 
    [\vG: \vL(R)] \, = \, [\vG: \vL] \, [\vL: \vL(R)] \, = \,
   m  \Sig^{}_{\nts\vL}(R) \ts ,
  \]
  which proves the second claim.

  The first claim can be proved by applying
  Lemma~\ref{csl-lem:csl-dual}.  It is well known that $\vL\subseteq
  \vG$ implies $\vG^{*} \subseteq \vL^{*}$. Since
  $\Sig^{}_\vG(R)=\Sig^{}_{\vG^{*}}(R)$ by
  Lemma~\ref{csl-lem:csl-dual}, for any lattice $\vG$, the result now
  follows immediately from the first part of the proof.
\end{proof}

Lemma~\ref{csl-lem:Sig-sublat} provides us with some useful bounds on the
coincidence indices of a sublattice. In certain cases, we can even get
sharper bounds \cite{csl-LZ9,csl-habil}. As an example, we mention the
following result, which is a special case of
\cite[Thm.~3.1.10]{csl-habil} or \cite[Thm.~2.2]{csl-LZ9} (with $u=1$
in the notation used there).

\begin{lemma}\label{csl-lem:Sig-sublat-c}
  Let\/ $\vL$ be a sublattice of\/ $\vG$ of index\/ $m$, and let\/ $R\in
  \OC(\vG)$ be such that\/ $\vL \ts \cap R(t + \vL)=\varnothing$ for all\/
  $t\in \vG\setminus \vL$. Then, $\Sig^{}_{\nts\vL}(R)$ divides\/
  $\Sig^{}_\vG(R)$.  \qed
\end{lemma}

Note that this result is the basis of the concept of \emph{colour
  coincidences}; compare~\cite{csl-LZ9,csl-Loquias10,csl-LZ15}.

\begin{remark}
  Lemma~\ref{csl-lem:Sig-sublat-c} is only useful in practice if it is
  reasonably easy to check the condition $\vL \cap R(t +
  \vL)=\varnothing$ for all $t\in \vG\setminus \vL$.  This is possible
  if the points of $\vL$ and $\vG\setminus \vL$ lie on different
  shells, that is, if the sets $\bigl\{|x| \,\big\vert\, x \in
  \vL\bigr\}$ and $\bigl\{|x| \,\big\vert\, x \in \vG\setminus
  \vL\bigr\}$ are disjoint.  This way, one can show that the three
  classes of cubic lattices have the same coincidence indices, as we
  shall see later in Section~\ref{csl-sec:cubic}. \exend
\end{remark}

\begin{remark}
  The shelling structure\index{shelling} of lattices is a well-studied
  problem.  It leads to $\Theta$-series, which are nicely summarised
  in~\cite[Sec.~2.2.3]{csl-Conway}. The problem has also been
  investigated for embedded $\ZZ\ts$-modules such as rings of
  cyclotomic integers in the plane~\cite{csl-BG03}, for the icosian
  ring in $4$-space~\cite{csl-Moody94}, or for $\ZZ\ts$-modules in
  $3$-space with icosahedral symmetry~\cite{csl-Weiss}.  Also,
  \mbox{Penrose{\ts}-type} tilings have been considered, where the
  notion of an \emph{averaged shelling} was introduced~
  \cite{csl-BG00}.  In the latter case, an interpretation of the
  results in a wider setting is still missing.  \exend
\end{remark}

\subsection{Multiple coincidences}
\label{csl-sec:mcsl}

We can generalise our considerations on CSLs by looking at
intersections of more than two commensurate lattices.  The analogous
step for modules will briefly be discussed in
Section~\ref{csl-sec:csl-csm}.  This problem is interesting for
various reasons.  On the one hand, these intersections naturally occur
in the discussion of the counting functions for CSLs; see
Section~\ref{csl-sec:count-csl} and~\cite{csl-pzcsl7}. On the other
hand, they are important in crystallography in connection with
multiple junctions~\cite{csl-gerts1,csl-gerts2,csl-gerts3}.  Another
interesting application arises in the theory of lattice quantisers
where one usually deals with rather complex lattices. There, one hopes
to simplify the problem by representing a complex lattice as the
intersection of simpler lattices~\cite{csl-Slo02a,csl-Slo02b}.

In fact, intersections of more than two isometric commensurate copies
of a lattice have already been discussed
in~\cite{csl-BG,csl-pzmcsl1,csl-BZ07,csl-habil}. Let us first recall the
corresponding definitions.

\begin{definition}
  Let $\vG\subseteq \RR^d$ be a lattice and assume $R_i\in\OC(\vG)$, with
  $i\in\{1,\ldots m\}$.  The lattice
\[
 \vG(R^{}_1,\ldots,R^{}_m)\, := \, 
  \vG\cap R^{}_1 \vG\cap\ldots\cap R^{}_m \vG \, =\, 
 \vG(R^{}_1)\cap\ldots\cap \vG(R^{}_m)
\]
is then called a \emph{multiple} CSL (MCSL) of order
$m$.\index{coincidence~site~lattice!multiple} Its index\index{index}
in $\vG$ is denoted by $\Sig(R^{}_1,\ldots,R^{}_m)$.
\end{definition}

In order to distinguish CSLs of the type $\vG(R) = \vG \cap R \vG$
from multiple CSLs, we will occasionally use the term 
simple\index{coincidence~site~lattice!simple} or
ordinary\index{coincidence~site~lattice!ordinary} CSL for $\vG(R)$.

Note that $\Sig(R^{}_1,\ldots,R^{}_m)$ is finite since
$\vG(R^{}_1,\ldots,R^{}_m)$ is a finite intersection of mutually
commensurate lattices~\cite{csl-Baake-rev}. In particular, an
immediate consequence of the second isomorphism theorem for groups is the
following result.

\begin{lemma}[{\cite[Lemma~3.3.1]{csl-habil}}]\label{csl-lem:sig12}
  For\/ $R^{}_{1},R^{}_{2}\in\OC(\vG)$, one has
\[
  \Sig(R^{}_1,R^{}_2)\, =\, 
  \frac{\Sig(R^{}_1)\, \Sig(R^{}_2)}{\Sig_+(R^{}_1,R^{}_2)}\ts ,
\]
where\/ $\Sig_+(R^{}_1,R^{}_2)$ is the index of the direct sum\/
$\vG_+(R^{}_1,R^{}_2)= \vG(R^{}_1)+ \vG(R^{}_2)$ in the 
original lattice\/ $\vG$.\qed
\end{lemma}

More generally, one has the following relation.

\begin{lemma}[{\cite[Lemma~3.3.2]{csl-habil}}]\label{csl-lem:sig1n}
  For any\/ $R_i\in\OC(\vG)$,
\[
  \Sig(R^{}_1,\ldots,R^{}_m)\, =\, 
  \frac{\Sig(R^{}_1,\ldots,R^{}_{m-1})\, \Sig(R^{}_m)}
  {\Sig_+(R^{}_1,\ldots,R^{}_{m-1};R^{}_m)}\ts ,
\]
where\/ $\Sig_+(R^{}_1,\ldots,R^{}_{m-1};R^{}_m)$ is the index of
\[
   \vG_+(R^{}_1,\ldots,R^{}_{m-1};R^{}_m)\, =\, 
   \vG(R^{}_1,\ldots,R^{}_{m-1})+ \vG(R^{}_m)
\]
in\/ $\vG$. In particular, $\Sig(R^{}_1,\ldots,R^{}_m)$ divides\/
$\Sig(R^{}_1)\cdot\ldots\cdot\Sig(R^{}_m)$.\qed
\end{lemma}

This result allows us to infer some basic properties of the
coincidence spectrum.  Recall from Remark~\ref{csl-rem:spec} that the
simple coincidence spectrum was defined as
$\sigma(\vG)=\{\Sig^{}_\vG(R) \mid R\in \OC(\vG)\}$. Likewise, we
introduce the \emph{multiple coincidence
  spectrum}\index{coincidence~spectrum!multiple} as the set
\[
  \sigma_\infty(\vG)
  =\{\Sig(R^{}_1,\ldots,R^{}_m)\mid R^{}_i\in \OC(\vG), m\in\NN\}. 
\]
Clearly, we have
\begin{equation}\label{csl-eq:dieletzte}
  \sigma(\vG) \,\subseteq\, \sigma_\infty(\vG) \,\subseteq\, 
  \widehat{\sigma}(\vG)\ts ,
\end{equation}
where $\widehat{\sigma}(\vG)$ is the set of all positive integers that
divide an integer from the (multiplicative) semigroup generated by
$\sigma(\vG)$. We shall come back to this relation and possible
consequences at the end of Section~\ref{csl-sec:mcsl-cub-3}.

\subsection{MCSLs and monotiles}
\label{csl-sec:csl-mono}

In \cite[Sec.~5.7.7]{csl-TAO}, the SCD monotile due
to Schmitt, Conway and Danzer for $\RR^3$ is
discussed.  This convex tile, together with translated and rotated
copies (but no reflected copies), allows to form periodic
two-dimensional layers $L$, which can only be stacked vertically by
rotating the layers by a fixed irrational rotation $R$.  In
particular, any tiling $\cT$ of $\RR^d$ obtained this way must have
the form
\begin{equation}\label{csl-monostacks}
  \cT \, = \bigcup_{m\in\ZZ} (mc + R^m L)\ts , 
\end{equation}
where $c$ is a suitable vector orthogonal to the plane of the layer
$L$; compare \cite[Eq.~(5.7)]{csl-TAO} and~\cite{csl-Dan95,csl-BF05}.
As $R^n\ne \one$ for any $n\in\ZZ\setminus\{ 0\}$, any resulting
tiling of $\RR^{d}$ is aperiodic.\index{aperiodic}

Let us analyse this construction in some more detail, in terms of
MCSLs. Let $L$ be one fixed layer of an SCD tiling.  If $\vG$ is the
group of translations that leaves $L$ invariant, then the stack of
$n+1$ layers $\cT = \bigcup_{m=0}^n (mc + R^m L)$ is invariant under
the MCSL $\vG_n:=\vG \cap R\vG \cap \ldots \cap R^n\vG$, with
$\vG^{}_{0}=\vG$.  As $\bigcap_{n\in\Nnull} \vG_n = \{ 0\}$, the tiling
is aperiodic; compare~\cite[Lemma~5.8 and Rem.~5.12]{csl-TAO}.

If we pursue these ideas further, we see that we can construct
monotiles in all odd dimensions \mbox{$2m+1\geq 3$}.  Let us start
with a lattice $\vG\in\RR^d$ for $d=2m$ and assume that $\vG$ has a
coincidence rotation\index{rotation!coincidence} $R$ such that
$R^n\ne \one$ for any $n\in\ZZ\setminus\{ 0\}$. We choose a unit cell
$U$ (possibly convex or a parallelohedron, with suitable markers) of
the CSL $\vG\cap R\vG$ such that no lattice point of $\vG$ or $R\vG$
is on the boundary. We can always choose $U$ in such a way that it
tiles $\RR^d$ only periodically, with $\vG$ as the corresponding
lattice of periods.  We define a prototile $T$ in $\RR^{d+1}$ as
$U\times [0,1]$ and add markings on the bottom and the top of $T$ as
follows. On the bottom, we mark each lattice point of $\vG$ that is
contained in $U$ (to avoid any complication, we choose some mark
without any symmetry) and on top we mark the lattice points of $R\vG$
(with the same marks just rotated by $R$). This guarantees that we can
stack these layers of tiles vertically only by rotating them by
$R$. Hence, the only tilings we can get are tilings of the
form~\eqref{csl-monostacks}, with $L$ replaced by $\vG$.

As $R^n\ne \one$ for any $n\in\ZZ\setminus\{ 0\}$, the tiling is not
periodic in the remaining (transversal) direction. To exclude any
periodicity in a direction parallel to the layers, we need
$\bigcap_{n\in\Nnull} \vG_n = \{ 0\}$. Such an $R$ exists for the
square lattice. In fact, each coincidence rotation $R$ that is not a
symmetry of the square lattice\index{square~lattice!CSL} has this
property.  Likewise, $\ZZ^{2m}$ has infinitely many coincidence
rotations $R$ that satisfy $\bigcap_{n\in\Nnull} \vG_n = \{ 0\}$. In
particular, we may choose $R$ as the direct product of two-dimensional
coincidence rotations, each of which fails to be a symmetry of the
square lattice.

However, note that, although all these tilings are aperiodic, they are
not \emph{strongly aperiodic},\index{strongly~aperiodic} as there is
still a skew rotation\index{rotation!skew} symmetry left, which means
that the symmetry group contains a subgroup isomorphic to $\ZZ$;
compare~\cite[Def.~5.22]{csl-TAO}. In this sense, also the original
SCD tiling is aperiodic, but not strongly aperiodic. To the best of
our knowlegde, no strongly aperiodic monotile in $3$-space is known.

With this restriction, the above construction establishes the
existence of monotiles in odd dimensions. For even dimensions, the
analogous construction fails, as the corresponding lattice then has
odd dimension and any coincidence rotation of it leaves at least one
lattice direction invariant. Whether monotiles exist in even
dimensions is still an open problem. Only in dimension $d=2$, a
monotile for the Euclidean plane (with next-to-nearest neighbour local
rules) was discovered by Joan Taylor; see \cite[Sec.~5.7.6]{csl-TAO}
and references therein for a more detailed account of the tiling, its
properties (due to Socolar and Taylor) and predecessors (due to
Penrose).

\subsection{Counting functions}
\label{csl-sec:count-csl}

As sketched in Section~\ref{csl-sec:square}, we are interested in
several enumeration problems. In particular, for a given index, we are
after the number of coincidence isometries and the number of CSLs. For
a fixed lattice $\vG$, we shall denote the number of CSLs of a given
index $n$ by $c^{}_{\vG}(n)$. As the same CSL can be generated by
several coincidence isometries, it is not useful to deal with the
total number of coincidence isometries directly, but it is more
convenient to use a properly normalised counting function instead.

If $S$ is a symmetry operation of $\vG$, we have $\vG(RS)=\vG(R)$
for any coincidence isometry $R$. This means that the number of
coincidence isometries with a given index is a multiple of
$\card(\OG(\vG))$, where $\OG(\vG)$ is the symmetry group of
$\vG$. Thus, we prefer to deal with the function
$c^{\mathsf{iso}}_{\vG}(n)$, which counts the coincidence isometries
modulo the symmetry group. Then, the number of coincidence isometries
of a given index $n$ is given by
$\card(\OG(\vG))\ts\ts c^{\mathsf{iso}}_{\vG}(n)$. Likewise, we define
$c^{\mathsf{rot}}_{\vG}(n)$ for all coincidence rotations, now counted
modulo $\SO(\vG) = \OG (\vG) \cap \SO (d)$. This guarantees
$c^{\mathsf{iso}}_{\vG}(1)=c^{\mathsf{rot}}_{\vG}(1)=c^{}_{\vG}(1)$.

Let us mention that
$c^{\mathsf{rot}}_{\vG}(n)=c^{\mathsf{iso}}_{\vG}(n)$ holds whenever
there exists an orientation-reversing symmetry operation. In
particular, $c^{\mathsf{rot}}_{\vG}(n)=c^{\mathsf{iso}}_{\vG}(n)$
holds for every lattice $\vG$ in odd dimensions.

Recall from Remark~\ref{csl-rem:sym-rel} that two coincidence
isometries $R$ and $R'$ are called \emph{symmetry
  related}\index{symmetry~related}, if there exists a symmetry
operation $S\in \OG(\vG)$ such that $R'=RS$. As symmetry-related
coincidence isometries generate the same CSL, it follows that
$c^{\mathsf{iso}}_{\vG}(n)$ is an upper bound for
$c^{}_{\vG}(n)$. However, these two numbers differ in general, as
non-symmetry-related coincidence isometries may still generate the
same CSL.

\begin{example}\label{csl-ex: Zsqrt3}
  As an example for differing counting functions for lattices versus
  isometries, we consider the rectangular lattice
  $\vG=\ZZ[\ii\sqrt{3}]\subset\RR^2$, which is a sublattice of the
  hexagonal lattice $\vL=\ZZ[\omega]$ with
  $\omega=\frac{1+\ii\sqrt{3}}{2}$. Then, one has the
  inclusions $2\vL\subset\vG\subset\vL$ with indices
  $[\vL:\vG]=[\vG:2\vL]=2$. As $\omega^{k}$ with $k\not\equiv 0 \bmod 3$
  is a symmetry operation for $\vL$ but not for $\vG$, we infer
  $\Sig^{}_{\vG}(\omega^{k}) > 1 = \Sig_{\vL}(\omega^{k})$ for
  $k\in\{1,2\}$.  It follows from Lemma~\ref{csl-lem:Sig-sublat} that
  one in fact has $\Sig^{}_{\vG}(\omega^{k})=2=[\vG:2\vL]$. Together
  with $\vG(\omega^{k})\supseteq 2 \vL$, this gives $\vG(\omega^{k}) =
  2 \vL$ for $k\in\{1,2\}$. As $\omega$ and $\omega^2$ fail to be
  symmetry related, this implies
  $c^{\mathsf{iso}}_{\vG}(2)=c^{\mathsf{rot}}_{\vG}(2)>c^{}_{\vG}(2)$. In
  fact, a more detailed analysis yields $c^{\mathsf{rot}}_{\vG}(2)=2 >
  1= c^{}_{\vG}(2)$.

  This example can easily be generalised as follows. Whenever one has
  a lattice $\vL\subset\vG$ such that the index $[\vG:\vL]=p$ is a
  prime and such that $\OG(\vG) \subset \OG(\vL)$ with
  $[\OG(\vL):\OG(\vG)]\geq 3$, one can infer
  $c^{\mathsf{iso}}_{\vG}(p)>c^{}_{\vG}(p)$ by analogous
  arguments. Moreover, if $p$ is not in the coincidence spectrum of
  $\vL$, one can even show that
  \[
  c^{\mathsf{iso}}_{\vG}(p)\, =\, [\OG(\vL):\OG(\vG)] -1
  \, >\,  1 \, =\,  c^{}_{\vG}(p)\ts .
  \]
  This follows from $\Sig^{}_{\vG}(R)=p$ together with the observation
  that $\vG(R)=\vL$ for any isometry
  \mbox{$R\in\OG(\vL)\setminus\OG(\vG)$}. \exend
\end{example}

In several important examples, all these counting functions are
multiplicative, which suggests the use of generating functions of
Dirichlet series type to determine their asymptotic growth rate, as we
have done in several examples so far. In general, however, the
counting functions fail to be multiplicative, though we have the
following weaker result.

\begin{theorem}[\cite{csl-pzcsl7,csl-habil}]\label{csl-theo:cisosup}
  The arithmetic function\/ $c^{\mathsf{iso}}_\vG(n)$,
  $c^{\mathsf{rot}}_\vG(n)$ and\/ $c^{}_{\vG}(n)$ are
  supermultiplicative, that is, $c^{\mathsf{iso}}_\vG(mn)\geq
  c^{\mathsf{iso}}_\vG(m)\, c^{\mathsf{iso}}_\vG(n)$ holds for coprime
  integers\/ $m$ and\/ $n$, and likewise for the other
  functions.\qed
\end{theorem}

Given the close relationship of similar sublattices and coincidence
site lattices, which we will analyse below, one might be tempted to
assume that the counting functions $b^{\mathsf{pr}}_\vG(n)$ and
$b^{}_\vG(n)$ for similar sublattices are multiplicative if and only
if the corresponding counting functions $c^{}_\vG(n)$ and
$c^{\mathsf{iso}}_\vG(n)$ are multiplicative. However, this is not
true.  In fact, SSLs seem to be more prone to violation of
multiplicativity than CSLs. For instance, for $\vG=\ZZ\times 5\ZZ$,
multiplicativity is violated for $b^{\mathsf{pr}}_\vG(n)$ and
$b^{}_\vG(n)$, compare~\cite{csl-BSZ-sim}, while $c^{}_\vG(n)$ and
$c^{\mathsf{iso}}_\vG(n)$ are still multiplicative~\cite{csl-JonasF}.

We expect that the connection between $c^{\mathsf{iso}}_\vG(n)$ and
$c^{}_\vG(n)$ must be closer, and in fact one has the following
result.

\begin{theorem}[\cite{csl-pzcsl7,csl-habil}]\label{csl-theo:csl-mult}
  If the arithmetic function\/ $c_\vG^{\mathsf{iso}}(n)$ is multiplicative,
  then so is the function\/ $c^{}_\vG(n)$. \qed
\end{theorem}

It is presently unknown whether the converse holds or not. As the
counting functions $c^{}_\vG(n)$ and $c_\vG^{\mathsf{iso}}(n)$ are
generally not multiplicative, it is desirable to have some criteria
when they are.  For $c^{}_\vG(n)$, we have the following result.

\begin{theorem}[\cite{csl-pzcsl7,csl-habil}]\label{csl-thm:mc}
  For a lattice $\vG\subset \RR^d$,
  the following statements are equivalent.
  \begin{enumerate}\itemsep=2pt
  \item \label{mc1} The arithmetic function\/ $c^{}_{\vG}(m)$ is
    multiplicative.\index{multiplicative~function}
  \item \label{mc2} Every simple CSL\/ $ \vG(R)$ has a representation of
    the form\/ 
    \[
      \vG(R)\, = \, \vG(R^{}_1)\cap\ldots\cap \vG(R^{}_n)
    \]
    with all indices\/ $\Sig^{}_{\vG}(R^{}_i)$ being powers of distinct primes.
  \item \label{mc3} Every MCSL\/ $\vG(R^{}_1,\ldots,R^{}_n)$ of order\/ $n$
    has a representation of the form\/ 
\[ 
   \vG(R^{}_1,\ldots,R^{}_n) \, = \, \vG_1\cap\ts\ldots\ts\cap \vG_k
\]
    where the\/ $\vG_i$ are MCSLs of order
    at most\/ $n$ whose indices\/ $\Sig_i$ are powers of
    distinct primes.\qed
  \end{enumerate}
\end{theorem}
Let us mention that the representation $\vG(R)=
\vG(R^{}_1)\cap\ts\ldots\ts\cap \vG(R^{}_n)$, if it exists, is unique up to
the order of the $\vG(R^{}_i)$. In fact, if $\Sigma(R)=p^{r_1}_{1}
\cdots p^{r_n}_{n}$ is the prime factorisation of $\Sigma(R)$ and
$m_{i}:=\Sigma(R)\, p^{-r_i}_{i}$, then $\vG(R^{}_i)$ can be
calculated via
\[
  \vG(R^{}_i) \, =\, \Bigl(\myfrac{1}{m_{i}} \vG(R)\Bigr) \cap \vG.
\]
Note that the right-hand side is always a sublattice of $\vG$ of index
$p^{r_i}_{i}$. The key in proving Theorem~\ref{csl-thm:mc} is to show
that it is actually a CSL if $c^{}_{\vG}(m)$ is multiplicative. On the
other hand, one can show that $\vG(R^{}_1)\cap\ldots\cap \vG(R^{}_n)$
is always a simple CSL if the indices are coprime, which allows one
to count all CSLs that have such a representation.  Analogous results
hold for MCSLs; compare~\cite{csl-pzcsl7,csl-habil}.

A similar criterion exists for $c_\vG^{\mathsf{iso}}(n)$. In order to
formulate it, we need some terminology.  We call a bijection
$\pi=\{p^{}_1,p^{}_2\ldots\}$ from the positive integers onto the
prime numbers an ordering of the prime numbers. We call a
decomposition of a coincidence isometry $R=R^{}_1\cdots R^{}_n$ a
$\pi$-decomposition of $R$ if, for any $i$, $\Sig^{}_{\vG}(R^{}_i)$ is a
power of $p^{}_{i}$ (we allow $\Sig^{}_{\vG}(R^{}_i)=p^{0}_{i}=1$).  It is
clear that any $\pi$-decomposition is unique up to point group
elements.
 
\begin{theorem}[\cite{csl-pzcsl7,csl-habil}]\label{csl-thm:mcrot}
  The following statements are equivalent.
\begin{enumerate}\itemsep=2pt
\item \label{mcrot1} The arithmetic function\/
  $c^{\mathsf{iso}}_\vG(m)$ is multiplicative.\index{multiplicative~function}
\item \label{mcrot2} There exists an ordering\/ $\pi$ of the prime
  numbers such that any coincidence isometry\/ $R$ has a $($unique\/$)$
  $\pi$-decomposition.
\item \label{mcrot3} For any ordering\/ $\pi$ of the prime numbers,
  there exists a\/ $\pi$-decom\-position of every coincidence
  isometry\/ $R$.  \qed
\end{enumerate}
\end{theorem}

\subsection{Generalisations to $\ZZ\ts$-modules}
\label{csl-sec:csl-csm}

The considerations on CSLs can be generalised to embedded
$\ZZ\ts$-modules.\index{module!embedded} As most of the definitions
and results depend only on the algebraic properties, their
generalisation is straightforward. However, some of our previous
proofs involved the use of the dual lattice, which has no immediate
counterpart for $\ZZ\ts$-modules. In these cases, some care and new
approaches are needed.

We recall from Definition~\ref{csl-def:commM} that two embedded
$\ZZ\ts$-modules $M_1$ and $M_2$ are called commensurate, $M_1 \sim
M_2$, if their intersection $M_1 \cap M_2$ has finite index in both
$M_1$ and $M_2$. The notion of a coincidence site lattice can now
easily be transferred to the case of modules as follows.

\begin{definition}
  Let $M \subset \RR^d$ be a properly embedded $\ZZ\ts$-module of
  finite rank, and consider $R\in \OG(d,\RR)$. If $M \sim R M$, then
  $M(R):=M \cap R M$ is called a \emph{coincidence site
    module}\index{coincidence~site~module} (CSM). Then, $R$
  is called a \emph{coincidence
    isometry}\index{coincidence~isometry}. The corresponding index
  $\Sig^{}_{\! M}(R):=[M: M(R)]$ is called its \emph{coincidence
    index}\index{coincidence~index}.
\end{definition}

Again, we are interested in the sets
\begin{align}
  \OC(M)& \, :=\, \{R\in \OG(d,\RR)\mid M \sim R M \ts \}\\
\intertext{and}
  \SOC(M)&\, :=\, \{R\in \OC(M)\mid \det(R) = 1 \}.
\end{align}
As expected, these sets are indeed groups.

\begin{theorem}
  If\/ $M \subset \RR^d$ is a properly embedded\/ $\ZZ\ts$-module, the set
  of all coincidence isometries, $\OC(M)$, forms a subgroup of\/
  $\OG(d,\RR)$. Likewise, the group\/ $\SOC(M)$ is a subgroup of\/
  $\SO(d,\RR)$. \qed
\end{theorem}

Lemmas~\ref{csl-lem:comm-csl} and~\ref{csl-lem:similar-csl}
immediately generalise as follows.

\begin{lemma}[{\cite[Lemmas~3.1.2 and 3.1.3]{csl-habil}}]
   The\/ $\OC$-groups are equal for commensurate modules.
   Moreover, similar modules have conjugate\/ $\OC$-groups. In
   particular, one has\/ $\OC(\alpha R M) = R\, \OC(M)\, R^{-1}$ and\/
   $\Sig^{}_{\alpha RM}(S) = \Sig^{}_{\nts M}(R^{-1} S R)$. \qed
\end{lemma}

Obviously, there is no analogue of Lemma~\ref{csl-lem:csl-dual}. Thus,
it is not evident whether an analogue of Lemma~\ref{csl-lem:ind-inv}
exists. Fortunately, it does, but its proof requires some results on
irreducible polynomials over the ring $\ZZ$; compare~\cite{csl-habil}.

\begin{theorem}[{\cite[Thm.~3.1.6]{csl-habil}}]\label{csl-theo:Sigma-Rinv}
  Let\/ $M \subseteq \RR^d$ be an embedded\/ $\ZZ\ts$-module of finite
  rank.  For any\/ $R\in \OC(M)$, we have\/ $\Sig^{}_{\nts M}(R) =
  \Sig^{}_{\nts M}(R^{-1})$.\qed
\end{theorem}

Again, it is interesting to compare the coincidence indices of modules
with those of their submodules.

\begin{theorem}[{\cite[Thm.~3.1.9]{csl-habil}}]\label{csl-theo:Sig-submod}
  Let\/ $N$ be a submodule of\/ $M$ of index\/ $m$.  Then,
  $\Sig^{}_{\nts M}(R)$ divides\/ $m\Sig^{}_{\nts N}(R)$ and\/ 
  $\Sig^{}_{\nts N}(R)$ divides\/ $m\Sig^{}_{\nts M}(R)$.\qed
\end{theorem}

Whereas the second statement of Lemma~\ref{csl-lem:Sig-sublat} can be
generalised immediately, the first claim of
Theorem~\ref{csl-theo:Sig-submod} requires a different approach, as we
generally lack the notion of a dual module.  The proof is algebraic in
nature and can be found in~\cite{csl-habil}; compare also
\cite{csl-LZ9}, where a similar approach for lattices is described.

\subsection{Similar versus coincidence submodules}
\label{csl-sec:sslcsl}

After we have dealt with similar sublattices and coincidence site
lattices and their generalisations, let us return to the connections
between them.  It is clear that there are substantial connections, as
became obvious from the groups we defined along the way. In line with
Section~\ref{csl-sec:square}, let us illustrate this in more detail
with the square lattice, the latter once again identified with
$\ZZ[\ii]$, the ring of Gaussian integers.\index{Gaussian~integer}

\begin{example}
  We know from Theorem~\ref{csl-theo:nongenmul} and
  Example~\ref{csl-ex:ssl-sq} that $\SOS(\ZZ[\ii])$ is given by
  \[
   \SOS (\ZZ[\ii]) \, 
    = \, \Bigl\{ \myfrac{z}{\lvert z \rvert} \;\big|\; 
   0\ne z\in \ZZ[\ii] \Bigr\} 
   \, \simeq\,  C_8 \times \ZZ^{(\aleph_0)}.
  \]
  In comparison, we have
  \[
   \SOC (\ZZ[\ii]) \, 
    = \, \Bigl\{ \myfrac{z}{\bar{z}} \;\big|\; 
   0\ne z\in \ZZ[\ii] \Bigr\} 
   \, = \, \Bigl\{ \myfrac{z^2}{\lvert z \rvert^2} \;\big|\; 
   0\ne z\in \ZZ[\ii] \Bigr\} 
   \, \simeq\,  C_4 \times \ZZ^{(\aleph_0)} \ts ,
  \]
  where $C_4$ is the groups of units of $\ZZ[\ii]$, while a full set
  of generators of $\ZZ^{(\aleph_0)}$ is provided by
  $\{\frac{\ts\ts \pi_{\nts p}\ts\ts}{\overline{\pi_{\nts p}}} =
  \frac{\pi_{\nts p}^2}{\lvert\pi_{\nts p}\rvert^2}\mid p\equiv 1
  \bmod{4} \}$,
  where, for each $p$ of this kind, $\pi_{\nts p}$ is one of the
  Gaussian primes\index{prime!Gaussian} with
  $\pi_{\nts p} \overline{\pi_{\nts p}} = p$.  Comparing these with
  the set of generators for $\SOS(\ZZ[\ii])$ in
  Example~\ref{csl-ex:ssl-sq}, one sees that all generators of
  $\SOC(\ZZ[\ii])$ are squares of generators of $\SOS(\ZZ[\ii])$, and
  we infer that
  \[
   \SOS (\ZZ[\ii])/ \SOC (\ZZ[\ii]) \,\simeq\, C_2^{(\aleph_0)}\ts ,
  \]
  which means that the factor group $\SOS (\ZZ[\ii])/ \SOC (\ZZ[\ii])$
  is an infinite Abelian $2$-group; compare~\cite{csl-svenja1}.  \exend
\end{example}

Let us now see how this observation can be put on a more general
basis. We will formulate the main results immediately for modules;
compare~\cite{csl-pzsslcsl1,csl-habil}.  For the special cases of
lattices, we refer to~\cite{csl-svenja1}.  The corresponding results
for a special class of modules, namely the
$\cS$-lattices\index{S-lattice@$\cS$-lattice} from
Definition~\ref{csl-def:S-lat}, can be found in~\cite{csl-svenja2}.

\begin{lemma}[{\cite[Lemma~3.2.1]{csl-habil}}]\label{csl-lem:OC-scal}
  Let\/ $M\subseteq \RR^d$ be a finitely generated free\/
  $\ZZ\ts$-module. Then,
\begin{enumerate}\itemsep=2pt
\item $R\in \OC(M)$ if and only if\/ $1 \in \scal_{\nts M}(R)$.
\item $R\in \OG(M)$ if and only if\/ $1 \in \Scal_{\nts M}(R)$.\qed
\end{enumerate}
\end{lemma}

Here, $\OG(M)$ is the point symmetry group of $M$.  An immediate
consequence for lattices is the following result.

\begin{corollary}\label{csl-lem:OC-den}
If\/ $\vG \subset \RR^d$ is a lattice, one has\/ $R\in \OC(\vG)$ if
and only if\/ $R\in \OS(\vG)$ together with\/ $\den^{}_\vG(R)\in
\NN$.\qed
\end{corollary}

It is often helpful to know some connections between the coincidence
indices and the corresponding denominators;
compare~\cite{csl-pzsslcsl1}.

\begin{lemma}\label{csl-theo:denG-Sig1}
  Let\/ $\vG$ be a lattice in\/ $\RR^d$. For any\/ $R\in \OC(\vG)$,
  one has
\begin{enumerate}\itemsep=2pt
\item \label{csl-enu:denG-Sig1}
     $\lcm\left(\den^{}_\vG(R),\den^{}_\vG(R^{-1})\right)$ 
       divides\/ $\Sig^{}_\vG(R)$;
\item\label{csl-enu:denG-Sig2}
     $\Sig^{}_\vG(R)$ divides\/ 
     $\gcd\left(\den^{}_\vG(R),\den^{}_\vG(R^{-1})\right)^d$.
  \item \label{csl-enu:denG-Sig3}
    $\Sig^{}_\vG(R)^2$ divides\/
    $\lcm\left(\den^{}_\vG(R),\den^{}_\vG(R^{-1})\right)^d$.
\end{enumerate}
\end{lemma}
\begin{proof}
  For~\eqref{csl-enu:denG-Sig1}, recall that $\vG(R)$ has index
  $\Sig(R)$ in $\vG$, thus
  \[
    \Sig(R) \vG \,\subseteq\, \vG(R) \,\subseteq\, R \vG, 
  \]
  or, equivalently, $\Sig(R) R^{-1} \vG \subseteq \vG$. Consequently,
  $\Sig(R)$ is a multiple of $\den(R^{-1})$. By symmetry, $\den(R)$ is
  a divisor of $\Sig(R^{-1})=\Sig(R)$ as well, and
  claim~\eqref{csl-enu:denG-Sig1} follows.

  For~\eqref{csl-enu:denG-Sig2}, we exploit that $\den(R)$ is an
  integer for $R\in\OC(\vG)$. Consequently, $\den(R) R \vG$ is a
  sublattice of both $\vG$ and $R \vG$, wherefore one has $\den(R) R
  \vG \subseteq \vG(R)$.  Comparing the indices of $\den(R) R \vG$ and
  $\vG(R)$ in $\vG$ shows that $\Sig(R)$ divides $\den(R)^d$. Using
  $\Sig(R^{-1})=\Sig(R)$ as above yields~\eqref{csl-enu:denG-Sig2}.
  
  Finally, let $a:=\lcm\left(\den(R),\den(R^{-1})\right)$. Then, $a \vG$ and
  $a R \vG$ are both sublattices of $\vG$ and of $R \vG$, hence
  $a (\vG + R \vG)$ is a sublattice of $\vG \cap R \vG$ with index
  \[
    [R \cap R \vG: a (\vG + R \vG)]\, =\, \myfrac{a^d}{\Sig(R)^2}\ts ,
  \]
  as
  $\Sig(R)=[\vG: \vG(R)]=[\vG + R \vG:\vG\ts ]$. Hence $\Sig(R)^2$
  divides $a^d$.
\end{proof}

The situation becomes particularly simple for planar lattices, where
we get the following result by recalling
$\den^{}_\vG(R)=\den^{}_\vG(R^{-1})$.

\begin{corollary}[{\cite[Cor.~2.6]{csl-pzsslcsl1}}]\label{csl-cor:den-Sig2D}
  Let $\vG$ be a lattice in $\RR^2$. Then, for any $R\in\OC(\vG)$, one
  has $\Sig^{}_\vG(R)\, =\, \den^{}_\vG(R)\ts$.\qed
 \end{corollary}

Our main result follows from Theorem~\ref{csl-theo:scalgroupM}.

\begin{theorem}[{\cite[Thm.~3.2.2]{csl-habil}}]
Let\/ $M\subset\RR^{d}$ be an embedded\/ $\ZZ\ts$-module of finite
rank. Then, the kernel of the homomorphism
\[
\begin{split}
    \phi \! : \, \OS(M) & \,\longrightarrow\, 
    \RR_+/ (\scal^{}_{\nts M}(\one)\cap\RR_+)\ts , \\
    R & \,\longmapsto\, \scal^{}_{\nts M}(R) \cap \RR_+ \ts ,
\end{split}    
\]
   is the group\/ $\OC(M)$. Thus, $\OC(M)$ is a normal subgroup of\/
   $\OS(M)$, and\/ $\OS(M)/\OC(M)$ is Abelian.\qed
\end{theorem}

This result was first proved for lattices in~\cite{csl-svenja1} and
later generalised to $\cS$-lattices in~\cite{csl-svenja2}.

If $M\subseteq \RR^d$ is a lattice or an $\cS$-lattice, all elements of
$\OS(M)/\OC(M)$ have finite order. In particular, their order is a
divisor of $d$; see~\cite{csl-svenja1,csl-svenja2}.

\begin{theorem}\label{csl-theo:order-d}
  Let\/ $M\subseteq \RR^d$ be a lattice or an\/
  $\cS$-lattice\index{S-lattice@$\cS$-lattice}.  Then, the factor
  group given by\/ $\OS(M)/\nts \nts \OC(M)$ is the direct sum of cyclic 
  groups of prime power order that divide\/ $d$.\qed
\end{theorem}

The close relationship between SSLs and CSLs is also reflected in the
following condition for two CSLs to be equal.

\begin{lemma}[{\cite[Lemma 3.4.2]{csl-habil}}]
\label{csl-lem:equal-csl}
  Let\/ $\vG \subset \RR^d$ be a lattice.  Assume that\/ $R^{}_1,
  R^{}_2\in \OC(\vG)$ generate the same CSL, so\/
  $\vG(R^{}_1)=\vG(R^{}_2)$.  Then, one has\/
  $\Sig(R^{}_1)=\Sig(R^{}_2)$ together with\/ 
  $\den(R_1^{-1})=\den(R_2^{-1})$.
\end{lemma}
\begin{proof}
  The statement about $\Sig$ is trivial. For the denominator, observe that
\[
  \den(R_1^{-1})\vG\, \subseteq\, 
  \vG(R^{}_1)\, =\, \vG(R^{}_2)\,\subseteq\, R^{}_2\vG.
\]
  Consequently,
\[
  \den(R_1^{-1})R_2^{-1}\vG\,\subseteq\,\vG,
\]
  which shows that $\den(R_1^{-1})$ is a multiple of $\den(R_2^{-1})$.
  Then, by symmetry, $\den(R_2^{-1})$ is a multiple of
  $\den(R_1^{-1})$ as well, and the claim follows.
\end{proof}
This result is particularly useful in the following examples, when we
have to characterise those coincidence isometries that generate the
same CSL.  Let us start our series of illustrations with some examples
in the plane.

\section{(M)CSMs of planar modules with 
$N$-fold symmetry}
\label{csl-sec:csl-nfold}

We can generalise the results of the square lattice to all rings
$\ZZ[\xi_n]$ of cyclotomic integers which are PIDs;
compare~\cite{csl-Pleasants,csl-BG}. Thus, let $n$ be one of the
numbers given in Eq.~\eqref{csl-pid-cyclo}.  We have seen in
Section~\ref{csl-sec:two-dim} that the similar submodules are then
exactly the non-trivial ideals of $\ZZ[\xi_n]$, and that the
similarity rotations are given by $\frac{v}{\lvert v \rvert}$ with
$v\in \ZZ[\xi_n]$.

As any of these $29$ modules is also a ring, we have $\mathrm{MR}
(\ZZ[\xi_n]) =\ZZ[\xi_n]$.  This implies that the coincidence
rotations\index{rotation!coincidence} are precisely given by
\mbox{$\ee^{\ii\varphi}=\frac{v}{\lvert v \rvert}$} for which $\lvert
v \rvert^2= v\bar{v}$ is a square in $\ZZ[\xi_n]$. In other words,
using the unique prime factorisation up to units in $\ZZ[\xi_{n}]$,
the coincidence rotations are precisely the rotations of the form
$\varepsilon \frac{w}{\bw}$ with $0\ne w\in\ZZ[\xi_{n}]$, where
$\varepsilon$ is a unit in $\ZZ[\xi_n]$. Here, we may assume that
$\frac{w}{\bw}$ is a reduced fraction, which means that $w$ and $\bw$
are coprime.  Under this assumption, one finds
\begin{equation}
  \ZZ[\xi_n]\cap \varepsilon \myfrac{w}{\bw}\ZZ[\xi_n]\, = \, 
  w\ts\ZZ[\xi_n]\ts .
\end{equation}

To find the possible values of $w$, we mention that a prime
$\omega\in \ZZ[\xi_n]$ can be a factor of $w$ only if
$\frac{\omega}{\bom}$ is not a unit in $\ZZ[\xi_n]$. Thus, we only
have to consider the so-called \emph{complex splitting
  primes}.\index{prime!splitting} To expand on this, consider the
prime factorisation of a rational prime $p$ over the real subring
$\scO_n=\ZZ[\xi_{n}+\bar{\xi}_{n}]$, which is the ring of integers of
the maximal real subfield $\QQ(\xi_n + \bar{\xi}_n)$ of
$\QQ(\xi_n)$. Let $\pi$ be a prime in $\scO_n$. Now, the complex
splitting primes are those primes $\pi$ that split as
$\pi=\omega_\pi \bom_\pi$ over $\ZZ[\xi_n]$, with $\omega_\pi$ and
$\bom_\pi$ being non-associated primes in $\ZZ[\xi_n]$, which means
that $\frac{\omega_\pi}{\bom_\pi}$ is not a unit.  Thus, the possible
values of $w$ are of the form
\begin{equation}
  w \, = \, \varepsilon \prod_{\pi} \omega_\pi^{t_\pi^+} 
   \bom_\pi^{t_\pi^-},
\end{equation}
where $\varepsilon$ is a unit, $t_\pi^+ t_\pi^- = 0$, and the product
runs over all primes $\pi\in\scO_n$ that divide $w \ts\bw$. In other
words, any coincidence rotation\index{rotation!coincidence} in
$\SOC(\ZZ[\xi_n])$ can be written as a finite product
\begin{equation}
  \ee^{\ii\varphi} \, = \, \varepsilon'\ts \myfrac{w}{\bw} \, = \, \varepsilon'
  \prod_{\pi} \left(\frac{\omega_\pi}{\bom_\pi} \right)^{\! t_\pi},
\end{equation}
with $t^{}_\pi = t_\pi^+ - t_\pi^-$, where $\pi$ runs over the complex
splitting primes of $\scO_{n}$ and where $\varepsilon'$ is again a
unit.

Any complex splitting prime $\pi\in\scO_n$ lies over a unique rational
prime $p$, which is the norm of $\pi$ in $\scO_{n}$. Then, one also
calls $p$ a complex splitting prime of the field
extension\index{number~field} $\QQ(\xi_{n})/\QQ$. The set of all such
rational primes is abbreviated as $\cC_{n}$ and thus consists of all
rational primes that split in the final step from
$\QQ(\xi_{n}+\bar{\xi}_{n})$ to $\QQ(\xi_{n})$. To expand on the
structure of the primes and their splitting, we recall that the index
\mbox{$\bigl[\ZZ[\xi_n]: \omega_\pi \ZZ[\xi_n]\bigr]=p^{\ell_p}$}
depends only on $p$, where $\ell_p$ is an integer which we will
specify below. As a result, the CSM
$\ZZ[\xi_n]\cap \varepsilon \frac{w}{\bw}\ZZ[\xi_n]= w\ts\ZZ[\xi_n]$
has index
\begin{equation}
   \Sig^{}_{\ZZ[\xi_n]} \bigl( \varepsilon \ts \tfrac{w}{\bw}\bigr) 
   \, = \prod_{\pi} p^{\ell_p \lvert t_\pi \rvert}
\end{equation}
with $t_{\pi}$ as introduced above.  Thus, the possible coincidence
indices are products of the so-called \emph{basic
  indices}\index{index!basic} $p^{\ell_p}$, and the coincidence
spectrum is the (multiplicative) monoid generated by these basic
indices. In other words,
\begin{equation}
  \sigma \bigl(\ZZ[\xi_n]\bigr) \, =\,  
    \Bigl\{ \prod_{p\in\cC_{n}} p^{\ell_p t_p} \,\big\vert\, t_p\in\NN,
        \ts\text{only finitely many $t_p\ne 0$} \Bigr\},
\end{equation}
where $\cC_n$ is the set of complex splitting primes as introduced
above.\index{prime!splitting}

As $\ZZ[\xi_n]$ is a PID for the list of $n$ we consider here, the
counting function $c_n(m)=c^{}_{\ZZ[\xi_n]}(m)$ is
multiplicative,\index{multiplicative~function} wherefore it suffices
to determine it for $m=p^{\ell_p}$.  This is now a purely
combinatorial task, and one finally arrives at the following result.

\begin{theorem}[{\cite[Thm.~3]{csl-Pleasants} and \cite[Thm.~1]{csl-BG}}]
  \label{csl-thm:psi-xin}
  Let\/ $n$ be one of the\/ $29$ numbers from
  Eq.~\eqref{csl-pid-cyclo}.  Then, the generating function for the
  number\/ $c_n(m)=c^{}_{\ZZ[\xi_n]}(k)$ of CSMs of\/ $\ZZ[\xi_n]$ of
  index\/ $k$ is given by\index{Dirichlet~series}
  \[
    \Psi^{}_{\ZZ[\xi_n]}(s) \, =  \, \sum_{k=1}^\infty \frac{c_n(k)}{k^s}
       \, =  \, \frac{\zeta^{}_{\KK_n}(s)}{\zeta^{}_{\LL_n}(2s)} 
      \begin{cases}
        (1 + p^{-s})^{-1}, & \text{ if } n=p^r, \\
        1, & \text{ otherwise,}
      \end{cases}
  \]
  where\/ $\zeta^{}_{\KK_n}(s)$ and\/ $\zeta^{}_{\LL_n}(2s)$ are the
  Dedekind zeta\index{zeta~function!Dedekind} functions of the number
  field\/ $\KK_n = \QQ(\xi_n)$ and its maximal real subfield\/
  $\LL_n = \QQ(\xi_n + \bar{\xi}_n)$, respectively.  If\/ $\cC_n$
  denotes the set of complex splitting primes\index{prime!splitting}
  for the field extension\/ $\KK_n/\QQ$, then\/
  $\Psi^{}_{\ZZ[\xi_n]}(s)$ has the Euler product
  expansion\index{Euler~product}
  \[
    \Psi^{}_{\ZZ[\xi_n]}(s) \, = \prod_{p\in\cC_n} \left(
    \frac{1 + p^{-\ell_p s}}{1 - p^{-\ell_p s}}
    \right)^{\frac{m_p}{2}},
  \]
  with certain integers\/ $\ell_p$ and\/ $m_p$ as follows.  If\/ $p
  \nmid n$, one has $m_p=\frac{\phi(n)}{\ell_p}$ where\/ $\ell_p$ is
  the smallest positive integer such that\/ $p^{\ell_p}\equiv 1
  \bmod{n}$.  If\/ $p | n$ together with\/ $n=p^t r$, where\/ $r$
  and\/ $p$ are coprime, one has\/ $m_p=\frac{\phi(r)}{\ell_p}$
  where\/ $\ell_p$ is the smallest positive integer such that\/
  $p^{\ell_p}\equiv 1 \bmod{r}$. \qed
\end{theorem}

For explicit values of $\ell_p$ and $m_p$, see~\cite[Tables~1
  and~2]{csl-BG}.  The first terms of $\Psi^{}_{\ZZ[\xi_n]}(s)$ for
all $n$ from Eq.~\eqref{csl-pid-cyclo} are listed
in~\cite[Table~4]{csl-BG}.

The explicit expression of $\Psi^{}_{\ZZ[\xi_n]}(s)$ in terms of zeta
functions allows us to determine the asymptotic behaviour of $c_n(k)$.
Here, $\Psi^{}_{\ZZ[\xi_n]}(s)$ is a meromorphic function that is
analytic in the half-plane $\{ \Real(s) >1 \}$ and has a simple pole
at $s=1$, which results in linear growth for the summatory function of
$c_n(k)$. In particular, using Theorem~\ref{csl-thm:meanvalues}, we
get the following result.

\begin{corollary}[{\cite[Cor.~1]{csl-BG}}]\label{csl-cor:asym-xin}
  The asymptotic behaviour\index{asymptotic~behaviour} of the number\/
  $c_n(k)$ of CSMs of\/ $\ZZ[\xi_n]$ of index\/ $k$ is given by
  \[
    \sum_{k \leq x} c_n(k)\, \sim \, \gamma^{}_n\, x
  \]
 as\/ $x\to\infty$, where\/ $\gamma^{}_n$ is the residue of\/
 $\Psi^{}_{\ZZ[\xi_n]}(s)$ at\/ $s=1$, which is given by
  \[
    \gamma^{}_n\, = \, \frac{\alpha^{}_n}{\zeta^{}_{\LL_n}(2)} 
         \begin{cases}
           p/(p+1), & \text{ if } n=p^r,\\
           1, & \text{ otherwise,}
         \end{cases}
  \]
  with\/ $\alpha^{}_n:= \Res_{s=1}\, \bigl(\zeta^{}_{\KK_n}(s)\bigr)$. \qed
\end{corollary}

Note that the constants $\alpha^{}_n$ and $\gamma^{}_n$ can be
calculated by expressing $\zeta^{}_{\KK_n}(s)$ and
$\zeta^{}_{\LL_n}(s)$ in terms of Riemann's
zeta\index{zeta~function!Riemann} function $\zeta(s)$ and certain
\mbox{$L\ts$-series}; compare~\cite[Sec.~4]{csl-BG}.  For some
examples including $n\in\{3, 4, 5, 7, 8, 12\}$, we refer
to~\cite[Sec.~4]{csl-Pleasants}, where the average
$\gamma^{}_n = \lim_{x\to\infty} \frac{1}{x} \sum_{k \leq x} c_n(k)$
has been evaluated explicitly.  Numerical values for $\alpha^{}_n$ and
$\gamma^{}_n$ are listed in~\cite[Table~3]{csl-BG}.

Let us continue with multiple coincidences. As any MCSM is an
intersection of simple CSMs, we see that
\begin{equation}\label{csl-eq:mcsm}
  \begin{split}
    \ZZ[\xi_n] \cap \varepsilon_1 \frac{w^{}_1}{\bw^{}_1}\ZZ[\xi_n] \cap
    \ldots \cap \varepsilon_k
     \frac{w^{}_k}{\bw^{}_k}\ZZ[\xi_n] \qquad \\[2mm]
    \mbox{ }\qquad =\,  w^{}_1 \ZZ[\xi_n] \cap \ldots \cap w^{}_k \ZZ[\xi_n]
    \, =\,  w \ts \ZZ[\xi_n]
  \end{split}
\end{equation}
with $w=\lcm(w^{}_1, \ldots, w^{}_k)$. Again, any MCSM is an ideal of
$\ZZ[\xi_n]$, but $w$ is more general now. Nevertheless, $w$ is still
of the form of a finite product,
\begin{equation}
  w\, = \, \varepsilon \prod_{\pi} \omega_\pi^{t_\pi^+} \bom_\pi^{t_\pi^-},
\end{equation}
but now without any further restriction on the non-negative integers
$t_\pi^+$ and $t_\pi^-$. This shows that the coincidence spectrum does
not change, so that
\begin{equation}\label{csl-eq:specstable}
  \sigma \bigl(\ZZ[\xi_n]\bigr)\, = \,
  \sigma_{\infty} \bigl(\ZZ[\xi_n]\bigr);
\end{equation}
compare~\cite[Cor.~2]{csl-BG}.

It follows from Eq.~\eqref{csl-eq:mcsm} that any MCSM can actually be
written as the intersection of only two simple CSMs. This allows one
to determine the number $c^{\infty}_n(k)$ of MCSMs of $\ZZ[\xi_n]$ of
index $k$. The result reads as follows.

\begin{theorem}[{\cite[Thm.~3]{csl-Pleasants} and \cite[Thm.~1]{csl-BG}}]
  \label{csl-thm:mcsm-nfold}
  Let\/ $n$ be one of the\/ $29$ numbers from
  Eq.~\eqref{csl-pid-cyclo}.  Then, the generating function for the
  number\/ $c^{\infty}_n(k)=c^{\infty}_{\ZZ[\xi_n]}(k)$ of CSMs of\/
  $\ZZ[\xi_n]$ of index\/ $k$ is given by\index{Dirichlet~series}
  \[
  \Psi^{\infty}_{\ZZ[\xi_n]}(s) \, =
      \sum_{k=1}^\infty \frac{c^{\infty}_n(k)}{k^s}
      \, = \prod_{p\in\cC_n} \Bigl(
    \myfrac{1}{1 - p^{-\ell_p s}} \Bigr)^{m_p},
  \]
  where\/ $\cC_n$ denotes the set of complex splitting
  primes\index{prime!splitting} for the field extension\/ $\KK_n/\QQ$
  and the integers\/ $\ell_p$ and\/ $m_p$ are those from
  Theorem~$\ref{csl-thm:psi-xin}$.  \qed
\end{theorem}

This nice generating function is due to the fact that we actually
count all ideals whose index $m$ factors into primes contained in
$\cC_n$. As $\Psi^{\infty}_{\ZZ[\xi_n]}(s)$ still has a simple pole at
$s=1$, using Theorem~\ref{csl-thm:meanvalues} once more, we get a
linear growth behaviour again. The determination of the residue is a
bit more complicated here, as $\Psi^{\infty}_{\ZZ[\xi_n]}(s)$ cannot
be represented via zeta functions in a simple way. Still, one has
the following result.

\begin{corollary}[{\cite[Cor.~1]{csl-BG}}]\label{csl-cor:asym-mcsl-xin}
    The summatory function\/ $\sum_{k \leq x} c^{\infty}_n(k)$ has the
  asymptotic behaviour\index{asymptotic~behaviour}
  \[
    \sum_{k \leq x} c^{\infty}_n(k)\, \sim \, \beta^{}_n \, x
  \]
  as\/ $x\to\infty$, with the growth constant\/ $\beta^{}_n =
  \Res_{s=1}\, \bigl(\Psi^{\infty}_{\ZZ[\xi_n]}(s)\bigr) = q^{}_n \gamma^{}_n$.
  Here, $\gamma^{}_n$ is defined as in
  Corollary\/~$\ref{csl-cor:asym-xin}$, and\/ $q^{}_n$ is given by
  \[
    \hspace{4cm}
    q^{}_n \, := \, \lim_{s\to 1} 
             \frac{\Psi^{\infty}_{\ZZ[\xi_n]}(s)}{\Psi^{}_{\ZZ[\xi_n]}(s)}
       \, = \prod_{\ell=1}^{\infty} 
                \bigl(\Psi^{}_{\ZZ[\xi_n]}(2^\ell)\bigr)^{2^{-\ell}}.
    \hspace{3.8cm}\qed
  \]
\end{corollary}

The last formula in Corollary~\ref{csl-cor:asym-mcsl-xin} is a
consequence of the representation
\begin{equation}
  \begin{split}
  \Psi^{\infty}_{\ZZ[\xi_n]}(s)\, 
    & = \ts\bigl(\Psi^{\infty}_{\ZZ[\xi_n]}
      (2^{L+1}s)\bigr)^{2^{-(L+1)}}
      \prod_{\ell=0}^{L} 
        \bigl(\Psi^{}_{\ZZ[\xi_n]}
        (2^{\ell}s)\bigr)^{2^{-\ell}} \\
    & = \,\Psi^{}_{\ZZ[\xi_n]}(s)  \prod_{\ell=1}^{\infty} 
        \bigl(\Psi^{}_{\ZZ[\xi_n]}
        (2^{\ell}s)\bigr)^{2^{-\ell}},
  \end{split}
\end{equation}
which holds for any integer $L\geq 0$;
compare~\cite[Prop.~2]{csl-BG}. As the infinite product converges
rapidly, $q_n$, and thus $\beta^{}_n$, can be calculated numerically
in an efficient way; see~\cite[Table~3]{csl-BG} for a list of values
of $\beta^{}_n$.

\begin{example}\label{csl-ex:mcsl-square}
  Let us once more consider the square lattice for
  illustration. Theorem~\ref{csl-thm:mcsm-nfold} implies that the
  generating function for its MCSLs\index{square~lattice!multiple~CSL}
  reads\index{Dirichlet~series}
  \[
  \begin{split}
    \Psi^{\infty}_{\minisquare}(s) \; & = \,
    \sum_{m=1}^{\infty} \frac{c_{\minisquare}^{\infty}(k)}{k^{s}}
    \; = \; \left( 1 + 2^{-s}\right)^{-1}
    \frac{\zeta^{}_K(s)}{\zeta(2s)}\\[1mm]
     &  = 
    \prod_{p\equiv1 \ts\ts (4)} \myfrac{1}{(1-p^{-s})^2}
     \; = \; \Psi^{}_{\minisquare}(s)
    \prod_{p\equiv1 \ts\ts (4)} \myfrac{1}{1-p^{-2s}}
    \ts ,
  \end{split}
  \]
  where we have employed the notation
  $c_{\minisquare}^{\infty}(k)=c_{4}^{\infty}(k)$ for the number of
  MCSLs. The latter is a multiplicative function, whose values for
  (positive) prime powers are given
  by\index{counting~function!square~lattice}
  \[
    c_{\minisquare}^{\infty}(p^r)=
    \begin{cases}
      r+1, & \text{if } p \equiv 1 \bmod{4}, \\
      0, & \text{otherwise}.
    \end{cases}
  \]
  The first terms of the expansion read
  \[
    \Psi^{\infty}_{\minisquare}(s) \, = 
    1 + \myfrac{2}{5^s} + \myfrac{2}{13^s} + \myfrac{2}{17^s} + \myfrac{3}{25^s}
    + \myfrac{2}{29^s} + \myfrac{2}{37^s} + \myfrac{2}{41^s} + \myfrac{2}{53^s}
    + \ldots,
  \]
  and a comparison with $\Psi^{}_{\minisquare}(s)$ from
  Eq.~\eqref{csl-eq:csl-square} yields
  \[
    \Psi^{\infty}_{\minisquare}(s)-\Psi^{}_{\minisquare}(s)
    \, = \, \myfrac{1}{25^s} + \myfrac{2}{125^s} + \myfrac{1}{169^s} +
    \myfrac{1}{289^s} + 
    \myfrac{2}{325^s} + \myfrac{2}{425^s} + \myfrac{3}{625^s} + \ldots;
  \]
  compare~\cite[Table~5]{csl-BG}. Note that no additional MCSLs exist
  for square-free indices. The first terms of
  $\Psi^{\infty}_{\minisquare}(s)-\Psi^{}_{\minisquare}(s)$ indicate
  that most MCSLs actually are simple CSLs, which is confirmed by the
  asymptotic growth rates of the summatory functions,
  \[
    \gamma^{}_{\minisquare}:=\gamma^{}_4=\frac{1}{\pi}\approx 0.318{\ts}310 
    \ \text{ and } \
    \beta^{}_{\minisquare}:=\beta^{}_4\approx 0.336{\ts}193\ts ,
    \]
  of the simple and multiple CSLs, respectively;
  compare~\cite[Table~3]{csl-BG}.

  Furthermore, note that the simple CSLs\index{square~lattice!CSL} are
  all primitive SSLs, whereas the additional
  MCSLs\index{square~lattice!multiple~CSL} are all non-primitive
  SSLs. In fact, an SSL is an MCSL if and only if its index factors
  into primes $p\equiv 1 \bmod{4}$ only.
  
  The possible coincidence indices are precisely the positive odd
  integers that are products of primes $p\equiv 1 \bmod{4}$ only. In
  other words, the coincidence
  spectra\index{coincidence~spectrum!square~lattice} of the square
  lattice are given by
  \[
    \sigma(\ZZ^2)\, =\,\sigma^{}_{\infty}(\ZZ^2)\, =\,
    \{ \text{all finite products of primes $\equiv 1 \bmod{4}$}\}
  \]
  and thus agree in this case.  \exend
\end{example}

Let us now turn our attention to some important examples in three and
four dimensions, where quaternions will play a fundamental role;
compare Section~\ref{csl-sec:quat}.  On the one hand, following
Cayley,\index{rotation!Cayley~parametrisation} rotations in three
and four dimensions can be parametrised conveniently by quaternions,
which allows us to exploit the algebraic structure of certain rings of
quaternions, including the rings $\JJ$, $\LL$ and $\II$. On the other
hand, these rings are either four-dimensional lattices themselves,
like $\JJ$ and $\LL$, or they are related to lattices.  For instance,
the lattice $A_4$ is related to the icosian ring $\II$; see
Section~\ref{csl-sec:a4}. Likewise, the projections of $\JJ$ and $\LL$
onto the three-dimensional imaginary subspace yield the body-centred
and primitive cubic lattices, respectively. Moreover, $\II$ is a
$\ZZ[\tau]$-lattice of rank $4$ in the sense of
Definition~\ref{csl-def:S-lat}.

\section{The cubic lattices}
\label{csl-sec:cubic}

The three-dimensional cubic lattices are among the most important
lattices in crystallography,\index{crystallography} and the study of
their coincidences is a classic
problem~\cite{csl-Ranga66,csl-Grimmer74,csl-GBW,csl-Grimmer84}.
Later, these lattices have been revisited in a more mathematical
context \cite{csl-Baake-rev,csl-Z2}.  Here, the key tool is the ring
$\JJ$ of Hurwitz
quaternions,\index{quaternion!Hurwitz}\index{Hurwitz~ring} since it
turns out that any coincidence rotation\index{rotation!coincidence} of
a three-dimensional cubic lattice can be parametrised by a Hurwitz
quaternion; compare~\mbox{\cite[Sec.~2.5.4]{csl-TAO}} as well
as~\cite{csl-BM} and references therein for some general background.

Let us first define our setting. We use the conventions
of~\cite[Ex.~3.2]{csl-TAO} and define
\begin{equation}\label{csl-cubiclat}
  \Gpc\, := \,\ZZ^3, \quad \Gbcc \, := \, \ZZ^3 \cup (u+\ZZ^3)\ts , 
   \quad \Gfcc \, :=\, \Gbccstar\ts ,
\end{equation}
with $u=\frac{1}{2}(1,1,1)$. Here, the index $\mathsf{pc}$ indicates
that this lattice is a primitive cubic lattice, and likewise
$\mathsf{bcc}$ and $\mathsf{fcc}$ denote the body-centred and the
face-centred cubic lattices, respectively.

Traditionally, one starts with the primitive cubic lattice, partly due
to the fact that this lattice allows the easiest treatment with
elementary methods. We will deviate from this tradition here, as the
body-centred lattice allows for the nicest description of its
coincidence site lattices.

\begin{fact}
One has\/ $\Gbcc\simeq\Imag (\JJ)$ and\/ $\Gpc\simeq\Imag (\LL)$.\qed
\end{fact}

Recall that $\JJ$ is a maximal order\index{maximal~order} and a
principal ideal ring, whereas $\LL$ is neither. This indicates that
$\Gbcc$ is easier to deal with, because we can exploit the arithmetic
properties of $\JJ$ while relying on its ideal structure.

The first step in determining the CSLs of $\vG$ is the determination
of $\OC(\vG)$. Since the point reflection $I \! : \, x \mapsto -x$ is
a symmetry operation of all three-dimensional lattices, it is actually
sufficient to determine $\SOC(\vG)$. We get the following well-known
result; compare~\cite{csl-Baake-rev, csl-BLP96, csl-habil}.

\begin{theorem}\label{csl-theo:cub-equalSig}
  Let\/ $\Gpc, \Gbcc, \Gfcc \subset \RR^3$ be the primitive, the 
  body-centred, and the face-centred cubic lattice of
  Eq.~\eqref{csl-cubiclat}, respectively. Then, one has\/
  $\OC(\Gpc) = \OC(\Gbcc) = \OC(\Gfcc) = \Orth(3,\QQ)$ together with
\[
    \Sig^{}_{\mathsf{pc}}(R) \, = \,
    \Sig^{}_{\mathsf{bcc}}(R) \, = \,
    \Sig^{}_{\mathsf{fcc}}(R)
\]
for all\/ $R\in \Orth(3,\QQ)$.  
\end{theorem}

\begin{proof}
  The equality of the three $\OC$-groups is a consequence of the fact
  that the three cubic lattices are mutually commensurate.  The
  explicit form of the $\OC$-group is most easily seen for the lattice
  $\Gpc=\ZZ^{3}$, since the standard basis of $\RR^3$ is also a
  lattice basis of $\ZZ^3$.

  Note that $\Gpc \subset \Gbcc$ is a sublattice of index $2$. One
  easily verifies that $|x|^2$ is an integer for all $x\in \Gpc$ and
  that $4\ts |x|^2\equiv 3 \bmod 4$ for all $x\in \Gbcc \setminus \Gpc$.
  Hence, an application of Lemma~\ref{csl-lem:Sig-sublat-c} shows that
  $\Sig^{}_{\mathsf{pc}}(R)$ divides $\Sig^{}_{\mathsf{bcc}}(R)$. The
  reverse divisibility property can be obtained by considering the
  dual lattice $\Gbccstar=\Gfcc$.  In particular, $|x|^2$ is even for
  all $x\in \Gbccstar$ and odd for all $x\in \Gpc \setminus
  \Gbccstar$.
\end{proof}

Note that this result was already proved by Grimmer, Bollmann and
Warrington~\cite{csl-GBW}. Actually, they used a similar method
in their proof, and Lemma~\ref{csl-lem:Sig-sublat-c} is a natural
generalisation of their approach.

\begin{remark}
  Let us note that $\OC(\vG) = \OS(\vG) = \OG(3,\QQ)$ holds for all
  cubic lattices of Eq.~\eqref{csl-cubiclat}.  We have determined
  $\OC(\vG)$ explicitly above, but we could have argued more
  abstractly by using the connection of $\OS(\vG)$ and $\OC(\vG)$ as
  laid out in Section~\ref{csl-sec:sslcsl}.  It follows from
  Theorem~\ref{csl-theo:order-d} that all elements of
  $\OS(\vG)/\OC(\vG)$ have an order that divides $3$. On the other
  hand, the cubic lattices are rational lattices, which implies that
  all elements of $\OS(\vG)/\OC(\vG)$ have an order at most $2$. Thus,
  we indeed have $\OC(\vG)=\OS(\vG)$. Moreover, as $\vG$ is
  commensurate to $\ZZ^3$, the elements of $\OC(\vG)$ are exactly the
  rational orthogonal matrices, $\OG(3,\QQ)$.  \exend
\end{remark}

As any rotation in $\OG(3,\QQ)$ can be parametrised by a rational
quaternion, we can parametrise the coincidence
rotations\index{rotation!coincidence} by primitive Lipschitz or
Hurwitz quaternions.  Contrary to the traditional approach in
crystallography, we opt for primitive Hurwitz quaternions here;
compare~\cite{csl-Baake-rev}.  In particular, via
Eq.~\eqref{csl-eq:Rpar}, one finds
\begin{align}
  \SOC(\Gbcc)\, =\,\{R(q) \mid q\in\JJ\} \,=\, 
  \{R(q) \mid q\in\JJ \mbox{ is primitive}\}.
\end{align}

The first step in determining the coincidence index is the calculation
of the denominator $\den^{}_\vG(R(q))$. From Eq.~\eqref{csl-eq:Rpar},
we see that $\den^{}_\vG(R(q))$ must be a divisor of $|q|^2$. Taking
into account that the greatest common divisor of all matrix entries of
$R(q)$ is a power of $2$, we get the following result.

\begin{corollary}\label{csl-coro:cub-den}
  For any cubic lattice\/ $\vG$ in the setting of
  Eq.~\eqref{csl-cubiclat}, we have\/
  $\den^{}_\vG(R(q))=\frac{|q|^2}{2^\ell}$, where\/ $q$ is a primitive
  Hurwitz\index{quaternion!Hurwitz} quaternion and\/ $\ell$ is the
  maximal exponent such that\/ $2^\ell \big\vert |q|^2$.  \qed
\end{corollary}

Note that $\ell$ is either $0$ or $1$, depending on whether $|q|^2$ is
odd or even. If one chooses to use primitive Lipschitz quaternions,
one gets $\ell\in\{0,1,2\}$ instead. Furthermore, note that the
denominators for any similarity rotation $R$ and its inverse are the
same, $\den^{}_\vG(R^{-1})=\den^{}_\vG(R)$, as $R^{-1}(q)=R(\bar{q})$.

\begin{proposition}\label{csl-theo:cub-Sig}
  For any cubic lattice\/ $\vG\subset\RR^{3}$
  as in Eq.~\eqref{csl-cubiclat}, we have
  \[
   \Sig^{}_\vG(R(q)) \, = \, \den^{}_\vG(R(q)) \, = \, 
   \frac{|q|^2}{2^\ell}\ts , 
  \]
where\/ $q$ is a primitive Hurwitz quaternion and\/ $\ell$ is the
maximal exponent such that\/ $2^\ell \big\vert |q|^2$.
\end{proposition}

\begin{proof}
  From Theorem~\ref{csl-theo:denG-Sig1}, we know that the index
  $\Sig^{}_\vG(R(q))$ is a multiple of $\den^{}_\vG(R(q)) =
  \frac{|q|^2}{2^\ell}$ and a divisor of $\den^{}_\vG(R(q))^2$. As the
  latter is odd, so is $\Sig^{}_\vG(R(q))$, and it is thus sufficient
  to show that $\Sig^{}_\vG(R(q))$ divides $|q|^2$.

  By Theorem~\ref{csl-theo:cub-equalSig}, the coincidence indices are
  the same for all cubic lattices. Hence, it suffices to prove that
  $\Sig^{}_{\mathsf{bcc}}(R(q))$ divides $|q|^2$. We observe
  $R(q)\Imag (xq) = \Imag (qx)$, which implies that $R(q)\Imag (\JJ
  q)=\Imag (q\JJ)$, from which we infer that $\Imag (q\JJ) \subseteq
  \Gbcc(R(q))$. Consequently, $\Sig_{\mathsf{bcc}}(R(q))$ divides
  the index \mbox{$[\Imag (\JJ) : \Imag (q\JJ)]$}.

  In order to determine the latter, we note that $[\JJ: q\JJ\ts
  ]=|q|^4$ for any $q\in\JJ$. Moreover, one has
\[
  \bigl[\bigl(\JJ \cap \Real(\HH)\bigr):
  \bigl((q\JJ) \cap \Real(\HH)\bigr)\bigr] \,=\, |q|^2,
\] 
  where $\Real(\HH)$ is to be understood as the real axis.  Hence
  $\mbox{$[\Imag (\JJ) : \Imag (q\JJ)]$}=\frac{[\JJ: q\JJ\ts ]}
  {[\Real (\JJ): \Real (q\JJ)]}=|q|^2$, and
  $\Sig_{\mathsf{bcc}}(R(q))$ thus divides $|q|^2$.
\end{proof}

If $\den^{}_\vG(R(q))$ is square-free, there also exists a 
simple alternative proof. Since we have
$\den^{}_\vG(R)=\den^{}_\vG(R^{-1})$ for the
cubic lattices, Theorem~\ref{csl-theo:denG-Sig1} tells us that
$\Sig^{}_\vG(R)^2$ divides $\den^{}_\vG(R)^3$, and if $\den^{}_\vG(R)$
is square-free, we may infer that $\Sig^{}_\vG(R)=\den^{}_\vG(R)$.

\begin{remark}\label{csl-rem:spec-cub}
  It follows from Proposition~\ref{csl-theo:cub-Sig} that the
  coincidence indices are odd positive integers. Moreover, Lagrange's
  four-square theorem \cite{csl-Grosswald} tells us that any positive
  integer is a sum of four squares. Hence, for any odd $n$, there
  exists a Hurwitz quaternion $q$ such that $n=\lvert q\rvert^2$. This
  implies that any odd positive integer is realised as a coincidence
  index, or in other words, the coincidence spectrum
  \index{coincidence~spectrum!cubic~lattices} of any cubic lattice is
  precisely the set of positive odd integers, so $\sigma(\Gbcc) =
  \sigma(\Gpc) = \sigma(\Gfcc)=2\ts\Nnull+1$.  \exend
\end{remark}

\begin{proposition}\label{csl-theo:bcc-csl}
  If\/ $q$ is a primitive Hurwitz quaternion with\/ $|q|^2$ odd, one
  has the relation\/ $\Gbcc(R(q)) = \Imag (q\JJ)$.
\end{proposition}

\begin{proof}
We have seen $\Imag (q\JJ) \subseteq \Gbcc(R(q))$ in the proof of
Proposition~\ref{csl-theo:cub-Sig}.  If $|q|^2$ is odd, then both
sublattices have the same index,
\[
\Sig_{\mathsf{bcc}}(R(q))\, =\, |q|^2\, =\,
    [\Imag (\JJ) : \Imag (q\JJ)]\ts ,
\]
and hence $\Imag (q\JJ) = \Gbcc(R(q))$.
\end{proof}

If $|q|^2$ is even, $q$ can be written as $q=rs$ with $r,s\in\JJ$,
where $|r|^2$ is odd and $|s|^2=2^\ell$. As $R(s)$ is a symmetry
operation of $\Gbcc$, we see that $\Gbcc(R(q)) = \Gbcc(R(r)) = \Imag
(r\JJ)$.

An analogous result exists for the primitive cubic lattice $\ZZ^3$ and
can be stated as follows; compare~\cite[Thm.~3.5.5]{csl-habil}.

\begin{proposition}\label{csl-theo:pc-csl}
  If\/ $q$ is a primitive Lipschitz quaternion with\/ $|q|^2$ odd, 
  one has the relation\/ $\Gpc(R(q)) = \Imag  (q\ts\LL) $.
\end{proposition}

\begin{proof}
  From Proposition~\ref{csl-theo:bcc-csl}, we infer that
  \[
     \Gbcc(R(q)) \cap \Imag (\LL) \, = \, 
     \Imag (q\JJ) \cap \Imag (\LL) \ts .
  \]
  As $\Gpc(R(q)) \nts\subseteq\nts \Gbcc(R(q)) \cap \Imag (\LL)$, and both
  $\Gpc(R(q))$ and $\Gbcc(R(q)) \cap \Imag (\LL)$ have index $2$ in
  $\Gbcc(R(q))$, we also infer $\Gpc(R(q)) = \Gbcc(R(q)) \cap \Imag
  (\LL)$. A similar argument applied to $\Imag (q\LL) \subseteq \Imag
  (q\JJ) \cap \Imag (\LL)$ shows that one has $\Imag (q\LL) = \Imag (q\JJ)
  \cap  \Imag (\LL)$, which completes the proof.
\end{proof}

Again, in analogy to the situation for $\Gbcc$, we can find a
quaternion $r\in\LL$ such that $\Gpc(R(q)) = \Imag (r\LL)$ if $|q|^2$
is even.

Let us return to the CSLs of
$\Gbcc$. Proposition~\ref{csl-theo:bcc-csl} shows that any CSL of
$\Gbcc$ is the projection $\Imag (q\JJ)$ of an ideal $q\JJ$ of
$\JJ$. On the other hand, whenever $q$ is an odd primitive quaternion,
$\Imag (q\JJ)$ is a CSL of $\Gbcc$.  If we can show that there is a
bijection between the set of ideals $\{ q\JJ \mid q \text{ is
  primitive and odd}\}$ and the set of CSLs, then we can easily count
the CSLs of a given index, as the number of ideals of a fixed index is
well known~\cite{csl-Vigneras}.  The first step into this direction is
the following result.

\begin{lemma}
  Let\/ $q,r\in\JJ$ such that\/ $|q|^2$ and\/ $|r|^2$ are odd. Then,
  one has\/ $\Imag (q\JJ) \subseteq \Imag (r\JJ) $ if and only if\/ $
  q\JJ \subseteq r\JJ $.
\end{lemma}

\begin{proof}
  Only the `only if' part is non-trivial. $\Imag (q\JJ) \subseteq
  \Imag (r\JJ)$ implies that $|r|^2$ divides $|q|^2$. Now,
\[
   \Imag (r\JJ) \, = \, \Imag (r\JJ) + \Imag (q\JJ) 
   \, = \, \Imag (r\JJ + q\JJ) \, = \, \Imag (s\JJ) \ts , 
\]
which shows that $|r|^2=|s|^2$, where $s$ is the greatest common left
divisor of $r$ and $q$.  Hence $s^{-1}r\in \JJ$, but as $|s^{-1}r|=1$,
it must be a unit.  Thus $q \JJ \subseteq s \JJ = r \JJ$.
\end{proof}

From this, we infer the following result; compare~\cite{csl-BLP96} for
a similar result in a more general context.

\begin{corollary}
  Let\/ $q,r\in\JJ$ such that\/ $|q|^2$ and\/ $|r|^2$ are odd. Then,
  one has\/ $\Imag (q\JJ) = \Imag (r\JJ) $ if and only if\/ $ q\JJ =
  r\JJ $.  \qed
\end{corollary}

In other words, putting the previous steps together, we have proved
the following result.

\begin{lemma}\label{csl-lem:bcc-bij}
  The mapping\/ $q \JJ \mapsto \Gbcc(R(q))$, which maps the set of
  left ideals generated by primitive quaternions with\/ $|q|^2$ odd
  onto the set of CSLs of\/ $\Gbcc$, is a bijection.  \qed
\end{lemma}

An analogous result can be proved for the other cubic lattices as
well. 

\begin{theorem}\label{csl-thm:csls-ideals}
  The mapping\/ $q \JJ \mapsto \vG_{a}(R(q))=\Imag (q\JJ)\cap
  \vG_{a}$, with fixed type
  $a\in\{\mathsf{pc},\mathsf{bcc},\mathsf{fcc}\}$, defines a bijection
  between the set of left ideals generated by primitive quaternions
  with\/ $|q|^2$ odd and the set of CSLs of\/ $\vG_{a}$. 
\end{theorem}

\begin{proof}
  From $\vG^{}_{a}(R(q)) \subseteq \vG^{}_{a}$ and $\vG^{}_{a}(R(q))
  \subseteq \Gbcc(R(q)) = \Imag (q\JJ)$, we see that we must have
  $\vG^{}_{a}(R(q))
  \subseteq \Imag (q\JJ) \cap \vG^{}_{a}$. As $[\Gbcc:\vG_{a}]$ is a
  power of 2 and the coincidence indices are always odd, index
  considerations show that we even have $\vG^{}_{a}(R(q)) = \Imag
  (q\JJ) \cap \vG^{}_{a}$. Now, the theorem is a consequence of the
  bijection in Lemma~\ref{csl-lem:bcc-bij}, where index considerations
  confirm that $\Imag (q\JJ)=\Imag (q'\JJ)$ holds if and only if
  $\Imag (q\JJ) \cap \vG^{}_{a} = \Imag (q'\JJ) \cap \vG^{}_{a}$.
\end{proof}

So far, we get the following result for the arithmetic
functions\index{function!arithmetic} that count the number of CSLs and
coincidence isometries for a given index, where we use
$c^{}_{\mathsf{bcc}}(n):=c^{}_{\Gbcc}(n)$ for simplicity.

\begin{corollary}
  For the cubic lattices according to Eq.~\eqref{csl-cubiclat}, one
  has\index{counting~function!cubic~lattice}
  \[
  c^{\mathsf{iso}}_{\mathsf{bcc}}(n)=c^{}_{\mathsf{bcc}}(n)
   \, =\, c^{\mathsf{iso}}_{\mathsf{pc}}(n)\, =\, c^{}_{\mathsf{pc}}(n)
   \, =\, c^{\mathsf{iso}}_{\mathsf{fcc}}(n)\, =\,
   c^{}_{\mathsf{fcc}}(n)\ts .\smallskip
   \]
\end{corollary}

\begin{proof}
  It follows from Theorem~\ref{csl-thm:csls-ideals} that the number
  of CSLs of any cubic lattice for a given index is given by the number
  of left ideals generated by primitive $q$ with $|q|^2$ odd, hence
  $c^{}_{\mathsf{bcc}}(n)=c^{}_{\mathsf{pc}}(n)=c^{}_{\mathsf{fcc}}(n)$.
  As the coincidence indices of a given coincidence isometry are the same for
  all cubic lattices, we also have
  $c^{\mathsf{iso}}_{\mathsf{bcc}}(n)=c^{\mathsf{iso}}_{\mathsf{pc}}(n)
  =c^{\mathsf{iso}}_{\mathsf{fcc}}(n)$.\smallskip
  
  It remains to show
  $c^{}_{\mathsf{bcc}}(n)=c^{\mathsf{iso}}_{\mathsf{bcc}}(n)$. By
  Section~\ref{csl-sec:count-csl},
  $c^{\mathsf{iso}}(n)=c^{\mathsf{rot}}(n)$ holds for any lattice in
  odd dimensions and for any $n\in\NN$. It thus suffices to show that
  $c^{}_{\mathsf{bcc}}(n)=c^{\mathsf{rot}}_{\mathsf{bcc}}(n)$.  Recall
  that any coincidence rotation\index{rotation!coincidence} can be
  parametrised either by an odd primitive quaternion $q$ or by a
  primitive quaternion $q (1+\ii)$, where $q$ is again odd. As the
  rotation $R(1+\ii)$ is a symmetry operation of all three cubic
  lattices, $q$ and $q (1+\ii)$ generate the same CSL.  As all
  symmetry rotations are generated by quaternions $u$ or $(1+\ii)u$,
  where $u$ is a unit, Theorem~\ref{csl-thm:csls-ideals} implies
  $c^{}_{\mathsf{bcc}}(n)=c^{\mathsf{rot}}_{\mathsf{bcc}}(n)
  =c^{\mathsf{iso}}_{\mathsf{bcc}}(n)$.
\end{proof}

Actually, we can calculate $c^{}_{\vG}(n)$ explicitly. We first note
that $c^{}_{\vG}(n)$ is multiplicative, as $\JJ$ is a principal ideal
ring and thus has an essentially unique prime factorisation. Let us
recall that uniqueness is a bit subtle here, since $\JJ$ is not
Abelian, and the prime factorisation depends on the ordering of the
factors in general. But, if we fix an ordering (by requiring that the
norm of the prime factors should increase monotonically, say), the
prime factors are unique up to units.  Thus, $c^{}_{\vG}(n)$ is
determined by its values for prime powers, and, in particular, we
have\index{counting~function!cubic~lattice}
\begin{equation}\label{csl-eq:c-cub}
    c^{}_{\mathsf{bcc}}(p^r) \, = \, (p+1) \ts\ts p^{r-1} 
\end{equation}
if $p$ is an odd prime, as $c^{}_{\mathsf{bcc}}(p^r)$ is the number of
primitive ideals of norm~$p^{2r}$; see~\cite[Ch.~~10]{csl-Hurwitz}.
Furthermore, note that $24\, c^{}_{\mathsf{bcc}}(p^r)$ is the number
of primitive quaternions of norm $p^r$ and $8\ts\ts
c^{}_{\mathsf{bcc}}(p^r)$ is the number of primitive representations
of $p^r$ as a sum of four squares, which follows easily from the total
number of representations; compare~\cite{csl-Grosswald,csl-Hardy}.
Thus, $8\ts\ts c^{}_{\mathsf{bcc}}(m)$ is the number of primitive
representations of $m$ as a sum of four squares, if $m$ is
odd,\footnote{{\ts\ts}This is part of Jacobi's four-square theorem
  \cite{csl-Grosswald}, which states that the number of ways to
  represent $m$ as the sum of four squares is $8$ times the sum of its
  divisors (if $m$ is odd) and $24$ times the sum of its odd divisors
  (if $m$ is even).}  and $\ts c^{}_{\mathsf{bcc}}(m)=0$ for $m$ even.
Hence, we obtain an explicit expression for the generating function;
see also~\cite{csl-Baake-rev} and~\cite[Sec.~2]{csl-BLP96}.

\begin{theorem}[{\cite[Props.~3.2 
and 3.3]{csl-Baake-rev}}]\label{csl-theo:CSL-cub} For any cubic
  lattice\/ $\vG\subset\RR^3$ in the setting of
  Eq.~\eqref{csl-cubiclat}, we have
   \[
   \OC(\vG) \, = \, \OS(\vG) \, = \,
   \OC(\Gbcc) \, = \, \OG(3,\QQ) \ts . 
   \]
   In particular, if\/ $q$ is a primitive Hurwitz quaternion and\/
   $\ell$ is the maximal exponent such that\/ $2^\ell \big\vert
   |q|^2$, then the coincidence index\index{index} is given by
   \[
     \Sig^{}_\vG(R(q)) \, = \, \Sig^{}_{\mathsf {bcc}}(R(q)) \, = \, 
     \den^{}_\vG(R(q)) \, = \, \frac{|q|^2}{2^\ell}\ts . 
   \]
   Moreover, we have\/ $\Psi^{}_{\!\vG}(s) =
   \Psi^{\mathsf{iso}}_{\!\vG}(s) =\Psi^{}_{\!\mathsf{bcc}}(s)$, which
   is given by the equation
{\allowdisplaybreaks[4]
\begin{align*}
     \Psi^{}_{\mathsf {bcc}}(s) \, &=
     \sum_{m=1}^{\infty}\frac{c^{}_{\vG}(m)}{m^s} \, = \prod_{p\neq
       2}{\frac{1+p^{-s}}{1-p^{1-s}}}\\[1mm] &= \, \frac{1}{1+2^{-s}}
     \, \frac{\zeta_{\JJ}(s/2)}{\zeta(2s)} \, = \,
     \frac{1-2^{1-s}}{1+2^{-s}} \,
     \frac{\zeta(s)\,\zeta(s-1)}{\zeta(2s)}\\[3mm] &= \,
     1+\tfrac{4}{3^s}+\tfrac{6}{5^s}+\tfrac{8}{7^s}+\tfrac{12}{9^s}
     +\tfrac{12}{11^s}+\tfrac{14}{13^s}+
     \tfrac{24}{15^s}+\tfrac{18}{17^s}+\tfrac{20}{19^s}+
     \tfrac{32}{21^s}+
     \cdots ,
\end{align*}}
   where all positive odd integers appear in the denominator. \qed
\end{theorem}

Here, we have made use of the zeta function $\zeta^{}_{\JJ}$ of the
Hurwitz ring from
Eq.~\eqref{csl-zetaHurwitz}\index{zeta~function!Hurwitz~ring}, which
counts the non-trivial left ideals of $\JJ$. We observe that
$\zeta^{}_{\JJ}(s)$ and $\Psi^{}_{\mathsf{bcc}}(s)$ differ by the
factors $\frac{1}{1+2^{-s}}$ and $\frac{1}{\zeta(2s)}$. Note that the
term $(1+2^{-s})\,\zeta(2s)$ is the generating function for the
two-sided ideals of $\JJ$. But as the two-sided ideals only generate
the trivial CSL $\vG(R)=\vG$, they do not contribute to
$\Psi^{}_{\mathsf{bcc}}(s)$, hence their contribution to
$\zeta^{}_{\JJ}(s)$ has to be factored out to obtain the generating
function $\Psi^{}_{\mathsf{bcc}}(s)$.

It follows from the properties of Riemann's
zeta\index{zeta~function!Riemann} function that
$\Psi^{}_{\mathsf{bcc}}(s)$ is a meromorphic function of $s$. In
particular, $\Psi^{}_{\mathsf{bcc}}$ is analytic in the half-plane
$\{\Real(s)\geq 2\}$, and its rightmost pole is located at
$s=2$. Using Delange's theorem\index{Delange~theorem}
(Theorem~\ref{csl-thm:meanvalues}), we find the asymptotic growth
behaviour\index{asymptotic~behaviour} (compare~\cite{csl-Baake-rev}
and~\cite[Sec.~2]{csl-BLP96})
\begin{equation}
   \sum_{n\leq x} c^{}_{\mathsf{bcc}}(n) \, = \, 
   \myfrac{3x^2}{\pi^2} + \scO(x^2)\, ,
   \quad \text{ as } x \to \infty \ts.
\end{equation}

In contrast to the CSLs of the square and the triangular lattice in
the plane, the CSLs of the cubic lattice generally fail to be similar
sublattices, and usually have lower symmetries; see~\cite{csl-Z2} for
details.

\begin{remark} It is an interesting question what kind of grain
  boundaries are compatible with CSLs of cubic lattices, as the
  geometric situation in \mbox{$3$-space} is certainly richer than in
  the plane. Now, a large number of CSLs for cubic lattices can be
  written as $\vG(R)$, where $R$ is actually a rotation through $\pi$
  around an axis in a lattice direction $v$. These are precisely the
  rotations parametrised by a quaternion $q=(0,v)$;
  compare~\cite{csl-Grimmer73}.

  The lattice planes perpendicular to $v$ through a point $nv$ with
  $n\in\ZZ$ are invariant under a rotation about $v$ through
  $\pi$. Any of these can act as a defect-free (or perfect) grain
  boundary between two crystal halves, and the entire configuration
  appears as a stacking fault; see Figure~\ref{csl-fig:stack} for an
  illustration of a stacking sequence that corresponds to a CSL with
  index $\Sig=3$ and hexagonal symmetry. Note that the order of the
  layers is reversed in the rotated half.

  In contrast to cubic lattices, a rotation $R$ through $\pi$ about a
  lattice vector $v$ is not necessarily a coincidence rotation for a
  general lattice. However, if $R$ is a coincidence rotation, the
  corresponding lattice planes orthogonal to $v$ are invariant under
  $R$, and analogous stacking faults may occur.\index{stackings}

  Apart from their obvious relevance to the twinning structure of
  cubic crystals, coincidence isometries in the form of rotations
  through $\pi$ or simple reflections are useful generators for more
  complicated coincidence isometries in higher dimensions. In fact,
  this leads to one of the few approaches to higher dimensions known
  so far; see Section~\ref{csl-sec:higher} below for more.\exend
\end{remark}

\begin{figure}
\includegraphics[scale=0.65]{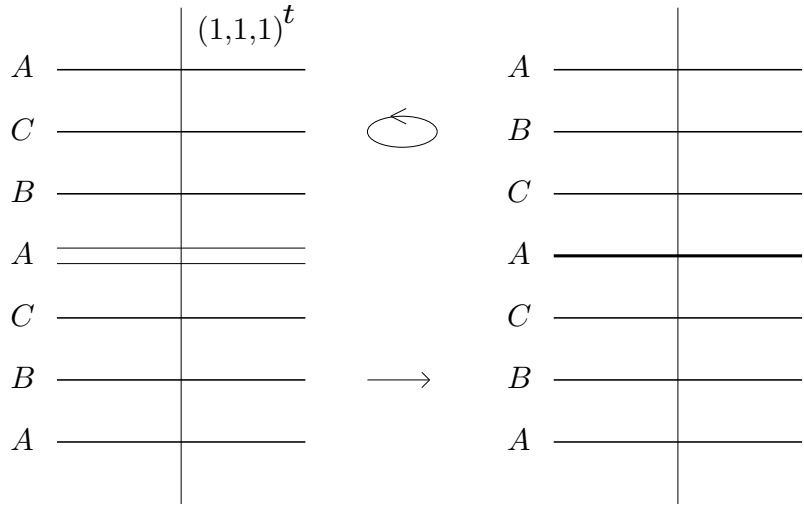}
\caption{Sketch of a stacking fault in a cubic crystal. The upper half
  is rotated through an angle $\pi$ about the $(1,1,1)^t$-axis. This
  keeps the $A$ layers fixed and interchanges layers of types $B$ and
  $C$.}
\label{csl-fig:stack}
\end{figure}

\begin{remark}
  The results for the cubic lattices can be generalised to certain
  embedded $\ZZ\ts$-modules of the form $\Imag(\cO)$, where $\cO$ is a
  maximal order\index{maximal~order} in a quaternion
  algebra~\cite{csl-BLP96}. The situation is quite convenient in the
  case of quaternion algebras $\HH(K)$ over a real algebraic number
  field $K$ such that both $K$ and $\HH(K)$ have class number $1$. In
  particular, apart from the Hurwitz ring\index{Hurwitz~ring} $\JJ$,
  this includes the icosian ring\index{icosian~ring} $\II\subset
  \HH(\QQ(\mbox{\small $\sqrt{5}$}\,))$ and the cubian ring $\KK
  \subset \HH(\QQ(\mbox{\small $\sqrt{2}$}\,))$;
  see~\cite{csl-BM,csl-BLP96} for details.
  
  The counterpart to the odd primitive quaternions are the so-called
  $\cO$-reduced quaternions. If $q$ is $\cO$-reduced, many of our
  results for the cubic lattices can be reformulated for $\cO\subset
  \HH(K)$. In particular, the coincidence index is given by
  $\Sig_{\cO}(R(q))= N(|q|^2)$, where $N$ is the norm in the number
  field $K$; compare~\cite[Prop.~5]{csl-BLP96}. This follows from the
  explicit expression for the CSMs
  \[
  \Imag(\cO)\cap\Imag(q\cO q^{-1})\, = \,
  \Imag(\cO \cap q\cO q^{-1}) \, = \, \Imag (q\cO);
  \]
  compare~\cite[Lemmas~4, 5~and 6]{csl-BLP96}. Moreover, there still
  is a bijection between the CSMs $\Imag (q\cO)$ and the left ideals
  $q\cO$; see~\cite[Thm.~1]{csl-BLP96}. This makes it possible to
  count the CSMs and to write down an explicit expression for the
  generating function~\cite[Thm.~2]{csl-BLP96}, namely
  \[
    \Psi^{\mathsf{iso}}_{\cO}(s)\, = \, \Psi^{}_{\cO}(s) \, = \,
    \frac{\zeta^{}_{\cO}(s/2)}{\zeta^{}_{\cO.\cO}(s/2)} \, = \,
    E(s)\ts\frac{\zeta^{}_{K}(s)\,\zeta^{}_{K}(s-1)}{\zeta^{}_{K}(2s)}.
  \]
  Here, $\zeta^{}_{\cO}(s)$ and $\zeta^{}_{\cO.\cO}(s)$ denote the
  zeta functions of the left and the two-sided ideals of $\cO$,
  respectively, whereas $\zeta^{}_{K}(s)$ is the zeta function of $K$
  and $E(s)$ is either $1$ or an additional analytic factor that takes
  care of the extra contributions from (finitely many) ramified
  primes.\index{prime!ramified} As a consequence, one gets the
  asymptotic behaviour~\cite[Cor.~1]{csl-BLP96}
  \[
  \sum_{n\leq x} c^{}_{\cO}(n) \, \sim \, 
   \rho^{}_{\cO} \myfrac{x^2}{2} \, , \quad \text{ as } x \to \infty \ts
  \]
  for some $\rho^{}_{\cO}\in\RR_{+}$.    \exend
\end{remark}

\section{The four-dimensional hypercubic lattices}
\label{csl-sec:cub4}

Let us continue with some examples in $4$-space, and let us start with
the hypercubic lattices. So far, in all our examples, the generating
functions for the number of coincidence rotations (modulo symmetries)
and the number of CSLs coincided, as two different coincidence
rotations\index{rotation!coincidence} generated the same CSL if and
only if they were symmetry related\index{symmetry~related}.  This is
no longer the case in the examples to come.

\subsection{The centred hypercubic lattice $D^{*}_{4}$}
\label{csl-sec:cub4-d4}

As we have already seen in Section~\ref{csl-sec:ssl-hycub}, any
similarity rotation can be parametrised by a pair of $\JJ$-primitive
Hurwitz quaternions,\index{quaternion!Hurwitz} where $\JJ=D^{*}_{4}$
as lattices in our setting.  In fact, it follows from
Corollary~\ref{csl-lem:OC-den} and Eq.~\eqref{csl-denD4} that
$R=R(p,q)$ is a coincidence rotation\index{rotation!coincidence} of
$\JJ$ if and only if $|pq|\in\NN$.  A pair $(p,q)\in \JJ\!\times\!
\JJ$ with $|pq|\in\NN$ is called
\emph{admissible}.\index{quaternion!admissible~pair} Thus, $R(p,q)$ is
a coincidence rotation\index{rotation!coincidence} of $\JJ$ if and
only if $R(p,q)$ can be parametrised by an admissible pair of
$\JJ$-primitive Hurwitz quaternions. As a consequence, we have the
following result.

\begin{fact} 
$\SOC(\JJ)=\SO(4,\QQ)$. \qed
\end{fact}

However, it turns out that primitive quaternions are not an optimal
choice in this case, and we prefer a suitably scaled pair. To find
such a pair, note first that $|pq|^2$ is a square in $\NN$ for any
admissible pair, and so is $|pq|^2/\gcd(|p|^2,|q|^2)^2$. As the two
factors
\[
  \frac{|q|^2}{\gcd(|p|^2, |q|^2)} 
  \quad \text{and} \quad
  \frac{|p|^2}{\gcd(|p|^2, |q|^2)}
\]
are coprime, they must be squares as well.  Hence, we can define the
(coprime) integers
\begin{equation} 
  \alpha^{}_p\, :=\, \sqrt{\frac{|q|^2}{\gcd(|p|^2, |q|^2)}} 
  \quad\text{and}\quad
  \alpha^{}_q\, :=\, \sqrt{\frac{|p|^2}{\gcd(|p|^2, |q|^2)}}\ts .
\end{equation}
Of course, $(x,y)=(\alpha^{}_p \ts p,\alpha^{}_q \ts q)$ defines the
same rotation as $(p,q)$.  However, we can deal more easily with
$(x,y)$ since $|x|^2=|y|^2$. Moreover, the octuple $(x,y)=(\alpha^{}_p
\ts p,\alpha^{}_q \ts q)$ is primitive for primitive $p$ and $q$, in
the sense that $\frac{1}{n}(\alpha^{}_p \ts p,\alpha^{}_q \ts q)\in
\JJ\! \times\! \JJ$ if and only if $n \in \{\pm 1 \}$. This guarantees
that there exist quaternions $v,w\in\JJ$ such that
$\ip{x}{v}+\ip{y}{w}=1$. We shall call a pair of quaternions with
these two properties an \emph{extended admissible
  pair},\index{quaternion!admissible~pair} and denote it by
$(p^{}_{\alpha},q^{}_{\alpha})=(\alpha^{}_p \ts p,\alpha^{}_q q)$.

Clearly, scaling quaternions does not change the rotation $R(p,q)$.
On the other hand, there are a lot of rotations that yield the same
CSL, namely all rotations that only differ by a symmetry operation of $\JJ$.
Let us denote the corresponding group by
\[
\SO(\JJ) \, :=\, \{ R\in\SO(4,\RR) \mid R\JJ=\JJ\}\ts ,
\]
which is a group of order $24^2=576$.  Recall that we call two
coincidence rotations\index{rotation!coincidence} $R,R'$
symmetry related\index{symmetry~related} if there exists an
$S\in\SO(\JJ)$ such that $R'=RS$ holds.

Let us have a closer look at symmetry-related rotations.  It follows
from $R(p,q)\JJ = \frac{1}{|pq|}p\JJ\bar q$ that
$R(p,q)\JJ=R(p',q')\JJ$ if and only if
\[
\myfrac{1}{|p p'|} \ts \bar{p} p'\JJ \, = \, 
\myfrac{1}{|q q'|}  \ts\JJ\ts \bar{q} q'.
\]
This means that $(p,q)$ and $(p\ts r, q\ts r)$ are symmetry related if
and only if $r$ is a quaternion such that $r\JJ$ is a two-sided
ideal. Apart from scaling factors and units, the only non-trivial such
quaternion is $r=1+\ii$;
see~\cite{csl-Vigneras,csl-Koecher,csl-duVal,csl-Hurwitz}.  Thus,
$R(p,q)\JJ=R(p\ts r,q\ts r)\JJ$, and, as $r$ is the only prime
quaternion (up to units) of norm $|r|^2=2$, we can find, for any
rotation $R \in \SOC(\JJ)$, a pair of quaternions $(p,q)$ with $|p|^2$
and $|q|^2$ odd such that $R$ is symmetry related to $R(p,q)$. We can
thus confine our considerations to the latter rotations, and we will
call an extended admissible pair $(p,q)$ with $|p|^2$ and $|q|^2$ odd
an \emph{odd extended admissible} pair.\index{quaternion!admissible~pair}

In fact, we can express all CSLs in terms of odd extended admissible
pairs as follows.
\begin{lemma}
If\/ $(p,q)$ is an odd extended admissible pair, one has
\[
p\JJ+\JJ\bar q \, \subseteq \, \JJ\cap\frac{p\JJ\bar q}{|pq|}.
\]
\end{lemma}
\begin{proof}
Clearly, $p\JJ\subseteq\JJ$ and $\JJ\bar q\subseteq\JJ$, thus giving
$p\JJ+\JJ\bar q\subseteq \JJ$. On the other hand, since $|p|^2=|q|^2$,
one has
\begin{align}
p\JJ \, = \, \frac{p\JJ q\bar q}{|q|^2}
\, \subseteq \, \frac{p\JJ \bar q}{|q|^2} \, = \,
\frac{p\JJ \bar q}{|pq|},
\end{align}
and a similar argument for $\JJ\bar q$ yields
$p\JJ+\JJ\bar q\subseteq \frac{p\JJ\bar q}{|pq|}$.
\end{proof}

The first step for the converse inclusion is the following result,
where we return to the more general case of extended admissible pairs
for a moment.

\begin{lemma}
  If\/ $(p,q)$ is an extended admissible pair, one has
\[
  2\left(\JJ\cap\frac{p\JJ\bar q}{|pq|}\right)\, 
  \subseteq\; p\JJ+\JJ\bar q\ts .
\]
\end{lemma}

\begin{proof}
  Let $x\in \JJ\cap\frac{p\JJ\bar q}{|pq|}$. Then, there exists a
  $y\in\JJ$ such that $x=\frac{p\ts y\bar q}{|pq|}$. Since $(p,q)$ is
  an extended admissible pair, there exist quaternions $v,w\in\JJ$
  such that $\ip{p}{v}+\ip{q}{w}=1$. Consequently,
\begin{align*}
  2\ts x& \, = \,
  2\bigl(\ip{p}{v}+\ip{q}{w}\bigr)x \, = \,
  2\ip{p}{v}x+2\ts x\ip{q}{w}\, = \,
  p\bar v x + v\bar p x + xq\bar w + xw\bar q\\
  & \, = \, p\bar v x + vy\bar q + py\bar w +
  xw\bar q  \, \in \,  p\JJ+\JJ\bar q \ts ,
\end{align*}
where we have made use of the identity $\ip{a}{b\ts}=\frac{1}{2}(a\bar
b + b\bar a)$.
\end{proof}

Trivially, since $|p|^2=|q|^2$, one has
\[
  |p|^2\left(\JJ\cap\frac{p\JJ\bar q}{|pq|}\right)\, =\; 
  |p|^2\JJ\cap p\JJ\bar q \, \subseteq \, p\JJ+\JJ\bar q\ts .
\]
If we restrict again to odd extended admissible pairs, we get
\[
  \JJ\cap\frac{p\JJ\bar q}{|pq|}
  \, =\, 2\left(\JJ\cap\frac{p\JJ\bar q}{|pq|}\right)
  +|p|^2\left(\JJ\cap\frac{p\JJ\bar q}{|pq|}\right)\subseteq \,
  p\JJ+\JJ\bar q\ts ,
\]
since $|p|^2$ is odd.  Hence, we have proved the following result.

\begin{theorem}\label{csl-theo:csl-cub4}
Let\/ $(p,q)$ be an odd extended admissible pair. Then,
\[
  \JJ\cap\frac{p\JJ\bar q}{|pq|} \, =\,  p\JJ+\JJ\bar q\ts ,
\]
so each CSL of the centred hypercubic lattice is of the form\/
$p\JJ+\JJ\bar q$ for a suitable odd extended admissible pair. \qed
\end{theorem}

This explicit expression of the CSLs of $\JJ$ in terms of a sum of
ideals of $\JJ$ is very useful, as it does not only help to calculate
their indices, but it also allows us to determine which coincidence
rotations yield the same CSL.

Let us first state the result for the index.

\begin{theorem}[{\cite[Theorem 4.1.6]{csl-habil}}]\label{csl-theo:index-D4}
  If\/ $(p,q)$ is an odd extended admissible pair, one has\/
  $\Sig(R(p,q))=|p|^2$.
\end{theorem}

\begin{proof}[Sketch of proof]
The idea of the proof is to exploit the equation
\[
  p\JJ\, \subseteq\,  p\JJ+\JJ\bar q
 \, =\, \JJ\cap\frac{p\JJ\bar q}{|pq|}\, \subseteq\, \JJ
\]
to show $\Sig(R(\bar p,q))\Sig(R(p,q))=|p|^4$. By proving that the index
$\Sig(R(p,q))$ divides $|p|^2$, one then infers $\Sig(R(p,q))=|p|^2$.
For the rather technical details, we refer to~\cite{csl-habil}.
\end{proof}

\begin{remark}\label{csl-rem:cub4-index2p}
  It may be useful to formulate the index\index{index} in terms of
  primitive admissible pairs. Let $p,q$ be primitive odd quaternions
  with associated extended pair $(p^{}_{\alpha},
  q^{}_{\alpha})=(\alpha^{}_p \ts p, \alpha^{}_q \ts q)$. Then,
  \begin{align*}
    \Sig(R(p,q)) & \, = \, \alpha_p^2 \, |p|^2 \, =\,  \alpha_q^2 \, |q|^2 
    \, = \, \alpha^{}_p \ts\alpha^{}_q \, |pq| \\[1mm]
    & \, = \, \lcm\bigl(|p|^2,|q|^2\bigr)
    \, = \, \alpha_p^2 \ts \alpha_q^2 \gcd\bigl(|p|^2,|q|^2\bigr). 
  \end{align*}
  Note that $|pq|$ is the denominator of $R(p,q)$. This shows that, in
  general, $\den(R)$ and $\Sig(R)$ do \emph{not} coincide for the lattice
  $D_4^{*}$, which is in contrast to the three-dimensional cubic
  lattices. In fact, $\den(R)=\Sig(R)$ holds if and only if
  $\alpha^{}_p=\alpha^{}_q=1$.  \exend
\end{remark}

\begin{remark}\label{csl-rem:cub4-spec}
  This explicit expression for the coincidence indices allows us to
  determine the coincidence spectrum. As in
  Remark~\ref{csl-rem:spec-cub}, we conclude that $|p|^2$ and $|q|^2$
  run through all odd positive integers, and the possible coincidence
  indices thus are exactly the odd positive integers. In other words,
  the coincidence spectrum of $D_4^{*}$ and $D_4^{}$, which we know to
  be similar lattices, is the set of all odd positive integers,
  \[
    \Sig\bigl(\SOC(D^{*}_{4})\bigr) \, = \, \Sig\bigl(\SOC(D^{}_{4})\bigr)
    \, = \, 2\ts\NN_{0} +1.
  \]
  This is exactly the same spectrum we have found for the
  three-dimensional cubic lattices; compare
  Remark~\ref{csl-rem:spec-cub}.  As $D_4^{*}$ has reflections among
  its symmetry operations, this is also the full spectrum
  $\Sig\bigl(\OC(D^{*}_{4})\bigr)=\Sig\bigl(\SOC(D^{*}_{4})\bigr)$ by
  Remark~\ref{csl-rem:spec}.  \exend
\end{remark}

Our next task is to enumerate the coincidence isometries of $D_4^{*}$.
Since the point group of $D_4^{*}$ contains $24^2=576$ rotations, the
number of coincidence rotations\index{rotation!coincidence} of a given
index $n$ can be written as $576 \, c^{\mathsf{rot}}_{D_4^{*}}(n)$.
As the point group contains also reflections, the number of
coincidence isometries\index{isometry!coincidence} is twice this
number, $1152\, c^{\mathsf{rot}}_{D_4^{*}}(n)$.

By Theorem~\ref{csl-theo:index-D4}, counting the number of coincidence
rotations is equivalent to counting the number of odd extended
admissible pairs. We first observe that
$c^{\mathsf{rot}}_{D_4^{*}}(n)$ is a multiplicative function, which
follows from the essentially unique prime factorisation in
$\JJ$. Indeed, if $(p,q)$ and $(r,s)$ are odd extended admissible
pairs with $|p|^2=m$ and $|r|^2=n$ for $m,n$ coprime, $(p\ts r,q\ts
s)$ is an odd extended admissible pair with $|pr|^2=mn$.  Conversely,
any odd extended admissible pair $(p,q)$ with $|p|^2=mn$ can be
decomposed into odd extended admissible pairs with index $m$ and $n$,
respectively. As this decomposition is unique up to units,
multiplicativity follows.

Thus, we only need to compute $c^{\mathsf{rot}}_{D_4^{*}}(n)$ for $n$
being a prime power. In the following, let $\pi$ denote a rational
prime (we choose $\pi$ here as we have used $p$ for quaternions
already).  As odd extended admissible pairs consist of odd quaternions
only, $c^{\mathsf{rot}}_{D_4^{*}}(2^r)=0$. Hence, $\pi$ is always odd
in what follows.  It is now a purely combinatorial task to determine
$c^{\mathsf{rot}}_{D_4^{*}}(\pi^r)$. The number of primitive
quaternions $p$ with norm $|p|^2=\pi^r$ is given by $24\ts f(\pi^r)$
with $f(\pi^r)=(\pi \nts +\nts 1) \ts \pi^{r-1}$ for $r\geq 1$; compare
Eq.~\eqref{csl-eq:c-cub}.  Any odd extended admissible pair $(p,q)$
with $|p|^2=\pi^r$ can be obtained from a primitive admissible pair
$(p^{}_1,q^{}_1)$ with $|p^{}_1|^2=\pi^{r'}\!\! ,$
$|q^{}_1|^2=\pi^{r''}\!\! ,$ $r=\max(r',r''),$ and $r'-r''$
even. Hence,
\begin{equation}\label{csl-eq:dim4cub-crot}
\begin{split}
  c^{\mathsf{rot}}_{D_4^{*}}(\pi^r)\,  & =\, f(\pi^r)^2 
  + 2\sum_{s=1}^{[r/2]} f(\pi^r) \,f(\pi^{r-2s})\\
    & =\, \myfrac{\pi\!+\!1}{\pi\!-\!1}\,\pi^{r-1}\,(\pi^{r+1}+\pi^{r-1}-2)\ts .
\end{split}
\end{equation}
Let us summarise this result in the following theorem, where we change
the notation and use $p$ to denote a rational
prime.

\begin{theorem}\label{csl-theo:dim4cub-d4rot}
  The number of coincidence rotations\index{rotation!coincidence} of\/
  $D_4^{*}$ of index\/ $n$ is given by\/
  $576 \, c^{\mathsf{rot}}_{D_4^{*}}(n)$, where\/
  $c^{\mathsf{rot}}_{D_4^{*}}(n)$ is a multiplicative arithmetic
  function. It is determined by\/ $c^{\mathsf{rot}}_{D_4^{*}}(2^r)=0$
  for\/ $r\geq 1$ together with
\[
    c^{\mathsf{rot}}_{D_4^{*}}(p^r)\, =\, 
    \frac{p\nts\nts +\!1}{p\nts\nts -\!1}\, p^{r-1}\, (p^{r+1}+p^{r-1}-2)
\]
if\/ $p$ is an odd prime and\/ $r\geq 1$.  \qed
\end{theorem}

The multiplicativity of $c^{\mathsf{rot}}_{D_4^{*}}(n)$ guarantees
that the corresponding Dirichlet series\index{Dirichlet~series}
generating function can be written as an Euler
product,\index{Euler~product}
\begin{equation}\label{csl-4dim-frotpcub}
\begin{split}
  \Psi_{D_4^{*}}^{\mathsf{rot}}(s)
  & \, = \, \sum_{n=1}^\infty \frac{c^{\mathsf{rot}}_{D_4^{*}}(n)}{n^s}
  \, = \, \prod_{p\ne 2}\frac{(1+p^{-s})(1+p^{1-s})}{(1-p^{1-s})(1-p^{2-s})}\\[2mm]
  & \, = \,   \myfrac{1-2^{1-s}}{1+2^{-s}}\myfrac{1-2^{2-s}}{1+2^{1-s}}
    \frac{\zeta(s)\,\zeta(s-1)^2\,\zeta(s-2)}{\zeta(2s)\,\zeta(2s-2)}\ts ,
\end{split}
\end{equation}
where the first few terms read as follows,
\[
\Psi_{D_4^{*}}^{\mathsf{rot}}(s)
=\textstyle{1+\frac{16}{3^s}+\frac{36}{5^s}+\frac{64}{7^s}+\frac{168}{9^s}
+\frac{144}{11^s}+\frac{196}{13^s}+\frac{576}{15^s}+\frac{324}{17^s}
+\frac{400}{19^s}+\frac{1024}{21^s}+\cdots}
\]

It is remarkable that $\Psi_{D_4^{*}}^{\mathsf{rot}}(s)$ can be
expressed in terms of the cubic generating function
$\Psi^{}_{\!\mathsf{bcc}}(s)$ from Theorem~\ref{csl-theo:CSL-cub},
which follows immediately from its explicit expression in terms of
zeta functions from Eq.~\eqref{csl-4dim-frotpcub}. In particular, one
has
\begin{equation}
  \Psi_{D_4^{*}}^{\mathsf{rot}}(s)\, =\, 
  \Psi^{}_{\!\mathsf{bcc}}(s)\Psi^{}_{\!\mathsf{bcc}}(s-1)  .
\end{equation}
This explicit expression shows that $\Psi_{D_4^{*}}^{\mathsf{rot}}(s)$
is a meromorphic function in the complex plane.  Its rightmost pole is
at $s=3$, with residue $\frac{630}{\pi^6}\ts \zeta(3)$. Using
Theorem~\ref{csl-thm:meanvalues}, we obtain the asymptotic
behaviour\index{asymptotic~behaviour}
\[
  \sum_{n\leq x} c^{\mathsf{rot}}_{D_4^{*}}(n) \, \sim\,  
  \myfrac{210}{\pi^6}\ts \zeta(3)\, x^3 \, \approx\,
  0.262{\ts}570\, x^3
\]
as $x\to\infty$. 

Next, we want to calculate the number $c^{}_{D_4^{*}}(n)$ of distinct
CSLs of a given index $n$. In contrast to the three-dimensional cubic
lattices, where we have $c^{\mathsf{rot}}_{}(n)=c^{}_{}(n)$, it turns
out that $c^{}_{D_4^{*}}(n)$ and $c^{\mathsf{rot}}_{D_4^{*}}(n)$
generally differ.  Clearly, we have the upper bound
$c^{}_{D_4^{*}}(n)\leq c^{\mathsf{rot}}_{D_4^{*}}(n)$.  To calculate
$c^{}_{D_4^{*}}(n)$, we must determine which coincidence rotations
generate the same CSL.

One knows from Lemma~\ref{csl-lem:equal-csl} that two CSLs can only
agree if the corresponding coincidence indices are the same.  In
addition, the denominators of the inverses must be equal, but as
$\den(R)=\den(R^{-1})$, we infer that the denominators must be the
same as well. However, these conditions are not yet sufficient.  In
fact, we need additional conditions, which are a bit technical;
compare~\cite{csl-BZ08} and see~\cite[Thm.~4.1.12]{csl-habil} for a
proof.

\begin{theorem}\label{csl-theo:4dim-equalcslJ}
  Let\/ $(q^{}_1,p^{}_1)$ and\/ $(q^{}_2,p^{}_2)$ be two primitive
  admissible pairs of odd quaternions.  Then, the relation
  \[
    \JJ\cap \frac{p^{}_1 \JJ \bar{q}^{}_1}{|p^{}_1q^{}_1|}
    \, =\, \JJ\cap \frac{p^{}_2 \JJ \bar{q}^{}_2}{|p^{}_2q^{}_2|}
  \]
  holds if and only if the following conditions are satisfied\/
  $(\nts $up to units$\ts )$:
  \begin{enumerate}\itemsep=3pt
  \item $|p^{}_1q^{}_1|=|p^{}_2q^{}_2|$, 
  \item $\lcm\bigl(|p^{}_1|^2,|q^{}_1|^2\bigr)
    =\lcm\bigl( |p^{}_2|^2,|q^{}_2|^2\bigr)$,
  \item $\gcld\bigl( p^{}_1,|p^{}_1q^{}_1|\bigr)
    =\gcld\bigl( p^{}_2,|p^{}_2q^{}_2|\bigr)$, and 
  \item $\gcld\bigl( q^{}_1,|p^{}_1q^{}_1|\bigr)
    =\gcld\bigl( q^{}_2,|p^{}_2q^{}_2|\bigr)$. \qed
  \end{enumerate}
\end{theorem}

Note that the first two conditions correspond to the aforementioned
condition that the coincidence indices and the denominator are the
same (recall from Remark~\ref{csl-rem:cub4-index2p} that $\Sig(R(p,q))
= \lcm(|p|^2,|q|^2)$ and \mbox{$\den(R(p,q)) = |pq|$}, if $(p,q)$ is a
primitive admissible pair of odd quaternions).

\begin{remark}\label{csl-rem:4dim-equalcslJ}
  One gets an equivalent set of conditions for the equality of two
  CSLs if one replaces conditions (1) and (2) in
  Theorem~\ref{csl-theo:4dim-equalcslJ} by $|p^{}_1|^2=|p^{}_2|^2$ and
  $|q^{}_1|^2=|q^{}_2|^2$.  It is obvious that the two conditions
  $|p^{}_1|^2=|p^{}_2|^2$ and $|q^{}_1|^2=|q^{}_2|^2$ imply that the
  denominators $|p^{}_1q^{}_1|=|p^{}_2q^{}_2|$ and the coincidence indices
  $\lcm(|p^{}_1|^2,|q^{}_1|^2)=\lcm(|p^{}_2|^2,|q^{}_2|^2)$ are the
  same. The reverse direction is more complicated, as the two
  conditions $|p^{}_1q^{}_1|=|p^{}_2q^{}_2|$ and
  $\lcm(|p^{}_1|^2,|q^{}_1|^2)=\lcm(|p^{}_2|^2,|q^{}_2|^2)$ alone only
  yield $\gcd(|p^{}_1|^2,|q^{}_1|^2)=\gcd(|p^{}_2|^2,|q^{}_2|^2)$, but
  not $|p^{}_1|^2=|p^{}_2|^2$ and $|q^{}_1|^2=|q^{}_2|^2$ directly. In
  fact, we need both of the other two conditions,
  $\gcld(p^{}_1,|p^{}_1q^{}_1|)=\gcld(p^{}_2,|p^{}_2q^{}_2|)$ and
  \mbox{$\gcld(q^{}_1,|p^{}_1q^{}_1|)=\gcld(q^{}_2,|p^{}_2q^{}_2|)$},
  to establish $|p^{}_1|^2=|p^{}_2|^2$ and $|q^{}_1|^2=|q^{}_2|^2$ as
  well; compare~\cite[Proof of Thm.~4.1.12 and
    Rem.~4.1.13]{csl-habil}. \exend
\end{remark}

We are now ready to count the number $c^{}_{D_4^{*}}(n)$ of CSLs.  It
follows from Theorem~\ref{csl-theo:csl-mult} that $c^{}_{D_4^{*}}(n)$
is multiplicative, since $c^{\mathsf{iso}}_{D_4^{*}}(n)$ is
multiplicative.  As there are no CSLs of even index,
$c^{}_{D_4^{*}}(n)$ is completely determined by
$c^{}_{D_4^{*}}(\pi^r)$ for odd rational primes $\pi$ and
$r\in\NN$. The latter can be calculated by counting the number of odd
primitive admissible pairs that satisfy the conditions in
Theorem~\ref{csl-theo:4dim-equalcslJ} or in
Remark~\ref{csl-rem:4dim-equalcslJ}. Thus,
\begin{equation}
  c^{}_{D_4^{*}}(\pi^r) \, = \, 
  f(\pi^r)^2 + 2\sum_{s=1}^{[r/2]} f(\pi^{r-s}) f(\pi^{r-2s})\ts ,
\end{equation}
where $f(\pi^r)=(\pi +1) \, \pi^{r-1}$ for $r\geq 1$ as above.  Note
that this expression is very similar to
Eq.~\eqref{csl-eq:dim4cub-crot}, the only difference being that one
factor $f(\pi^{r})$ is replaced by $f(\pi^{r-s})$, where the latter
counts the number of distinct $\gcld(p,|pq|)$ with $|p|^2=\pi^r$ and
$|q|^2=\pi^{r-2s}$.

Evaluating the sum yields the following result, where we again switch
to $p$ to denote a rational prime.

\begin{theorem}\label{csl-theo:dim4cub-d4csl}
  The number of distinct CSLs of\/ $D_4^{*}$ of index\/ $n$ is given
  by\/ $c^{}_{D_4^{*}}(n)$. Here, $c^{}_{D_4^{*}}(n)$ is a
  multiplicative arithmetic function, which is completely determined
  by\/ $c^{}_{D_4^{*}}(2^r)=0$ for\/ $r\geq 1$ together with
  \begin{align*}
    c^{}_{D_4^{*}}(p^r)&=
    \begin{cases}
      \frac{(p\ts\ts +1)^2}{p^3-1}\bigl(p^{2r+1}+p^{2r-2}-2\ts 
       p^{\frac{r-1}{2}}\bigr),
      &\mbox{if $r\geq 1$ is odd,}\\[2mm]
      \frac{(p\ts\ts +1)^2}{p^3-1}
      \bigl(p^{2r+1}+p^{2r-2}-2\ts \frac{p^2\nts +1}{p\ts\ts +1}\ts 
      p^{\frac{r-2}{2}}\bigr),
      &\mbox{if $r\geq 2$ is even,}
    \end{cases}
  \end{align*}
  for odd primes\/ $p$. Then,
  \begin{align*}
    \Psi_{D_4^{*}}(s)^{}&=\sum_{n=1}^\infty \frac{c^{}_{D_4^{*}}(n)}{n^s}\,
    =\prod_{p\ne 2}\frac{1+p^{-s}+2p^{1-s}+2p^{-2s}+p^{1-2s}+p^{1-3s}}%
    {(1-p^{2-s})(1-p^{1-2s})}\\[1mm]
    &=\textstyle{1+\frac{16}{3^s}+\frac{36}{5^s}+\frac{64}{7^s}+\frac{152}{9^s}
    +\frac{144}{11^s}+\frac{196}{13^s}+\frac{576}{15^s}+\frac{324}{17^s}
    +\frac{400}{19^s}+\frac{1024}{21^s}+\cdots}
  \end{align*}
is the corresponding Dirichlet series\index{Dirichlet~series}. \qed
\end{theorem}

Unfortunately, unlike before, there is no nice representation of
$\Psi_{D_4^{*}}(s)$ as a product of zeta functions. Nevertheless, we
can use Theorem~\ref{csl-thm:meanvalues} to calculate the asymptotic
behaviour as follows.

Note that $\Psi_{D_4^{*}}^{}(s)$ is quite similar to
$\Psi_{D_4^{*}}^{\mathsf{rot}}(s)$; compare
Eq.~\eqref{csl-4dim-frotpcub}. In fact, differences between the
corresponding counting functions occur only for those integers that are
divisible by the square of an odd prime. Thus, the rightmost pole of
$\Psi_{D_4^{*}}^{}(s)$ is still at $s=3$, which is the same as for
$\Psi_{D_4^{*}}^{\mathsf{rot}}(s)$.  This implies the asymptotic
behaviour $\sum_{n\leq x} c^{}_{D_4^{*}}(n) \sim c\ts x^3$ as $x \to
\infty$ for some positive constant $c$.  To be more specific, we
consider the ratio
\begin{equation}
  \frac{\Psi_{D_4^{*}}^{}(s)}{\Psi_{D_4^{*}}^{\mathsf{rot}}(s)}
  \, = \prod_{p\ne 2} \left( 1
  -\frac{2\ts (p^2-1)\ts p^{-2s}}{(1+p^{-s})(1+p^{1-s})(1-p^{1-2s})}
  \right),
\end{equation}
where the right-hand side defines an analytic function in the open
half-plane \mbox{$\bigl\{\Real(s)>\frac{3}{2} \bigr\}$} with
\begin{align}\label{csl-eq:dim4cub-gamma}
  \gamma&\, := \,
          \lim_{s\to 3} \frac{\Psi_{D_4^{*}}^{}(s)}{\Psi_{D_4^{*}}^{\mathsf{rot}}(s)}
  \, =  \prod_{p\ne 2} \left( 1
  -\frac{2\ts (p^2-1)\ts p^{-6}}{(1+p^{-2})(1+p^{-3})(1-p^{-5})}
  \right) \\[1mm]
    \nonumber
    &\, \approx \, 0.976{\ts}966{\ts}019 \, < \, 1\ts .
\end{align}
Hence, $\sum_{n\leq x} c^{}_{D_4^{*}}(n)$ grows by a factor $\gamma$
slower than $\sum_{n\leq x} c^{\mathsf{rot}}_{D_4^{*}}(n)$. In
particular, we obtain \index{asymptotic~behaviour}
\[
  \sum_{n\leq x} c^{}_{D_4^{*}}(n) \,\sim\,  
  \myfrac{210}{\pi^6}\ts\ts\zeta(3)\ts\ts\gamma\, x^3
  \,\approx\, 0.256{\ts}522\, x^3  ,
\]
as $x\to\infty$.  This shows that $\sum_{n\leq x}
c^{\mathsf{rot}}_{D_4^{*}}(n)$ and $\sum_{n\leq x} c^{}_{D_4^{*}}(n)$
differ by less than $2.5\%$ asymptotically, which means that it is
quite rare that two coincidence rotations that are not symmetry
related\index{symmetry~related} generate the same CSL.

As we have enumerated the distinct CSLs, we might ask the question of
how many non-equivalent CSLs there are, where we call two CSLs $\vL_1$
and $\vL_2$ \emph{equivalent} if there is an $R\in\OG(\JJ)$ such that
$\vL_2=R\vL_1$. This question has not completely been answered yet,
but some partial results can be found in~\cite{csl-Z3}.

\subsection{The primitive hypercubic lattice $\ZZ^4$}
\label{csl-sec:cub4-z4}

Let us move on to the primitive hypercubic lattice, which we identify
with $\ZZ^4$ or, in terms of quaternions, with the ring of Lipschitz
quaternions $\LL$.\index{quaternion!Lipschitz} As $\ZZ^4$ and $D_4^{*}$
are commensurate, they have the same group of coincidence
rotations\index{rotation!coincidence}, which means
\[
  \SOC(\ZZ^4)\, =\, \SOC(D_4^{*})\, =\, \SO(4,\QQ)\ts .
\] 
Moreover, we have $D_4^{}\subset \ZZ^4 \subset D_4^{*}$, where $\ZZ^4$
is a sublattice of $D_4^{*}$ of index~$2$. Thus, by
Theorem~\ref{csl-theo:Sig-submod}, the coincidence indices of the two
lattices can differ at most by a factor of $2$.  This implies that we
have either $\Sig^{}_{\ZZ^4}(R) = \Sig^{}_{D_4^{*}}(R)$ or
$\Sig^{}_{\ZZ^4}(R) = 2 \Sig^{}_{D_4^{*}}(R)$ for a given coincidence
rotation $R$.  Actually, both cases do occur.

This becomes immediately clear if we recall that the primitive
hypercubic lattice $\ZZ^4$ has a smaller symmetry group than
$D_4^{*}$. In particular, $\SO(\ZZ^{4})$ contains only $192$
rotations, so that
\[
  [\SO(D_4^{*}):\SO(\ZZ^4)]\, =\,
  [\OG(D_4^{*}):\OG(\ZZ^4)]\, =\, 3\ts .  
\]
As a consequence, every class of
symmetry-related\index{symmetry~related} coincidence rotations of
$D_4^{*}$ splits into three classes of $\ZZ^4$.  In particular, all
rotations in \mbox{$\SO(D_4^{*})\setminus\SO(\ZZ^4)$} are coincidence
rotations\index{rotation!coincidence} for $\ZZ^4$ of index $2$, so we
have one class with coincidence index $1$ and two classes with index
$2$.

The same pattern also emerges for the other coincidence
rotations$\,$---$\,$and, more generally, for coincidence isometries as
well. In particular, every class of symmetry-related coincidence
rotations of $D_4^{*}$ splits into three classes, one of which has the
same coincidence index as before,
$\Sig^{}_{\ZZ^4}(R)=\Sig^{}_{D_4^{*}}(R)$, while the other two classes
have index $\Sig^{}_{\ZZ^4}(R)=2\Sig^{}_{D_4^{*}}(R)$.  To see this,
we recall from Theorem~\ref{csl-theo:denG-Sig1} that $\den_{\ZZ^4}(R)$
divides $\Sig^{}_{\ZZ^4}(R)$, while $\Sig^{}_{\ZZ^4}(R)$ divides
$\den_{\ZZ^4}(R)^4$. Consequently, $\Sig^{}_{\ZZ^4}(R)$ is even if
and only if $\den_{\ZZ^4}(R)$ is. In other words,
\begin{equation}
  \qquad \Sig^{}_{\ZZ^4}(R)\, =\, 
  \lcm\bigl( \Sig^{}_{D_4^{*}}(R), \den_{\ZZ^4}(R) \bigr);
\end{equation}
compare~\cite{csl-Baake-rev}. If $(p,q)$ is an odd primitive
admissible pair, we have
\begin{align}
  \den_{\ZZ^4}(R(p,q))\, &=\, 
  \begin{cases}
    |pq|, & \text{if } \ip{p}{q} \in \ZZ, \\
    2 \ts |pq|, & \text{if } \ip{p}{q} \notin \ZZ,
  \end{cases}
\intertext{while, if $(p,q)$ is an even primitive admissible pair, one gets}
  \den_{\ZZ^4}(R(p,q))\, &=\, 
  \begin{cases}
    \frac{|pq|}{2}, & \text{if } \ip{p}{q} \text{ is even,} \\
    |pq|, & \text{if } \ip{p}{q} \text{ is odd.}
  \end{cases}
\end{align}
Checking for all possible combinations of units, we see that
every class of symmetry-related\index{symmetry~related} coincidence
rotations\index{rotation!coincidence} of $D_4^{*}$ indeed splits into three
classes, one of which has odd denominator and coincidence index
$\Sig^{}_{\ZZ^4}(R)=\Sig^{}_{D_4^{*}}(R)$, while the other two classes
have even denominator and coincidence index
$\Sig^{}_{\ZZ^4}(R)=2\Sig^{}_{D_4^{*}}(R)$.

\begin{remark}\label{csl-rem:spec-z4}
  These relations mean that the coincidence spectrum of $\ZZ^4$ is
  larger than the coincidence spectrum of $D_4^{*}$ and $D_4^{}$. In
  particular, we conclude from Remark~\ref{csl-rem:cub4-spec} that the
  coincidence spectrum of $\ZZ^4$ is the set
  \[
    \hspace{3.2cm}
    \Sig\bigl(\OC(\ZZ^4)\bigr)\, = \, 
    \Sig\bigl(\SOC(\ZZ^4)\bigr)\, = \,
    (2 \Nnull +1 ) \cup (4 \Nnull +2 ). 
    \hspace{3cm} \mbox{\exend}
  \]
\end{remark}

In order to also get an explicit expression for the CSLs, we consider
the following chain of inclusions
\begin{equation}\label{csl-eq:chain}
  D_4^{}\cap R D_4^{} \,\subseteq\, \ZZ^4 \cap R \ZZ^4
  \,\subseteq\, D_4^{*}\cap R D_4^{*} \cap \ZZ^4 \,\subset\,  
  D_4^{*}\cap R D_4^{*}
\end{equation}
for any $R\in\SOC(D_4^{*})$. As
$\Sig^{}_{D_4^{}}(R)=\Sig^{}_{D_4^{*}}(R)$ by
Lemma~\ref{csl-lem:csl-dual}, and also $[D_4^{*}:D_4^{}]=4$, we
conclude that $[(D_4^{*}\cap R D_4^{*}):(D_4^{}\cap R
  D_4^{})]=4$. Moreover, with $[D_4^{*}:\ZZ^4]=2$, this shows
$[(D_4^{*}\cap R D_4^{*} \cap \ZZ^4):(D_4^{}\cap R D_4^{})]=2$, as
$\Sig^{}_{D_4^{*}}(R)$ is always odd. Thus, we are left with two
possibilities, namely either with $\ZZ^4 \cap R \ZZ^4 = D_4^{*}\cap R
D_4^{*} \cap \ZZ^4 = \ZZ^4 \cap R D_4^{*}$, in which case 
$\Sig^{}_{\ZZ^4}(R)=\Sig^{}_{D_4^{*}}(R)$, or with $\ZZ^4 \cap R \ZZ^4 =
D_4^{}\cap R D_4^{}$, where we have
$\Sig^{}_{\ZZ^4}(R)=2\Sig^{}_{D_4^{*}}(R)$ instead.

Let us summarise these results as follows.

\begin{proposition}\label{csl-prop:dim4cub-cslz4}
  For any coincidence rotation\index{rotation!coincidence}\/
  $R\in\SOC(\ZZ^4)$, the coincidence index\index{index} is
  \[
  \Sig^{}_{\ZZ^4}(R) \,= \,\lcm
  \bigl( \Sig^{}_{D_4^{*}}(R), \den_{\ZZ^4}(R) \bigr),
  \]
  which is even if and only if\/ $\den_{\ZZ^4}(R)$ is even. Moreover,
    \[
    \ZZ^4 \cap R \ZZ^4 \, =\, 
    \begin{cases}
      (D_4^{*}\cap R D_4^{*}) \cap \ZZ^4 = \ZZ^4 \cap R D_4^{*} ,
        & \text{ if $\Sig^{}_{\ZZ^4}(R)$ is even,}\\
      D_4^{}\cap R D_4^{},
        & \text{ if $\Sig^{}_{\ZZ^4}(R)$ is odd,}
    \end{cases}
  \]
is the  corresponding CSL.\qed
\end{proposition}

This allows us to determine the number of coincidence
rotations\index{rotation!coincidence}, which is  $192\,
c_{\ZZ^4}^{\mathsf{rot}}(n)$, as the symmetry group $\SO(\ZZ^4)$ has
order $192$. By the above considerations, each class of
symmetry-related\index{symmetry~related} coincidence rotations splits
into three classes, one with coincidence index
$\Sig^{}_{\ZZ^4}(R)=\Sig^{}_{D_4^{*}}(R)$, and two with index
$\Sig^{}_{\ZZ^4}(R)=2\Sig^{}_{D_4^{*}}(R)$.  This
gives
\begin{align}\label{csl-eq:dim4cub-crot-z4}
  c_{\ZZ^4}^{\mathsf{rot}}(n) \, = \,
  \begin{cases}
    c_{D_4^{*}}^{\mathsf{rot}}(n), & \text{ if $n$ is odd,} \\
    2 \, c_{D_4^{*}}^{\mathsf{rot}}\bigl(\frac{n}{2}\bigr), 
    & \text{ if $n$ is even.}
  \end{cases}
\end{align}
As $c_{D_4^{*}}^{\mathsf{rot}}(n)$ is multiplicative, so is
$c_{\ZZ^4}^{\mathsf{rot}}(n)$, and the corresponding Dirichlet series
again admits an Euler product expansion.  In particular, we have the
following result; compare~\cite{csl-Baake-rev,csl-Z3}.

\begin{theorem}\label{csl-theo:dim4cub-z4rot}\index{Dirichlet~series}
  The generating function for the number\/ $c_{\ZZ^4}^{\mathsf{rot}}(n)$
  of coincidence rotations\index{rotation!coincidence} of\/ $\ZZ^4$ is
  given by
  \begin{align*}
    \Psi_{\ZZ^4}^{\mathsf{rot}}(s)
    &\, =\, \sum_{n=1}^\infty \frac{c_{\ZZ^4}^{\mathsf{rot}}(n)}{n^s}
    \, =\, (1+2^{1-s})\, \Psi_{D_4^{*}}^{\mathsf{rot}}(s)\\[1mm]
    & \, = \,   \frac{(1-2^{1-s})(1-2^{2-s})}{1+2^{-s}}
    \frac{\zeta(s)\,\zeta(s-1)^2\,\zeta(s-2)}{\zeta(2s)\,\zeta(2s-2)}\\[1mm]
    &\, =\, (1+2^{1-s})\prod_{p\ne 2}
    \frac{(1+p^{-s})(1+p^{1-s})}{(1-p^{1-s})(1-p^{2-s})}\\
    \intertext{with the first terms being given by}
    \Psi_{\ZZ^4}^{\mathsf{rot}}(s)
    &\, =\, 1+\myfrac{2}{2^s}+\myfrac{16}{3^s}+\myfrac{36}{5^s}+\myfrac{32}{6^s}
    +\myfrac{64}{7^s}+\myfrac{168}{9^s}+\myfrac{72}{10^s}
    +\myfrac{144}{11^s}+\myfrac{196}{13^s}\\[1mm]
    &\qquad\, +\myfrac{128}{14^s}
    +\myfrac{576}{15^s}+\myfrac{324}{17^s}+\myfrac{336}{18^s}+\myfrac{400}{19^s}
    +\myfrac{1024}{21^s}+\myfrac{288}{22^s}
    +\cdots
  \end{align*}
  It is a meromorphic function in the complex plane, whose rightmost
  pole is located at\/ $s=3$, with residue\/
  $\frac{1575}{2\pi^6}\ts\zeta(3)$.  Consequently, as $x \to\infty$,
  we have the asymptotic
  behaviour\index{asymptotic~behaviour}
  \[
    \sum_{n\leq x} c^{\mathsf{rot}}_{\ZZ^4}(n) \,\sim\, 
    \myfrac{525}{2\ts \pi^6}\ts\zeta(3)\, x^3
    \,\approx\, 0.328{\ts}212 \, x^3.
  \]
\end{theorem}

\begin{proof}
  It follows from Eq.~\eqref{csl-eq:dim4cub-crot-z4} that
  $\Psi_{\ZZ^4}^{\mathsf{rot}}(s)$ is obtained from
  $\Psi_{D_4^{*}}^{\mathsf{rot}}(s)$ by adding a factor
  $1+2^{1-s}$. As the latter is analytic, the analytic structure of
  $\Psi_{\ZZ^4}^{\mathsf{rot}}(s)$ is the same as that of
  $\Psi_{D_4^{*}}^{\mathsf{rot}}(s)$ (see
  Theorem~\ref{csl-theo:dim4cub-d4rot} and the comments thereafter),
  except for poles located at $s=1+ \frac{(2n+1) \pi}{\log(2)} \ii$,
  which are cancelled by the factor $1+2^{1-s}$. An application of
  Theorem~\ref{csl-thm:meanvalues} finally yields the asymptotic
  behaviour.
\end{proof}

In a similar way, we can enumerate the CSLs. It follows from
Proposition~\ref{csl-prop:dim4cub-cslz4} that each CSL of $D_4^{*}$
corresponds to exactly one pair of CSLs of $\ZZ^4$, one of which has
odd index, while the other one has even index. Note that the explicit
expressions for the CSLs in Proposition~\ref{csl-prop:dim4cub-cslz4}
guarantee that two CSLs of $\ZZ^4$ are only equal if the corresponding
CSLs of $D_4^{*}$ are equal. This implies that the number of CSLs of
$\ZZ^4$ is given by
\begin{equation}\label{csl-eq:dim4cub-ccsl-z4}
  c_{\ZZ^4}^{}(n) \, =\,
  \begin{cases}
    c_{D_4^{*}}^{}(n), & \text{ if $n$ is odd,} \\
    c_{D_4^{*}}^{}\nts\bigl(\frac{n}{2}\bigr), & \text{ if $n$ is even.}
  \end{cases}
\end{equation}
This yields the following result.

\begin{theorem}\label{csl-theo:dim4cub-z4csl}
  The generating function for the number\/ $c_{\ZZ^4}^{}(n)$ of CSLs
  of\/ $\ZZ^4$ is given by\index{Dirichlet~series}
  \begin{align*}
    \Psi_{\ZZ^4}(s)&\, =\, (1+2^{-s})\ts\Psi_{D_4^{*}}(s)\\[1mm]
    &\, =\, (1+2^{-s})\prod_{p\ne 2}\frac{1+p^{-s}+2p^{1-s}+2p^{-2s}+p^{1-2s}+p^{1-3s}}
    {(1-p^{2-s})(1-p^{1-2s})}\\[1mm]
    &\, =\, 1+\myfrac{1}{2^s}+\myfrac{16}{3^s}+\myfrac{36}{5^s}+\myfrac{16}{6^s}
    +\myfrac{64}{7^s}+\myfrac{152}{9^s}+\myfrac{36}{10^s}
    +\myfrac{144}{11^s}+\myfrac{196}{13^s}\\[2mm]
    &\qquad\, +\myfrac{64}{14^s}
    +\myfrac{576}{15^s}+\myfrac{324}{17^s}
    +\myfrac{152}{18^s}+\myfrac{400}{19^s}
    +\myfrac{1024}{21^s}+\myfrac{144}{22^s}
    +\cdots.
  \end{align*}
  It is a meromorphic function in the half-plane\/
  $\bigl\{\Real(s)>\frac{3}{2}\bigr\}$, whose rightmost pole is
  located at\/ $s=3$, with residue\/
  $\frac{2835}{4\pi^6}\ts\ts\zeta(3)\ts\ts\gamma$, where $\gamma$ is the
  constant from Eq.~\eqref{csl-eq:dim4cub-gamma}.  Consequently,
  we have the asymptotic behaviour\index{asymptotic~behaviour}
  \[
    \sum_{n\leq x} c^{}_{\ZZ^4}(n) \,\sim\, 
    \myfrac{945}{4\ts\pi^6}\ts\zeta(3)\ts\ts\gamma\, x^3
    \,\approx\, 0.288{\ts}587 \, x^3,
  \]
  as\/ $x\to \infty$.  \qed
\end{theorem}

Let us now turn our attention to the corresponding problem of embedded
modules, with special focus on the golden ratio.\index{golden~mean}

\section{More on the icosian ring}
\label{csl-sec:icosianall}

The icosian ring, which is a maximal order in the quaternion algebra
$\HH(\QQ(\mbox{\small $\sqrt{5}$}\,))$, is an interesting example of a
$\ZZ\ts$-module of rank $8$ that is embedded in $\RR^{4}$. At the same
time, it is a $\ZZ[\tau]$-module of rank $4$, and thus an interesting
object in our context in its own right. Beyond this, as we already saw
in the context of SSLs, it is a powerful tool for the description of
the root lattice $A_{4}$. Here, we analyse the coincidence structure,
first via the CSLs for $A_{4}$ and then via the CSMs for $\II$ itself.

\subsection{Coincidences of the root lattice $A_4$}
\label{csl-sec:a4}

Recall from Section~\ref{csl-sec:sim-a4} that $A_{4}$ can be
represented as
\[
   L\, = \,\{x\in \II\mid x=\widetilde{x}\ts\}\ts,
\]
which brings in the icosian ring, $\II$. As $\JJ$ and $\II$ share a
lot of properties, we expect the calculation of the CSLs to be
similar. Indeed, this is true, and we may thus skip various details;
see~\cite{csl-BGHZ08,csl-HZ10,csl-habil} for details.  However, recall
that we needed a pair of quaternions to characterise the CSLs of
$\JJ$.  Here, we only need a \emph{single} quaternion $q$, as the
coincidence rotations of $A_{4}$ can be parametrised by admissible
pairs of the form $(q,\widetilde{q}\ts)$. Consequently, we call a
quaternion $q\in\II$ \emph{admissible},\index{quaternion!admissible}
if $|q\widetilde{q}\ts|^2=\nr(|q|^2)$ is a square in $\NN$. In fact,
$x \mapsto \frac{1}{|q\widetilde{q}|} qx\widetilde{q}$ defines a
coincidence rotation\index{rotation!coincidence} of $A_{4}$ in the
above representation if and only if $q\in\II$ is admissible.

In the case of the hypercubic lattices in four dimensions, it was
useful to deal with an extended admissible pair of primitive
quaternions.  Here, we define the notion of an extended primitive
admissible quaternion as follows.  Let $q\in\II$ be primitive and
admissible.  Then, $\frac{|q\widetilde{q}|^2}{\gcd(|q|^2,
  |\widetilde{q}|^2)^2}$ is a square in $\ZZ[\tau]$. Here, $\gcd$
refers to the greatest common divisor in $\ZZ[\tau]$, which is well
defined up to a unit as $\ZZ[\tau]$ is a Euclidean domain.  Now,
$\frac{|q|^2}{\gcd(|q|^2, |\widetilde{q}|^2)}\in\ZZ[\tau]$ and
$\frac{|\widetilde{q}|^2}{\gcd(|q|^2, |\widetilde{q}|^2)}\in\ZZ[\tau]$
are relatively prime in $\ZZ[\tau]$.  Since their product is a square,
they must be squares (up to units) in $\ZZ[\tau]$, too (we have unique
prime factorisation). If the units have been chosen appropriately, we
may assume that $\frac{|q|^2}{\gcd(|q|^2,
  |\widetilde{q}|^2)}\in\ZZ[\tau]$ and
$\frac{|\widetilde{q}|^2}{\gcd(|q|^2, |\widetilde{q}|^2)}\in\ZZ[\tau]$
are squares in $\ZZ[\tau]$.  Hence, we may take the root (where we may
choose the positive one) and define
\begin{equation}
  \alpha^{}_q\, :=\, 
  \sqrt{\frac{|\widetilde{q}\ts|^2}{\gcd(|q|^2,
      |\widetilde{q}\ts|^2)}}\ts,\quad
  \alpha^{}_{\widetilde{q}}\, :=\, 
  \alpha'_q=\sqrt{\frac{|q|^2}{\gcd(|q|^2,
      |\widetilde{q}\ts|^2)}}\ts ,
\end{equation}
which are unique up to units. Note further that the last equality only
holds up to a unit.

\begin{definition}
  Let $q\in\II$ be a primitive admissible quaternion. Then,
  $q^{}_{\alpha}:=\alpha^{}_q \ts q$ is called an 
  \emph{extended admissible quaternion} (corresponding to $q$).
\end{definition}

Of course, this definition is unique only up to units in $\ZZ[\tau]$,
but this does not matter as units of $\ZZ[\tau]$ cancel out in the
definition of the coincidence rotations.  The key result in the
determination of the CSLs is the following characterisation.

\begin{theorem}[{\cite[Thms.~2 and~3]{csl-BGHZ08}}]\label{dim4a4-theo:csl-a4}
  Let\/ $q\in\II$ be a primitive admissible quaternion and\/ $q^{}_{\alpha}$
  its extension. Then,
  \[
  L\cap \frac{q L\widetilde{q}}{|q\widetilde{q}\ts|}\, =\, 
  L_{q^{}_{\alpha}}\, :=\, (q^{}_{\alpha}\II+
  \II \ts\ts\widetilde{q}^{}_{\alpha})\cap L\ts .
  \]
  Moreover, its coincidence index\/ $\Sig^{}_{A_4}(q)$ is given by
  \[
  \qquad\qquad\qquad\qquad\qquad\qquad\;
  \Sig^{}_{A_4}(q)\, =\, |q^{}_{\alpha}|^2
  \, =\, \lcm\bigl(|q|^2, |\widetilde{q}\ts|^2\bigr)\ts .
  \qquad\qquad\qquad\qquad\qquad\quad \qed
  \]
\end{theorem}

This allows us to determine the multiplicative counting function
$c^{\mathsf{rot}}_{\! A_4}$, which is explicitly given by
\cite[Eq.~(5.29)]{csl-habil}
\begin{equation}\label{csl-eq:dim4a4-crot}
   c^{\mathsf{rot}}_{\! A_4}(p^r) \, = \,
     \begin{cases}
       6 \cdot 5^{2r-1}, & \mbox{if } p=5,\\[1mm]
       \frac{p\ts\ts +1}{p\ts\ts -1}\ts p^{r-1}(p^{r+1}+p^{r-1}-2),
         & \mbox{if } p\equiv \pm 1\; (5), \\[1mm]
       p^{2r} \! + p^{2r-2},
         & \mbox{if } p\equiv \pm 2 \; (5).
     \end{cases}
\end{equation}
The result now reads as follows.

\begin{theorem}[{\cite[Thm.~4]{csl-BGHZ08}}]\label{csl-theo:Psi-A4}
  Let\/ $120\, c^{\mathsf{rot}}_{\! A_4}(m)$ be the number of
  coincidence rotations\index{rotation!coincidence} of index\/ $m$ of
  the root lattice\/ $A_4$, as specified by
  Eq.~\eqref{csl-eq:dim4a4-crot}. Then, with\/ $K=\QQ(\mbox{\small
    $\sqrt{5}$}\,)$, the Dirichlet series\index{Dirichlet~series}
  generating function for\/ $c^{\mathsf{rot}}_{\! A_4}(m)$ reads
  \begin{align*}
  \Psi_{\! A_4}^{\mathsf{rot}}(s)
  &\, =\, \sum_{n\in\NN} \frac{c^{\mathsf{rot}}_{\! A_4}(n)}{n^s}
   \, =\, \frac{\zeta_K^{}(s-1)}{1+5^{-s}}
    \frac{\zeta(s)\ts\zeta(s-2)}{\zeta(2s)\ts\zeta(2s-2)}   \\[1mm]
  &\, =\, \myfrac{1+5^{1-s}}{1-5^{2-s}}
    \prod_{p\equiv\pm 1(5)}\frac{(1+p^{-s})(1+p^{1-s})}{(1-p^{1-s})(1-p^{2-s})}
    \prod_{p\equiv\pm 2(5)}\frac{1+p^{-s}}{1-p^{2-s}} \\[1mm]
  &\, =\, 1+\tfrac{5}{2^s}+\tfrac{10}{3^s}+\tfrac{20}{4^s}+\tfrac{30}{5^s}
    +\tfrac{50}{6^s}+\tfrac{50}{7^s}+\tfrac{80}{8^s}+\tfrac{90}{9^s}
    +\tfrac{150}{10^s}+\tfrac{144}{11^s}
    +\cdots,
  \end{align*}
  and the coincidence spectrum is\/ $\NN$. \qed
\end{theorem}

The function $\Psi_{\! A_4}^{\mathsf{rot}}$ is meromorphic in the
entire complex plane, and its rightmost pole is a simple pole at
$s=3$, with residue
\begin{equation}\label{csl-rhoA4rot}
\begin{split}
  \rho^{\mathsf{rot}}_{\! A_4}\, &=\, 
   \Res_{s=3}\bigl(\Psi_{\! A_4}^{\mathsf{rot}}(s)\bigr)
    \, =\, \myfrac{125}{126}\ts
      \frac{\zeta_K^{}(2)\ts\zeta(3)}{\zeta(6)\ts\zeta(4)}\\
     &=\,  \frac{450\sqrt{5}}{\pi^6}\,\zeta(3)
    \, \approx\, 1.258{\ts}124,
\end{split}
\end{equation} 
where the last equation follows by inserting the special values
\[
  \zeta(4)\, =\, \myfrac{\pi^4}{90}\ts ,\quad
  \zeta(6)\, =\, \myfrac{\pi^6}{945}\ts , \quad
  \zeta_K^{}(2)\, =\, \myfrac{2\pi^4}{75\sqrt{5}}\ts , \quad
  L(1,\chi_5)\, = \, \frac{2\log(\tau)}{\sqrt{5}}\ts ,
\]
and Ap\'{e}ry's constant $\zeta(3)\approx 1.202{\ts}056{\ts}903$; see
\cite{csl-BM,csl-BGHZ08} and references therein.

A familiar argument based on Theorem~\ref{csl-thm:meanvalues} gives us
the asymptotic growth rate of $c^{\mathsf{rot}}_{\! A_4}(m)$ as
follows.

\begin{corollary}
  With the residue\/ $\rho^{\mathsf{rot}}_{\! A_4}$ from
  Eq.~\eqref{csl-rhoA4rot}, the summatory asymptotic
  behaviour\index{asymptotic~behaviour} of\/ $c^{\mathsf{rot}}_{\!
    A_4}(m)$ is given by
  \[
    \sum_{m\leq x} c^{\mathsf{rot}}_{\!A_4}(m) \,\sim\, 
     \rho^{\mathsf{rot}}_{\! A_4} \ts \frac{x^3}{3} 
     \, \approx\,  0.419{\ts}375 \, x^3 ,\vspace{-2mm}
  \]
   as\/ $x\to\infty$.\qed
\end{corollary}

As we shall see later in Corollary~\ref{csl-cor:growth-csl-a4}, the
number of coincidence rotations and the number of CSLs of a given
index grow much faster than the number of SSLs. This is due to the
fact that the index of a primitive SSL is $\den^{}_{A_4}(q)^4$,
whereas the coincidence index $\Sig^{}_{\! A_4}(q)$ is much smaller
and satisfies the inequality $\den^{}_{A_4}(q)\leq \Sig^{}_{\!
  A_4}(q)\leq \left(\den^{}_{A_4}(q)\right)^2$.\smallskip

The key result in counting the number of distinct CSLs is the
following.

\begin{theorem}[{\cite[Thm.~7]{csl-HZ10}}]\label{dim4a4-theo:equalL}\sloppy
  Assume that\/ $q^{}_1$ and\/ $q^{}_2$ are admissible.  Then, one
  has\/ $L(R(q^{}_1))=L(R(q^{}_2))$ if and only if\/
  $|q^{}_1|^2=|q^{}_2|^2$ and\/
  $\gcld(q^{}_1,|q^{}_1\widetilde{q}^{}_1|/c)
  =\gcld(q^{}_2,|q^{}_2\widetilde{q}^{}_2|/c)$, where\/ $c=\sqrt{5}$
  if\/ $|q^{}_1|^2=|q^{}_2|^2$ is divisible by\/ $5$, and\/ $c=1$
  otherwise.\qed
\end{theorem}

From this result, one can derive the following explicit expression for the
counting function \cite[Eq.~(5.93)]{csl-habil}
\[
   c^{}_{\! A_4}(p^r) \,=\,
   \begin{cases}
     6 \cdot 5^{2r-2}, & \mbox{if } p=5,\\[1mm]
     \frac{(p\ts\ts +1)^2}{p^3-1}\bigl(p^{2r+1}+p^{2r-2}-2\ts
     p^{\frac{r-1}{2}}\bigr),
     &\mbox{if $p \equiv \pm1\; (5)$, $r$ odd,}\\[1mm]
     \frac{(p\ts\ts +1)^2}{p^3-1}\bigl(p^{2r+1}+p^{2r-2}-2\ts
     \frac{p^2\nts +1}{p\ts\ts +1}
     \ts p^{\frac{r-2}{2}} \bigr),
     &\mbox{if $p \equiv \pm1 \; (5)$, $r$ even,}\\[1mm]
     p^{2r} \! + p^{2r-2},
         & \mbox{if } p\equiv \pm 2 \; (5).
   \end{cases}
\]
We can now summarise as follows.

\begin{theorem}[{\cite[Thm.~5.5.6]{csl-habil}}]\label{csl-theo:PsiL-CSL}
  Let\/ $c^{}_{\! A_4}(m)$ be the number of CSLs of the
  root lattice\/ $A_4$ of index\/ $m$. The Dirichlet 
  series\index{Dirichlet~series} generating function
  for\/ $c^{}_{\! A_4}(m)$ reads
  \begin{align*}
  \quad
  \Psi^{}_{\! A_4}(s)\, & =\, \sum_{n\in\NN} \frac{c^{}_{\! A_4}(n)}{n^s} \,
  =\,\Bigl(1+\myfrac{6\cdot 5^{-s}}{1-5^{2-s}}\Bigr)
  \prod_{p\equiv\pm 2(5)}\frac{1+p^{-s}}{1-p^{2-s}} \\[2mm]
  & \qquad\qquad\qquad\quad \times \!
    \prod_{p\equiv\pm 1(5)}\!\!\frac{1+p^{-s}+2p^{1-s}+2p^{-2s}+p^{1-2s}+p^{1-3s}}
    {(1-p^{2-s})(1-p^{1-2s})}\\[2mm]
  &=\,1+\tfrac{5}{2^s}+\tfrac{10}{3^s}+\tfrac{20}{4^s}+\tfrac{6}{5^s}
  +\tfrac{50}{6^s}+\tfrac{50}{7^s}+\tfrac{80}{8^s}+\tfrac{90}{9^s}
  +\tfrac{30}{10^s}+\tfrac{144}{11^s}+\tfrac{200}{12^s}+\tfrac{170}{13^s}
  +\cdots.    \quad\qed
  \end{align*}
\end{theorem}

In order to compare $\Psi_{\! A_4}(s)$ and $\Psi^{\mathsf{rot}}_{\!
  A_4}(s)$, we consider the function
\[
  \psi^{}_{\! A_4}(s)\,   :=\, 
   \frac{\Psi_{\! A_4}(s)}{\Psi^{\mathsf{rot}}_{\! A_4}(s)} 
      = \, \Bigl( 1 - \myfrac{24\cdot 5^{-s}}{1+5^{1-s}} \Bigr)
          \prod_{p\equiv\pm 1(5)}\!
           \left( 1 - \frac{2\ts (p^2-1)\ts p^{-2s}}
          {(1+p^{-s})(1+p^{1-s})(1-p^{1-2s})} 
             \right). 
\]
It is analytic in the open half-plane $\{\Real(s)> \frac{3}{2}\}$, as
the Euler product converges there. This proves that $\Psi^{}_{\!
  A_4}(s)$ is a meromorphic function in the open half-plane
$\{\Real(s)> \frac{3}{2}\}$. Its rightmost pole is a simple pole at
$s=3$ with residue
\begin{equation}\label{csl-rhoA4}
  \rho^{}_{\! A_4}\, =\, \Res_{s=3}\bigl(\Psi_{\! A_4}^{}(s)\bigr)
      \, = \, \psi^{}_{\! A_4}\nts (3)\, \rho^{\mathsf{rot}}_{\! A_4}
      \, \approx \, 1.025{\ts}695, 
\end{equation}
where $\psi^{}_{\! A_4}(3)\approx 0.815{\ts}257{\ts}622 < 1$ 
has been calculated numerically. Finally, we apply
Theorem~\ref{csl-thm:meanvalues}, which gives us the asymptotic growth
rate as follows.

\begin{corollary}\label{csl-cor:growth-csl-a4}
  With the residue\/ $\rho^{}_{\! A_4}$ from Eq.~\eqref{csl-rhoA4}, 
  the summatory asymptotic behaviour\index{asymptotic~behaviour}
  of\/ $c^{}_{\! A_4}(m)$ is given by
  \[
    \sum_{m\leq x} c^{}_{\! A_4}(m) \,\sim\, \rho^{}_{\! A_4} \myfrac{x^3}{3} 
      \,\approx\, 0.341{\ts}898 \, x^3,\vspace{-2mm}  
   \]
   as\/ $x\to\infty$.\qed
\end{corollary}

Comparing the growth rate of the number of CSLs with that of the
coincidence rotations, we see that the former is roughly 20\% lower
than the latter. As we shall see shortly, this difference is much
bigger than in the case of the icosian ring. Yet, it is still more an
exception than a rule that two coincidence rotations that are not
symmetry related\index{symmetry~related} generate the same
CSL.

\subsection{Coincidences of $\II$}
\label{csl-sec:a4-I}

Here, we want to consider the CSMs of the icosian ring $\II$ itself,
which is an interesting example of an embedded module in $4$-space.
It is also a $\ZZ[\tau]$-lattice in $\RR^4$ in the sense of
Definition~\ref{csl-def:S-lat}.

The methods to find the CSMs are basically a combination of the tools
we used in Sections~\ref{csl-sec:cub4} and~\ref{csl-sec:a4}, as we
deal with admissible pairs of quaternions in $\II$ here.  Thus, we
will keep the presentation concise and refer to~\cite{csl-habil} for
details.

As $\II$ is a $\ZZ(\tau)$-lattice, we have
$\scal_\II(\one)=\QQ(\tau)^{\times}$ and $\Scal_\II(\one)=\ZZ(\tau)$.
Correspondingly, we call a pair $(p,q)\in\II\!\times\!\II$
\emph{primitive admissible} if $p,q$ are primitive and $|pq|^2$ is a
square in $\ZZ[\tau]$.  It follows that the coincidence
rotations\index{rotation!coincidence} are precisely those rotations
$R(q,p)x=q x p/|pq|$ that can be parametrised by a primitive
admissible pair; compare~\cite{csl-habil}.

As before, it makes sense to define   
\begin{equation} 
  \alpha^{}_{q}\,:=\,\sqrt{\frac{|p|^2}{\gcd(|q|^2, |p|^2)}} 
   \quad\text{and}\quad
  \alpha^{}_{p}\,:=\,\sqrt{\frac{|q|^2}{\gcd(|q|^2, |p|^2)}}
\end{equation}
for any primitive admissible pair $(q,p)$, where $\alpha^{}_q$ and
$\alpha^{}_{p}$ are again defined up to a unit (now in $\ZZ[\tau]$).
Correspondingly, we call
$(q^{}_{\alpha},p^{}_{\alpha})=(\alpha^{}_q \ts q,\alpha^{}_{p} \ts p)$
the extension of the primitive admissible pair $(q,p)$.  This implies
\begin{equation}\label{csl-dim4a4-eq:qapa}
  |q^{}_{\alpha}|^2\, =\, |p^{}_{\alpha}|^2
  \, = \, |q^{}_{\alpha} p^{}_{\alpha}|  \; \quad
  \text{(up to a unit).}
\end{equation}

The CSMs of $\II$ can now completely be characterised as follows,
which is the analogue of Theorem~\ref{csl-theo:csl-cub4} for the
Hurwitz ring $\JJ$.

\begin{theorem}[{\cite[Thm.~5.4.2 and ~5.4.4]{csl-habil}}]\label{csl-theo:csm-I}
  Let\/ $(q^{}_{\alpha},p^{}_{\alpha})$ be the extension of the
  primitive admissible pair\/ $(q,p)$.  Then, one has
  \[
  \II\cap \frac{q \II p}{|qp|}\,=\,
  q^{}_{\alpha}\II + \II\ts p^{}_{\alpha}\ts .
  \]
  The index\index{index} of this CSM in\/ $\II$ is given by
  \[
  \qquad\quad\quad\qquad\quad
  \Sig^{}_\II\bigl(R(q,p)\bigr)\, =\, \nr\bigl(\lcm(|q|^2,|p|^2)\bigr)
    \, =\, \nr\bigl(|q^{}_{\alpha}|^2\bigr)
    \, =\, \nr\bigl(|p^{}_{\alpha}|^2\bigr).  
    \qquad\; \qquad \quad\;  \qed
  \]
\end{theorem}

This allows us to calculate the number of coincidence
rotations\index{rotation!coincidence} of a given index $m$, which is
given by $7200\, c^{\mathsf{rot}}_{\II}(m)$, where the factor $7200$
is the order of $\SO(\II)$, the rotation symmetry group of $\II$, and
$c^{\mathsf{rot}}_{\II}(m)$ is a multiplicative
function\index{multiplicative~function} which is completely determined
by\index{counting~function!icosian~ring}
\[
  c^{\mathsf{rot}}_{\II}(p^r)\, =\, 
  \begin{cases}
    3 \cdot 5^{r-1}(13\cdot 5^{r-1}-1), & \text{if } p=5, r\geq 1, \\[1mm]
    h(p,r), & \text{if } p\equiv \pm 1\; (5) \text{ and } r\geq 1,\\[1mm]
    \frac{p^2+1}{p^2-1}\, p^{2r-2} ( p^{2r+2}+p^{2r-2}-2 ), 
    & \text{if } p\equiv \pm 2\; (5) \text{ and }r\geq 2 \text{ even, }\\[1mm]
    0, & \text{if } p\equiv \pm 2\; (5) \text{ and } r\geq 1 \text{ odd,}
  \end{cases}
\]
with\vspace{-2mm}
\begin{equation}
\begin{split}
 h(p,r)
   & \, =\,  2\ts p^{2r-2}(p+1)^2 -4 \ts p^{r-2}
        \frac{p^{r-1}-1}{(p-1)^3} (3 p^2 + 1)(p+1) \\[1mm]
   &\qquad  + (r-1) \frac{(p+1)^2}{(p-1)^2} p^{r-2}  \bigl(p^{r-2} (p^2 + 1)^2 
       +  4 \bigr) . 
\end{split}
\end{equation}
Thus, we can calculate the corresponding generating function.

\begin{theorem}[{\cite[Thm.~5.4.5]{csl-habil}}]\label{csl-theo:Psi-rot-I}
  Let\/ $7200\, c^{\mathsf{rot}}_{\II}(m)$ be the number of
  coincidence rotations\index{rotation!coincidence} of the icosian
  ring\/ $\II$. Then, the Dirichlet series\index{Dirichlet~series}
  generating function for\/ $c^{\mathsf{rot}}_{\II}(m)$ reads
  \[
  \begin{split}
  \Psi_{\II}^{\mathsf{rot}}(s) \,
   & = \sum_{n\in\NN} \frac{c^{\mathsf{rot}}_{\II}(n)}{n^s}
    \, = \, \zeta_{\ts\II}^{\mathsf{pr}}(s) \,\zeta_{\ts\II}^{\mathsf{pr}}(s-1)
   \\[1mm]
   & = \,\frac{(1+5^{-s})(1+5^{1-s})}{(1-5^{1-s})(1-5^{2-s})} 
    \prod_{p\equiv\pm 1(5)} \!
      \! \left(\frac{(1+p^{-s})(1+p^{1-s})}{(1-p^{1-s})(1-p^{2-s})}
          \right)^{\! 2} \! \!
    \prod_{p\equiv\pm 2(5)} \!
      \frac{(1+p^{-2s})(1+p^{2-2s})}{(1-p^{2-2s})(1-p^{4-2s})}
    \\[2mm]
   &=\, \textstyle{1+\frac{25}{4^s}+\frac{36}{5^s}+\frac{100}{9^s}
    +\frac{288}{11^s}+\frac{440}{16^s}+\frac{400}{19^s}
    +\frac{900}{20^s}+\frac{960}{25^s}+\frac{1800}{29^s}+\frac{2048}{31^s}
    +\frac{2500}{36^s}
    +\cdots }
  \end{split}
  \]
  with\/ $\zeta_{\ts\II}^{\mathsf{pr}}(s)$
  as\index{zeta~function!icosian~ring} given in
  Eq.~\eqref{csl-eq:zetaIpr}.  In particular, the possible coincidence
  indices are exactly those numbers that can be represented by the
  quadratic form\/ $k^2+k\ell-\ell^2=\nr(k+\ell\tau)$.  \qed
\end{theorem}

$\Psi_{\II}^{\mathsf{rot}}(s)$ is a meromorphic function in the entire
complex plane, whose rightmost pole is a simple pole at $s=3$ with
residue (see~\cite[Eq.~(5.61)]{csl-habil})
\begin{equation}\label{csl-eq:rhoI}
\begin{split}
    \rho_\II^{\mathsf{rot}} \, & :=\, 
    \Res_{s=3}\ts\bigl(\Psi_{\II}^{\mathsf{rot}}(s)\bigr)
      \, = \, \frac{\zeta_K^{}(2)^2\zeta_K^{}(3)}{
        \zeta_K^{}(4)\ts\zeta_K^{}(6)}L(1,\chi_5) \\
     & \hphantom{:}= \,   \frac{3^5\, 5^7\, 7 \sqrt{5}}{268 \ts \pi^{12}}\log(\tau) 
             \ts \zeta_K^{}(3)
        \, \approx\,  0.593{\ts}177. 
\end{split}
\end{equation}

Using Theorem~\ref{csl-thm:meanvalues}, we get the following
asymptotic behaviour.

\begin{corollary}[{\cite[Cor.~5.4.6]{csl-habil}}]
  The asymptotic behaviour\index{asymptotic~behaviour} of the
  summatory function of\/ $c^{\mathsf{rot}}_{\II}(m)$, as\/ $x \to \infty$, is
  \[
    \sum_{m\leq x} c^{\mathsf{rot}}_{\II}(m) \,\sim\, 
     \rho_\II^{\mathsf{rot}} \myfrac{x^3}{3} 
      \,\approx\, 0.197{\ts}726\, x^3, 
 \]
 with\/ $\rho_\II^{\mathsf{rot}}$ as given in
 Eq.~\eqref{csl-eq:rhoI}.  \qed
\end{corollary}

In order to enumerate the CSMs themselves, we need a criterion that
tells us which rotations generate the same CSM. This is given by the
following result, which is the analogue of
Theorem~\ref{csl-theo:4dim-equalcslJ} for $\JJ$, and of
Theorem~\ref{dim4a4-theo:equalL} for the lattice $A_4$.

\begin{theorem}[{\cite[Thm.~5.4.13]{csl-habil}}]
  \label{csl-4dim-equalI}
  Let\/ $(q^{}_1,p^{}_1)$ and\/ $(q^{}_2,p^{}_2)$ be two primitive
  admissible pairs. Then, the identity
  \[
  \II\cap \frac{q^{}_1 \II\ts p^{}_1}{|q^{}_1p^{}_1|}\, =\, 
  \II\cap \frac{q^{}_2 \II\ts p^{}_2}{|q^{}_2p^{}_2|}
  \]
  holds if and only if the following conditions are satisfied
  $($up to units$\ts )$.
  \begin{enumerate}\itemsep=2pt
  \item $|q^{}_1p^{}_1|=|q^{}_2p^{}_2|$, 
  \item $\lcm(|q^{}_1|^2,|p^{}_1|^2)=\lcm(|q^{}_2|^2,|p^{}_2|^2)$, 
  \item $\gcld(q^{}_1,|p^{}_1q^{}_1|)=\gcld(q^{}_2,|p^{}_2q^{}_2|)$, and
  \item $\gcrd(p^{}_1,|p^{}_1q^{}_1| =\gcrd(p^{}_2,|p^{}_2q^{}_2|)$.    \qed
  \end{enumerate}
\end{theorem}

The effect of these criteria is that it is now a purely combinatorial
task to calculate $c^{}_{\II}(m)$ and the corresponding Dirichlet
series. For explicit expressions for $c^{}_{\II}(m)$,
see~\cite[Eq.~5.79]{csl-habil}.

\begin{theorem}[{\cite[Thm.~5.4.14]{csl-habil}}]\label{csl-theo:Psi-I}
  Let\/ $c^{}_{\II}(m)$ be the number of CSMs of the icosian ring\/
  $\II$ of index\/ $m$. Then, the Dirichlet
  series\index{Dirichlet~series} generating function for\/
  $c^{}_{\II}(m)$ reads\vspace{-1mm}
  \[
  \begin{split}
   \qquad  \;\Psi_{\II}(s)\,
   & = \sum_{n\in\NN} \frac{c^{}_{\II}(n)}{n^s} 
   \, = \, \myfrac{1+11\cdot 5^{-s}+7\cdot 5^{-2s}+5^{1-3s}}%
      {(1-5^{2-s})(1-5^{1-2s})}\\[1mm]
   &\qquad\qquad\qquad\;\times\! \prod_{p\equiv\pm 1(5)}
     \left(\frac{1+p^{-s}+2p^{1-s}+2p^{-2s}+p^{1-2s}+p^{1-3s}}%
     {(1-p^{2-s})(1-p^{1-2s})}\right)^{\!2}\\[2mm]
   &\qquad\qquad\qquad\;\times\! \prod_{p\equiv\pm 2(5)}
       \frac{1+p^{-2s}+2p^{2-2s}+2p^{-4s}+p^{2-4s}+p^{2-6s}}%
       {(1-p^{4-2s})(1-p^{2-4s})}\\[3mm]
   &=\,\textstyle{1+\frac{25}{4^s}+\frac{36}{5^s}+\frac{100}{9^s}
     +\frac{288}{11^s}+\frac{410}{16^s}+\frac{400}{19^s}
     +\frac{900}{20^s}+\frac{912}{25^s}+\frac{1800}{29^s}
     +\frac{2048}{31^s} +\cdots.}   \qquad \;\; \qed
  \end{split}
  \]
\end{theorem}

We are not aware of a representation of $\Psi_{\II}(s)$ in terms of
zeta functions. Nevertheless, we can determine its analytic
properties.  We note that the Euler product
{\allowdisplaybreaks[4]\begin{align}
  \lefteqn{\hspace{-4em}
    \psi^{}_\II(s)\,:=\,\frac{\Psi_{\II}(s)}{\Psi_{\II}^{\mathsf{rot}}(s)}
    \,=\,
    \left( 1 - \myfrac{48\cdot 5^{-2s}}{(1+5^{-s})(1+5^{1-s})(1-5^{1-2s})} \right)}
      \nonumber \\[1mm]
      &\ \qquad\qquad\quad \times \prod_{p\equiv\pm 1(5)}
     \left( 1 - \frac{2(p^2-1)p^{-2s}}{(1+p^{-s})(1+p^{1-s})(1-p^{1-2s})} 
          \right)^{\! 2} \\[1mm]
      &\ \qquad\qquad\quad \times \prod_{p\equiv\pm 2(5)}    
      \left( 1 - \frac{2(p^4-1)p^{-4s}}{(1+p^{-2s})(1+p^{2-2s})(1-p^{2-4s})} 
         \right) \nonumber
\end{align}}
converges for $\Real(s)>\frac{3}{2}$, which implies that
$\Psi_{\II}(s)$ is meromorphic in the half-plane given by
\mbox{$\bigl\{\Real(s)>\frac{3}{2}\bigr\}$}. Moreover, the rightmost pole of
$\Psi_{\II}(s)$ is a simple pole located at $s=3$, with residue
\begin{equation}
  \rho_\II^{}\, :=\, \Res_{s=3}\ts\bigl(\Psi_{\II}^{}(s)\bigr)
    \,=\, \psi^{}_\II(3) \ts\rho_\II^{\mathsf{rot}}
    \,\approx\, 0.587{\ts}063. 
\end{equation}
Here, $\psi^{}_\II(3)\approx 0.989{\ts}691{\ts}798 < 1$ 
was calculated numerically.  Finally, we apply
Theorem~\ref{csl-thm:meanvalues} to obtain the asymptotic behaviour.

\begin{corollary}[{\cite[Cor.~5.4.15]{csl-habil}}]
  The asymptotic behaviour\index{asymptotic~behaviour} of the
  summatory function of\/ $c^{}_{\II}(m)$, as\/ $x\to \infty$, is
  \begin{align}
    \sum_{m\leq x} c^{}_{\II}(m) \,\sim\, \rho_\II^{} \myfrac{x^3}{3} 
      \,\approx\, 0.195{\ts}688\, x^3, 
  \end{align}
  with $\rho_\II^{}$ as given above. \qed
\end{corollary}
Note that $\rho_\II^{}$ and $\rho_\II^{\mathsf{rot}}$
differ by just about 1\%.  Thus, in most cases, two coincidence
rotations that are not symmetry related\index{symmetry~related}
generate different CSMs.

\section{Multiple CSLs of the cubic lattices}
\label{csl-sec:mcsl-cub}

So far, we have mostly considered ordinary (or simple) CSLs and
CSMs. The problem of finding all \emph{multiple} CSLs (MCSLs) is more
difficult than determining all CSLs. In fact, there are only few cases
where the problem of multiple coincidences has been solved so
far. These include the two-dimensional lattices and modules of
$n$-fold symmetry~\cite{csl-Baake-rev}, which we discussed in
Section~\ref{csl-sec:csl-nfold}, and the three-dimensional cubic
lattices, which we want to discuss here;
compare~\cite{csl-pzmcsl1,csl-habil}.

Let us recall from Section~\ref{csl-sec:cubic} that any coincidence
rotation $R$ of the cubic lattices can be parametrised by primitive
Hurwitz quaternions. Moreover, there is a bijection between the CSLs
of the body-centred cubic lattice and the ideals $q \JJ$ generated by
odd primitive quaternions. In particular, we have $\Gbcc=\Imag(\JJ)$
and $\Gbcc(R(q ))=\Imag(q \JJ)$ with $\Sig(R(q))=|q |^2$ if $q $ is a
primitive odd quaternion. If $q $ is an even primitive quaternion,
then $\Sig(R(q))=\frac{|q |^2}{2}$. In this case, $q$ can be written
as the product $r\,(1,1,0,0)$ of an odd primitive quaternion with an
even one, and the corresponding CSL can be written as $\Gbcc(R(q
))=\Imag(r\JJ)$.

Consequently, it is sufficient to consider CSLs generated by primitive
odd quaternions.  Just as in the case of ordinary CSLs, we start with
the analysis of the body-centred cubic lattice and later derive the
MCSLs of the other cubic lattices in the setting of
Eq.~\eqref{csl-cubiclat}.\smallskip

Let us first discuss the coincidence spectrum. We know from
Remark~\ref{csl-rem:spec-cub} that the ordinary coincidence spectrum
for all three types of cubic lattices is $2\ts\NN_{0} +1$. Moreover,
we have seen in Section~\ref{csl-sec:mcsl} that
$\Sig(R^{}_1,\ldots,R^{}_m)$ divides the product
$\Sig(R^{}_1)\cdot\ldots\cdot\Sig(R^{}_m)$. Thus, the spectrum of
indices of MCSLs is again the set of positive odd integers.

\begin{proposition}\label{csl-prop:spectrum}
  Let\/ $\vG$ be any cubic lattice. The $($multiple$\ts )$ coincidence
  spectrum of\/ $\vG$ is given by\/ $\ts 2\ts\NN_{0} +1$.\qed
\end{proposition}

Hence, no new indices occur. Nevertheless, additional lattices emerge
and the multiplicity of the corresponding index will increase. We have
seen that $c^{}_\vG(m)$ is a multiplicative function. By
Theorem~\ref{csl-thm:mc}, this implies that any ordinary CSL can be
written as 
\[ 
      \vG(R) \, = \, \vG(R^{}_1)\cap\ldots\cap \vG(R^{}_n) \ts ,
\]
where the
indices $\Sig^{}_\vG(R^{}_i)$ are powers of distinct primes. In this
case, we know that the MCSL $\vG(R^{}_1)\cap\ldots\cap \vG(R^{}_n)$ 
agrees with an
ordinary CSL.  However, if the indices of the $\vG(R_i)$ are not
relatively prime, the corresponding MCSL is, in general, not equal to
an ordinary CSL.

More generally, by an application of Theorem~\ref{csl-thm:mc}, the
multiplicativity of $c^{}_\vG(m)$ guarantees that \emph{any} MCSL
$\vG(R^{}_1,\ldots,R^{}_n)$ can be written as the intersection of
MCSLs $\vG_k$ of prime power index. Furthermore, the $\vG_k$ can be
chosen in such a way that they are intersections of at most $n$
ordinary CSLs.  Thus, we may restrict our analysis of MCSLs to those
MCSLs whose index is a prime power.

To become more concrete, we mention that the decomposition of CSLs
into CSLs of prime power index corresponds to the prime factorisation
in $\JJ$.  In particular, if $|q|^2=\pi_1^{\alpha_1}\cdot\ldots\cdot
\pi_k^{\alpha_k}$ is the prime factorisation of $|q|^2$ in $\NN$ and
$p^{}_{i}:=\gcld(q,\pi_i^{\alpha_i})$, the aforementioned
decomposition is now given by $\vG(R(q))= \vG(R(p^{}_1))\cap\ldots\cap
\vG(R(p^{}_k))$. Note that $q$ is a common right multiple of all
$p^{}_{i}$. Conversely, if the $p^{}_{i}$ are primitive odd
quaternions such that all $|p^{}_{i}|^2$ are relatively prime, then
any least common right multiple $q$ is primitive and odd, and we have
$\vG(R(q))= \vG(R(p^{}_1))\cap\ldots\cap \vG(R(p^{}_k))$.  Likewise,
if we define \mbox{$p^{}_{ij}=\gcld(q_{i},\pi^{\alpha_{ij}}_{j})$},
where the $\alpha^{}_{ij}$ are the exponents in the prime
factorisation $|q_i|^2=\pi^{\alpha_{i1}}_{1}\cdot\ldots\cdot
\pi^{\alpha_{ik}}_{k}$, then the corresponding decomposition of the
MCSL reads $\vG(R(q^{}_1),\ldots,R(q^{}_n))=\vG_1\cap\ldots\cap \vG_k$
with the lattices $\vG_\ell=\vG(R(p^{}_{1\ell}))\cap\ldots\cap
\vG(R(p^{}_{n\ell}))$.

Moreover, this guarantees the multiplicativity of the corresponding
counting functions $c^{(\infty)}_{}(m)$ and $c^{(k)}_{}(m)$, where
$c^{(\infty)}_{}(m)$ is the number of all MCSLs of a given index $m$
and $c^{(k)}_{}(m)$ the corresponding number of all MCSLs that can be
written as the intersection of at most $k$ ordinary CSLs.

As we want to enumerate the distinct MCSLs, it is an essential
question under what condition two MCSLs are equal. A preliminary
result is the following, which generalises
Lemma~\ref{csl-lem:equal-csl} to the present situation.

\begin{lemma}[{\cite[Lemma~6.1.2]{csl-habil}}]\label{csl-lem:equal-mcsl}
  Let\/ $\vG$ be any cubic lattice and assume that\vspace{-2mm}
\[
  \vG\bigl(R(q^{}_1),\ldots,R(q^{}_n)\bigr)\, =\, 
  \vG\bigl(R(q'_1),\ldots,R(q'_m)\bigr) ,
\]
where\/ $q^{}_{i}$ and\/ $q'_{j}$ are primitive odd quaternions. Then, we have
\[
    \Sig_\vG^{}
    \bigl(R(q^{}_1),\ldots,R(q^{}_n)\bigr)\, =\,
    \Sig_\vG^{}\bigl(R(q'_1),\ldots,R(q'_m)\bigr)
\]
together with\/ $\,\lcm\left(|q^{}_1|^2,\ldots, |q^{}_n|^2\right) = 
    \lcm\left(|q'_1|^2,\ldots, |q'_n|^2\right)$. \qed
\end{lemma}

The conditions of the lemma are necessary, but by no means
sufficient. For ordinary CSLs, we have the much stronger condition
$q\JJ=q'\JJ$, and we expect additional conditions for MCSLs. Let us
start with the case $n=2$.

\subsection{Intersections of two CSLs}
\label{csl-sec:mcsl-cub-2}

As the body-centred cubic lattice $\vG=\Gbcc=\Imag(\JJ)$ has the most
convenient representation in terms of quaternions, we start with this
lattice.  The first step to determine all possible MCSLs
$\vG(R^{}_1,R^{}_2)$ that can be written as the intersection of at
most two ordinary CSLs is the calculation of their indices. We note
that
\[
   \vG_+(R^{}_1,R^{}_2) \, := \,
   \vG(R^{}_1)+\vG(R^{}_2) \, = \,
   \Imag(q^{}_1\JJ+q^{}_2\JJ) \, = \,
   \Imag(q \JJ) \ts ,
\]
  where $q =\gcld(q^{}_{1},q^{}_{2})$. Hence, recalling
  that we may assume $|q_i|^2$ to be odd, we have
\begin{equation}\label{csl-eq:sig12c}
  \Sig(R^{}_1,R^{}_2)\, =\, \frac{|q^{}_1|^2\ts |q^{}_2|^2}{|q |^2} .
\end{equation}
In the case that $|q^{}_1|^2$ and $|q^{}_2|^2$ are relatively prime,
this reduces to
$\Sig(R^{}_1,R^{}_2)=|q^{}_1|^2|q^{}_2|^2$.  This is the
aforementioned case where the MCSL is equal to an ordinary
CSL. Another special case occurs when $q^{}_1$ is a left divisor of
$q^{}_2$. Here, we have $\vG(R^{}_2)\subseteq \vG(R^{}_1)$, and the
MCSL $\vG(R^{}_1,R^{}_2)=\vG(R^{}_2)$ is again an ordinary CSL. In
order to understand the general situation, we start with the case that
both $|q_i|^2$ are powers of the same rational prime $p$.

Actually, the case of MCSLs of prime power index is sufficient,
because we can recover the general case from this one, as we have
mentioned before. We are mainly interested in the case of two
different CSLs none of which is a sublattice of the other one, so
neither $q^{}_1$ nor $q^{}_2$ is a right multiple of the other
one. Fortunately, we do not need to exclude the latter case
explicitly, as all formulas include the case of ordinary CSLs
implicitly.

Recall that we have an explicit expression for ordinary CSLs, namely
$\Gbcc(R(q))=\Imag(q\JJ)$. An analogous expression for MCSLs is given
by the following result.

\begin{lemma}[{\cite[Lemma~6.2.2]{csl-habil}}]\label{csl-lem:mcsl2b}
  Let\/ $q^{}_1$ and\/ $q^{}_{2}$ be primitive quaternions with\/
  $|q_i|^2=p^{\alpha_i}$, where\/ $p$ is the same odd prime for both
  quaternions. Let\/ $q $ be a least common right multiple of\/
  $q^{}_1$ and\/ $q^{}_2$.  Then, we have
\[\;\;\quad\qquad\qquad\qquad\quad
  \Gbcc(R^{}_1,R^{}_2)\,=\,
   \Imag(q \JJ+q^{}_1\JJ\ts{\bar{q}}^{}_2)\,=\, \Imag(q
  \JJ+q^{}_2\JJ\ts{\bar{q}}^{}_1)\ts .\qquad\qquad\qquad
  \quad\;\;\qed
\]
\end{lemma}

Note that $q \JJ+q^{}_1\JJ\ts{\bar{q}}^{}_2$ need not be an ideal. If
not, $\Gbcc(R^{}_1,R^{}_2)$ is neither an ordinary CSL nor a multiple
of one. Further, note that $\Imag(q \JJ)$ is a multiple of an ordinary
CSL as $q$, in general, is not primitive here.

When enumerating MCSLs, we must make sure that we do not count any
MCSL twice.  Thus, we need a criterion when two MCSLs are equal. This
is provided by the following result.

\begin{theorem}[{\cite[Thm.~6.2.3]{csl-habil}}]\label{csl-eq2csl}
  Let\/ $q^{}_i$ with\/ $1\le i\le  4$ be primitive quaternions
  such that\/ \mbox{$|q^{}_i|^2=p^{\alpha_i}$}, where\/ $p$ is an odd
  rational prime and where\/
  $\alpha^{}_1\geq\alpha^{}_2\geq\alpha^{}_4$ and\/
  \mbox{$\alpha^{}_3\geq\alpha^{}_4$}. Let\/ $q^{}_{ij}$ with\/
  $|q^{}_{ij}|^2=p^{\alpha_{ij}}$ be the greatest common left divisor
  of\/ $q^{}_i$ and\/ $q^{}_j$. In addition, if\/
  $\alpha^{}_1=\alpha^{}_2$, let\/ $\alpha^{}_{13}\geq\alpha^{}_{23}$,
  and if\/ $\alpha^{}_3=\alpha^{}_4$, let\/
  $\alpha^{}_{13}\geq\alpha^{}_{14}$. Then, with $R_{i}=R(q^{}_i)$, we
  have
  \[
  \Gbcc(R^{}_1)\cap\Gbcc(R^{}_2)\, =\, \Gbcc(R^{}_3)\cap\Gbcc(R^{}_4)
  \]
  if and only if the conditions\/ \mbox{$\alpha^{}_1=\alpha^{}_3$},
  $\alpha^{}_2-\alpha^{}_{12}=\alpha^{}_4-\alpha^{}_{34}$,
  $\alpha^{}_1-\alpha^{}_{13}\leq\min(\alpha^{}_4-\alpha^{}_{34},\alpha^{}_{34})$
  and\/ $\alpha^{}_4-\alpha^{}_{24}\leq
  \mbox{$\min(\alpha^{}_4-\alpha^{}_{34},\alpha^{}_{34})$}$ are
  satisfied.\qed
\end{theorem}

Note that the ordering conditions on the $\alpha$ coefficients do not
put any restrictions on the applicability of the theorem, since we can
always interchange the role of the $q^{}_{i}$ such that these
conditions are met.

\begin{remark}
  The two conditions $\alpha^{}_1=\alpha^{}_3$ and
  $\alpha^{}_2-\alpha^{}_{12}=\alpha^{}_4-\alpha^{}_{34}$ correspond
  to the two conditions in Lemma~\ref{csl-lem:equal-mcsl}. The first
  one means that the least common multiples of the denominators must
  be the same, and the second follows from the equality of the
  indices, which gives
  $\alpha^{}_1+\alpha^{}_2-\alpha^{}_{12}=\alpha^{}_3 +
  \alpha^{}_4-\alpha^{}_{34}$.  Furthermore, the condition
  $\alpha^{}_1-\alpha^{}_{13}\leq \alpha^{}_4-\alpha^{}_{34}$ can
  easily be understood by considering
\[
\begin{split}
  \hspace{2.7cm}
    \Gbcc(R^{}_1)\cap\vG(R^{}_3)
    \,&\supseteq\, \Gbcc(R^{}_1)\cap\Gbcc(R^{}_2)\cap
                    \Gbcc(R^{}_3)\cap\Gbcc(R^{}_4)\\[1mm]
    &=\, \Gbcc(R^{}_3)\cap\Gbcc(R^{}_4)\ts . \hspace{5.75cm}
    \mbox{\exend}
\end{split}
\]
\end{remark}

Theorem~\ref{csl-eq2csl} is not very intuitive, but we can understand
it better by comparing the quaternions involved. It basically tells us
how different the quaternions $q^{}_1,q^{}_3$ and $q^{}_2,q^{}_4$ may
be; see~\cite{csl-habil} for details. This allows us to calculate
the counting function for MCSLs that are the intersection of at most
two ordinary CSLs.

\begin{theorem}[{\cite[Thm.~6.2.4]{csl-habil}}]
  Let\/ $p$ be an odd prime number. Then, the number\/
  $c^{(2)}_{\mathsf{bcc}}(p^r)$ of distinct MCSLs of\/ $\Gbcc$ of
  index\/ $p^r$ that are an intersection of at most two ordinary CSLs
  is given by
\[
\begin{split}
c^{(2)}_{\mathsf{bcc}}(p^r)\, = \;\, & \myfrac{r+1}{2}\,(p+1)\,p^{r-1}+
\left(\myfrac{r}{2}-1\right)p^{r-2}-
\left(\myfrac{r}{2}-\left[\myfrac{r}{2}\right]\right)p^{r-4}\\[2mm]
&\mbox{}+\frac{p^{r-1}-p^{r-2[r/3]-1}}{p^2-1}+
\frac{p^{4[r/3]-r+2}-p^{4[r/2]-r-2}}{2(p^2-1)}\ts ,
\end{split}
\]
where\/ $[x]$ denotes the Gau{\ss} bracket.  \qed
\end{theorem} 

As $c^{(2)}_{\mathsf{bcc}}(n)$ is a multiplicative function, we can find an
explicit expression for its Dirichlet series generating function as usual.

\begin{theorem}\label{csl-theo:c2gen}
  Let\/ $c^{(2)}_{\mathsf{bcc}}(m)$ be the number of distinct MCSLs of
  index\/ $m$ that are an intersection of at most two ordinary CSLs.
  Then, $c^{(2)}_{\mathsf{bcc}}(\Sig)$ is a multiplicative arithmetic
  function whose Dirichlet series is given by\index{Dirichlet~series}
  {\allowdisplaybreaks[4]
  \begin{align*}
  \Psi^{(2)}_{\mathsf{bcc}}(s)\,  :=&\hphantom{i}
   \sum_{n=1}^\infty\frac{c^{(2)}_{\mathsf{bcc}}(n)}{n^s}  \, =
     \!\prod_{p\in\PP\setminus\{2\}}\! \psi_2(p,s) \\[2mm]
      =& \hphantom{I}  \myfrac{1-2^{1-3s}}{1+2^{-3s}}\,
        \frac{\zeta(3s-1)\ts\zeta(3s)}{\zeta(6s)}\,
           \varphi^{(2)}_{\mathsf{bcc}}(s)\, \Psi^{}_{\mathsf{bcc}}(s) \\[2mm]
   =& \hphantom{I} \frac{(1-2^{1-s})(1-2^{1-3s})}{(1+2^{-s})(1+2^{-3s})}\,
   \frac{\zeta(s-1)\ts\zeta(s)\ts\zeta(3s-1)\ts\zeta(3s)}
        {\zeta(2s)\ts\zeta(6s)}\,
    \varphi^{(2)}_{\mathsf{bcc}}(s) \\[3mm]
  =& \hphantom{I} \textstyle{1+\frac{4}{3^s}+\frac{6}{5^s}+\frac{8}{7^s}+
  \frac{18}{9^s}+\frac{12}{11^s}
  +\frac{14}{13^s}+\frac{24}{15^s}+\frac{18}{17^s}
  +\frac{20}{19^s}}+\frac{32}{21^s}\\[2mm]
  &\mbox{}\quad \textstyle{\mbox{}+\frac{24}{23^s}+
   \frac{45}{25^s}+\frac{76}{27^s}+\frac{30}{29^s}+\frac{32}{31^s}
   +\frac{48}{33^s}+\frac{48}{35^s}+\frac{38}{37^s}+\frac{56}{39^s}
  +\cdots} ,
  \end{align*}}
  where\/ $\psi_2(p,s)$ is the Euler factor corresponding to\/
  $c^{(2)}_{\mathsf{bcc}}(p)$, which is given by
\[
  \psi^{}_{2}(p,s) \, := \sum_{r=1}^\infty\frac{c^{(2)}_{}(p^r)}{p^{rs}} \, = \, 
   \frac{(1+p^{-s})(1+p^{-3s})}{(1-p^{1-s})(1-p^{1-3s})} \ts\times\ts C(p,s) 
\]
with
\[
  C(p,s)\, = \,  \left( 1 + \frac{p^{-2s}(p^2+p)}{2(1+p^{-s})(1-p^{1-s})}
             - \frac{p^{-4s}(p+1)}{(1+p^{-s})(1-p^{1-s})(1+p^{-3s})} \right),
\]
while\/ $\varphi^{(2)}_{\mathsf{bcc}}(s)$ is then given by
\[
   \varphi^{(2)}_{\mathsf{bcc}}(s) \; = 
   \prod_{p\in\PP\setminus\{2\}}\! C(p,s) \ts ,
\]
where the product runs over all odd rational primes. \qed
\end{theorem}

The explicit knowledge of $\Psi^{(2)}_{\mathsf{bcc}}(s)$ allows us to
find its analytic properties. We know from Section~\ref{csl-sec:cubic}
that $\Psi^{}_{\mathsf{bcc}}(s)$ is meromorphic function of $s$, whose
rightmost pole is located at $s=2$.  Furthermore,
$\varphi^{(2)}_{\mathsf{bcc}}(s)$ converges absolutely in the
half-plane $\bigl\{\Real(s)>\frac{3}{2}\bigr\}$, which guarantees its
analyticity there. Thus, we get the following asymptotic behaviour.

\begin{corollary}[{\cite[Cor.~6.2.6]{csl-habil}}]\label{csl-cor:asymp2}
  The asymptotic behaviour\index{asymptotic~behaviour} of the
  summatory function of\/ $c^{(2)}_{\mathsf{bcc}}(m)$ is given by
  \[
    \sum_{m\leq x}c^{(2)}_{\mathsf{bcc}}(m)
    \,\sim\, \frac{\rho^{(2)}_{\mathsf{bcc}}}{2}x^2
    \,\approx\, 0.356{\ts}491 \, x^2 ,
  \]
  as\/ $x\to\infty$, with
  \begin{align*} \qquad\qquad\quad\;\;
    \rho^{(2)}_{\mathsf{bcc}}\, &:= 
    \,\Res_{s=2}\ts\bigl(\Psi^{(2)}_{\mathsf{bcc}}(s)\bigr)
    \,=\, \myfrac{124}{325}\,
      \frac{\zeta(2)\ts\zeta(6)\ts\zeta(5)}{\zeta(4)\ts\zeta(12)}\,
                                  \varphi^{(2)}_{\mathsf{bcc}}(2)\\[1mm]
       &\hphantom{:}= \,\myfrac{3866940}{691\ts \pi^8}\,\zeta(5)\,
         \varphi^{(2)}_{\mathsf{bcc}}(2)
       \,\approx\, 0.712{\ts}983\ts .    
   \qquad\qquad\qquad\quad\;\,\qed
  \end{align*}
\end{corollary}

If we compare the asymptotic growth rates for ordinary CSLs and
\mbox{MCSLs}, we see that the latter is not much bigger than the
former. This shows that most MCSLs are ordinary CSLs. This behaviour
is not surprising, since
$c^{(2)}_{\mathsf{bcc}}(m)=c^{}_{\mathsf{bcc}}(m)$ for square-free
indices $m$. Thus, all terms $n^{-s}$ with $n$ square-free are missing
in the expansion of
$\Psi^{(2)}_{\mathsf{bcc}}(s)-\Psi^{}_{\mathsf{bcc}}(s)$, whose first
terms are given by
\[
\begin{split}
  \Psi^{(2)}_{\mathsf{bcc}}(s)-\Psi^{}_{\mathsf{bcc}}(s)\, =& \hphantom{I}
  \textstyle{\frac{6}{9^s}+\frac{15}{25^s}+\frac{40}{27^s}
  +\frac{36}{45^s}+\frac{28}{49^s}+\frac{48}{63^s}+\frac{60}{75^s}
  +\frac{174}{81^s}+\frac{72}{99^s}}\\[1mm]
  &\hphantom{I}\textstyle{\mbox{}\quad\ts +\frac{84}{117^s}
  +\frac{66}{121^s}+\frac{156}{125^s}+\frac{240}{135^s}+\frac{112}{147^s}
  +\frac{108}{153^s}
  +\cdots}
\end{split}
\]
For the determination of the counting function, it was sufficient to
have an explicit expression for $\Gbcc(R^{}_{1},R^{}_{2})$ for prime
power indices.  Nevertheless, we can give an explicit expression for
MCSLs with general index as well, which generalises
Lemma~\ref{csl-lem:mcsl2b}.

\begin{theorem}[{\cite[Thm.~6.2.7]{csl-habil}}]\label{csl-theo:mcsl2}
  Let\/ $q^{}_1$ and $q^{}_{2}$ be primitive odd quaternions and let\/
  $q$ be their least common right multiple. Then, one has\/
  $\Gbcc\bigl(R(q^{}_1),R(q^{}_2)\bigr)= 
  \Imag(q \JJ+q^{}_1\JJ\ts{\bar{q}}^{}_2)= \Imag(q
  \JJ+q^{}_2\JJ\ts{\bar{q}}^{}_1)$.  \qed
\end{theorem}

\subsection{Intersections of three or more CSLs}
\label{csl-sec:mcsl-cub-3}

We can go one step further and analyse MCSLs which are the
intersection of at most three ordinary CSLs. Again, it is sufficient
to consider only MCSLs of prime power index. Also in this case, we get
an explicit expression for the MCSLs as follows.

\begin{theorem}[{\cite[Thm.~6.3.7]{csl-habil}}]\label{csl-theo1:3p}
  Let\/ $q_i$ with\/ $i\in\{1,2,3\}$ be odd primitive quaternions with
  prime power norm\/ $|q_i|^2=p^{\alpha_i}$, such that\/
  $|q^{}_1|^2\geq |q^{}_i|^2$.  Let\/ ${m}_{ij}=\lcrm (q^{}_i,q^{}_j)$
  and\/ ${g}_{ij}=\gcld(q^{}_i,q^{}_j)$. Let\/ $|{m}_{12}|^2\geq
  |{m}_{13}|^2$. Then,
  \[
  \Gbcc\bigl(R(q^{}_1),R(q^{}_2),R(q^{}_3)\bigr)
  \, =\, \Imag\bigl({m}_{12}\JJ+nq^{}_1\JJ\ts{\bar{q}}^{}_2\bigr),
  \]
  where\/
  $n=\max\Bigl(1,\frac{|q^{}_3|^2}{|{g}^{}_{13}|^2|{g}^{}_{23}|^2}\Bigr)$.\qed
\end{theorem}

Note that the expression for the triple CSL in
Theorem~\ref{csl-theo1:3p} is very similar to the expression for the
double CSL in Lemma~\ref{csl-lem:mcsl2b}.  In fact, the only
difference is that an additional factor $n$ occurs. If $n=1$, the
triple CSL is just the intersection of two ordinary CSLs, since one
has the relation
$\Gbcc(R(q^{}_1),R(q^{}_2),R(q^{}_3)) =
\Gbcc(R(q^{}_1),R(q^{}_2))\subseteq\Gbcc(R(q^{}_1),R(q^{}_3))$
in this case. Let us note in passing that this yields a criterion for
$\vG \bigl( R(q^{}_{1}),R(q^{}_{2})\bigr) \subseteq \vG \bigl(
R(q^{}_{1}), R(q^{}_{3})\bigr)$. In particular, under the
assumptions of Theorem~\ref{csl-theo1:3p}, this inclusion holds 
if and only if
\[
    \frac{\lvert q^{}_{3} \rvert^{2}}
    {\lvert g^{}_{13} \rvert^{2} \,
     \lvert g^{}_{23} \rvert^{2}}  
    \, \le \, 1 \ts .
\]
But even if $n>1$, the triple CSL is just a multiple of a double CSL,
as we have the following result.

\begin{theorem}[{\cite[Thm.~6.3.8]{csl-habil}}]\label{csl-theo:p3CSL}
  Let\/ $\vG'$ be a sublattice of\/ $\Gbcc$ of prime power index\/
  $p^\alpha$.  Then, $\vG'$ can be represented as the intersection of
  three ordinary CSLs,
  \[
  \vG' \, =\, \Gbcc(R^{}_1)\cap\Gbcc(R^{}_2)\cap\Gbcc(R^{}_3)\ts ,
  \]
  if and only if there exists\/ $\beta\in \Nnull$ together with two
  coincidence rotations\/ $R'_1$ and\/ $R'_2$ such that\/
  $\vG'=p^\beta(\Gbcc(R'_1)\cap\Gbcc(R'_2))$.  The integer\/ $\beta$ is
  determined uniquely by\/ $\vG'$. \qed
\end{theorem}

Thus, we have established a one-to-one correspondence between
intersections of three ordinary CSLs and multiples of intersections of
two ordinary CSLs.  This allows us to express
$c^{(3)}_{\mathsf{bcc}}(p^r)$ in terms of
$c^{(2)}_{\mathsf{bcc}}(p^r)$ as follows.

\begin{corollary}[{\cite[Cor.~6.3.9]{csl-habil}}]
  Let\/ $p$ be an odd prime number. Then,
  \[
  c^{(3)}_{\mathsf{bcc}}(p^r)\, =
  \sum_{0\leq n\leq r/3}c^{(2)}_{\mathsf{bcc}}(p^{r-3n})\ts ,
  \]
  where\/ $c^{(3)}_{\mathsf{bcc}}(m)$ and\/
  $c^{(2)}_{\mathsf{bcc}}(m)$ denote the number of MCSLs of index\/
  $m$ that can be written as an intersection of $($up to$\ts )$ three
  and two ordinary CSLs, respectively.\qed
\end{corollary}

As $c^{(3)}_{\mathsf{bcc}}$ is once again multiplicative, we can
easily infer its generating function as
follows.\index{multiplicative~function}

\begin{theorem}[{\cite[Thm.~6.3.10]{csl-habil}}]\label{csl-theo:c3gen}
  Let\/ $c^{(3)}_{\mathsf{bcc}}(m)$ be the number of distinct MCSLs of
  index\/ $m$ that are an intersection of at most three ordinary
  CSLs. Then, $c^{(3)}_{\mathsf{bcc}}(m)$ is a multiplicative
  arithmetic function whose Dirichlet series\index{Dirichlet~series}
  is given by
\[
    \Psi^{(3)}_{\mathsf{bcc}}(s)\, :=
    \sum_{n=1}^\infty\frac{c^{(3)}_{\mathsf{bcc}}(n)}{n^s}\, =\, 
    (1-2^{-3s})\, \zeta(3s)\, \Psi^{(2)}_{\mathsf{bcc}}(s)
\]
where\/ $\Psi^{(2)}_{\mathsf{bcc}}(s)$ is given by
Theorem~$\ref{csl-theo:c2gen}$. One finds
\begin{align*}
    \Psi^{(3)}_{\mathsf{bcc}}(s) &\,=\, 
     \textstyle{1+\frac{4}{3^s}+\frac{6}{5^s}+\frac{8}{7^s}+
     \frac{18}{9^s}+\frac{12}{11^s}
      +\frac{14}{13^s}+\frac{24}{15^s}+\frac{18}{17^s}+\frac{20}{19^s}}
      +\frac{32}{21^s}\\[2mm]
    &\textstyle{\mbox{}\qquad\,+\frac{24}{23^s}+\frac{45}{25^s}
      +\frac{77}{27^s} +\frac{30}{29^s}+\frac{32}{31^s}
      +\frac{48}{33^s}+\frac{48}{35^s}+\frac{38}{37^s}+\frac{56}{39^s}
      +\cdots}
\end{align*}
for the leading terms.\qed
\end{theorem}

Familiar arguments involving Theorem~\ref{csl-thm:meanvalues} yield
the following asymptotic behaviour.

\begin{corollary}[{\cite[Cor.~6.3.11]{csl-habil}}]\label{csl-cor:asymp3}
  The asymptotic behaviour\index{asymptotic~behaviour} of the summatory 
  function of\/ $c^{(3)}_{\mathsf{bcc}}(m)$, as\/ $x\to\infty$, is given by
\[
    \sum_{m\leq x}c^{(3)}_{\mathsf{bcc}}(m)
      \, =\, \frac{\rho^{(3)}_{\mathsf{bcc}}}{2}\, x^2 
     \, \approx\,  0.357{\ts}007 \, x^2  ,\vspace{-1mm}  
\]
  where
\begin{align*}
  \qquad\qquad\, \rho^{(3)}_{\mathsf{bcc}}\, :=
  \Res_{s=2}\,\bigl(\Psi^{(3)}_{\mathsf{bcc}}(s)\bigr)\,
    &=\,\myfrac{63}{64}\,\zeta(6)\,\rho^{(2)}_{\mathsf{bcc}}
     \, =\, \myfrac{1953}{5200}\,
      \frac{\zeta(2)\ts\zeta(6)^2\ts\zeta(5)}{\zeta(4)\ts\zeta(12)}\,
       \varphi^{(2)}_{\mathsf{bcc}}(2)\\[1mm]
       & = \,\myfrac{64449}{11056\ts \pi^2}\,
       \zeta(5)\,\varphi^{(2)}_{\mathsf{bcc}}(2)
       \,\approx\, 0.714{\ts}014\ts . \qquad\quad\qquad\qquad\; \qed   
\end{align*}
\end{corollary}

Comparing these results with Corollary~\ref{csl-cor:asymp2}, we see
that the difference in the growth rate is significantly below
1\%. This small difference is not surprising as genuinely triple CSLs
can only occur for indices that are divisible by $p^3$ for some odd
$p$. In particular, the first such lattice occurs for the index
$\Sig=27$. The fact that new MCSLs are rather rare is also illustrated
by the first terms of the expansion
\[
\begin{split}
\Psi^{(3)}_{\mathsf{bcc}}(s)&-\Psi^{(2)}_{\mathsf{bcc}}(s)\,
=\,\Psi^{(2)}_{\mathsf{bcc}}(s)\, \bigl((1-2^{-3s})\ts\zeta(3s)-1\bigr)\\[1mm]
&=\, \textstyle{\frac{1}{27^s}+\frac{4}{81^s}+\frac{1}{125^s}+\frac{6}{135^s}
+\frac{8}{189^s}+\frac{18}{243^s}+\frac{12}{297^s}+\frac{1}{343^s}
+\frac{14}{351^s}
+\cdots}
\end{split}
\]
Here, all terms $n^{-s}$ with $n$ cube-free are missing, which is just
a reformulation of the fact that
$c^{(3)}_{\mathsf{bcc}}(n)=c^{(2)}_{\mathsf{bcc}}(n)$ for these $n$.

Finally, let us mention that any triple CSL is just a multiple of a
double CSL for general index $m$, as we have the following
generalisation of Theorem~\ref{csl-theo:p3CSL}.

\begin{theorem}[{\cite[Thm.~6.3.12]{csl-habil}}]
  Let\/ $R_i$ with\/ $i\in\{1,2,3\}$ be coincidence rotations of\/
  $\Gbcc$. Then, there exist rotations\/ $R^{\prime}_1$ and\/
  $R^{\prime}_{2}$ together with an integer\/ $n\in\NN$ such that\/
  $\Gbcc(R^{}_1,R^{}_2,R^{}_3)=n\Gbcc(R'_1,R'_2)$. Conversely, for any
  sublattice of the form\/ $n\Gbcc(R'_1,R'_2)$, there exist
  coincidence rotations\/ $R_i$ with\/ $i\in\{1,2,3\}$ such that
  one has the coincidence relation
  $\Gbcc(R^{}_1,R^{}_2,R^{}_3)=n\Gbcc(R'_1,R'_2)$.\qed
\end{theorem}

In fact, this yields all MCSLs, as any MCSL of $\Gbcc$ can be written
as the intersection of three ordinary CSLs.

\begin{theorem}[{\cite[Thm.~6.4.3]{csl-habil}}]
  Let\/ $R^{}_1,\ldots, R^{}_n$ be a finite number of coincidence
  rotations of\/ $\Gbcc$. Then, there exist coincidence rotations\/
  $R'_1$, $R'_2$ and $R'_3$ such that one has\/ $\Gbcc(R^{}_1,\ldots,
  R^{}_n)=\Gbcc(R'_1,R'_2,R'_3)$.\qed
\end{theorem}

Consequently, no new MCSLs emerge if we consider intersections of more
than three ordinary CSLs. Hence, the total number of MCSLs of given
index $m$ is already given by $c^{(3)}_{\mathsf{bcc}}(m)$, which means
that, for all $n\geq 3$, we have
\[
  c^{(\infty)}_{\mathsf{bcc}}(m)\, =\, c^{(n)}_{\mathsf{bcc}}(m)
  \, =\, c^{(3)}_{\mathsf{bcc}}(m)\ts .
\]
A similar phenomenon has been observed in two dimensions, where the
set of MCSLs stabilises already for $n=2$; compare
Section~\ref{csl-sec:csl-nfold} and~\cite{csl-BG}.

So far, we have only discussed the body-centred cubic
lattice. However, we know from the ordinary CSLs that all three types
of cubic lattices have the same group of coincidence rotations, the
same spectrum of indices and the same multiplicity function. In fact,
this remains true in the case of MCSLs, too;
compare~\cite[Thms.~6.5.2, 6.5.4 and~6.5.5]{csl-habil}.

\begin{theorem}
  Let\/ $R^{}_1,\ldots, R^{}_n$ be coincidence rotations for the cubic
  lattices in the setting of Eq.~\eqref{csl-cubiclat}.  Then,
  \[
  \Sig_{\mathsf{pc}}^{}(R^{}_1,\ldots , R^{}_n)\, =\, 
  \Sig_{\mathsf{bcc}}^{}(R^{}_1,\ldots , R^{}_n)\, =\, 
  \Sig_{\mathsf{fcc}}^{}(R^{}_1,\ldots , R^{}_n)\ts .
  \]
  Moreover, \vspace{-2mm}
  \begin{align*}
    \Gpc(R^{}_1,\ldots,R^{}_n)&\, = \,
    \Gbcc(R^{}_1,\ldots,R^{}_n)\cap\Gpc\quad\text{and}\\[2mm]
    \Gfcc(R^{}_1,\ldots,R^{}_n)&\, = \,
    \Gbcc(R^{}_1,\ldots,R^{}_n)\cap\Gfcc
  \end{align*}
  relate the CSLs of the different cubic lattices. \qed
\end{theorem}

This result implies that the counting functions are equal for all three cubic
lattices, too. In particular, we have
\begin{align}
  c^{(n)}_{\mathsf{pc}}(m)\, =\, c^{(n)}_{\mathsf{fcc}}(m)\, =\,
   c^{(n)}_{\mathsf{bcc}}(m)
\end{align}
for any $n\in\NN\cup\{\infty\}$, and the corresponding generating
functions are equal as well.

Finally, let us mention an application to
crystallography.\index{crystallography} One object of interest to
crystallographers are so-called triple
junctions~\cite{csl-gerts2,csl-gerts1,csl-gerts3}.  Roughly speaking,
\emph{triple junctions}\index{triple~junction} are three crystal
grains meeting in a straight line. This means that there are three
pairs of grains sharing a common plane (grain boundary). They give
rise to three simple CSLs and to a double CSL, which is the
intersection of the former. In our terms, the latter is an MCSL
$\vG\cap R^{}_1\vG\cap R^{}_2\vG$, whereas the former are the simple
CSLs $\vG\cap R^{}_1\vG$, $\vG\cap R^{}_2\vG$ and $R^{}_1\vG\cap
R^{}_2\vG$, respectively. An important question is the relation of the
indices of these lattices.

Let us denote the indices of the simple CSLs by
$\Sig^{}_{i}:=\Sig(R^{}_{i})$, where $R^{}_3:=R_1^{-1}R^{}_2$. Let
$q^{}_{1}$ and $q^{}_{2}$ be the quaternions that parametrise
$R^{}_{1}$ and $R^{}_{2}$, respectively.  Then, $R^{}_{3}$ is
generated by ${\bar{q}}^{}_{1}q^{}_{2}$, which is not a primitive
quaternion in general. The corresponding primitive quaternion is given
by $q^{}_3:=\frac{{\bar{q}}_1q^{}_2}{|q^{}_{12}|^2}$, where
$q^{}_{12}=\gcld(q^{}_1,q^{}_2)$.  Hence, we can immediately reproduce
Gertsman's result~\cite{csl-gerts1} for the index
$\Sig^{}_3=\frac{\Sig^{}_1\Sig^{}_2}{\Sig_{12}^2}$, where
$\Sig^{}_{12}:=\Sig(R(q^{}_{12}))$ is the index that corresponds to
the rotation $R(q^{}_{12})$.  On the other hand, we know from
Lemma~\ref{csl-lem:sig12} and Eq.~\eqref{csl-eq:sig12c} that
\[
  \Sig(R^{}_1,R^{}_2)\, =\, \frac{\Sig^{}_1\Sig^{}_2}{\Sig^{}_{12}}
  \, =\, \Sig^{}_{12}\, \Sig^{}_3 \ts .
\]
Now, we define $q_1':=q_{12}^{-1}q^{}_1$ and
$q_2':=q_{12}^{-1}q^{}_2$. Then, we may write
\begin{equation}\label{csl-eq:qqq}
  q^{}_1\,=\, q^{}_{12}\ts q_1'\ts ,\quad
  q^{}_2\, =\, q^{}_{12}\ts q_2'\ts ,\quad
  q^{}_3\, =\, {\bar{q}}_1^{\ts\prime}q_2'
\end{equation}
and, correspondingly, we may also decompose the rotations
$R^{}_{i}$ into the `basic' constituents
$R^{}_{12}:=R(q^{}_{12})$, $R_1':=R(q_1')$ and $R_2':=R(q_2')$. We
note that the corresponding indices are multiplicative,
\[
\Sig(R^{}_1)\, =\, \Sig(R^{}_{12})\Sig(R_1')\ts , \;
\Sig(R^{}_2)\, =\, \Sig(R^{}_{12})\Sig(R_2')\ts , \;
\Sig(R^{}_3)\, =\, \Sig(R_1')\Sig(R_2')\ts .
\]
Furthermore, we see
${\bar{q}}_1'=\gcld({\bar{q}}_1,q^{}_3)=:q^{}_{13}$ and
${\bar{q}}_2'=\gcld({\bar{q}}_2,{\bar{q}}_3)=:q^{}_{23}$, whence
Eq.~\eqref{csl-eq:qqq} may be written in a more symmetric way as
\[
q^{}_1\, =\, q^{}_{12}{\bar{q}}^{}_{13}\, =\, 
\frac{q^{}_2{\bar{q}}_3}{|q^{}_{23}|^2}\ts , \quad
q^{}_2\, =\, q^{}_{12}{\bar{q}}^{}_{23}\, =\, 
\frac{q^{}_1 q^{}_3}{|q^{}_{13}|^2}\ts , \quad
q^{}_3\, =\, {\bar{q}}^{}_{13}{\bar{q}}^{}_{23}\, =\, 
\frac{{\bar{q}}_1 q^{}_2}{|q^{}_{12}|^2}\ts .
\]
If we define the corresponding indices in the obvious way, we see
that the index $\Sig(R^{}_1,R^{}_2)$ can be written as
\begin{align*}
\Sig(R^{}_1,R^{}_2)\, &=\, \frac{\Sig^{}_1\Sig^{}_2}{\Sig^{}_{12}}
\, =\, \frac{\Sig^{}_1\Sig^{}_3}{\Sig^{}_{13}}
\, =\, \frac{\Sig^{}_2\Sig^{}_3}{\Sig^{}_{23}}
\, =\, \Sig^{}_{12}\Sig^{}_3=\Sig^{}_{13}\Sig^{}_2=\Sig^{}_{23}\Sig^{}_1\\[1mm]
&=\, \Sig^{}_{12}\Sig^{}_{13}\Sig^{}_{23}
\, =\, \Sig^{}_{12}\Sig_1'\Sig_2'\, =\, 
\bigl(\Sig^{}_1\Sig^{}_2\Sig^{}_3\bigr)^{\!\frac{1}{2}}.
\end{align*}
The last expression was proved by different methods
in~\cite{csl-gerts1}.  Note that we can express $\Sig(R^{}_1,R^{}_2)$
either in terms of the simple indices $\Sig^{}_1,\Sig^{}_2,\Sig^{}_3$
or in terms of the `reduced' indices
$\Sig^{}_{12},\Sig^{}_{13},\Sig^{}_{23}$, which somehow describe the
`common' part of $R^{}_1$, $R^{}_2$ and $R^{}_3$.  Note that
$R^{}_{12}$, $R^{}_{13}$ and $R^{}_{23}$ contain the complete
information about the triple junction. In particular, we can write
$\vG(R^{}_1,R^{}_2)$ as $\vG(R^{}_1,R^{}_2)=
R^{}_{12}(R_{12}^{-1}\vG\cap R_{13}^{-1}\vG\cap R_{23}^{-1}\vG)$.\smallskip

As we have now solved the problem of MCSLs of the cubic lattices, it
is natural to ask whether these results can be
generalised. Unfortunately, not much is known about MCSLs in
dimensions $d>3$, not even for the $A_4$-lattice or the hypercubic
lattices. This is not too surprising in view of the fact that the
computation of MCSLs is substantially more difficult than the
determination of ordinary CSLs. Indeed, even for the $4$-dimensional
root lattices, no explicit expression for the MCSLs is known, which
makes the corresponding enumeration problem intractable along the
explicit route we have taken above.

Still, there are several interesting questions to address. A striking
feature of our examples is the stabilisation property of the
coincidence spectra and of the MCSLs themselves.  For the
planar lattices and modules of Section~\ref{csl-sec:csl-nfold}, any
MCSM can be represented as the intersection of at most two ordinary
CSMs, whereas for the cubic lattices up to three ordinary CSLs are
needed. One might suspect that, in dimension $d$, any MCSL can be
written as the intersection of at most $d$ ordinary CSLs, but this
seems too difficult to decide at the moment.

A somewhat easier problem is the stabilisation phenomenon of the
coincidence spectra. For the cubic lattices as well as for the planar
lattices and modules of Section~\ref{csl-sec:csl-nfold}, we have
$\sigma_\infty(\vG)=\sigma(\vG)$; compare
Proposition~\ref{csl-prop:spectrum} and
Eq.~\eqref{csl-eq:specstable}. Similarly, we have
$\sigma_\infty(\vG)=\sigma(\vG)$ for the lattices $A_4$, $D_4^{*}$ and
$\ZZ^4$. For $A_4$ and $D_4^{*}$, this follows immediately from
$\sigma(\vG)=\widehat{\sigma}(\vG)$ and
Eq.~\eqref{csl-eq:dieletzte}. For $\ZZ^4$, one has to argue
differently, as $\sigma(\vG)\neq\widehat{\sigma}(\vG)$. Here, index
considerations similar to those in Eq.~\eqref{csl-eq:chain} do the
job.

There are two further (somewhat extremal) situations where we can
prove stabilisation. If the simple spectrum is a finite set, which is
equivalent to the finiteness of the set of CSLs, the coincidence
spectra must stabilise after a finite number of steps, as the set of
all MCSLs is finite as well.  This happens for the rather large class
of planar lattices that have exactly two coincidence reflections;
compare~\cite{csl-BSZ-well}. The second situation is the case
$\sigma(\vG)=\NN$, where we obviously have
$\sigma_\infty(\vG)=\sigma(\vG)$. This happens for $\vG=\ZZ^d$ for
$d\geq 5$, as we shall see below.

\section{Results in higher dimensions}
\label{csl-sec:higher}

For dimensions $d\geq 4$, not much is known about CSLs in general, let
alone CSMs.  However, if $\vG$ is
\emph{rational},\index{lattice!rational} we have some results on the
possible indices.  In this case, the group $\OC(\vG)$ is generated by
coincidence reflections.  To be more concrete, let $R_{v}\!:\,
\RR^d\longrightarrow\RR^d$, $x \mapsto
x-2\frac{\ip{v}{x}}{\ip{v}{v}}v$ denote the reflection in the plane
perpendicular to $v\in\RR^d$. As a first result, we mention the
following characterisation of rational lattices.

\begin{theorem}[{\cite[Thm.~3.2 and Cor.~3.3]{csl-zou1}}]\label{csl-thm:zou1}
  A lattice\index{lattice!rational} \/ $\vG\subset \RR^d$ is
  ra\-tional{\ts\ts}\footnote{{\ts}meaning rational in the wider
    sense; compare Footnote~\ref{csl-foot:rat-lat} on
    page~\pageref{csl-foot:rat-lat}.}  if and only if any reflection\/
  $R_{v}$ with\/ $v\in\vG\setminus\{0\}$ is a coincidence reflection.
  \qed
\end{theorem}

As we have plenty of coincidence reflections for rational lattices, it
is not a surprise that they generate the group $\OC(\vG)$. In
particular, we have the following analogue of the classic
Cartan--Dieudonn\'e theorem (see \cite{csl-Cartan,csl-Garling}) for
coincidence isometries.

\begin{theorem}[{\cite[Thm.~3.1 and
    Thm.~3.5]{csl-zou1}}]\label{csl-thm:zou2}
  Let\/ $\vG\subset \RR^d$ be a rational lattice $($in the wider
  sense$\ts )$.  Then, any coincidence isometry of\/ $\vG$ is a
  product of at most\/ $d$ coincidence reflections generated by
  lattice vectors of\/ $\vG$.  \qed
\end{theorem}

Theorem~\ref{csl-thm:zou2} allows us to determine the coincidence
spectrum for some rational lattices.\index{lattice!rational} As an
example, we consider~$\vG=\ZZ^d$. In this case, $\OC(\ZZ^d)$ is
generated by the reflections $R_{v}$, where $v$ runs through all
non-zero primitive lattice vectors. The coincidence index
$\Sig(R_{v})$ can be calculated explicitly and, for primitive~$v$, is
given by \cite[Thm.~3.2]{csl-zou2} as
\begin{equation}\label{csl-eq:zd-refl-index}
  \Sig(R_{v}) \, =\, 
  \begin{cases}
    \ip{v}{v}, & \mbox{if $\ip{v}{v}$ is odd,}\\
    \frac{1}{2}\ip{v}{v}, & \mbox{if $\ip{v}{v}$ is even.}
  \end{cases}
\end{equation}
As any positive integer $n$ can be written as the sum of four squares,
there exists a primitive vector $v\in\ZZ^d$ with $\ip{v}{v}=2n$ for
$d\geq 5$ (choose one of the components to be $1$, which guarantees
the primitivity, and adjust the other components to get length
$\ip{v}{v}=2n$). Hence, in $\ZZ^d$ with $d\geq 5$, all positive
integers occur as a coincidence index of some reflection, which gives
us the coincidence spectrum; compare~\cite{csl-zou2}.

\begin{fact}
  The coincidence spectrum of\/ $\ZZ^d$ for\/ $d\geq 5$ is\/
  $\NN$. \qed
\end{fact}

\begin{remark}
  Previously, we have seen that the coincidence spectrum of $\ZZ^d$ is
  a proper subset of $\NN$ for $2\leq d \leq 4$; compare
  Example~\ref{csl-ex:mcsl-square} and Remarks~\ref{csl-rem:spec-cub}
  and~\ref{csl-rem:spec-z4}.  Although Theorem~\ref{csl-thm:zou2}
  guarantees that the coincidence reflections generate $\OC(\ZZ^d)$,
  it is not evident whether they yield the whole coincidence
  spectrum. But, in fact, this is indeed the case. Moreover, for
  $2\leq d \leq 4$, it follows that $n$ is the index of a coincidence
  reflection if and only if $n$ is the index of a coincidence
  rotation.

This is obvious for $d=2$, as there is an index-preserving bijection
between coincidence reflections and coincidence rotations (observe
that any reflection is the product of complex conjugation with a
rotation).

For $d=3$ or $d=4$, there is no such bijection. Nevertheless, we get
the possible indices for coincidence reflections by evaluating
Eq.~\eqref{csl-eq:zd-refl-index}. This is straightforward for $d=4$,
where we conclude that exactly all odd positive integers and all
positive integers of the form $4n+2$ occur as coincidence indices for
some coincidence reflection. These are the same indices
 we found for the coincidence rotations of $\ZZ^4$ in
Section~\ref{csl-sec:cub4}; compare Remark~\ref{csl-rem:spec-z4}.

For $d=3$, evaluating Eq.~\eqref{csl-eq:zd-refl-index} is more
difficult.  Recall that any integer that is not of the form $4^k m$
with $m \equiv 7\bmod 8$ can be written as the sum of three
squares. Hence, for any odd $n$, there exists a vector $v\in\ZZ^3$
such that $\ip{v}{v}=2n$.  In fact, there even exists a primitive $v$,
since a positive integer $m\not\equiv 0 \bmod{4}$ can be represented
as a sum of three integers if and only if it has a primitive
representation; see~\cite{csl-CoHi} for an explicit formula for the
number of primitive representations. Thus, there is a coincidence
reflection of index $n$ for any positive odd $n$. Recall from
Section~\ref{csl-sec:cubic} that there are coincidence rotations of
index $n$ for all positive odd $n$ as well.  \exend
\end{remark}

Theorem~\ref{csl-thm:zou1} can be generalised to $\cS$-lattices, as
its proof is algebraic in nature. The analogue of a rational lattice
can be characterised as follows.\index{S-lattice@$\cS$-lattice}

\begin{theorem}[{\cite[Thm.~3.2]{csl-Huck09}}]\label{csl-thm:huck1}
  Let\/ $M\subseteq \RR^d$ be an\/ $\cS$-lattice, and let\/ $K$ be the
  field of fractions of\/ $\cS$. Then, the following properties are
  equivalent.
  \begin{enumerate}\itemsep=2pt
    \item For all\/ $u,v\in M$ and\/ $w\in M\setminus\{0\}$, we have\/ 
      $\frac{\ip{u}{v}}{\ip{w}{w}}\in K$;
    \item $R_{v}$ is a coincidence reflection for any\/ $v\in
      M\setminus\{0\}$.  \qed
  \end{enumerate}
\end{theorem}
For any $\cS$-lattice that satisfies the properties of
Theorem~\ref{csl-thm:huck1}, we have the following generalisation of
Theorem~\ref{csl-thm:zou2}, which again is an analogue of the
Cartan--Dieudonn\'e theorem.

\begin{theorem}[{\cite[Thm.~3.1]{csl-Huck09}}]
  Let\/ $M\subseteq \RR^d$ be an\/ $\cS$-lattice, and let\/ $K$ be the
  field of fractions of\/ $\cS$. Let\/ $M$ satisfy the conditions of
  Theorem~$\ref{csl-thm:huck1}$. Then, any coincidence isometry of\/
  $M$ can be written as the product of at most\/ $d$ coincidence
  reflections generated by non-zero vectors of\/ $M$.  \qed
\end{theorem}

To get more concrete results in dimensions $d\geq 5$, it would be nice
to have an explicit parametrisation for the coincidence
isometries. For dimensions \mbox{$d=3$} and \mbox{$d=4$}, we profitted
from the parametrisation of rotation by quaternions. An obvious
candidate for higher dimensions is Cayley's
parametrisation\index{rotation!Cayley~parametrisation} of rotations
in terms of Clifford algebras. At present, however, we are not aware
of any concrete results in this direction for $d\geq 5$.\medskip

\noindent\textit{Acknowledgements}.
It is a pleasure to thank J.~Br\"{u}dern, U.~Grimm, C.~Huck,
R.V.~Moody, U.~Rehmann, R.~Scharlau and C.~Voll for cooperation and
helpful comments.


\begin{thebibliography}{AO\ts 1}           

\itemsep=2pt
      
\bibitem{csl-akh-sal10a} 
  Akhtarkavan E.\ and Salleh M.F.M.\ (2010).
  Multiple description lattice vector quantization using
  multiple ${A}_4$ quantizers, 
  \textit{IEICE Electron.\ Express} \textbf{7}, 1233--1239.
 
\bibitem{csl-akh-sal12} 
  Akhtarkavan E.\ and Salleh M.F.M.\ (2012).
  Multiple descriptions coinciding lattice vector
  quantizer for wavelet image coding,
  \textit{IEEE Trans.\ Image Processing} 
  \textbf{21}, 653--661.

\bibitem{csl-Apostol}
  Apostol T.M.\ (1984).
  \textit{Introduction to Analytic Number Theory}
  (Springer, New York); 5th corr.\ printing (1998).

\bibitem{csl-Baake-rev}
  Baake M.\ (1997).
  Solution of the coincidence problem in dimensions $d\le 4$.
  In \textit{The Mathematics of Long-Range Aperiodic Order},
  Moody R.V.\ (ed.), NATO ASI Series C 489,
  pp.~9--44 (Kluwer, Dordrecht);
  rev.\ version: \texttt{arXiv:math.MG/0605222}.

\bibitem{csl-BF05}
  Baake M.\ and Frettl\"oh D.\ (2005).
  SCD patterns have singular diffraction,
  \textit{J. Math. Phys} \textbf{46}, 033510: 1--10. 
  \texttt{arXiv:math-ph/0411052}.

\bibitem{csl-BG03}
  Baake M.\ and Grimm U.\ (2003).
  A note on shelling,
  \textit{Discr.\ Comput.\ Geom.} \textbf{30}, 573--589.\newline
  \texttt{arXiv:math.MG/0203025}.

\bibitem{csl-BG2}
  Baake M.\ and Grimm U.\ (2004).
  Bravais colourings of planar modules with $N$-fold symmetry,
  \textit{Z.\ Krist.} \textbf{219}, 72--80.
  \texttt{arXiv:math.CO/0301021}.         

\bibitem{csl-BG}
  Baake M.\ and Grimm U.\ (2006).
  Multiple planar coincidences with $N$-fold symmetry,
  \textit{Z.\ Krist.} \textbf{221}, 571--581.
  \texttt{arXiv:math.MG/0511306}.   

\bibitem{csl-TAO}
  Baake M.\ and Grimm U.\ (2013).
  \textit{Aperiodic Order. Vol.~1:\ A Mathematical Invitation}
  (Cambridge University Press, Cambridge). 

\bibitem{csl-BG00}
  Baake M., Grimm U., Joseph D.\ and Repetowicz P.\ (2000). 
  Averaged shelling for quasicrystals,
  \textit{Mat.\ Sci.\ Eng.\ A} \textbf{294--296}, 441--445.
  \texttt{arXiv:math.MG/9907156}.

\bibitem{csl-BGHZ08}
  Baake M., Heuer M., Grimm U.\ and Zeiner P.\ (2008).
  Coincidence rotations of the root lattice ${A}_4$,
  \textit{European J.\ Combin.} \textbf{29}, 1808--1819.
  \texttt{arXiv:0709.1341}.

\bibitem{csl-BHM}
  Baake M., Heuer M.\ and Moody R.V.\ (2008).
  Similar sublattices of the root lattice ${A}_4$,
  \textit{J.\ Algebra} \textbf{320}, 1391--1408.
  \texttt{arXiv:math.MG/0702448}.   

\bibitem{csl-BM1}
  Baake M.\ and Moody R.V.\ (1998).
  Similarity submodules and semigroups. In
  \textit{Quasicrystals and Discrete Geometry},
  Patera J.\ (ed.), Fields Institute Monographs,
  vol.\ 10, pp.~1--13 (AMS, Providence, RI).

\bibitem{csl-BM}
   Baake M.\ and Moody R.V.\ (1999).
   Similarity submodules and root systems in four dimensions,
   \textit{Can.\ J.\ Math.} \textbf{51}, 1258--1276.
   \texttt{arXiv:math.MG/9904028}.         

\bibitem{csl-BLP96}
  Baake M., Pleasants P.A.B.\ and Rehmann U.\ (2007).
  Coincidence site modules in $3$-space,
  \textit{Discr.\ Comput.\ Geom.} \textbf{38}, 111--138.
  \texttt{arXiv:math.MG/0609793}.

\bibitem{csl-BSZ-sim}
  Baake M., Scharlau R.\ and Zeiner P.\ (2011).
  Similar sublattices of planar lattices,
  \textit{Can.\ J.\ Math.} \textbf{63}, 1220--1237.
  \texttt{arXiv:0908.2558}.

\bibitem{csl-BSZ-well}
   Baake M., Scharlau R.\ and Zeiner P.\ (2014).
   Well-rounded sublattices of planar lattices,
   \textit{Acta Arithm.} \textbf{166.4}, 301--334.
   \texttt{arXiv:1311.6306}.

\bibitem{csl-BZ07} 
  Baake M.\ and Zeiner P.\ (2007).
  Multiple coincidences in dimensions $d\leq 3$,
  \textit{Philos.\ Mag.} \textbf{87}, 2869--2876.

\bibitem{csl-BZ08} 
  Baake M.\ and Zeiner P.\ (2008).
  Coincidences in 4 dimensions,
  \textit{Philos.\ Mag.} \textbf{88}, 2025--2032.\newline
  \texttt{arXiv:0712.0363}.

\bibitem{csl-Boll70}
   Bollmann W.\ (1970).
   \textit{Crystal Defects and Crystalline Interfaces}
   (Springer, Berlin).

\bibitem{csl-Boll82}
   Bollmann W.\ (1982).
   \textit{Crystal Lattices, Interfaces, Matrices}
   (\copyright W.\ Bollmann, Geneva).

\bibitem{csl-Borevic-Sh} 
  Borevich I.\ and Shafarevich I.\ (1966).
  \textit{Number Theory}
  (Academic Press, New York).

\bibitem{csl-Cartan}
  Cartan \'{E}.\ (1981).
  \textit{The Theory of Spinors}, reprint
  (Dover, New York).
  
\bibitem{csl-Cassels}
   Cassels J.W.S.\ (1971).
   \textit{An Introduction to the Geometry of Numbers},
   2nd corr.\ printing (Springer, Berlin).

\bibitem{csl-Cassels-Q}
   Cassels J.W.S.\ (1971).
   \textit{Rational Quadratic Forms}
   (Academic Press, London).

\bibitem{csl-consloa99}
    Conway J.H., Rains E.M.\ and Sloane N.J.A.\ (1999).
    On the existence of similar sublattices,
    \textit{Can.\ J.\ Math.} \textbf{51}, 1300--1306.
    \texttt{arXiv:math.CO/0207177}.

\bibitem{csl-Conway}
   Conway J.H.\ and Sloane N.J.A.\ (1999).
   \textit{Sphere Packings, Lattices and Groups}, 
   3rd ed.\ (Springer, New York).

\bibitem{csl-CS03}
   Conway J.H.\ and Smith D.A.\ (2003).
   \textit{On Quaternions and Octonions:\ Their Geometry, 
   Arithmetic,  and Symmetry}
   (A.K.\ Peters,  Wellesley, MA).

\bibitem{csl-CoHi}
    Cooper S.\ and Hirschhorn M.\ (2007).
    On the number of primitive representations of integers
    as sums of squares,
    \textit{Ramanujan J.} \textbf{13}, 7--25.

\bibitem{csl-Cox}
   Cox D.A.\ (2013).
   \textit{Primes of the Form $x^2 + n\ts y^2$},
   2nd ed.\ (Wiley, Hoboken, NJ).

\bibitem{csl-Coxeter}
   Coxeter H.S.M.\ (1973).
   \textit{Regular Polytopes}, 3rd ed.\ (Dover, New York).

\bibitem{csl-Urban}
   Dai M.X.\ and Urban K.\ (1993).
   Twins in icosahedral \textsf{Al-Cu-Fe},
   \textit{Philos.\ Mag.\ Lett.} \textbf{67},  67--71.

\bibitem{csl-Dan95}
   Danzer L.\ (1995).
   A family of $3D$-spacefillers not permitting any periodic or
   quasi\-periodic tiling.
   In \textit{Aperiodic '94}, Chapuis G.\ and Paciorek W.\ (eds.),
   pp.~11--17 (World Scientific, Singapore).

\bibitem{csl-Slo02a}
   Diggavi S.N., Sloane N.J.A.\ and Vaishampayan V.A.\ (2002).
   Asymmetric multiple description lattice vector quantizers, 
   \textit{IEEE Trans.\ Inf.\ Theory} \textbf{48}, 174--191.

\bibitem{csl-MagnusD}
   D\"umke M.\ (2011).
   \textit{Koinzidenzgitter von Ordnungen imagin\"ar quadratischer
     Zahl\-k\"orper}, Diploma thesis (Bielefeld University).

\bibitem{csl-duVal}
   du Val P.\ (1964).
   \textit{Homographies, Quaternions and Rotations}
   (Clarendon Press, Oxford).

\bibitem{csl-JonasF}
  Freiberger J.\ (2008).
  \textit{Koinzidenzgitter von Rechteckgittern}, Diploma thesis
  (Bielefeld University). 

\bibitem{csl-Friedel11}
   Friedel G.\ (1911).
    \textit{Le\c{c}ons de {C}ristallographie} (Hermann, Paris).

\bibitem{csl-Garling}
  Garling D.J.H.\ (2011).
  \textit{Clifford Algebras:\ An Introduction}
  (Cambridge University Press, Cambridge).
    
\bibitem{csl-gerts1} 
   Gertsman V.Y.\ (2001).
   Geometrical theory of triple junctions of {CSL} boundaries, 
   \textit{Acta Cryst.\ A} \textbf{57}, 369--377.

\bibitem{csl-gerts2} 
   Gertsman V.Y.\ (2001).
   Coincidence site lattice theory of multicrystalline ensembles, 
   \textit{Acta Cryst.\ A} \textbf{57}, 649--655.

\bibitem{csl-gerts3} 
   Gertsman V.Y.\ (2002).
   On the auxiliary lattices and dislocation reactions at triple 
   junctions, \textit{Acta Cryst.\ A} \textbf{58}, 155--161.

\bibitem{csl-SvenjaDiss}
  Glied S.\ (2010).
  \textit{Coincidence and Similarity Isometries 
   of Modules in Euclidean Space},
  PhD thesis (Bielefeld University). 

\bibitem{csl-svenja2}
  Glied S.\ (2011).
  Similarity and coincidence isometries for modules,
  \textit{Can.\ Math.\ Bull.} \textbf{55}, 98--107.
  \texttt{arXiv:1005.5237}.

\bibitem{csl-svenja1}
  Glied S.\  and Baake M.\ (2008).
  Similarity versus coincidence rotations of lattices,
  \textit{Z.\ Krist.} \textbf{223}, 770--772.
  \texttt{arXiv:0808.0109}.

\bibitem{csl-Grimmer73}
  Grimmer H.\ (1973).
  Coincidence rotations for cubic lattices,
  \textit{Scripta Met.} \textbf{7}, 1295--1300.

\bibitem{csl-Grimmer74}
  Grimmer H.\ (1974).
  Disorientations and coincidence rotations for cubic lattices,
  \textit{Acta Cryst.\ A} \textbf{30},  685--688.

\bibitem{csl-Grimmer84}
  Grimmer H.\ (1984).
  The generating function for coincidence site lattices
  in the cubic system,
  \textit{Acta Cryst.\ A} \textbf{40}, 108--112.

\bibitem{csl-Grimmer89} 
  Grimmer H.\ (1989).
  Systematic determination of coincidence orientations
  for all hexagonal lattices with axial ratio $c/a$
  in a given interval,
  \textit{Acta Cryst.\ A} \textbf{45}, 320--325.

\bibitem{csl-GBW}
  Grimmer H., Bollmann W.\ and Warrington D.H.\ (1974).
  Coincidence-site lattices and complete pattern-shift lattices
  in cubic crystals,
  \textit{Acta Cryst.\ A} \textbf{30}, 197--207.

\bibitem{csl-GW87} 
  Grimmer H.\  and Warrington D.H.\ (1987).
  Fundamentals for the description of hexagonal lattices
  in general and in coincidence orientation,
  \textit{Acta Cryst.\ A} \textbf{43}, 232--243.
       
\bibitem{csl-Grosswald}
  Grosswald E.\ (1985).
  \textit{Representations of Integers as Sums of Squares}
  (Springer, New York).

\bibitem{csl-Gruber}
  Gruber B.\ (1997).
  Alternative formulae for the number of sublattices,
  \textit{Acta Cryst.\ A} \textbf{53}, 807--808.       

\bibitem{csl-Lekker}
  Gruber P.M.\ and Lekkerkerker C.G.\ (1987).
  \textit{Geometry of Numbers}, 2nd ed.\
  (North-Holland, Amsterdam).

\bibitem{csl-Hardy}
  Hardy G.H.\ and Wright E.M.\ (2008).
  \textit{An Introduction to the Theory of Numbers},
  6th ed.\ (Oxford University Press, Oxford).

\bibitem{csl-HZ10}
  Heuer M.\ and Zeiner P.\ (2010).
  CSLs of the root lattice $A_4$,
  \textit{J.\ Phys.:\ Conf.\ Ser.} \textbf{226}, 
  012024:\ 1--6. \texttt{arXiv:1301.2001}. 

\bibitem{csl-Huck09}
  Huck C.\ (2009).
  A note on coincidence isometries of modules in Euclidean space,
  \textit{Z.\ Krist.} \textbf{224}, 341--344.
  \texttt{arXiv:0811.3551}.

\bibitem{csl-Humphreys}
  Humphreys J.E.\ (1992).
  \textit{Reflection Groups and Coxeter Groups}, 
  2nd corr.\ printing (Cambridge University Press, Cambridge).

\bibitem{csl-Hurwitz}
  Hurwitz A.\ (1919).
  \textit{Vorlesungen \"uber die Zahlentheorie der Quaternionen}
  (Sprin\-ger, Berlin).

\bibitem{csl-Rost}
  Knus M.-A., Merkurjev~A., Rost~M.\ and Tignol~J.-P.\ (1998).
  \textit{The Book of Involutions}
  (AMS, Providence, RI).

\bibitem{csl-Koecher}
  Koecher M.\ and Remmert R.\ (1991).
  Hamilton's quaternions. In \textit{Numbers},
  Ebbinghaus, H.-D., et al.\ (eds.), GTM 123,
  pp.\ 189--220 (Springer, New York).

\bibitem{csl-KW49}
  Kronberg M.L.\ and Wilson F.H.\ (1949).
  Secondary recrystallization in copper,
  \textit{Trans.\ AIME} \textbf{185}, 501--514.

\bibitem{csl-Loquias10}
 Loquias M.J.C.\ (2010). 
 \textit{Coincidences and Colorings of Lattices and $\ZZ\ts$-modules}, 
 PhD thesis (Bielefeld University). 

\bibitem{csl-LZ10}
  Loquias M.J.C.\ and Zeiner P.\ (2010). 
  Coincidence isometries of a shifted square lattice, 
  \textit{J.\ Phys.:\ Conf.\ Ser.} \textbf{226}, 
  012026:\ 1--10. \texttt{arXiv:1002.0519}.

\bibitem{csl-LZ9}
  Loquias M.J.\ and Zeiner\ P.\ (2011).
  Colourings of lattices and coincidence site lattices,
  \textit{Philos.\ Mag.} \textbf{91}, 2680--2689.
  \texttt{arXiv:1011.1001}.

\bibitem{csl-LZ14}
  Loquias M.J.\ and Zeiner P.\ (2014).
  The coincidence problem for shifted lattices and
  crystallographic point packings, 
  \textit{Acta Cryst.~A} \textbf{70}, 656--669.
  \texttt{arXiv:1301.3689}.

\bibitem{csl-LZ15}
  Loquias M.J.\ and Zeiner P.\ (2015).
  Coincidence indices of sublattices and coincidences of colorings,
  \textit{Z.\ Krist.} \textbf{230}, 749--759. \texttt{arXiv:1506.00028}.

\bibitem{csl-Lueck}
  L\"uck R.\ (1979).
  Pythagoreische Zahlen f\"ur den dreidimensionalen Raum,
  \textit{Phys.\ Bl\"atter} \textbf{35}, 72--75.

\bibitem{csl-Moody}
  Moody R.V.\ and Patera J.\ (1993).
  Quasicrystals and icosians,
  \textit{J.\ Phys.\ A:\ Math.\ Gen.} \textbf{26}, 2829--2853.

\bibitem{csl-Moody94}
  Moody R.V.\ and Weiss A.\ (1994).
  On shelling $E_8$ quasicrystals,
  \textit{J.~Number Theory} \textbf{47}, 405--412.

\bibitem{csl-Patera}
  Patera J.\ (1997).
  Non-crystallographic root systems and quasicrystals.
  In \textit{The Mathematics of Long-Range Aperiodic Order},
  Moody, R.V.\ (ed.), NATO ASI Series C 489, pp.\ 443--465
  (Kluwer, Dordrecht).

\bibitem{csl-Pleasants}
  Pleasants P.A.B., Baake M.\ and Roth J.\ (1996).
  Planar coincidences for $N$-fold symmetry,
  \textit{J.\ Math.\ Phys.} \textbf{37}, 1029--1058;
  rev.\ version: \texttt{arXiv:math.MG/0511147}.

\bibitem{csl-Radu95}
  Radulescu O.\ (1995).
  An elementary approach to the crystallography of twins in
  icosahedral quasicrystals, 
  \textit{J.~Phys.\ I (France)} \textbf{5}, 719--728.

\bibitem{csl-Radu}
  Radulescu O.\ and Warrington D.H.\ (1995).
  Arithmetic properties of module directions in quasicrystals,
  coincidence modules and coincidence quasilattices,
  \textit{Acta Cryst.\ A} \textbf{51}, 335--343.

\bibitem{csl-Ranga66}
  Ranganathan S.\ (1966).
  On the geometry of coincidence-site lattices,
  \textit{Acta Cryst.} \textbf{21}, 197--199. 

\bibitem{csl-Ranga90}
  Ranganathan S.\ (1990).
  Coincidence-site lattices, superlattices and quasicrystals,
  \textit{Trans.\ Indian Inst.\ Met.} \textbf{43}, 1--7. 

\bibitem{csl-Reiner}
  Reiner I.\ (2003).
  \textit{Maximal Orders}, 
  reprint (Clarendon Press, Oxford).

\bibitem{csl-Rodr11}
  Rodr\'{i}guez-Andrade M.A., Arag\'on-Gonz\'alez G., 
  Arag\'on J.L.\ and G\'omez-Rodr\'{i}\-guez A.\ (2011).
  Coincidence lattices in the hyperbolic plane,
  \textit{Acta Cryst.\ A} \textbf{67}, 35--44.

\bibitem{csl-Ruth1}
  Rutherford, J.S.\ (1992).
  The enumeration and symmetry-significant properties of derivative lattices,
  \textit{Acta Cryst.\ A} \textbf{48}, 500--508.

\bibitem{csl-Ruth4}
  Rutherford, J.S.\ (2009).
  Sublattice enumeration. IV. Equivalence classes of plane sublattices by 
  parent Patterson symmetry and colour lattice group type,
  \textit{Acta Cryst.\ A} \textbf{65}, 156--163.

\bibitem{csl-Sass}
  Sass L.S.\ (1985).
  Grain boundary structure. In \textit{Encyclopedia of Materials 
  Science and Engineering}, Bever M.B.\ (ed.), vol.~3,
  pp.~2041--2045 (Pergamon, Oxford).

\bibitem{csl-Schwarz}
  Schwarzenberger R.L.E.\ (1980).
  \textit{$N\nts$-Dimensional Crystallography}
  (Pitman, London).

\bibitem{csl-Serre}
  Serre J.-P.\ (1993).
  \textit{A Course in Arithmetic}, 
  4th corr.\ printing (Springer, New York).

\bibitem{csl-OEIS}
  Sloane N.J.A.\ (ed.).
  \textit{The On-Line Encyclopedia of Integer Sequences},
  available at \texttt{https://oeis.org/}

\bibitem{csl-Slo02b} 
  Sloane N.J.A.\ and Beferull-Lozano B.\ (2003).
  Quantizing using lattice intersections. In
  \textit{Discrete and Computational Geometry},
  Aronov B., Basu S., Pach J.\ and Sharir M.\
  (eds.), pp.~799--824 (Springer, Berlin).
  \texttt{arXiv:math.CO/0207147}.

\bibitem{csl-Tenenbaum}
  Tenenbaum G.\ (1995).
  \textit{Introduction to Analytic and Probabilistic Number 
    Theory} (Cambridge University Press, Cambridge).

\bibitem{csl-Vigneras}
  Vign\'{e}ras M.-F.\ (1980).
  \textit{Arithm\'etique des Alg\`ebres de Quaternions},
  LNM 800  (Springer, Berlin).

\bibitem{csl-Warrington}
  Warrington D.H.\ (1993).
  Coincidence site lattices in quasicrystal tilings,
  \textit{Mat.~Science Forum} \textbf{126--128}, 57--60.
      
\bibitem{csl-Wash}      
  Washington L.C.\ (1997).
  \textit{Introduction to Cyclotomic Fields},
  2nd ed.\ (Springer, New York).

\bibitem{csl-Weiss}
  Weiss A.\ (2000).
  On shelling icosahedral quasicrystals. In 
  \textit{Directions in Mathematical Quasicrystals}, 
  Baake M.\ and Moody R.V.\ (eds.), CRM Monograph Series,
  vol.~13, pp. 161--176
  (AMS, Providence, RI).

\bibitem{csl-Zagier}
  Zagier D.B.\ (1981). 
  \textit{Zetafunktionen und quadratische K\"{o}rper}
  (Springer, Berlin).

\bibitem{csl-Zass}
  Zassenhaus H.J.\ (1958).
  \textit{The Theory of Groups},
  2nd ed.\ (Chelsea, New York).
      
\bibitem{csl-Z2}
  Zeiner P.\ (2005).
  Symmetries of coincidence site lattices of cubic lattices,
  \textit{Z.\ Krist.} \textbf{220}, 915--920. 
  \texttt{arXiv:math.MG/0605525}.   
      
\bibitem{csl-Z3}
  Zeiner P.\ (2006).    
  Coincidences of hypercubic lattices in 4 dimensions,
  \textit{Z.\ Krist.} \textbf{221}, 105--114.
  \texttt{arXiv:math.MG/0605526}.  

\bibitem{csl-pzmcsl1} 
  Zeiner P.\ (2006).
  Multiple CSLs for the body centered cubic lattice, 
  \textit{J.\ Phys.:\ Conf.\ Ser.} \textbf{30}, 163--167.
  \texttt{arXiv:math.MG/0605521}.  

\bibitem{csl-pzcsl7}
  Zeiner P.\ (2010).
  Multiplicativity in the theory of coincidence site lattices,
  \textit{J.\ Phys.:\ Conf.\ Ser.} \textbf{226}, 
  012025:\ 1--6. \texttt{arXiv:1212.4528}.

\bibitem{csl-pzsslcsl1}
  Zeiner P.\ (2014).
  Similar submodules and coincidence site modules,
  \textit{Acta Phys.\ Pol.\ A} \textbf{126}, 641--645.
  \texttt{arXiv:1402.5013}.

\bibitem{csl-habil}
  Zeiner P.\ (2015).
  \textit{Coincidence Site Lattices and Coincidence Site Modules},
  Habilitation thesis (Bielefeld University). 

\bibitem{csl-zou1}
  Zou Y.M.\ (2006).
  Indices of coincidence isometries of the hypercubic lattice $\ZZ^n$,
  \textit{Acta Cryst.\ A} \textbf{62}, 454--458.

\bibitem{csl-zou2}
  Zou Y.M.\ (2006),
  Structures of coincidence symmetry groups,
  \textit{Acta Cryst.\ A} \textbf{62}, 109--114.

\end{thebibliography}
\end{document}